\documentclass{article}
\usepackage{cite}
\usepackage{amssymb,amsthm}
\usepackage{mathrsfs}
\usepackage[tbtags]{amsmath}
\usepackage{enumerate}
\usepackage{amsmath,amssymb,amsthm,mathrsfs,dsfont}

\usepackage[titletoc]{appendix}

\usepackage{dutchcal}
\usepackage{titlesec,hyperref}
\usepackage{color}
\usepackage{verbatim}
\usepackage{fancyhdr}
\usepackage[margin=1cm]{geometry}
\usepackage{mathrsfs}
\usepackage{boondox-cal}
\usepackage{indentfirst}

\newtheorem{theo}{Theorem}[section]
\newtheorem{lemm}[theo]{Lemma}

\newtheorem{rema}[theo]{Remark}
\numberwithin{equation}{section}

\allowdisplaybreaks

\begin{document}
	\title{Global uniform regularity and large time behavior of solutions to three dimensional compressible
MHD equations with vanishing vertical magnetic resistivity in half space}
	\author{
		$\mbox{Song Gao}^1$ \footnote{Corresponding author. Email: gaoshanbudi@sjtu.edu.cn}, \quad
        $\mbox{Jiahong Wu}^2$ \footnote{Email: jwu29@nd.edu}\quad and\quad
		$\mbox{Feng Xie}^1$ \footnote{Email: tzxief@sjtu.edu.cn}\\
		\quad
		$^1\mbox{School}$  of Mathematical Sciences, Shanghai Jiao Tong University,\\
		 Shanghai 200240, P. R. China\\
$^2\mbox{Department}$ of Mathematics, University of Notre Dame,\\
Notre Dame 46556, USA\\
	}
	\date{}
	\maketitle
\vspace{-1em}
	\begin{abstract}
\begin{sloppypar}
This paper aims to establish the global regularity and large time behavior of solutions to the three-dimensional (3D) compressible magnetohydrodynamics (MHD) equations with vanishing vertical magnetic resistivity in the upper half-space with no-slip boundary condition on velocity and perfectly conducting boundary condition on magnetic field.
By exploiting anisotropic Sobolev inequalities and elaborated estimates, we are able to achieve global-in-time uniform regularity estimates of solutions, which are independent of small vertical resistivity coefficient $\varepsilon$. These uniform global regularity estimates allow us to pass to the limit as \(\varepsilon \to 0\) and obtain the convergence to the corresponding MHD system without vertical magnetic resistivity globally in time. Moreover, the $H^{1}\left(\mathbb{R}_{+}^{3}\right)$ decay rates of solutions to the original system are also derived based on the detailed analysis on semigroup of the related linearized operator in half space and the uniform energy estimates achieved. In contrast to the decay estimates for the limit system obtained by Gao and Xie \cite{JDE}, the present decay estimate exhibits a slower rate, attributable to the absence of higher-order normal derivative estimates, which results from the occurrence of boundary layers. Finally, we combine both sets of regularity and decay results to rigorously prove an explicit time-uniform $L^2$ convergence rate of order $\varepsilon^{\frac{1}{4}}$ for the vanishing vertical magnetic resistivity limit process.
\end{sloppypar}

\end{abstract}
	
\vspace{-1em}

	\tableofcontents
	
	\section{Introduction}
	
	The magnetohydrodynamics (MHD) equations play a significant role in the study of many phenomena in geophysics, astrophysics, cosmology and engineering, which reflect the basic physics laws governing the motion of electrically conducting fluids such as plasmas, liquid metals and electrolytes.
For a general introduction on MHD theory, we refer the reader to the monographs and survey articles \cite{new10, new12, new26, new45}.
In this paper, we focus on the following three dimensional (3D) compressible MHD equations with small vertical resistivity parameter $\varepsilon$:
	\begin{equation}\label{eq1}
		\left\{\begin{array}{*{5}{ll}}
			\partial_t\rho^{\varepsilon} + \operatorname{div}\left(\rho^{\varepsilon}\mathbf{v}^{\varepsilon}\right) = 0 &{\rm in} ~~ \mathbb{R}_+^3,\\
\rho^{\varepsilon}\left(\partial_t\mathbf{v}^{\varepsilon} + \mathbf{v}^{\varepsilon} \cdot \nabla\mathbf{v}^{\varepsilon}\right) - \mu\Delta\mathbf{v}^{\varepsilon} - \left(\mu + \lambda\right)\nabla\operatorname{div}\mathbf{v}^{\varepsilon} + \nabla P = \left(\nabla \times \mathbf{B}^{\varepsilon}\right) \times \mathbf{B}^{\varepsilon} \quad & {\rm in} ~~\mathbb{R}_+^3,\\
			\partial_t\mathbf{B}^{\varepsilon} - \left(\Delta_h+\varepsilon\partial_3^2\right)\mathbf{B}^{\varepsilon} + \mathbf{v}^{\varepsilon} \cdot \nabla\mathbf{B}^{\varepsilon} + \mathbf{B}^{\varepsilon}\operatorname{div}\mathbf{v}^{\varepsilon} = \mathbf{B}^{\varepsilon} \cdot \nabla\mathbf{v}^{\varepsilon} &{\rm in} ~~ \mathbb{R}_+^3,\\
 \operatorname{div}\mathbf{B}^{\varepsilon} = 0 \quad & {\rm in} ~~ \mathbb{R}_+^3,\\
 \left(\rho^{\varepsilon}, \mathbf{v}^{\varepsilon}, \mathbf{B}^{\varepsilon}\right)\big|_{t=0} = \left(\rho_{0}, \mathbf{v}_{0}, \mathbf{B}_{0}\right) &{\rm in} ~~ \mathbb{R}_+^3,
		\end{array}\right.
	\end{equation}
where $\mathbb{R}_+^3=\left\{ (x_1,x_2,x_3) \in \mathbb{R} \times \mathbb{R} \times \mathbb{R_+} \right\}$, and the unknowns \(\rho^{\varepsilon} \in \mathbb{R}_+\), \(\mathbf{v}^{\varepsilon} \in \mathbb{R}^3\) and \(\mathbf{B}^{\varepsilon} \in \mathbb{R}^3\) represent the density, velocity and magnetic field, respectively.
The pressure \(P = P(\rho^{\varepsilon})\) is a smooth function of the density, satisfying \(P'(\rho^{\varepsilon}) > 0\) with \(P'(\tilde{\rho}) = 1\), where \(\tilde{\rho} > 0\) is a constant reference density.
The viscosity coefficients \(\mu\) and \(\lambda\) satisfy \(\mu > 0\) and \(2\mu + 3\lambda \geq 0\).
For notational convenience, we write
\begin{equation*}
	\Delta_h \overset{def}{=}  \partial_1^2+\partial_2^2, \quad \partial_h \overset{def}{=} \left(\partial_1, \partial_2\right).
\end{equation*}
The system \eqref{eq1} is supplemented with the no-slip boundary condition on velocity and the perfectly conducting boundary condition on magnetic field:
\begin{equation}\label{eq2}
	\left(\mathbf{v}^{\varepsilon} , \partial_3\mathbf{B}^{\varepsilon}_h\right)\big|_{x_3=0}=\mathbf{0}.
\end{equation}
In this paper, we consider the case that a uniform normal component of initial magnetic field vanishes
on the boundary
\[
\left(\mathbf{B}_{0}\right)_3\big|_{x_3=0}=0,
\]
and along with the equation of magnetic field in \eqref{eq1} and boundary condition \eqref{eq2}, it implies
\begin{equation}\label{eq3}
  B_3^\varepsilon \big|_{x_3=0}=0.
\end{equation}
\enlargethispage{0.5cm}
Formally, as the parameter $\varepsilon \rightarrow 0^+$, system \eqref{eq1}
tends to the following system
	\begin{equation}\label{eq4}
		\left\{\begin{array}{*{5}{ll}}
			\partial_t\rho^{0} + \operatorname{div}\left(\rho^{0}\mathbf{v}^{0}\right) = 0 &{\rm in} ~~ \mathbb{R}_+^3,\\
\rho^{0}\left(\partial_t\mathbf{v}^{0} + \mathbf{v}^{0} \cdot \nabla\mathbf{v}^{0}\right) - \mu\Delta\mathbf{v}^{0} - \left(\mu + \lambda\right)\nabla\operatorname{div}\mathbf{v}^{0} + \nabla P = \left(\nabla \times \mathbf{B}^{0}\right) \times \mathbf{B}^{0} \quad & {\rm in} ~~\mathbb{R}_+^3,\\
			\partial_t\mathbf{B}^{0} - \Delta_h\mathbf{B}^{0} + \mathbf{v}^{0} \cdot \nabla\mathbf{B}^{0} + \mathbf{B}^{0}\operatorname{div}\mathbf{v}^{0} = \mathbf{B}^{0} \cdot \nabla\mathbf{v}^{0} &{\rm in} ~~ \mathbb{R}_+^3,\\
 \operatorname{div}\mathbf{B}^{0} = 0 \quad & {\rm in} ~~ \mathbb{R}_+^3,\\
 \left(\rho^{0}, \mathbf{v}^{0}, \mathbf{B}^{0}\right)\big|_{t=0} = \left(\rho_{0}, \mathbf{v}_{0}, \mathbf{B}_{0}\right) &{\rm in} ~~ \mathbb{R}_+^3,
		\end{array}\right.
	\end{equation}
together with the no-slip boundary condition
\begin{equation}\label{eq5}
		\mathbf{v}^{0}\big|_{x_3=0}=\mathbf{0}.
\end{equation}
The main goal of this paper is to establish the global regularity estimates and large time behavior of solutions to \eqref{eq1}--\eqref{eq3}, and as direct byproducts, we rigorously prove the global convergence of solutions to \eqref{eq1}--\eqref{eq3} to solutions to \eqref{eq4}--\eqref{eq5} as the parameter $\varepsilon$ goes to zero,
and establish an explicit time-uniform $L^2$ convergence rate for the vanishing vertical magnetic resistivity limit process.

\textbf{Related literature}: We briefly review three strands of related research: the vanishing viscosity limit of the  Navier–Stokes equations, and developments on the stability and decay rates of solutions to 3D Navier–Stokes and MHD equations.

Regarding the vanishing viscosity limit for compressible flows, Wang, Xin and Yong \cite{New46} established the uniform regularity and vanishing viscosity limit of solutions to compressible Navier-Stokes equations with Navier-slip boundary conditions in bounded domains.
For compressible MHD equations, Cui, Li and Xie \cite{CLX1, CLX2} derived the uniform regularity of solutions and proved the vanishing viscosity limit for compressible MHD equations with the no-slip boundary condition on the velocity and the perfectly conducting boundary condition on the magnetic field in the half space. For the corresponding limit process of compressible viscoelastic equations, we refer to \cite{GWX, WX}. It is noted that all of these results hold locally in time. More broadly, the global-in-time vanishing viscosity limit for both incompressible and compressible Navier–Stokes equations and MHD equations remains one of the central open problems in fluid mechanics;
we refer the reader to \cite{New23, New27, New30, New33} and references therein for the extensive literature on this topic.

Matsumura and Nishida \cite{23a} proved the global well-posedness of $H^3$ solutions to compressible Navier-Stokes equations for small perturbed initial data in the half-space \(\mathbb{R}_{+}^{3}\). Subsequently, Kagei and Kobayashi \cite{k1, k2} systematically investigated the large-time asymptotic behavior of solutions to the compressible Navier-Stokes equations under the no-slip boundary condition on velocity. In the second part of this paper, we try to extend the results in \cite{k1, k2} to 3D compressible viscous  MHD system with small vertical resistivity in \(\mathbb{R}_+^3\), where the no-slip boundary condition is imposed on velocity and the magnetic field satisfies the perfectly conducting boundary condition.

It is noted that the compressible MHD equations incorporating both dissipation and magnetic diffusion have garnered significant attention in recent research, owing to their broad physical applicability and the associated mathematical challenges, see, e.g., \cite{3.,11.,12.,18.,29.,33.} and among others.

Compared with the compressible MHD system that includes both viscosity and resistivity, far fewer results are currently available for the compressible viscous MHD equations without magnetic diffusion. The absence of magnetic diffusivity poses substantial challenges for establishing global well-posedness (even for small initial data) and analyzing stability near a background magnetic field. As specific examples, Wu and Wu \cite{31.}, Wu and Zhu \cite{32.} systematically demonstrated the global existence and stability of solutions to 2D compressible MHD equations lacking magnetic diffusion with the constraint that initial data are small. Meanwhile, Dong, Wu and Zhai \cite{6.} established the global existence of strong solutions for the \(2\frac{1}{2}\)-dimensional (2.5-D) compressible non-resistive MHD equations under the condition of small initial data.
When considering the horizontally infinite flat layer \(\Omega = \mathbb{R}^2 \times (0,1)\), Tan and Wang \cite{27a}  proved the global existence of smooth solutions to compressible MHD system, see also Li \cite{19b} for extensions to the heat-conductive fluids. On the 3D torus \(\mathbb{T}^3\), Wu and Zhai \cite{31b} tackled the MHD stability problem involving velocity dissipation, focusing on regions near a background magnetic field \(\mathbf{n} \in \mathbb{R}^3\) that satisfies the Diophantine condition. Later, Li, Xu and Zhai \cite{21a} extended this result to cover compressible viscous non-isentropic MHD flows that do not incorporate magnetic diffusion. Very recently, Wu and Zhai \cite{31a} studied the Cauchy problem in \(\mathbb{R}^3\) for the compressible viscous MHD equations with only horizontal magnetic diffusion, which is exactly the limit equation  \eqref{eq4} in this paper. Subsequently,  Gao and Xie \cite{JDE} established the global well-posedness of solutions to initial boundary value problem \eqref{eq4}--\eqref{eq5}, and gave the decay rates of solutions accordingly.
For the incompressible MHD system, Gao, Wu and Yao et al. \cite{9a,GPWY} established the global-in-time uniform regularity of solutions to the 3D anisotropic incompressible MHD equations in half space with slip boundary conditions in the vicinity of a background magnetic field, by exploiting the enhanced dissipation induced by the background magnetic field.

Due to the smallness of parameter $\varepsilon$ in \eqref{eq1}
and appearance of boundary which will lead to the possible formation of boundary layer, we try to establish the uniform energy estimate
under the framework of conormal Sobolev space rather than the classical
Sobolev space as in \cite{JDE}. The primary objective of this paper is to study the global well-posedness of strong solutions to \eqref{eq1}--\eqref{eq3} in the vicinity of the equilibrium state \((\tilde{\rho}, \mathbf{0}, \mathbf{0})\). We set \(\tilde{\rho} = 1\) without loss of generality,
then the perturbation $( a^{\varepsilon} ,\mathbf{v}^{\varepsilon},\mathbf{B}^{\varepsilon})$ around this equilibrium with $a^{\varepsilon} \overset{def}{=} \rho^{\varepsilon} - 1$ satisfies
	\begin{equation}\label{eq6}
 	\left\{\begin{array}{*{5}{ll}}
 		\partial_t a^{\varepsilon} + \operatorname{div}\mathbf{v}^{\varepsilon} = f^{\varepsilon}_1
\quad & {\rm in} ~~\mathbb{R}_+^3,\\
\partial_t\mathbf{v}^{\varepsilon} - \mu\Delta\mathbf{v}^{\varepsilon} - \left(\mu + \lambda\right)\nabla\operatorname{div}\mathbf{v}^{\varepsilon} + \nabla a^{\varepsilon} = f^{\varepsilon}_2 \quad & {\rm in} ~~\mathbb{R}_+^3,\\
 		\partial_t\mathbf{B}^{\varepsilon} - \left(\Delta_h+\varepsilon\partial_3^2\right)\mathbf{B}^{\varepsilon} = f^{\varepsilon}_3
\quad & {\rm in} ~~ \mathbb{R}_+^3,\\
\operatorname{div}\mathbf{B}^{\varepsilon} = 0 \quad & {\rm in} ~~ \mathbb{R}_+^3,\\
\left(a^{\varepsilon}, \mathbf{v}^{\varepsilon}, \mathbf{B}^{\varepsilon}\right)\big|_{t=0} = \left(a_{0}, \mathbf{v}_{0}, \mathbf{B}_{0}\right) \quad & {\rm in} ~~ \mathbb{R}_+^3,\\
 		\left(\mathbf{v}^{\varepsilon},\partial_3 \mathbf{B}^{\varepsilon}_h, B_3^\varepsilon\right) = \mathbf{0}\quad & {\rm on} ~~ \mathbb{R}^2 \times \left\{ x_3=0 \right\},
 	\end{array}\right.
 \end{equation}	
where  $a_0 \overset{def}{=} \rho_0 - 1$ and  \(f^{\varepsilon}_i\) (i=1,2,3) are nonlinear terms given by
\begin{equation*}
 \left\{
\begin{aligned}
&f^{\varepsilon}_1 \overset{def}{=} -\mathbf{v}^{\varepsilon} \cdot \nabla a^{\varepsilon} - a^{\varepsilon}\operatorname{div}\mathbf{v}^{\varepsilon},\\
&f^{\varepsilon}_2 \overset{def}{=} -\mathbf{v}^{\varepsilon} \cdot \nabla\mathbf{v}^{\varepsilon} + \mathbf{B}^{\varepsilon} \cdot \nabla\mathbf{B}^{\varepsilon} - \mathbf{B}^{\varepsilon}\nabla\mathbf{B}^{\varepsilon} - J\left(a^{\varepsilon}\right)\nabla a^{\varepsilon} \\
&\quad\quad - I\left(a^{\varepsilon}\right)\left(\mu\Delta\mathbf{v}^{\varepsilon} + \left(\mu + \lambda\right)\nabla\operatorname{div}\mathbf{v}^{\varepsilon} + \mathbf{B}^{\varepsilon} \cdot \nabla\mathbf{B}^{\varepsilon} - \mathbf{B}^{\varepsilon}\nabla\mathbf{B}^{\varepsilon}\right),\\
&f^{\varepsilon}_3 \overset{def}{=} -\mathbf{v}^{\varepsilon} \cdot \nabla\mathbf{B}^{\varepsilon} - \mathbf{B}^{\varepsilon}\operatorname{div}\mathbf{v}^{\varepsilon} + \mathbf{B}^{\varepsilon} \cdot \nabla\mathbf{v}^{\varepsilon}\\
\end{aligned}
\right.
\end{equation*}
with
\[
\mathbf{B}^{\varepsilon} \nabla \mathbf{B}^{\varepsilon} \overset{def}{=}  \sum_{i=1}^3 B^{\varepsilon}_i \nabla B^{\varepsilon}_i,\quad
I\left(a^{\varepsilon}\right) \overset{def}{=} \frac{a^{\varepsilon}}{a^{\varepsilon} + 1} \quad \text{and} \quad J\left(a^{\varepsilon}\right) \overset{def}{=} \frac{P'\left(a^{\varepsilon} + 1\right)}{a^{\varepsilon} + 1} - 1.
\]
We will turn to derive the global uniform energy estimate of solutions to the reformulated system \eqref{eq6} instead of \eqref{eq1} in what follows.

Accordingly, as $\varepsilon \to 0^+$, \eqref{eq4}--\eqref{eq5} will be changed into the following system
 \begin{equation}\label{eq7}
 	\left\{\begin{array}{*{5}{ll}}
 		\partial_t a^{0} + \operatorname{div}\mathbf{v}^{0} = f^{0}_1
\quad & {\rm in} ~~\mathbb{R}_+^3,\\
\partial_t\mathbf{v}^{0} - \mu\Delta\mathbf{v}^{0} - \left(\mu + \lambda\right)\nabla\operatorname{div}\mathbf{v}^{0} + \nabla a^{0} = f^{0}_2 \quad & {\rm in} ~~\mathbb{R}_+^3,\\
 		\partial_t\mathbf{B}^{0} - \Delta_h\mathbf{B}^{0} = f^{0}_3
\quad & {\rm in} ~~ \mathbb{R}_+^3,\\
\operatorname{div}\mathbf{B}^{0} = 0 \quad & {\rm in} ~~ \mathbb{R}_+^3,\\
\left(a^{0}, \mathbf{v}^{0}, \mathbf{B}^{0}\right)\big|_{t=0} = \left(a_{0}, \mathbf{v}_{0}, \mathbf{B}_{0}\right) \quad & {\rm in} ~~ \mathbb{R}_+^3,\\
 		\left(\mathbf{v}^{0}, B_3^0\right) = \mathbf{0}\quad & {\rm on} ~~ \mathbb{R}^2 \times \left\{ x_3=0 \right\},
 	\end{array}\right.
 \end{equation}	
where \(f^{0}_i\) (i=1,2,3) are nonlinear terms given by
\begin{equation*}
 \left\{
\begin{aligned}
&f^{0}_1 \overset{def}{=} -\mathbf{v}^{0} \cdot \nabla a^{0} - a^{0}\operatorname{div}\mathbf{v}^{0},\\
&f^{0}_2 \overset{def}{=} -\mathbf{v}^{0} \cdot \nabla\mathbf{v}^{0} + \mathbf{B}^{0} \cdot \nabla\mathbf{B}^{0} - \mathbf{B}^{0}\nabla\mathbf{B}^{0} - J\left(a^{0}\right)\nabla a^{0} \\
&\quad\quad - I\left(a^{0}\right)\left(\mu\Delta\mathbf{v}^{0} + \left(\mu + \lambda\right)\nabla\operatorname{div}\mathbf{v}^{0} + \mathbf{B}^{0} \cdot \nabla\mathbf{B}^{0} - \mathbf{B}^{0}\nabla\mathbf{B}^{0}\right),\\
&f^{0}_3 \overset{def}{=} -\mathbf{v}^{0} \cdot \nabla\mathbf{B}^{0} - \mathbf{B}^{0}\operatorname{div}\mathbf{v}^{0} + \mathbf{B}^{0} \cdot \nabla\mathbf{v}^{0}\\
\end{aligned}
\right.
\end{equation*}
with
\[
\mathbf{B}^0 \nabla \mathbf{B}^0 \overset{def}{=}  \sum_{i=1}^3 B^0_i \nabla B^0_i,\quad
I\left(a^{0}\right) \overset{def}{=} \frac{a^{0}}{a^{0} + 1} \quad \text{and} \quad J\left(a^{0}\right) \overset{def}{=} \frac{P'\left(a^{0} + 1\right)}{a^{0} + 1} - 1.
\]

To further facilitate the subsequent analysis on decay rates of solutions, we also introduce the following perturbed system in momentum variable. Set
\[
\mathbf{m}^{\varepsilon}\overset{def}{=}\left(a^{\varepsilon}+1\right)\mathbf{v}^{\varepsilon},
\]
whereupon the initial boundary value problem \eqref{eq6} can be rewritten into the following equivalent form
\begin{equation}\label{eq8}
 	\left\{\begin{array}{*{5}{ll}}
 		\partial_t a^{\varepsilon} + \operatorname{div}\mathbf{m}^{\varepsilon} = 0
\quad & {\rm in} ~~\mathbb{R}_+^3,\\
\partial_t\mathbf{m}^{\varepsilon} - \mu\Delta\mathbf{m}^{\varepsilon} - \left(\mu + \lambda\right)\nabla\operatorname{div}\mathbf{m}^{\varepsilon} + \nabla a^{\varepsilon} = \operatorname{div}F^{\varepsilon}_1+\operatorname{div}F^{\varepsilon}_2 \quad & {\rm in} ~~\mathbb{R}_+^3,\\
 		\partial_t\mathbf{B}^{\varepsilon} - \left(\Delta_h+\varepsilon\partial_3^2\right)\mathbf{B}^{\varepsilon} = F^{\varepsilon}_3
\quad & {\rm in} ~~ \mathbb{R}_+^3,\\
\operatorname{div}\mathbf{B}^{\varepsilon} = 0 \quad & {\rm in} ~~ \mathbb{R}_+^3,\\
\left(a^{\varepsilon}, \mathbf{m}^{\varepsilon}, \mathbf{B}^{\varepsilon}\right) \big|_{t=0}= \left(a_{0}, \mathbf{m}_{0}, \mathbf{B}_{0}\right) \quad & {\rm in} ~~ \mathbb{R}_+^3,\\
 		\left(\mathbf{m}^{\varepsilon},\partial_3 \mathbf{B}^{\varepsilon}_h,  B_3^\varepsilon\right) = \mathbf{0}\quad & {\rm on} ~~ \mathbb{R}^2 \times \left\{ x_3=0 \right\},
 	\end{array}\right.
 \end{equation}
where $\mathbf{m}^{\varepsilon} \in \mathbb{R}^3$ represents the momentum, and the initial momentum $\mathbf{m}_{0}\overset{def}{=}\left(a_{0}+1\right)\mathbf{v}_{0}$. $F^{\varepsilon}_{i} (i=1,2)$ denote the $3 \times 3$ matrices $\left(\left(F^{\varepsilon}_i\right)_{j k}\right)_{1 \leq j, k \leq 3}$ and $F^{\varepsilon}_3\overset{def}{=} f^{\varepsilon}_3$.
Specifically, $F^{\varepsilon}_i (i=1,2,3)$ are defined by
\begin{equation*}
\left\{
\begin{aligned}
\left(F^{\varepsilon}_1\right)_{jk} &\overset{def}{=} -\delta_{jk}\left(\left(a^{\varepsilon}\right)^2 \int_{0}^{1} (1-\theta) P''\left(a^{\varepsilon}\theta + 1\right) \mathrm{d}\theta\right) - \mu \partial_{k}\left( \frac{a^{\varepsilon} m^{\varepsilon}_j}{a^{\varepsilon}+1} \right) \\
&\quad -\delta_{jk}(\mu + \lambda) \operatorname{div}\left( \frac{a^{\varepsilon} \mathbf{m}^{\varepsilon}}{a^{\varepsilon}+1} \right) - \frac{m^{\varepsilon}_j m^{\varepsilon}_k}{a^{\varepsilon}+1}, \\
\left(F^{\varepsilon}_2\right)_{jk} &\overset{def}{=} B^{\varepsilon}_j B^{\varepsilon}_k - \frac{1}{2}\delta_{jk}\left|\mathbf{B}^{\varepsilon}\right|^2, \\
F^{\varepsilon}_3 &\overset{def}{=} -\left( \frac{\mathbf{m}^{\varepsilon}}{a^{\varepsilon}+1} \right) \cdot \nabla\mathbf{B}^{\varepsilon} - \mathbf{B}^{\varepsilon} \cdot \operatorname{div}\left( \frac{\mathbf{m}^{\varepsilon}}{a^{\varepsilon}+1} \right) + \mathbf{B}^{\varepsilon} \cdot \nabla\left( \frac{\mathbf{m}^{\varepsilon}}{a^{\varepsilon}+1} \right),
\end{aligned}
\right.
\end{equation*}
where the $j$-th component of $\operatorname{div}F^{\varepsilon}_{i} (i=1,2)$ is given by $\sum_{k=1}^{3} \partial_{k} \left(F^{\varepsilon}_i\right)_{j k}$.

Accordingly, as $\varepsilon \to 0^+$,  we set
\[
\mathbf{m}^{0}\overset{def}{=}\left(a^{0}+1\right)\mathbf{v}^{0},
\]
whereupon the initial boundary value problem \eqref{eq7} can be rewritten into the following equivalent form
\begin{equation}\label{eq08}
 	\left\{\begin{array}{*{5}{ll}}
 		\partial_t a^{0} + \operatorname{div}\mathbf{m}^{0} = 0
\quad & {\rm in} ~~\mathbb{R}_+^3,\\
\partial_t\mathbf{m}^{0} - \mu\Delta\mathbf{m}^{0} - \left(\mu + \lambda\right)\nabla\operatorname{div}\mathbf{m}^{0} + \nabla a^{0} = \operatorname{div}F^{0}_1+\operatorname{div}F^{0}_2 \quad & {\rm in} ~~\mathbb{R}_+^3,\\
 		\partial_t\mathbf{B}^{0} - \Delta_h\mathbf{B}^{0} = F^{0}_3
\quad & {\rm in} ~~ \mathbb{R}_+^3,\\
\operatorname{div}\mathbf{B}^{0} = 0 \quad & {\rm in} ~~ \mathbb{R}_+^3,\\
\left(a^{0}, \mathbf{m}^{0}, \mathbf{B}^{0}\right) \big|_{t=0}= \left(a_{0}, \mathbf{m}_{0}, \mathbf{B}_{0}\right) \quad & {\rm in} ~~ \mathbb{R}_+^3,\\
 		\left(\mathbf{m}^{0},  B_3^{0}\right) = \mathbf{0}\quad & {\rm on} ~~ \mathbb{R}^2 \times \left\{ x_3=0 \right\},
 	\end{array}\right.
 \end{equation}
where $\mathbf{m}^{0} \in \mathbb{R}^3$ represents the momentum, and the initial momentum $\mathbf{m}_{0}=\left(a_{0}+1\right)\mathbf{v}_{0}$. $F^{0}_{i} (i=1,2)$ denote the $3 \times 3$ matrices $\left(\left(F^{0}_i\right)_{j k}\right)_{1 \leq j, k \leq 3}$ and $F^{0}_3\overset{def}{=} f^{0}_3$.
Specifically, $F^{0}_i (i=1,2,3)$ are defined by
\begin{equation*}
\left\{
\begin{aligned}
\left(F^{0}_1\right)_{jk} &\overset{def}{=} -\delta_{jk}\left(\left(a^{0}\right)^2 \int_{0}^{1} (1-\theta) P''\left(a^{0}\theta + 1\right) \mathrm{d}\theta\right) - \mu \partial_{k}\left( \frac{a^{0} m^{0}_j}{a^{0}+1} \right) \\
&\quad -\delta_{jk}(\mu + \lambda) \operatorname{div}\left( \frac{a^{0} \mathbf{m}^{0}}{a^{0}+1} \right) - \frac{m^{0}_j m^{0}_k}{a^{0}+1}, \\
\left(F^{0}_2\right)_{jk} &\overset{def}{=} B^{0}_j B^{0}_k - \frac{1}{2}\delta_{jk}\left|\mathbf{B}^{0}\right|^2, \\
F^{0}_3 &\overset{def}{=} -\left( \frac{\mathbf{m}^{0}}{a^{0}+1} \right) \cdot \nabla\mathbf{B}^{0} - \mathbf{B}^{0} \cdot \operatorname{div}\left( \frac{\mathbf{m}^{0}}{a^{0}+1} \right) + \mathbf{B}^{0} \cdot \nabla\left( \frac{\mathbf{m}^{0}}{a^{0}+1} \right),
\end{aligned}
\right.
\end{equation*}
where the $j$-th component of $\operatorname{div}F^{0}_{i} (i=1,2)$ is given by $\sum_{k=1}^{3} \partial_{k} \left(F^{0}_i\right)_{j k}$.

\textbf{Notations}: For the convenience of notations, we take $L^p\left(\mathbb{R}_+^3\right)$, $p\geq1$ and $H^k\left(\mathbb{R}_+^3\right)$, $k\geq0$ for the usual $L^p$ and
Sobolev spaces on $\mathbb{R}_+^3$ with norms $\|\cdot\|_{L^p}$ and $\|\cdot\|_{H^k}$, respectively.  $(\cdot, \cdot)$ denotes the inner product in $L^2$.
The notation $A\lesssim B$ stands for $A\leq C B$ for some generic constant $C>0$.
$\nabla$ and $\partial$ denote the standard gradient, while $\nabla_h$ and $\partial_h$ stand for the horizontal gradient, $\operatorname{div}_h$
is the horizontal divergence, and $\Delta_h$ is the horizontal Laplace operator. For vector $\mathbf{v}=(v_1, v_2, v_3)$, $\mathbf{v}_h=(v_1, v_2)$ stands for the horizontal components.
$[q]$ denotes the greatest integer less than or equal to q.
Finally, we use \(\|f\|_{L_{x_{3}}^{r} L_{x_{2}}^{q} L_{x_{1}}^{p}}\) for the norm \(\left\| \left\| \left\| f \right\|_{L_{x_{1}}^p} \right\|_{L_{x_{2}}^q} \right\|_{L_{x_{3}}^r}\) and \(\|f\|_{L_{x_{3}}^{q} L_{x_{1} x_{2}}^{p}}\) for \(\left\| \left\| f \right\|_{L^{p}_{x_1x_2}} \right\|_{L^{q}_{x_3}}\).

To give a precise account of our main results, we define
 the conormal Sobolev spaces. Let $(Z_k)_{1 \le k \le 3}$ be a finite set of generators of vector fields tangent to the boundary $\mathbb{R}^2 \times \left\{x_3=0\right\}$,
 with $Z_k = \partial_k$, $k=1,2$ and $Z_3= \varphi (x_3) \partial_3$, where $\varphi(x_3)$ is any smooth bounded function such that $\varphi(0)=0$, $\varphi'(0) \neq 0$ and $\varphi(x_3)>0$ for every $x_3>0$. In this paper, we take \(\varphi(x_3) \overset{def}{=} \frac{x_3}{x_3+1}\). Define
	\begin{equation*}
	\begin{aligned}
H_{co}^m\left(\mathbb{R}_+^3\right)\overset{def}{=} \left\{ f \in
		L^2\left(\mathbb{R}_+^3\right)~\big|~Z^\alpha f \in L^2\left(\mathbb{R}_+^3\right),~~\forall 0\le |\alpha| \le m \right\},
	\end{aligned}
\end{equation*}
	where $Z^\alpha \overset{def}{=} Z^{\alpha_1}_1 Z^{\alpha_2}_2 Z^{\alpha_3}_3, \alpha=(\alpha_1, \alpha_2, \alpha_3)$ with $|\alpha|=\sum_{i=1}^3|\alpha_i|$.
	We also use the following notation, for every $m \in \mathbb{N}$:
	\begin{equation*}
		 \| f \|_{m}^2\overset{def}{=} \sum_{0\le|\alpha| \le m} \| Z^\alpha f \|_{L^2}^2.
\end{equation*}
Correspondingly, we define the local conormal Sobolev space $H_{co, loc}^{m-1}\left(\mathbb{R}_+^3\right)$ by extending the above global conormal Sobolev space to the local setting, which is essential for describing the strong convergence behavior to be discussed later. Specifically, we set
\begin{equation*}
H_{co, loc}^{m}\left(\mathbb{R}_+^3\right)\overset{\mathrm{def}}{=} \left\{ f \in L^2_{\mathrm{loc}}\left(\mathbb{R}_+^3\right) \big| Z^\alpha f \in L^2_{\mathrm{loc}}\left(\mathbb{R}_+^3\right),~~\forall 0\le |\alpha| \le m \right\},
\end{equation*}
where $L^2_{\mathrm{loc}}\left(\mathbb{R}_+^3\right)$ denotes the space of locally square-integrable functions on $\mathbb{R}_+^3$.

For every  integer  $m\ge 2$, we introduce the following energy functional
\begin{align*}
   E_{m}^{\varepsilon}(t)\overset{def}{=} & \sup_{0\leq \tau \leq t}\left\| \left(a^{\varepsilon}, \mathbf{v}^{\varepsilon}, \mathbf{B}^{\varepsilon}\right)(\tau)\right\|_{m} +\sup_{0\leq \tau \leq t} \left\| \partial_3\left(a^{\varepsilon}, \mathbf{v}^{\varepsilon}, \mathbf{B}^{\varepsilon}\right)(\tau)\right\|_{m-1} \\
     &+\sup_{0\leq \tau \leq t} \left\| \partial_3^2\left(\mathbf{v}^{\varepsilon},\varepsilon\mathbf{B}^{\varepsilon}\right)(\tau)\right\|_{m-2}+\sup_{0\leq \tau \leq t}\left\| \partial_{\tau}\left(a^{\varepsilon}, \mathbf{v}^{\varepsilon}, \mathbf{B}^{\varepsilon}\right)(\tau)\right\|_{m-2},
\end{align*}
and the dissipation functional as
\begin{align*}
D_{m}^{\varepsilon}(t)\overset{def}{=}{}& \Bigg( \int_{0}^{t}\bigg(\left\|\partial a^{\varepsilon} \right\|_{m-1}^2+\left\|(\partial\mathbf{v}^{\varepsilon}, \partial_h\mathbf{B}^{\varepsilon})\right\|_{m}^2
        +\left\|\partial_3\left(\partial\mathbf{v}^{\varepsilon}, \partial_h\mathbf{B}^{\varepsilon}\right)\right\|_{m-1}^2+\left\| \partial_{\tau}\left(a^{\varepsilon}, \mathbf{v}^{\varepsilon}, \mathbf{B}^{\varepsilon}\right)\right\|_{m-1}^2 \\
&+\left\| \partial_3\partial_{\tau}\mathbf{v}^{\varepsilon}\right\|_{m-2}^2+\varepsilon \left( \left\| \partial_3 \mathbf{B}^{\varepsilon}\right\|_{m}^2
           +\left\| \partial_3^2 \mathbf{B}^{\varepsilon}\right\|_{m-1}^2+\left\| \partial_3\partial_{\tau}\mathbf{B}^{\varepsilon}\right\|_{m-2}^2 \right)\bigg) \mathrm{d}\tau\Bigg)^{\frac{1}{2}}.
\end{align*}
Now, we state our first result about the global-in-time
uniform regularity estimate as follows.
\begin{theo}[Global well-posedness and uniform estimate of original system]\label{Th1}
	For every integer $m \ge 4$, assume the initial data $\left( a_{0} ,\mathbf{v}_{0},\mathbf{B}_{0}\right)\in H^m_{co}\left(\mathbb{R}_+^3\right)$,
$\partial_3 \left(a_{0}, \mathbf{v}_{0}, \mathbf{B}_{0}\right)\in H^ {m-1}_{co}\left(\mathbb{R}_+^3\right)$,
$\partial_3^2 \left(\mathbf{v}_{0}, \varepsilon\mathbf{B}_{0}\right)\in H^ {m-2}_{co}\left(\mathbb{R}_+^3\right)$
and satisfy the appropriate compatibility conditions. Then, there exists a small constant $\delta_0>0$ such that, if
\begin{equation}\label{eq9}
	\begin{aligned}
			\left\| \left( a_{0} ,\mathbf{v}_{0},\mathbf{B}_{0}\right) \right\|_{m}
			+\left\| \partial_3 \left(a_{0}, \mathbf{v}_{0}, \mathbf{B}_{0}\right) \right\|_{m-1}
+\left\| \partial_3^2\left(\mathbf{v}_{0}, \varepsilon\mathbf{B}_{0}\right) \right\|_{m-2}\le \delta_0,
		\end{aligned}
\end{equation}
the	system \eqref{eq6} has a unique global
        solution $\left(a^{\varepsilon}, \mathbf{v}^{\varepsilon}, \mathbf{B}^{\varepsilon}\right)$ satisfying
        \begin{equation}\label{eq10}
		E_m^{\varepsilon}(t)+D_m^{\varepsilon}(t)\lesssim \delta_0
       \end{equation}
    for all $t>0$.
    Furthermore, let $\left(a^{0}, \mathbf{v}^{0}, \mathbf{B}^{0}\right)$ be the global solution of limit system \eqref{eq7}
    supplemented
    with the same initial data $\left( a_{0} ,\mathbf{v}_{0},\mathbf{B}_{0}\right)$ as system \eqref{eq6}. Then, it holds
    \begin{equation*}
    \left(a^{\varepsilon}, \mathbf{v}^{\varepsilon}, \mathbf{B}^{\varepsilon}\right) \rightarrow \left(a^{0}, \mathbf{v}^{0}, \mathbf{B}^{0}\right) \text{~strongly~in~}
    L^\infty\left(0, t; H^{m-1}_{co, loc}\left(\mathbb{R}_+^3\right)\right)
    \end{equation*}
    for all $t>0$.
	\end{theo}
\begin{rema}
The appropriate compatibility conditions mentioned here and in what follows are analogous to those adopted in \cite{k2} and \cite{23a}. We refer the reader to these references for their detailed formulations.
\end{rema}
\begin{rema}
Since \(m \geq 4\), we can choose \(\delta_{0}>0\) small enough in \eqref{eq9} to guarantee that \(\left\|a^{\varepsilon}(t)\right\|_{L^{\infty}}\leq C \left\| a^{\varepsilon} \right\|_{2}^{\frac{1}{2}}\left\| \partial_3 a^{\varepsilon} \right\|_{2}^{\frac{1}{2}}  \leq \frac{1}{2}\) for all \(t \geq 0\),
where $C$ is a positive generic constant, by  Sobolev's inequality \eqref{a5} in Lemma \ref{Le1}.
Consequently, the density $\rho^{\varepsilon}$ in \eqref{eq1} has both uniform positive lower and upper bounds independent of $\varepsilon$.
\end{rema}
\begin{rema}
Once the uniform a priori energy estimate \eqref{eq10} holds true, it implies that no strong boundary layer appears in the process of vanishing vertical magnetic resistivity coefficient.
\end{rema}	
Furthermore, by incorporating the results on the decay rates of solutions to the related linearized compressible MHD equations and utilizing Duhamel's principle, we can derive the decay rates of solutions to the nonlinear problem \eqref{eq1}--\eqref{eq3}.
 Among these estimates, the $L^2$ decay rates of $(a^\varepsilon,\mathbf{m}^\varepsilon)$ are optimal. The precise statement of these results is presented in the following theorem.
\begin{theo}[Decay rates  of original system]\label{Th3}
For every integer $m \ge 6$, assume the initial data  $\left( a_{0} ,\mathbf{v}_{0},\mathbf{B}_{0}\right)\in H^m_{co}\left(\mathbb{R}_+^3\right)$,
$\partial_3 \left(a_{0}, \mathbf{v}_{0}, \mathbf{B}_{0}\right)\in H^ {m-1}_{co}\left(\mathbb{R}_+^3\right)$,
$\partial_3^2 \left(\mathbf{v}_{0}, \varepsilon\mathbf{B}_{0}\right)\in H^ {m-2}_{co}\left(\mathbb{R}_+^3\right)$
and satisfy the appropriate compatibility conditions.
    \begin{itemize}
        \item[(i)] There exists a positive constant $\delta_1$ such that if the initial perturbation $\left(a_{0}, \mathbf{m}_{0}\right) \in  L^1\left(\mathbb{R}_+^3\right)$,
        $\left(\mathbf{B}_{0},\partial_3 \mathbf{B}_{0}\right)$ $ \in L_{x_1 x_2}^1L_{x_3}^2\left(\mathbb{R}_+^3\right)$ and satisfy
    \begin{equation}\label{eq11}
      \begin{aligned}
      &\left\| \left( a_{0} ,\mathbf{v}_{0},\mathbf{B}_{0}\right) \right\|_{m}
			+\left\| \partial_3 \left(a_{0}, \mathbf{v}_{0}, \mathbf{B}_{0}\right) \right\|_{m-1}
+\left\| \partial_3^2\left(\mathbf{v}_{0}, \varepsilon\mathbf{B}_{0}\right) \right\|_{m-2}\\
&+ \left\| \left(a_{0}, \mathbf{m}_{0}\right) \right\|_{L^1}+\left\|\left(\mathbf{B}_{0},\partial_3 \mathbf{B}_{0}\right)\right\|_{L_{x_1 x_2}^1L_{x_3}^2} \le \delta_1,
    \end{aligned}
    \end{equation}
    then the solutions $\mathbf{u}^{\varepsilon}(t) \overset{def}{=} \left(a^{\varepsilon}, \mathbf{m}^{\varepsilon}\right)(t)$ and $\mathbf{B}^{\varepsilon}(t)$ to \eqref{eq8} satisfy
    \[
    \left\|\mathbf{u}^{\varepsilon}(t)\right\|_{H^1} = O\left(t^{-\frac{3}{4}}\right), \quad \left\|\mathbf{B}^{\varepsilon}(t)\right\|_{L^2} = O\left(t^{-\frac{1}{2}}\right),
    \]
    and
    \[
    \left\|\partial_h\mathbf{B}^{\varepsilon}(t) \right\|_{L^2} = O\left(t^{-1}\right), \quad \left\|\partial_3 \mathbf{B}^{\varepsilon}(t)\right\|_{L^2} = O\left(t^{-\frac{1}{2}+\delta_2}\right),
    \]
    as $t \to \infty$, where $\delta_2 \in \left(0, \frac{1}{32}\right)$ is an arbitrary constant.

        \item[(ii)] Also, in addition to the same assumptions on $(a_0, \mathbf{m}_0, \mathbf{B}_0)$ as in $(i)$, if $\int_{\mathbb{R}_+^3} a_0 \mathrm{d} \mathbf{x} \neq 0$, then
        \[
        \| \mathbf{u}^\varepsilon(t) \|_{L^2} \gtrsim  t^{-\frac{3}{4}},
        \]
        as $t \to \infty$.
    \end{itemize}
\end{theo}
\begin{rema}
Since $m\geq 6$, through a direct calculation, it can be deduced that $\left\|\left(a_{0}, \mathbf{m}_{0}\right)\right\|_{H^{1} \times L^2}\lesssim \delta_1$ from \eqref{eq11}. Consequently, it is possible for us to use system \eqref{eq8} to derive the decay rates of solutions to \eqref{eq1}--\eqref{eq3} instead of system \eqref{eq6}.
\end{rema}
\begin{rema}
  For the limit system \eqref{eq4}--\eqref{eq5}, applying the method in this paper yields results similar to Theorem \ref{Th1} and \ref{Th3}, which provides an alternative approach to prove the main results in \cite{JDE}.
\end{rema}

For the convenience of comparing the results of the original system and the limit system and the completeness of this paper, we list below the global well-posedness and decay rate results for the limit system established in \cite{JDE}.
\enlargethispage{0.5cm}
We define the following energy functionals for the limit system:
\begin{align*}
   E_{s}^{0}\left(t\right)\overset{def}{=} & \sup_{0\leq \tau \leq t}\left\| \left(a^0, \mathbf{v}^0, \mathbf{B}^0\right)\left(\tau\right)\right\|_{H^s} +\sum_{j=1}^{\left[\frac{s+1}{2}\right]}\sup_{0\leq \tau \leq t} \left\| \partial_{\tau}^{j} a^0\left(\tau\right)\right\|_{H^{s+1-2j}} \\
     &+ \sum_{j=1}^{\left[\frac{s}{2}\right]}\sup_{0\leq \tau \leq t} \left\| \partial_{\tau}^{j} \left(\mathbf{v}^0, \mathbf{B}^0\right)\left(\tau\right)\right\|_{H^{s-2j}},
\end{align*}
and
\begin{align*}
D_{s}^{0}(t)\overset{def}{=} \Bigg( &\int_{0}^{t}\left(\left\| \partial a^0 (\tau)\right\|_{H^{s-1}}^{2}+\left\| (\partial\mathbf{v}^0, \partial_h\mathbf{B}^0)(\tau)\right\|_{H^s}^{2}\right) \mathrm{d} \tau \\
&+\sum_{j=1}^{\left[\frac{s+1}{2}\right]} \int_{0}^{t}\left\| \partial_{\tau}^{j} (a^0, \mathbf{v}^0, \mathbf{B}^0)(\tau)\right\|_{H^{s+1-2j}}^{2} \mathrm{d} \tau \Bigg)^{\frac{1}{2}}.
\end{align*}
With these energy functionals, we state the global well-posedness and uniform estimate result for the limit system \eqref{eq7}.

\begin{theo}[Global well-posedness and uniform estimate of limit system]\label{Th01}
For every integer $s \ge 2$, assume the initial data $\left( a _0,\mathbf{v}_0,\mathbf{B}_0\right)\in H^s\left(\mathbb{R}_+^3\right)$ and satisfy the appropriate compatibility conditions. Then, there exists a positive number \(\eta_{0}\) such that if \(\left\| \left( a _0,\mathbf{v}_0,\mathbf{B}_0\right) \right\|_{H^s} \leq \eta_{0}\), then problem \eqref{eq7} admits a unique global solution \(\left( a^0 ,\mathbf{v}^0,\mathbf{B}^0\right)\) in the class
\[
a^0 \in C\left([0, \infty) ; H^{s}\right), \quad \partial_{t}^{j} a^0 \in C\left([0, \infty) ; H^{s+1-2 j}\right)\left(1 \leq j \leq\left[\frac{s+1}{2}\right]\right),
\]
\[
\partial_{t}^{j}\left(\mathbf{v}^0, \mathbf{B}^0\right) \in C\left([0, \infty) ; H^{s-2 j}\right) \quad\left(0 \leq j \leq\left[\frac{s}{2}\right]\right),
\]
\[
\partial a^0 \in L^{2}\left(0, \infty ; H^{s-1}\right), \quad \left(\partial\mathbf{v}^0, \partial_h\mathbf{B}^0\right) \in L^{2}\left(0, \infty ; H^{s}\right),
\]
\[
\partial_{t}^{j} \left(a^0, \mathbf{v}^0, \mathbf{B}^0\right) \in L^{2}\left(0, \infty ; H^{s+1-2 j}\right) \left(1 \leq j \leq \left[ \frac{s+1}{2}\right] \right).
\]
Moreover, the following inequality
\begin{equation}\label{0blm}
  E_{s}^{0}\left(t\right)+D_{s}^{0}\left(t\right) \lesssim \eta_0
\end{equation}
holds for all \(t \geq 0\).
\end{theo}

Based on the global well-posedness result above, we now present the large-time decay rate estimates for the solutions of the limit system \eqref{eq08}.

\begin{theo}[Decay rates of limit system]\label{Th03}
For every integer $s \ge 3$,
assume the initial data $\left(a_0, \mathbf{v}_0, \mathbf{B}_0\right) \in H^s\left(\mathbb{R}_+^3\right)$ and satisfy the appropriate compatibility conditions.
    \begin{itemize}
        \item[(i)] There exists a positive constant $\eta_1$ such that if the initial perturbation $\left(a_0, \mathbf{m}_0\right) \in  L^1\left(\mathbb{R}_+^3\right)$,
$\left(\mathbf{B}_0,\partial_3 \mathbf{B}_0\right) \in L_{x_1 x_2}^1 L_{x_3}^2 \left(\mathbb{R}_+^3\right)$ and
    \[
    \left\| \left(a_0, \mathbf{v}_0, \mathbf{B}_0\right) \right\|_{H^s} + \left\| \left(a_0, \mathbf{m}_0\right) \right\|_{L^1}+\left\| \left(\mathbf{B}_0,\partial_3 \mathbf{B}_0\right) \right\|_{L_{x_1 x_2}^1 L_{x_3}^2} \le \eta_1,
    \]
    then the solutions $\mathbf{u}^0\left(t\right) \overset{def}{=} \left(a^0, \mathbf{m}^0\right)\left(t\right)$ and $\mathbf{B}^0\left(t\right)$ to \eqref{eq08} satisfy
    \[
    \left\| \mathbf{u}^0\left(t\right) \right\|_{L^2} = O\left(t^{-\frac{3}{4}}\right), \quad \left\| \mathbf{B}^0\left(t\right) \right\|_{L^2} = O\left(t^{-\frac{1}{2}}\right),
    \]
    and
    \[
    \left\| \left(\partial\mathbf{u}^0, \partial_h\mathbf{B}^0\right)\left(t\right) \right\|_{L^2} = O\left(t^{-1}\right), \quad \left\| \partial_3 \mathbf{B}^0\left(t\right) \right\|_{L^2} = O\left(t^{-\frac{1}{2}}\right),
    \]
    as $t \to \infty$.

        \item[(ii)] Also, in addition to the same assumptions on $\left(a_0, \mathbf{m}_0, \mathbf{B}_0\right)$ as in $(i)$, if $\int_{\mathbb{R}_+^3} a_0 \mathrm{d} x \neq 0$, then
        \[
        \left\| \mathbf{u}^0\left(t\right) \right\|_{L^2} \gtrsim  t^{-\frac{3}{4}},
        \]
        as $t \to \infty$.
    \end{itemize}
\end{theo}

\begin{rema}
  For the limit system \eqref{eq4}--\eqref{eq5}, when $\mathbf{B}^0 = 0$, \eqref{eq08} is reduced to the isentropic compressible Navier-Stokes equations.
  From the results in \cite{k2}, the solution $\mathbf{u}^0\stackrel{\mathrm{def}}{=}\left(a^0,\mathbf{m}^0\right)$ has the following $H^1$ decay rates:
  \[
  \left\|\mathbf{u}^0(t)\right\|_{L^2} = O\left(t^{-\frac{3}{4}}\right), \quad
  \left\| \partial\mathbf{u}^0(t) \right\|_{L^2} = O\left(t^{-\frac{5}{4}}\right),
  \]
  when $\left\| \left(a_0, \mathbf{v}_0\right) \right\|_{H^s} + \left\| \left(a_0, \mathbf{m}_0\right) \right\|_{L^1}$ is sufficiently small with $s\geq 3$.

  From the results in \cite{JDE}, when
  \[
  \left\| \left(a_0, \mathbf{v}_0, \mathbf{B}_0\right) \right\|_{H^s} + \left\| \left(a_0, \mathbf{m}_0\right) \right\|_{L^1}+\left\|\left(\mathbf{B}_0,\partial_3 \mathbf{B}_0\right)\right\|_{L_{x_1 x_2}^1L_{x_3}^2}
  \]
  is sufficiently small with $s\geq 3$, the solution $\left(\mathbf{u}^0,\mathbf{B}^0\right)\stackrel{\mathrm{def}}{=}\left(a^0,\mathbf{m}^0,\mathbf{B}^0\right)$ has the following $H^1$ decay rates:
  \[
  \left\|\mathbf{u}^0(t)\right\|_{L^2} = O\left(t^{-\frac{3}{4}}\right), \quad \left\|\mathbf{B}^0(t)\right\|_{L^2} = O\left(t^{-\frac{1}{2}}\right),\]
\[  \left\| \left(\partial\mathbf{u}^0, \partial_h\mathbf{B}^0\right)(t) \right\|_{L^2} = O\left(t^{-1}\right), \quad \left\|\partial_3 \mathbf{B}^0(t)\right\|_{L^2} = O\left(t^{-\frac{1}{2}}\right).
  \]
  Compared with the results of the compressible Navier-Stokes equations, the decay rate of $\left\|\partial\mathbf{u}^0(t) \right\|_{L^2}$ in compressible MHD equations \eqref{eq4}--\eqref{eq5} is reduced from $O\left(t^{-\frac{5}{4}}\right)$ to $O\left(t^{-1}\right)$.
  This is because the limit equations \eqref{eq4}--\eqref{eq5} only have horizontal magnetic diffusion, leading to a low decay rate of $\left\|\mathbf{B}^{0}\right\|_{H^1}$.
  Moreover, the nonlinear terms of the momentum equation in \eqref{eq4}--\eqref{eq5} contain the Lorentz force $\left(\nabla \times \mathbf{B}^{0}\right) \times \mathbf{B}^{0}$, which further slows down the decay rate of $\left\|\partial\mathbf{u}^0(t) \right\|_{L^2}$ in \eqref{eq4}--\eqref{eq5}.

In contrast with the results in \cite{JDE}, we find that the decay rate of $ \left\| \left(\partial\mathbf{u}^{\varepsilon}, \partial_3 \mathbf{B}^{\varepsilon}\right)(t) \right\|_{L^2} $ is further reduced from  $O\left(t^{-1}, t^{-\frac12}\right)$ to $O\left(t^{-\frac34}, t^{-\frac12+\delta_2}\right)$. This is caused by the lower regularity in normal derivatives of solutions of the uniform energy estimate \eqref{eq10} obtained in this paper due to the appearance of weak boundary layers.
\end{rema}
By Theorems \ref{Th1} and \ref{Th01}, the local-in-time solutions of systems \eqref{eq1} and \eqref{eq4} have been extended to global-in-time solutions $\left(a^{\varepsilon}, \mathbf{v}^{\varepsilon}, \mathbf{B}^{\varepsilon}\right)$ and $\left(a^{0}, \mathbf{v}^{0}, \mathbf{B}^{0}\right)$, respectively. Although Theorem \ref{Th1} has established the strong convergence of $\left(a^{\varepsilon}, \mathbf{v}^{\varepsilon}, \mathbf{B}^{\varepsilon}\right)$ to $\left(a^{0}, \mathbf{v}^{0}, \mathbf{B}^{0}\right)$ in the local conormal Sobolev space, it does not provide an explicit convergence rate. Therefore, our final task is to derive an explicit, time-uniform convergence rate for the solutions between the original system \eqref{eq1} and the limit system \eqref{eq4}. This rate not only rigorously justifies the vanishing vertical magnetic resistivity limit, but also holds uniformly for all $t \geq 0$.
\begin{theo}[Convergence rate of vanishing vertical magnetic resistivity limit]\label{Th9}
For any integers $m \ge 6$ and $s \ge 3$,
assume the initial data $\left( a _0,\mathbf{v}_0,\mathbf{B}_0\right)\in H^m_{co}\left(\mathbb{R}_+^3\right) \cap H^s\left(\mathbb{R}_+^3\right)$,
$\partial_3 \left(a_{0}, \mathbf{v}_{0}, \mathbf{B}_{0}\right)\in H^ {m-1}_{co}\left(\mathbb{R}_+^3\right)$,
$\partial_3^2 \left(\mathbf{v}_{0}, \varepsilon\mathbf{B}_{0}\right)\in H^ {m-2}_{co}\left(\mathbb{R}_+^3\right)$,
$\left(a_0, \mathbf{m}_0\right) \in  L^1\left(\mathbb{R}_+^3\right)$,
$\left(\mathbf{B}_0,\partial_3 \mathbf{B}_0\right) \in L_{x_1 x_2}^1 L_{x_3}^2 \left(\mathbb{R}_+^3\right)$
and satisfy the appropriate compatibility conditions, and there exists a small positive constant $\delta_*$ such that
\begin{equation}\label{condition-convergence}
\begin{aligned}
&\left\| \left( a_{0} ,\mathbf{v}_{0},\mathbf{B}_{0}\right) \right\|_{m}
+\left\| \partial_3 \left(a_{0}, \mathbf{v}_{0}, \mathbf{B}_{0}\right) \right\|_{m-1}
+\left\| \partial_3^2\left(\mathbf{v}_{0}, \varepsilon\mathbf{B}_{0}\right) \right\|_{m-2}\\
&+\left\| (a_0, \mathbf{v}_0, \mathbf{B}_0) \right\|_{H^s} + \left\| \left(a_{0}, \mathbf{m}_{0}\right) \right\|_{L^1}
+\left\|\left(\mathbf{B}_{0},\partial_3 \mathbf{B}_{0}\right)\right\|_{L_{x_1 x_2}^1L_{x_3}^2}\le \delta_*.
\end{aligned}
\end{equation}
Then, the global solution $\left(\rho^{\varepsilon}, \mathbf{v}^{\varepsilon}, \mathbf{B}^{\varepsilon}\right)$ of system \eqref{eq1}
will converge to the global solution $\left(\rho^{0}, \mathbf{v}^{0}, \mathbf{B}^{0}\right)$ of system \eqref{eq4} with the convergence rate
\begin{equation}\label{conv-L2}
\sup_{0\leq \tau \leq t}\left\|\left(\rho^{\varepsilon}-\rho^{0}, \mathbf{v}^{\varepsilon}-\mathbf{v}^{0}, \mathbf{B}^{\varepsilon}-\mathbf{B}^{0}\right)(\tau)\right\|_{L^2\left(\mathbb{R}_+^3\right)}
\le C \varepsilon^{\frac{1}{4}}
\end{equation}
for any time $t>0$, where $C$ is a positive constant independent of $\varepsilon$ and $t$.
\end{theo}

\begin{rema}
By virtue of the uniform estimates in Theorems \ref{Th1} and \ref{Th01}, the $L^2$ convergence estimate \eqref{conv-L2} and the anisotropic Sobolev inequality \eqref{a5}, there holds
\[
\sup_{0\leq \tau \leq t}\left\|\left(\rho^{\varepsilon}-\rho^{0}, \mathbf{v}^{\varepsilon}-\mathbf{v}^{0}, \mathbf{B}^{\varepsilon}-\mathbf{B}^{0}\right)(\tau)\right\|_{L^{\infty}\left(\mathbb{R}_+^3\right)}
\le C \varepsilon^{\frac{1}{32}}
\]
for any $t>0$, where $C>0$ is a constant independent of $\varepsilon$ and $t$.
\end{rema}

In what follows, we will explain the main difficulties in proving Theorems \ref{Th1}, \ref{Th3} and \ref{Th9} as well as our strategies for overcoming them accordingly.
	
\textbf{Step 1: Analysis on global uniform regularity estimate independent of $\varepsilon$.}

First of all, compared with the global-in-time estimate for the limit system \eqref{eq4} in \cite{JDE},
the uniform estimate \eqref{eq10} only includes one order of vertical derivatives of $\left(a^{\varepsilon},\mathbf{B}^{\varepsilon}\right)$ that are independent of $\varepsilon$.
Therefore, since the Lorentz force $\left(\nabla \times \mathbf{B}^{\varepsilon}\right) \times \mathbf{B}^{\varepsilon}$ is a nonlinear term in the momentum equation, applying a similar argument as that in \cite{k2,23a} only yields estimates for the normal derivatives of velocity $\mathbf{v}^{\varepsilon}$ up to the second order.

Secondly, the vanishing vertical magnetic resistivity makes the estimates of nonlinear terms involving the magnetic field $\mathbf{B}^{\varepsilon}$ more intricate. To solve this issue, we apply various anisotropic inequalities to the nonlinear terms involving the magnetic field $\mathbf{B}^{\varepsilon}$, and ultimately achieve an optimal distribution of derivatives to close the uniform energy estimates.

Finally, due to the limited order of normal derivatives in the estimates available for $\left(a^{\varepsilon}, \mathbf{v}^{\varepsilon}, \mathbf{B}^{\varepsilon}\right)$ caused by the weak boundary layers, in order to close the whole energy estimate, we choose the above conormal Sobolev spaces as working space rather than the classical Sobolev spaces.
For estimating some nonlinear terms appearing therein, we can combine the given boundary condition and use techniques including the properties of commutators and  Hardy's inequality.

Therefore, we can use bootstrapping argument and the global uniform a priori energy estimates to finish the proof of Theorem \ref{Th1}.

\textbf{Step 2: Analysis on asymptotic behavior with respect to time.}

In contrast to the Cauchy problem, calculating decay rates of solutions to the initial-boundary value problem is more involved, as the standard Fourier transform method is no longer applicable here. Note that Kagei and Kobayashi previously established decay rate estimates of solutions to the linearized Navier-Stokes (NS) equations under the no-slip boundary condition in the half-space in \cite{k1,k2}. Some of their ideas can be modified to deal with the case considered here.

In addition, the magnetic field equation with small vertical magnetic diffusion considered here can be interpreted as a 3D nonlinear heat equation with a special diffusion operator $\Delta_h + \varepsilon\partial_3^2$. If we decompose the linearized magnetic field equation into the tangential component $\mathbf{B}^{\varepsilon}_h$ and the vertical component $B^{\varepsilon}_3$ separately, they can be viewed as initial-boundary value problems of the heat equation with homogeneous Neumann boundary conditions and homogeneous Dirichlet boundary conditions, respectively. The decay rates of solutions to the linearized magnetic field equation can be obtained by deriving the corresponding explicit solutions  and computing the decay rates of such explicit solutions via  Young's inequality in its convolution form.

If we denote the solution operator of the linearized magnetic field equation by $\mathcal{B}(t)$, then the corresponding solution to this equation can be written as $\mathcal{B}(t)\mathbf{B}_{0}$. We need to be extremely careful with the choice of the decay rate formula for $\mathcal{B}(t)\mathbf{B}_{0}$. Since $\varepsilon\rightarrow 0$ and $t\rightarrow+\infty$ in this paper, it is not clear which one converges faster, this may lead to the resulting decay rate being affected by $\varepsilon$. On one hand, if we employ the following decay rate estimate, which follows directly from \eqref{win3},
\begin{equation}\label{D1}
  \left\|\partial_{h}^{k}\mathcal{B}(t) \mathbf{B}_{0} \right\|_{L^{2}\left(\mathbb{R}^{3}_{+}\right)}\lesssim \left(\varepsilon t\right)^{-\frac{1}{4}}\cdot t^{-\frac{1+k}{2}}\left\|\mathbf{B}_{0} \right\|_{L^{1}\left(\mathbb{R}^{3}_{+}\right)},
\end{equation}
this will result in the decay rate depending on $\varepsilon$. To overcome this difficulty, we adopt the following decay rate formula in the form of the anisotropic $L^{p_i}$-norm
$\left\|\cdot\right\|_{L_{x_{1}}^{p_1}L_{x_{2}}^{p_2}L_{x_{3}}^{2}\left(\mathbb{R}^{3}_{+}\right)}$,
which follows directly from \eqref{win3},
\begin{equation}\label{D2}
  \left\|\partial_{h}^{k}\mathcal{B}(t) \mathbf{B}_{0} \right\|_{L^{2}\left(\mathbb{R}^{3}_{+}\right)}\lesssim t^{-\left(\frac{k}{2}+\frac{1}{2}\sum_{i=1}^{2}\left(\frac{1}{p_i}-\frac{1}{2}\right)\right)}
  \left\|\mathbf{B}_{0}\right\|_{L_{x_{1}}^{p_1}L_{x_{2}}^{p_2}L_{x_{3}}^{2}\left(\mathbb{R}^{3}_{+}\right)},
\end{equation}
which yields the decay rate formula independent of $\varepsilon$. On the other hand, since we design a weighted energy functional in the $H^1$ form to obtain the $H^1$ decay rate of the solution, we also need to choose an appropriate decay rate formula for $\partial_3\mathcal{B}(t)\mathbf{B}_{0}$. If we take the following form of the decay estimate, which follows directly from \eqref{win3},
\begin{equation*}
  \left\|\partial_{3}\mathcal{B}(t) \mathbf{B}_{0} \right\|_{L^{2}\left(\mathbb{R}^{3}_{+}\right)}\lesssim \left(\varepsilon t\right)^{-\frac{1}{2}}\cdot  t^{-\frac{1}{2}}
   \left\| \mathbf{B}_{0}\right\|_{L_{x_{1}x_{2}}^{1}L_{x_{3}}^{2}\left(\mathbb{R}^{3}_{+}\right)},
\end{equation*}
this leads to the same problem as \eqref{D1}, which indicates that the method of only using the anisotropic $L^{p_i}$-norm $\left\|\cdot\right\|_{L_{x_{1}}^{p_1}L_{x_{2}}^{p_2}L_{x_{3}}^{2}\left(\mathbb{R}^{3}_{+}\right)}$ in \eqref{D2} does not work here. To overcome this difficulty, we notice the boundary condition $\left(\mathbf{B}_{0}\right)_3\big|_{x_3=0} = 0$, which inspires us to transfer the first-order normal derivative from the solution $\mathcal{B}(t)\mathbf{B}_{0}$ to the initial data $\mathbf{B}_{0}$ via integration by parts, and thus we obtain the following formula
\begin{equation}\label{D3}
  \left\|\partial_{3}\mathcal{B}(t) \mathbf{B}_{0} \right\|_{L^{2}\left(\mathbb{R}^{3}_{+}\right)}\lesssim t^{-\frac{1}{2}} \left\|\partial_3 \mathbf{B}_{0} \right\|_{L_{x_{1}x_{2}}^{1}L_{x_{3}}^{2}\left(\mathbb{R}^{3}_{+}\right)},
\end{equation}
which gives the decay rate formula for $\partial_{3}\mathcal{B}(t) \mathbf{B}_{0}$ that is independent of $\varepsilon$.

Building on these facts, we thus derive the decay rates of solutions to the nonlinear problem considered in this paper by leveraging Duhamel's principle.
Due to the limited order of normal derivatives in the uniform energy estimate \eqref{eq10}, the estimates of the decay rates for the nonlinear terms are more difficult than those in \cite{JDE}. Accordingly, we complete the energy estimates by using   Hardy's inequality and conormal version interpolation
inequality \eqref{a8}. For the estimate of $\partial_3\mathcal{B}(t-\tau) F^{\varepsilon}_3(\tau)$, we note that the boundary condition in \eqref{eq8} implies $\left(F^{\varepsilon}_3\right)_3\big|_{x_3=0} = 0$, which allows us to transfer the first-order normal derivative from $\mathcal{B}(t-\tau) F^{\varepsilon}_3(\tau)$ to $F^{\varepsilon}_3(\tau)$ by the same method as in \eqref{D3}. Therefore, the decay rates of the nonlinear terms in this paper are relatively slow. Compared with the $H^1$ decay rates obtained in Theorem 1.3 of \cite{JDE}, the $L^2$ decay rates of the solutions are the same, while the $L^2$ decay rates of the first-order spatial derivatives of the solutions in this paper are slower.

\textbf{Step 3: Analysis on convergence rate.}

Having established the global uniform regularity and large-time decay rates of solutions to both the original system \eqref{eq1} and the limit system \eqref{eq4}, we now turn to deriving an explicit, time-uniform convergence rate for the vanishing vertical magnetic resistivity limit.
For notational convenience, we introduce the following notations:
$\bar{\rho} \overset{def}{=} \rho^\varepsilon - \rho^0$, $\bar{a} \overset{def}{=} a^\varepsilon - a^0$, $\bar{\rho} = \bar{a}$, $\bar{\mathbf{v}} \overset{def}{=} \mathbf{v}^\varepsilon - \mathbf{v}^0$, $\bar{\mathbf{B}} \overset{def}{=} \mathbf{B}^\varepsilon - \mathbf{B}^0$, $\bar{P} \overset{def}{=} P\left(\rho^\varepsilon\right) - P\left(\rho^0\right)$.
A key preliminary observation is that both the restricted order of normal derivatives in the uniform estimates for the original solutions \((a^\varepsilon, \mathbf{v}^\varepsilon, \mathbf{B}^\varepsilon)\) and the availability of only \(H^1\) decay rate estimates for the global solutions prevent us from establishing uniform estimate for the first-order normal derivative of \(\bar{\mathbf{B}}\), which in turn precludes the use of the Stokes elliptic regularity argument presented in Lemma \ref{GL6} to recover the dissipative estimate for the density difference \(\bar{a}\). As a result, we are restricted to using only the dissipation
\(\left\|\left(\partial\bar{\mathbf{v}},\partial_h \bar{\mathbf{B}}\right)\right\|_{L^2}\) to absorb nonlinear terms, and we can only achieve an \(L^2\) convergence rate rather than an \(H^1\)  convergence rate. To minimize the number of nonlinear terms that arise during the analysis of the difference system, we directly derive the difference equations from the original conservative forms \eqref{eq1} and \eqref{eq4} (yielding system \eqref{Con1}), rather than constructing them from the perturbed systems \eqref{eq6} and \eqref{eq7}.

The core challenges in establishing the \(L^2\) convergence rate for system \eqref{Con1} are twofold: obtaining a valid energy estimate for the density difference \(\bar{a}\), and handling the pressure gradient term \(\int_{\mathbb{R}_+^3}\bar{\mathbf{v}} \cdot \nabla \bar{P} \mathrm{d}\mathbf{x}\). To address these challenges, we first introduce the auxiliary function
\[
g(\rho) \overset{def}{=}\rho \int_{1}^{\rho} \frac{P(s) - P(1)}{s^2} \mathrm{d}s,
\]
which satisfies the relations
\[
\rho \nabla g'(\rho) = \nabla P(\rho), \quad g''(\rho) = \frac{P'(\rho)}{\rho},
\]
and for any fixed positive constant \(c\), if \(c \leq \rho \leq c^{-1}\), then \(g(\rho) \sim (\rho-1)^2\). We then define the auxiliary function
\[
\mathcal{G}\left(\rho^\varepsilon, \rho^0\right) \overset{def}{=} g\left(\rho^\varepsilon\right) - g\left(\rho^0\right) - g'\left(\rho^0\right) \bar{\rho},
\]
which satisfies the crucial equivalence \(\mathcal{G}\left(\rho^\varepsilon, \rho^0\right) \sim \bar{a}^2\) when both \(\rho^\varepsilon\) and \(\rho^0\) are uniformly bounded away from zero (which is guaranteed by Theorems \ref{Th1} and \ref{Th01} for small initial data). Through integration by parts and Taylor expansion of the pressure term, we transform the pressure gradient integral into
\begin{equation}\label{con1-summary}
\begin{aligned}
\int_{\mathbb{R}_+^3} \bar{\mathbf{v}} \cdot \nabla \bar{P} \mathrm{d}\mathbf{x}
={}& \frac{\mathrm{d}}{\mathrm{d}t} \int_{\mathbb{R}_+^3} \mathcal{G}\left(\rho^\varepsilon, \rho^0\right) \mathrm{d}\mathbf{x}  + \int_{\mathbb{R}_+^3} \frac{P'\left(\rho^0\right)}{\rho^0} \left(  \bar{\mathbf{v}}\cdot \nabla \rho^0 \right) \bar{\rho} \mathrm{d}\mathbf{x}  \\
&+ \int_{\mathbb{R}_+^3} \tilde{F}\left(\rho^0, \bar{\rho}\right) \bar{\rho}^2 \cdot \operatorname{div} \mathbf{v}^0 \mathrm{d}\mathbf{x},
\end{aligned}
\end{equation}
where
\[
\tilde{F}\left(\rho^0, \bar{\rho}\right) \overset{def}{=}  \int_{0}^{1} (1-\theta) P''\left(\rho^0 + \theta \bar{\rho}\right)  \mathrm{d}\theta.
\]
The identity \eqref{con1-summary} successfully extracts the density energy \(\int_{\mathbb{R}_+^3} \mathcal{G}\left(\rho^\varepsilon, \rho^0\right) \mathrm{d}\mathbf{x}\), leaving only two remainder terms.

The second remainder term \(\int_{\mathbb{R}_+^3} \tilde{F}\left(\rho^0, \bar{\rho}\right) \bar{\rho}^2 \cdot \operatorname{div} \mathbf{v}^0 \mathrm{d}\mathbf{x}\) presents a new difficulty: we cannot obtain a uniform decay rate for \(\left\|\operatorname{div} \mathbf{v}^0\right\|_{L^\infty}\), and thus cannot guarantee the time integrability of \(\left\|\operatorname{div} \mathbf{v}^0\right\|_{L^\infty}\). A simple estimate of the form \(\left\|\operatorname{div} \mathbf{v}^0\right\|_{L^\infty}\left\|\bar{a}\right\|_{L^2}^2\) would therefore fail to close the energy inequality. To overcome this, we expand \(P(\rho^\varepsilon)\) around \(\rho^0\) up to the third order in the Taylor expansion, and use the continuity equation for the limit system
\[
\operatorname{div} \mathbf{v}^0 = -\frac{1}{\rho^0}\left(\partial_t \rho^0 + \mathbf{v}^0 \cdot \nabla \rho^0\right)
\]
to replace the divergence term. We then introduce additional auxiliary functions
\[
J_1\left(\rho^0\right) = \int_{1}^{\rho^0} \frac{P''\left(s\right)}{2s} \mathrm{d}s, \quad J_2\left(\rho^0\right) = \int_{1}^{\rho^0} \frac{P'''\left(s\right)}{6s} \mathrm{d}s,
\]
and integrate by parts in time on the time derivative term. This procedure allows us to absorb the involved \(\operatorname{div} \mathbf{v}^0\) term into a modified energy functional, resulting in
\[
\int_{\mathbb{R}_+^3} \left(\mathcal{G}\left(\rho^\varepsilon, \rho^0\right)- J_1\left(\rho^0\right) \bar{a}^2-J_2\left(\rho^0\right) \bar{a}^3\right)\mathrm{d}\mathbf{x},
\]
which is still equivalent to \(\left\|\bar{a}\right\|_{L^2}^2\) for small \(\bar{a}\). All remaining terms arising in this derivation can be either absorbed into the dissipation terms or bounded by the global uniform energy estimates.

The final challenge is handling the term \(\left|\left(\varepsilon\partial_3\mathbf{B}^{\varepsilon},\partial_3\bar{\mathbf{B}}\right)\right|\).
To resolve this, we combine the weighted energy method with \(L^2\) decay rate of original system to obtain
\[
\varepsilon\int_{0}^{t}(1+\tau)^{1-\delta_2} \left\|\partial_3 \mathbf{B}^{\varepsilon}\right\|_{L^2}^2\mathrm{d}\tau \leq C.
\]
Together with the above estimate and by Hölder's inequality, it follows that
\[
\int_{0}^{t}\left|\left(\varepsilon\partial_3\mathbf{B}^{\varepsilon},\partial_3\bar{\mathbf{B}}\right)\right|\mathrm{d}\tau
\lesssim\varepsilon^{\frac{1}{2}}\left(\varepsilon\int_{0}^{t}(1+\tau)^{1-\delta_2}\left\|\partial_3\mathbf{B}^{\varepsilon}\right\|_{L^2}^{2}\mathrm{d}\tau\right)^{\frac{1}{2}}\left(\int_{0}^{t}(1+\tau)^{-1+\delta_2}\left\|\partial_3\bar{\mathbf{B}}\right\|_{L^2}^2\mathrm{d}\tau\right)^{\frac{1}{2}}.
\]
For the last factor, by Theorems \ref{Th3} and \ref{Th03}, \(\left\|\partial_3\mathbf{B}^\varepsilon\right\|_{L^2} = O\left(t^{-\frac{1}{2}+\delta_2}\right)\) and \(\left\|\partial_3\mathbf{B}^0\right\|_{L^2} = O\left(t^{-\frac{1}{2}}\right)\), so \(\left\|\partial_3\bar{\mathbf{B}}\right\|_{L^2} = O\left(t^{-\frac{1}{2}+\delta_2}\right)\). This gives
\[
\int_{0}^{t}(1+\tau)^{-1+\delta_2}\left\|\partial_3\bar{\mathbf{B}}\right\|_{L^2}^2\mathrm{d}\tau \lesssim \int_{0}^{t}(1+\tau)^{-2+3\delta_2}\mathrm{d}\tau \leq C.
\]
Combining these, we conclude
\[
\int_{0}^{t}\left|\left(\varepsilon\partial_3\mathbf{B}^{\varepsilon},\partial_3\bar{\mathbf{B}}\right)\right|\mathrm{d}\tau \lesssim \varepsilon^{\frac{1}{2}}.
\]
Combining all the above estimates, we arrive at the final differential inequality for the modified energy functional
\(
\bar{E}(t) \sim \left\|\left(\bar{a},\bar{\mathbf{v}},\bar{\mathbf{B}}\right)\right\|_{L^2},
\)
which reads
\[
\frac{\mathrm{d}}{\mathrm{d}t}\bar{E}(t)^2
\lesssim \bar{C}(t) \bar{E}(t)^2+\varepsilon^{\frac{1}{2}}\bar{K}(t),
\]
where \(\bar{C}(t)\) and \(\bar{K}(t)\) are coefficients satisfying \(\int_{0}^{t} \bar{C}(\tau) \mathrm{d}\tau \leq C\) and \(\int_{0}^{t} \bar{K}(\tau) \mathrm{d}\tau \leq C\) uniformly in \(t\) and \(\varepsilon\). Applying Gronwall's inequality with the initial condition \(\bar{E}(0) = 0\) (since both systems share the same initial data), we obtain
\[
\bar{E}(t)^2 \lesssim  \varepsilon^{\frac{1}{2}}\int_{0}^{t}  \bar{K}(s) e^{\int_{s}^{t} \bar{C}(r)\mathrm{d}r} \mathrm{d}s\lesssim
 \varepsilon^{\frac{1}{2}}.
\]
Recalling \(\bar{E}(t) \sim \left\|\left(\bar{a},\bar{\mathbf{v}},\bar{\mathbf{B}}\right)\right\|_{L^2}\), this yields the desired time-uniform \(L^2\) convergence rate
\[
\sup_{0\leq \tau \leq t}\left\|\left(\rho^{\varepsilon}-\rho^{0}, \mathbf{v}^{\varepsilon}-\mathbf{v}^{0}, \mathbf{B}^{\varepsilon}-\mathbf{B}^{0}\right)(\tau)\right\|_{L^2\left(\mathbb{R}_+^3\right)}
\lesssim  \varepsilon^{\frac{1}{4}}.
\]

The rest of this paper is organized as follows.
In Section  \ref{global-estimate}, we establish the global uniform estimate
for system  \eqref{eq1}--\eqref{eq3} under the condition of small initial data.
In Section  \ref{asymptotic-behavior}, we will establish the $H^1$ decay
rates for the system \eqref{eq1}--\eqref{eq3} under the condition \eqref{eq11}.
In Section \ref{Convergence rate of global solution}, we obtain the vanishing vertical magnetic resistivity limit with an explicit time-uniform \(L^2\) convergence rate via the regularity and decay results in Theorems \ref{Th1}, \ref{Th3}, \ref{Th01} and \ref{Th03}, rigorously justifying the asymptotic limit.
Finally, we collect some useful inequalities in Appendix \ref{ApA}.

\section{Global-in-time uniform regularity estimate}\label{global-estimate}
	In this section, we are devoted to deriving the a priori estimates of strong solutions \(\left( a^{\varepsilon} ,\mathbf{v}^{\varepsilon},\mathbf{B}^{\varepsilon}\right)\) to the compressible MHD system \(\eqref{eq6}\).
Throughout this section, we assume that
$m\geq 4$ and the solution satisfies the a priori assumption that
\begin{equation}\label{G1}
  E_m^{\varepsilon}(t)+D_m^{\varepsilon}(t) \leq \delta_3,
\end{equation}
 where \(\delta_3 > 0\) is a sufficiently small constant to be determined later.
\subsection{Conormal energy estimate}
\begin{lemm}\label{GL1}Assume \(\left(a^{\varepsilon},\mathbf{v}^{\varepsilon},\mathbf{B}^{\varepsilon}\right)\) is a solution to \eqref{eq6}, it holds that
\begin{equation}\label{G2}
  \begin{aligned}
& \left\|  Z^{\alpha} \left(a^{\varepsilon},\mathbf{v}^{\varepsilon}, \mathbf{B}^{\varepsilon}\right)(t)\right\|_{0}^{2}+\int_{0}^{t}  \left(\left\|  Z^{\alpha} \left(\frac{\mathrm{D} a^{\varepsilon}}{\mathrm{D}\tau},\partial\mathbf{v}^{\varepsilon}, \partial_h\mathbf{B}^{\varepsilon}\right)\right\|_{0}^{2}+\varepsilon\left\| \partial_3 Z^{\alpha} \mathbf{B}^{\varepsilon}\right\|_{0}^{2}\right) \mathrm{d}\tau \\
\lesssim{}&\left\|  Z^{\alpha} \left(a^{\varepsilon},\mathbf{v}^{\varepsilon}, \mathbf{B}^{\varepsilon}\right)(0)\right\|_{0}^{2}+\int_{0}^{t}  \left(\left\|  Z^{\alpha}\left( a^{\varepsilon}\operatorname{div}\mathbf{v}^{\varepsilon}\right)\right\|_{0}^{2}+\left|A _{\alpha}\right|+\left|\widetilde{A} _{\alpha}\right|\right) \mathrm{d}\tau
\end{aligned}
\end{equation}
 for \( |\alpha| \leq m \), where
\[
A_{\alpha}\overset{def}{=}\left( Z^{\alpha}f^{\varepsilon}_1,  Z^{\alpha} a^{\varepsilon}\right)+\left( Z^{\alpha} f^{\varepsilon}_2,  Z^{\alpha} \mathbf{v}^{\varepsilon}\right)+  \left(Z^{\alpha} f^{\varepsilon}_3,  Z^{\alpha} \mathbf{B}^{\varepsilon} \right)
\]
and
\begin{align*}
  \widetilde{A} _{\alpha}\overset{def}{=}{}&\left\|\left[Z^{\alpha}, \partial_3\right]\mathbf{v}^{\varepsilon}\right\|_0^2-\left( \left[Z^{\alpha}, \operatorname{div}\right]\mathbf{v}^{\varepsilon},  Z^{\alpha} a^{\varepsilon}\right)+\left(\varepsilon\left[Z^{\alpha}, \partial_3^2\right]\mathbf{B}^{\varepsilon},  Z^{\alpha} \mathbf{B}^{\varepsilon} \right)\\
  &+\left(\mu\left[Z^{\alpha}, \Delta\right]\mathbf{v}^{\varepsilon} + (\mu + \lambda)\left[Z^{\alpha}, \nabla\operatorname{div}\right]\mathbf{v}^{\varepsilon} - \left[Z^{\alpha}, \nabla\right]a^{\varepsilon},  Z^{\alpha} \mathbf{v}^{\varepsilon}\right).
\end{align*}
\end{lemm}

\begin{proof}
Denote the material derivative of \(a^{\varepsilon}\) as
\[
\frac{\mathrm{D} a^{\varepsilon}}{\mathrm{D} t}\overset{def}{=}\partial_{t} a^{\varepsilon}+  \mathbf{v}^{\varepsilon} \cdot \nabla a^{\varepsilon}.
\]
Apply the conormal derivatives $Z^{\alpha}$ to the equations in \eqref{eq6},
and we obtain
	 \begin{equation}\label{G3}
	 	\left\{\begin{array}{*{5}{ll}}
	 		\partial_t Z^{\alpha}a^{\varepsilon} + \operatorname{div}Z^{\alpha}\mathbf{v}^{\varepsilon} = Z^{\alpha}f^{\varepsilon}_1-\left[Z^{\alpha}, \operatorname{div}\right]\mathbf{v}^{\varepsilon}
\quad & {\rm in} ~~\mathbb{R}_+^3,\\
\partial_t Z^{\alpha}\mathbf{v}^{\varepsilon} - \mu\Delta Z^{\alpha}\mathbf{v}^{\varepsilon} - (\mu + \lambda)\nabla\operatorname{div}Z^{\alpha}\mathbf{v}^{\varepsilon} + \nabla Z^{\alpha}a^{\varepsilon} = Z^{\alpha}f^{\varepsilon}_2 \quad &\\
+ \mu\left[Z^{\alpha}, \Delta\right]\mathbf{v}^{\varepsilon} + (\mu + \lambda)\left[Z^{\alpha}, \nabla\operatorname{div}\right]\mathbf{v}^{\varepsilon} - \left[Z^{\alpha}, \nabla\right]a^{\varepsilon} \quad & {\rm in} ~~\mathbb{R}_+^3,\\
	 		\partial_tZ^{\alpha}\mathbf{B}^{\varepsilon} - \left(\Delta_h+\varepsilon\partial_3^2\right)Z^{\alpha}\mathbf{B}^{\varepsilon} = Z^{\alpha}f^{\varepsilon}_3+\varepsilon\left[Z^{\alpha}, \partial_3^2\right]\mathbf{B}^{\varepsilon}
\quad & {\rm in} ~~ \mathbb{R}_+^3,\\
	 		Z^{\alpha}\left(\mathbf{v}^{\varepsilon}, B_3^\varepsilon\right) = \mathbf{0},\quad \partial_3 Z^{\alpha}\mathbf{B}^{\varepsilon}_h=\mathbf{0}\quad & {\rm on} ~~ \mathbb{R}^2 \times \left\{ x_3=0 \right\}.
	 	\end{array}\right.
	 \end{equation}	
And we also have
\begin{equation}\label{G4}
   Z^{\alpha} \left(\frac{\mathrm{D} a^{\varepsilon} }{\mathrm{D}t}\right)=-   Z^{\alpha}\left( a^{\varepsilon} \operatorname{div}\mathbf{v}^{\varepsilon}\right)- Z^{\alpha} \operatorname{div} \mathbf{v}^{\varepsilon}.
\end{equation}
Let \(|\alpha| \leq m\). Taking the \(L^{2}\) inner product of \eqref{G3} with \(   Z^{\alpha} \left( a^{\varepsilon} ,\mathbf{v}^{\varepsilon},\mathbf{B}^{\varepsilon}\right)\) gives
\begin{equation*}
\begin{aligned}
&\frac{1}{2}\frac{\mathrm{d}}{\mathrm{d}t}\left\|  Z^{\alpha} \left(a^{\varepsilon},\mathbf{v}^{\varepsilon},\mathbf{B}^{\varepsilon}\right)(t)\right\|_{0}^{2} + \mu\left\| \nabla Z^{\alpha}\mathbf{v}^{\varepsilon}\right\|_0^2 + (\mu+\lambda)\left\|  \operatorname{div}Z^{\alpha}\mathbf{v}^{\varepsilon}\right\|_0^2+ \left\|\nabla_h Z^{\alpha}\mathbf{B}^{\varepsilon}\right\|_0^2
+ \varepsilon\left\|\partial_3 Z^{\alpha}\mathbf{B}^{\varepsilon}\right\|_0^2\\
={}&{}\left( Z^{\alpha} f^{\varepsilon}_1,  Z^{\alpha} a^{\varepsilon}\right) +  \left(Z^{\alpha} f^{\varepsilon}_2,  Z^{\alpha} \mathbf{v}^{\varepsilon}\right)+  \left(Z^{\alpha} f^{\varepsilon}_3,  Z^{\alpha} \mathbf{B}^{\varepsilon} \right)-\left( \left[Z^{\alpha}, \operatorname{div}\right]\mathbf{v}^{\varepsilon},  Z^{\alpha} a^{\varepsilon}\right) \\
&+  \left(\mu\left[Z^{\alpha}, \Delta\right]\mathbf{v}^{\varepsilon} + (\mu + \lambda)\left[Z^{\alpha}, \nabla\operatorname{div}\right]\mathbf{v}^{\varepsilon} - \left[Z^{\alpha}, \nabla\right]a^{\varepsilon},  Z^{\alpha} \mathbf{v}^{\varepsilon}\right)+  \left(\varepsilon\left[Z^{\alpha}, \partial_3^2\right]\mathbf{B}^{\varepsilon},  Z^{\alpha} \mathbf{B}^{\varepsilon} \right),
\end{aligned}
\end{equation*}
where the following fact is used,
\[
\left(\nabla  Z^{\alpha}  a^{\varepsilon} ,  Z^{\alpha} \mathbf{v}^{\varepsilon}\right)=-\left( Z^{\alpha}  a^{\varepsilon}, \operatorname{div}  Z^{\alpha} \mathbf{v}^{\varepsilon}\right).
\]
Combining with
\[
\left\|  Z^{\alpha} \left(\frac{\mathrm{D} a^{\varepsilon} }{\mathrm{D}t}\right)\right\|_{0}^{2} \lesssim\left\|  Z^{\alpha}\left( a^{\varepsilon}\operatorname{div}\mathbf{v}^{\varepsilon}\right)\right\|_{0}^{2}+\left\| Z^{\alpha} \operatorname{div} \mathbf{v}^{\varepsilon}\right\|_{0}^{2},
\]
which follows from \eqref{G4}, we get
\begin{equation*}
\begin{aligned}
& \frac{\mathrm{d}}{\mathrm{d}t}\left\| Z^{\alpha} \left(a^{\varepsilon},\mathbf{v}^{\varepsilon},\mathbf{B}^{\varepsilon}\right)(t)\right\|_{0}^{2}
+ \left\| Z^{\alpha} \left(\frac{\mathrm{D} a^{\varepsilon}}{\mathrm{D}\tau},\partial\mathbf{v}^{\varepsilon}, \partial_h\mathbf{B}^{\varepsilon}\right)\right\|_{0}^{2}
+ \varepsilon\left\|\partial_3 Z^{\alpha} \mathbf{B}^{\varepsilon}\right\|_{0}^{2}
\\
\lesssim{}&
\left\| Z^{\alpha}\left( a^{\varepsilon}\operatorname{div}\mathbf{v}^{\varepsilon}\right)\right\|_{0}^{2}
+ \bigl( Z^{\alpha}f^{\varepsilon}_1, Z^{\alpha} a^{\varepsilon}\bigr)
+ \bigl( Z^{\alpha} f^{\varepsilon}_2, Z^{\alpha} \mathbf{v}^{\varepsilon}\bigr)
+ \bigl(Z^{\alpha} f^{\varepsilon}_3, Z^{\alpha} \mathbf{B}^{\varepsilon} \bigr)\\
&+\left\|\left[Z^{\alpha}, \partial_3\right]\mathbf{v}^{\varepsilon}\right\|_0^2
- \bigl( \left[Z^{\alpha}, \operatorname{div}\right]\mathbf{v}^{\varepsilon}, Z^{\alpha} a^{\varepsilon}\bigr)+ \bigl(\varepsilon\left[Z^{\alpha}, \partial_3^2\right]\mathbf{B}^{\varepsilon}, Z^{\alpha} \mathbf{B}^{\varepsilon} \bigr)\\
&+ \bigl(\mu\left[Z^{\alpha}, \Delta\right]\mathbf{v}^{\varepsilon} + (\mu + \lambda)\left[Z^{\alpha}, \nabla\operatorname{div}\right]\mathbf{v}^{\varepsilon} - \left[Z^{\alpha}, \nabla\right]a^{\varepsilon}, Z^{\alpha} \mathbf{v}^{\varepsilon}\bigr).
\end{aligned}
\end{equation*}
By integrating it with respect to time, we obtain inequality \eqref{G2}.
\end{proof}
\subsection{Estimate of the first order time derivative}
\begin{lemm}\label{GL2}Assume \(\left(a^{\varepsilon},\mathbf{v}^{\varepsilon},\mathbf{B}^{\varepsilon}\right)\) is a solution to \eqref{eq6}, it holds that
\begin{equation}\label{G5}
  \begin{aligned}
& \left\|  Z^{\alpha}\partial_{t} \left(a^{\varepsilon},\mathbf{v}^{\varepsilon}, \mathbf{B}^{\varepsilon}\right)(t)\right\|_{0}^{2}+\int_{0}^{t}  \left(\left\| \partial_{\tau} Z^{\alpha} \left(\frac{\mathrm{D} a^{\varepsilon}}{\mathrm{D}\tau},\partial\mathbf{v}^{\varepsilon}, \partial_h\mathbf{B}^{\varepsilon}\right)\right\|_{0}^{2}
+\varepsilon\left\|  \partial_{\tau}  \partial_3Z^{\alpha}\mathbf{B}^{\varepsilon}\right\|_{0}^{2}\right) \mathrm{d}\tau \\
\lesssim{}&\left\|  Z^{\alpha}\partial_{t} \left(a^{\varepsilon},\mathbf{v}^{\varepsilon}, \mathbf{B}^{\varepsilon}\right)(0)\right\|_{0}^{2}+\int_{0}^{t}  \left(\left\| \partial_{\tau} Z^{\alpha}\left( a^{\varepsilon}\operatorname{div}\mathbf{v}^{\varepsilon}\right)\right\|_{0}^{2}+\left|B _{\alpha}\right|+\left|\widetilde{B} _{\alpha}\right|\right) \mathrm{d}\tau
\end{aligned}
\end{equation}
 for \( |\alpha| \leq m-2 \), where
\[
B_{\alpha}\overset{def}{=}\left( Z^{\alpha}\partial_{t}f^{\varepsilon}_1,  Z^{\alpha}\partial_{t} a^{\varepsilon}\right)+\left( Z^{\alpha}\partial_{t} f^{\varepsilon}_2,  Z^{\alpha} \partial_{t}\mathbf{v}^{\varepsilon}\right)+  \left(Z^{\alpha}\partial_{t} f^{\varepsilon}_3,  Z^{\alpha}\partial_{t} \mathbf{B}^{\varepsilon} \right)
\]
and
\begin{align*}
  \widetilde{B} _{\alpha}\overset{def}{=}{}&\left\|\left[Z^{\alpha}, \partial_3\right]\partial_{t}\mathbf{v}^{\varepsilon}\right\|_0^2-\left( \left[Z^{\alpha}, \operatorname{div}\right]\partial_{t}\mathbf{v}^{\varepsilon},  Z^{\alpha}\partial_{t} a^{\varepsilon}\right)+  \left(\varepsilon\left[Z^{\alpha}, \partial_3^2\right]\partial_{t}\mathbf{B}^{\varepsilon},  Z^{\alpha}\partial_{t} \mathbf{B}^{\varepsilon} \right)\\
  &+\left(\mu\left[Z^{\alpha}, \Delta\right]\partial_{t}\mathbf{v}^{\varepsilon} + (\mu + \lambda)\left[Z^{\alpha}, \nabla\operatorname{div}\right]\partial_{t}\mathbf{v}^{\varepsilon} - \left[Z^{\alpha}, \nabla\right]\partial_{t}a^{\varepsilon},  Z^{\alpha}\partial_{t} \mathbf{v}^{\varepsilon}\right).
\end{align*}	
\end{lemm}	
\begin{proof}
The proof is similar as that of Lemma \ref{GL1} and is thus omitted.
\end{proof}
\subsection{Estimate of the first order normal derivative}
In this subsection, we will derive the estimate of the first order normal derivative of $\left(a^{\varepsilon}, \mathbf{v}^{\varepsilon}, \mathbf{B}^{\varepsilon}\right)$.
\subsubsection{Estimate of the first order normal derivative of $\mathbf{v}^{\varepsilon}$}
\begin{lemm}\label{GL3}Assume $\left(a^{\varepsilon},\mathbf{v}^{\varepsilon},\mathbf{B}^{\varepsilon}\right)$ is a solution to \eqref{eq6}, it holds that
\begin{equation}\label{G6}
\begin{aligned}
& \left\| \partial Z^{\alpha} \mathbf{v}^{\varepsilon}(t)\right\|_{0}^{2}+\left\| \partial_{h} Z^{\alpha} \mathbf{B}^{\varepsilon}(t)\right\|_{0}^{2}+\varepsilon\left\| \partial_3 Z^{\alpha} \mathbf{B}^{\varepsilon}(t)\right\|_{0}^{2}+\int_{0}^{t}  \left\|\partial_{\tau} Z^{\alpha} \left(a^{\varepsilon}, \mathbf{v}^{\varepsilon}, \mathbf{B}^{\varepsilon}\right)\right\|_{0}^{2} \mathrm{d}\tau \\
\lesssim{}&\left\|  Z^{\alpha}  a^{\varepsilon} (0)\right\|_{0}^{2}+\left\|  \partial Z^{\alpha}  \mathbf{v}^{\varepsilon} (0)\right\|_{0}^{2}+\left\| \partial_{h} Z^{\alpha} \mathbf{B}^{\varepsilon}(0)\right\|_{0}^{2}+\varepsilon\left\| \partial_3 Z^{\alpha} \mathbf{B}^{\varepsilon}(0)\right\|_{0}^{2}\\
&+ \left\|  Z^{\alpha}  a^{\varepsilon} (t)\right\|_{0}^{2}+\int_{0}^{t}  \left(\left\| \partial Z^{\alpha} \mathbf{v}^{\varepsilon}\right\|_{0}^{2}+\left|C _{\alpha}\right|+\left|\widetilde{C}_{\alpha}\right|\right) \mathrm{d}\tau
\end{aligned}
\end{equation}
for $|\alpha| \leq m-1$, where
\[
C _{\alpha}\overset{def}{=}\left(Z^{\alpha}f^{\varepsilon}_1, \partial _{t}Z^{\alpha}a^{\varepsilon} \right) +\left( Z^{\alpha}f^{\varepsilon}_2, \partial _{t}Z^{\alpha}\mathbf{v}^{\varepsilon}\right)+\left( Z^{\alpha}f^{\varepsilon}_3, \partial _{t}Z^{\alpha}\mathbf{B}^{\varepsilon}\right)
\]
and
\begin{align*}
  \widetilde{C} _{\alpha}\overset{def}{=}{}&-\left( \left[Z^{\alpha}, \operatorname{div}\right]\mathbf{v}^{\varepsilon},  \partial _{t}Z^{\alpha} a^{\varepsilon}\right)+\left(\varepsilon\left[Z^{\alpha}, \partial_3^2\right]\mathbf{B}^{\varepsilon} ,  \partial _{t}Z^{\alpha} \mathbf{B}^{\varepsilon}\right)\\
  &+\left(\mu\left[Z^{\alpha}, \Delta\right]\mathbf{v}^{\varepsilon} + (\mu + \lambda)\left[Z^{\alpha}, \nabla\operatorname{div}\right]\mathbf{v}^{\varepsilon} - \left[Z^{\alpha}, \nabla\right]a^{\varepsilon},  \partial _{t}Z^{\alpha} \mathbf{v}^{\varepsilon}\right).
\end{align*}
\end{lemm}

\begin{proof}
Taking the $L^{2}$ inner product of \eqref{G3} with $  \partial_{t}Z^{\alpha}\left( a^{\varepsilon}, \mathbf{v}^{\varepsilon}, \mathbf{B}^{\varepsilon}\right)$ and combining with
\[
 \left(\nabla  Z^{\alpha}  a^{\varepsilon} , \partial_{t}Z^{\alpha} \mathbf{v}^{\varepsilon}\right)=  -\frac{\mathrm{d}}{\mathrm{d}t} \left( Z^{\alpha}  a^{\varepsilon} , \operatorname{div}  Z^{\alpha} \mathbf{v}^{\varepsilon}\right) + \left(\partial_{t}Z^{\alpha}  a^{\varepsilon} , \operatorname{div}  Z^{\alpha} \mathbf{v}^{\varepsilon}\right),
\]
we can get
\begin{align*}
& \frac{1}{2} \frac{\mathrm{d}}{\mathrm{d}t} \left(\mu\left\| \nabla  Z^{\alpha} \mathbf{v}^{\varepsilon}\right\|_{0}^{2}+(\mu+\lambda)\left\| \operatorname{div}  Z^{\alpha} \mathbf{v}^{\varepsilon}\right\|_{0}^{2}-2 \left( Z^{\alpha}  a^{\varepsilon} , \operatorname{div}  Z^{\alpha} \mathbf{v}^{\varepsilon}\right)\right. \\
&\left.+\left\| \partial_{h} Z^{\alpha} \mathbf{B}^{\varepsilon}(t)\right\|_{0}^{2}+\varepsilon\left\| \partial_3 Z^{\alpha} \mathbf{B}^{\varepsilon}(t)\right\|_{0}^{2}\right) + \left\| \partial_{t}Z^{\alpha}  \left(a^{\varepsilon}, \mathbf{v}^{\varepsilon}, \mathbf{B}^{\varepsilon}\right)\right\|_{0}^{2}+2  \left(\partial_{t}Z^{\alpha}  a^{\varepsilon} , \operatorname{div}  Z^{\alpha} \mathbf{v}^{\varepsilon}\right)\\
={}&\left(Z^{\alpha}f^{\varepsilon}_1, \partial _{t}Z^{\alpha}a^{\varepsilon} \right) +\left( Z^{\alpha}f^{\varepsilon}_2, \partial _{t}Z^{\alpha}\mathbf{v}^{\varepsilon}\right)+\left( Z^{\alpha}f^{\varepsilon}_3, \partial _{t}Z^{\alpha}\mathbf{B}^{\varepsilon}\right)-\left( \left[Z^{\alpha}, \operatorname{div}\right]\mathbf{v}^{\varepsilon},  \partial _{t}Z^{\alpha} a^{\varepsilon}\right)\\
&+\left(\mu\left[Z^{\alpha}, \Delta\right]\mathbf{v}^{\varepsilon} + (\mu + \lambda)\left[Z^{\alpha}, \nabla\operatorname{div}\right]\mathbf{v}^{\varepsilon} - \left[Z^{\alpha}, \nabla\right]a^{\varepsilon},  \partial _{t}Z^{\alpha} \mathbf{v}^{\varepsilon}\right)+\left(\varepsilon\left[Z^{\alpha}, \partial_3^2\right]\mathbf{B}^{\varepsilon} ,  \partial _{t}Z^{\alpha} \mathbf{B}^{\varepsilon}\right).
\end{align*}
Integrating the above inequality with respect to time variable in $[0, t]$ and applying  Young's inequality, we obtain \eqref{G6}.
\end{proof}	
\subsubsection{Estimate of the first order normal derivative of $a^{\varepsilon}$}
\begin{lemm}\label{GL4}Assume $\left(a^{\varepsilon},\mathbf{v}^{\varepsilon},\mathbf{B}^{\varepsilon}\right)$ is a solution to \eqref{eq6}, it holds that
\begin{equation}\label{G7}
\begin{aligned}
& \left\|  Z^{\alpha}\partial_3  a^{\varepsilon} (t)\right\|_{0}^{2}+\int_{0}^{t}  \left\|  Z^{\alpha}\partial_3 \left(a^{\varepsilon}, \frac{\mathrm{D} a^{\varepsilon} }{\mathrm{D}\tau}\right)\right\|_{0}^{2} \mathrm{d}\tau \\
\lesssim{}&\left\|  Z^{\alpha}\partial_3  a^{\varepsilon} (0)\right\|_{0}^{2}
+\int_{0}^{t}  \left(\left\| \partial_{h}  Z^{\alpha}\partial\mathbf{v}^{\varepsilon}\right\|_{0}^{2}+\left\|  \partial_{\tau}Z^{\alpha} \mathbf{v}^{\varepsilon}\right\|_{0}^{2}\right) \mathrm{d}\tau \\
&+\int_{0}^{t}  \left(\left\|  Z^{\alpha}\partial_3\left( a^{\varepsilon}\operatorname{div}\mathbf{v}^{\varepsilon}\right)\right\|_{0}^{2}+\left\|  Z^{\alpha}f^{\varepsilon}_2\right\|_{0}^{2}+\left|D_{\alpha}\right|\right) \mathrm{d}\tau\\
\end{aligned}
\end{equation}
for $|\alpha| \leq m-1$, where
\[
D_{\alpha}\overset{def}{=} \left(  Z^{\alpha}\partial_3\left(\mathbf{v}^{\varepsilon}\cdot \nabla a^{\varepsilon} \right), Z^{\alpha}\partial_3 a^{\varepsilon} \right).
\]
\end{lemm}
\begin{proof}
From the equation \eqref{eq6}, we have
\begin{align*}
\partial_{3}\left(\frac{\mathrm{D} a^{\varepsilon} }{\mathrm{D}t}\right)+  \partial_{3}^{2} v_{3}^{\varepsilon} & =-  \partial_{3} \operatorname{div}_h \mathbf{v}^{\varepsilon}_{h}-  \partial_{3}\left( a^{\varepsilon}\operatorname{div}\mathbf{v}^{\varepsilon}\right), \\
-(2\mu+\lambda) \partial_{3}^{2} v_{3}^{\varepsilon}+  \partial_{3}  a^{\varepsilon}  & =-\partial_{t} v^{\varepsilon}_{3}+\mu \Delta_h v_{3}^{\varepsilon}+(\mu+\lambda) \partial_{3} \operatorname{div}_h \mathbf{v}^{\varepsilon}_{h}+\left(f^{\varepsilon}_2\right)_{3}.
\end{align*}
Combining the above two equations to eliminate the $\partial_{3}^{2} v_{3}^{\varepsilon}$ term and applying  $Z^{\alpha}$ to both sides of the resulting identity, we obtain
\begin{align*}
& Z^{\alpha}\partial_3 \left(\frac{\mathrm{D} a^{\varepsilon} }{\mathrm{D}t}\right)  +\frac{1}{2\mu+\lambda} Z^{\alpha}\partial_3 a^{\varepsilon}  \\
={}&\frac{1}{2\mu+\lambda}\left(- Z^{\alpha}\partial_{t} v_{3}^{\varepsilon}+\mu \Delta_h Z^{\alpha}v_{3}^{\varepsilon}-\mu  Z^{\alpha}\partial_3\operatorname{div}_h\mathbf{v}^{\varepsilon}_{h}\right) \\
&-   Z^{\alpha}\partial_3\left( a^{\varepsilon}\operatorname{div}\mathbf{v}^{\varepsilon}\right)+\frac{1}{2\mu+\lambda} Z^{\alpha}\left(f^{\varepsilon}_2\right)_{3}.
\end{align*}
Taking the squared $L^2$-norm $\left(\|\cdot\|_0^2\right)$,  we obtain
\begin{align*}
& \frac{\mathrm{d}}{\mathrm{d}t} \left\|  Z^{\alpha}\partial_3  a^{\varepsilon} (t)\right\|_{0}^{2}+\left\|  Z^{\alpha}\partial_3 \left(a^{\varepsilon}, \frac{\mathrm{D} a^{\varepsilon} }{\mathrm{D}t}\right)\right\|_{0}^{2} \\
\lesssim{}&\left\| \partial_{h}  Z^{\alpha}\partial\mathbf{v}^{\varepsilon}\right\|_{0}^{2}+\left\|  \partial_{t}Z^{\alpha} \mathbf{v}^{\varepsilon}\right\|_{0}^{2}+\left\|  Z^{\alpha}\partial_3\left( a^{\varepsilon}\operatorname{div}\mathbf{v}^{\varepsilon}\right)\right\|_{0}^{2}+\left\|  Z^{\alpha} f^{\varepsilon}_2\right\|_{0}^{2}+\left|D_{\alpha}\right|.
\end{align*}
By integrating with respect to time, we obtain inequality \eqref{G7}.
\end{proof}

\subsubsection{Estimate of the first order normal derivative of $\mathbf{B}^{\varepsilon}$}
\begin{lemm}\label{GL5}Assume $\left(a^{\varepsilon},\mathbf{v}^{\varepsilon},\mathbf{B}^{\varepsilon}\right)$ is a solution to \eqref{eq6}, it holds that
\begin{equation}\label{G9}
\begin{aligned}
& \left\|  Z^{\alpha}\partial_3  \mathbf{B}^{\varepsilon}(t)\right\|_{0}^{2}+\int_{0}^{t}  \left(\left\| \partial_h Z^{\alpha}\partial_3 \mathbf{B}^{\varepsilon}\right\|_{0}^{2}+\varepsilon\left\| \partial_3 Z^{\alpha}\partial_3 \mathbf{B}^{\varepsilon}\right\|_{0}^{2}\right) \mathrm{d}\tau \\
\lesssim{}&\left\|  Z^{\alpha}\partial_3  \mathbf{B}^{\varepsilon}(0)\right\|_{0}^{2}+\int_{0}^{t}  \left(\left|E_{\alpha}\right|+\left|\widetilde{E}_{\alpha}\right|\right) \mathrm{d}\tau\\
\end{aligned}
\end{equation}
for $|\alpha| \leq m-1$, where
\[
E_{\alpha}\overset{def}{=}\left( Z^{\alpha} \partial_3 f^{\varepsilon}_3,  Z^{\alpha} \partial_3 \mathbf{B}^{\varepsilon}\right),
\]
and
\[
\widetilde{E}_{\alpha}\overset{def}{=}\left( \varepsilon\left[Z^{\alpha}, \partial_3^2\right]\partial_3 \mathbf{B}^{\varepsilon},  Z^{\alpha} \partial_3 \mathbf{B}^{\varepsilon}\right).
\]
\end{lemm}

\begin{proof}
Applying $Z^{\alpha}\partial_3$ to the magnetic equation in \eqref{eq6}, we get
\begin{equation}\label{G10}
\left\{
\begin{array}{*{5}{ll}}
\partial_{t} Z^{\alpha} \partial_3 \mathbf{B}^{\varepsilon} - \Delta_h Z^{\alpha} \partial_3 \mathbf{B}^{\varepsilon} - \varepsilon\partial_3^2 Z^{\alpha} \partial_3 \mathbf{B}^{\varepsilon} = Z^{\alpha} \partial_3 f^{\varepsilon}_3+\varepsilon\left[Z^{\alpha}, \partial_3^2\right]\partial_3 \mathbf{B}^{\varepsilon} \quad & {\rm in} ~~\mathbb{R}_+^3,\\
\left(Z^{\alpha}\partial_3 \mathbf{B}^{\varepsilon}_{h},\partial_3 Z^{\alpha}\partial_3 B^{\varepsilon}_{3}\right)=\mathbf{0} \quad & {\rm on} ~~ \mathbb{R}^2 \times \left\{ x_3=0 \right\}.
\end{array}
\right.
\end{equation}
Let $|\alpha| \leq m-1$. Taking the $L^{2}$ inner product of \eqref{G10} with $   Z^{\alpha} \partial_3 \mathbf{B}^{\varepsilon}$, we arrive at
\begin{align*}
\frac{1}{2}\frac{\mathrm{d}}{\mathrm{d}t}\left\| Z^{\alpha} \partial_3 \mathbf{B}^{\varepsilon}(t)\right\|_{0}^{2}
+ \left\| \nabla_h Z^{\alpha} \partial_3\mathbf{B}^{\varepsilon}\right\|_0^2
+ \varepsilon\left\| \partial_3 Z^{\alpha} \partial_3\mathbf{B}^{\varepsilon}\right\|_0^2
\\
= \bigl( Z^{\alpha} \partial_3 f^{\varepsilon}_3, Z^{\alpha} \partial_3 \mathbf{B}^{\varepsilon}\bigr)
+ \bigl( \varepsilon\left[Z^{\alpha}, \partial_3^2\right]\partial_3 \mathbf{B}^{\varepsilon}, Z^{\alpha} \partial_3 \mathbf{B}^{\varepsilon}\bigr).
\end{align*}
By integrating it with respect to time variable in $[0, t]$, we obtain inequality \eqref{G9}.
\end{proof}
\subsection{Estimate of  dissipation of $\left( a^{\varepsilon}, \mathbf{v}^{\varepsilon}\right)$}
\begin{lemm}\label{GL6}Assume $\left(a^{\varepsilon},\mathbf{v}^{\varepsilon},\mathbf{B}^{\varepsilon}\right)$ is a solution to \eqref{eq6}, it holds that
\begin{equation}\label{G11}
\begin{aligned}
&\int_{0}^{t}  \left(\left\| \partial Z^{\alpha}  a^{\varepsilon} \right\|_{0}^{2}+\left\| \partial^{2} Z^{\alpha} \mathbf{v}^{\varepsilon}\right\|_{0}^{2}\right) \mathrm{d}\tau \\
\lesssim{}& \int_{0}^{t}  \Bigg(\left\|  \partial  Z^{\alpha}\left(\frac{\mathrm{D} a^{\varepsilon} }{\mathrm{D}\tau}\right)\right\|_{0}^{2}+\left\|  Z^{\alpha} \partial \left(\frac{\mathrm{D} a^{\varepsilon} }{\mathrm{D}\tau}\right)\right\|_{0}^{2}+\left\|\partial_{\tau} Z^{\alpha}  \mathbf{v}^{\varepsilon}\right\|_{0}^{2} \\
&+\left\| Z^{\alpha} \partial\left( a^{\varepsilon}\operatorname{div}\mathbf{v}^{\varepsilon}\right)\right\|_{0}^{2}+\left\| \partial Z^{\alpha}\left( a^{\varepsilon}\operatorname{div}\mathbf{v}^{\varepsilon}\right)\right\|_{0}^{2}+\left\|  Z^{\alpha} f^{\varepsilon}_2\right\|_{0}^{2}+\left|\widetilde{F}_{\alpha}\right|\Bigg) \mathrm{d}\tau
\end{aligned}
\end{equation}
for $|\alpha| \leq m-1$, where
\[
\widetilde{F}_{\alpha}\overset{def}{=}  \left\| \left[Z^{\alpha}, \Delta\right]\mathbf{v}^{\varepsilon}\right\|_{0}^{2}+  \left\|\left[Z^{\alpha}, \nabla\right]a^{\varepsilon}\right\|_{0}^{2}+\left\|\partial\left[Z^{\alpha}, \operatorname{div}\right]\mathbf{v}^{\varepsilon}\right\|_{0}^{2}.
\]
\end{lemm}

\begin{proof}
To obtain \eqref{G11}, we use the following estimate for the Stokes system. Let $\left( a^{\varepsilon} , \mathbf{v}^{\varepsilon}\right)$ be the solution of the Stokes system
\begin{equation*}
\left\{
\begin{array}{*{5}{ll}}
-\mu \Delta \mathbf{v}^{\varepsilon} + \nabla a^{\varepsilon} = f \quad & \text{in } \mathbb{R}_+^3,\\[4pt]
\operatorname{div}\mathbf{v}^{\varepsilon} = g \quad & \text{in } \mathbb{R}_+^3,\\[4pt]
\mathbf{v}^{\varepsilon}= \mathbf{0} \quad & \text{on } \mathbb{R}^2\times\left\{x_3=0\right\}.
\end{array}
\right.
\end{equation*}
Then for any $\ell\in \mathbb{Z}$, $\ell\ge 0$, there holds
\[
\left\| \partial^{\ell+2} \mathbf{v}^{\varepsilon}\right\|_{0}+\left\| \partial^{\ell+1}  a^{\varepsilon} \right\|_{0} \lesssim\left\| \partial^{\ell} f\right\|_{0}+\left\| \partial^{\ell+1} g\right\|_{0},
\]
whose proof can be found in many references (See, e.g. \cite{4}).

It follows from \eqref{eq6} that
\[
\left\{
\begin{array}{ll}
-\mu \Delta Z^{\alpha}\mathbf{v}^{\varepsilon} + \nabla Z^{\alpha}a^{\varepsilon}
= -\partial_{t} Z^{\alpha}\mathbf{v}^{\varepsilon}
-(\mu+\lambda) Z^{\alpha}\nabla\left(a^{\varepsilon}\operatorname{div}\mathbf{v}^{\varepsilon}+\dfrac{\mathrm{D} a^{\varepsilon}}{\mathrm{D}t}\right)
+ Z^{\alpha}f^{\varepsilon}_2 \\
 + \mu \left[Z^{\alpha}, \Delta\right]\mathbf{v}^{\varepsilon} - \left[Z^{\alpha}, \nabla\right]a^{\varepsilon}
& \text{in } \mathbb{R}_+^3,\\[6pt]
\operatorname{div} Z^{\alpha}\mathbf{v}^{\varepsilon}
= - Z^{\alpha}\left(a^{\varepsilon}\operatorname{div}\mathbf{v}^{\varepsilon}+\dfrac{\mathrm{D} a^{\varepsilon}}{\mathrm{D}t}\right)
- \left[Z^{\alpha}, \operatorname{div}\right]\mathbf{v}^{\varepsilon}
& \text{in } \mathbb{R}_+^3,\\[6pt]
Z^{\alpha}\mathbf{v}^{\varepsilon} = \mathbf{0}
& \text{on } \mathbb{R}^2\times\left\{x_3=0\right\}.
\end{array}
\right.
\]
Thanks to the above estimate for the Stokes system,  we obtain the following inequality
\begin{align*}
& \left\| \partial Z^{\alpha}  a^{\varepsilon} \right\|_{0}^{2}+\left\| \partial^{2} Z^{\alpha} \mathbf{v}^{\varepsilon}\right\|_{0}^{2}\\
\lesssim{}& \left\|  \partial  Z^{\alpha}\left(\dfrac{\mathrm{D} a^{\varepsilon} }{\mathrm{D}t}\right)\right\|_{0}^{2}+\left\|  Z^{\alpha} \partial \left(\frac{\mathrm{D} a^{\varepsilon} }{\mathrm{D}t}\right)\right\|_{0}^{2}+\left\|\partial_{t} Z^{\alpha}  \mathbf{v}^{\varepsilon}\right\|_{0}^{2} \\
&+\left\| Z^{\alpha} \partial\left( a^{\varepsilon}\operatorname{div}\mathbf{v}^{\varepsilon}\right)\right\|_{0}^{2}+\left\| \partial Z^{\alpha}\left( a^{\varepsilon}\operatorname{div}\mathbf{v}^{\varepsilon}\right)\right\|_{0}^{2}+\left\|  Z^{\alpha} f^{\varepsilon}_2\right\|_{0}^{2}\\
&+\left\| \left[Z^{\alpha}, \Delta\right]\mathbf{v}^{\varepsilon}\right\|_{0}^{2}+  \left\|\left[Z^{\alpha}, \nabla\right]a^{\varepsilon}\right\|_{0}^{2}+\left\|\partial\left[Z^{\alpha}, \operatorname{div}\right]\mathbf{v}^{\varepsilon}\right\|_{0}^{2}.
\end{align*}
Then by integrating it with respect to time variable in $[0, t]$, the proof of inequality \eqref{G11} is done.
\end{proof}
\subsection{Estimate of the second order normal derivative of $\left(\mathbf{v}^{\varepsilon},\varepsilon\mathbf{B}^{\varepsilon}\right)$}
\begin{lemm}\label{GL7}Assume $\left(a^{\varepsilon},\mathbf{v}^{\varepsilon},\mathbf{B}^{\varepsilon}\right)$ is a solution to \eqref{eq6}, it holds that
\begin{equation}\label{G13}
\begin{aligned}
 \left\| Z^{\alpha}\partial_3^2  \left(\mathbf{v}^{\varepsilon},\varepsilon\mathbf{B}^{\varepsilon}\right)(t)\right\|_{0}^{2}
\lesssim{}&  \left\|  Z^{\alpha}\partial a^{\varepsilon} (t)\right\|_{0}^{2}+\left\|  Z^{\alpha} \partial_{h}\partial \mathbf{v}^{\varepsilon}(t)\right\|_{0}^{2}+\left\|  Z^{\alpha} \partial_{h}^2 \mathbf{B}^{\varepsilon}(t)\right\|_{0}^{2}\\
&+\left\|  Z^{\alpha} \partial_{t} \left(\mathbf{v}^{\varepsilon},\mathbf{B}^{\varepsilon}\right)(t)\right\|_{0}^{2}
+\left\| Z^{\alpha} \left(f^{\varepsilon}_2,f^{\varepsilon}_3\right)(t)\right\|_{0}^{2}
\end{aligned}
\end{equation}
for $|\alpha| \leq m-2$.
\end{lemm}
\begin{proof}
Applying $Z^{\alpha}$ to the equations of velocity and magnetic field in \eqref{eq6}, we get
\begin{align*}
-\mu  Z^{\alpha}\partial_3^2 \mathbf{v}^{\varepsilon}_{h}={}& - Z^{\alpha}\partial_{t}\mathbf{v}^{\varepsilon}_{h}+\mu  Z^{\alpha}\Delta_h\mathbf{v}^{\varepsilon}_{h}+(\mu+\lambda)  Z^{\alpha}\nabla_h \operatorname{div}\mathbf{v}^{\varepsilon}\\
& - Z^{\alpha}\nabla_h a^{\varepsilon} +  Z^{\alpha}\left(f^{\varepsilon}_2\right)_h,\\
-(2\mu+\lambda) Z^{\alpha}\partial_3^2 v_{3}^{\varepsilon}={}& -Z^{\alpha}\partial_{t}v_{3}^{\varepsilon}+\mu Z^{\alpha}\Delta_h v_{3}^{\varepsilon}+(\mu+\lambda)  Z^{\alpha}\partial_{3}\operatorname{div}_h\mathbf{v}^{\varepsilon}_{h} \\
& -  Z^{\alpha}\partial_{3} a^{\varepsilon} +Z^{\alpha}\left(f^{\varepsilon}_2\right)_{3},\\
-\varepsilon  Z^{\alpha}\partial_3^2 \mathbf{B}^{\varepsilon}={}& - Z^{\alpha}\partial_{t}\mathbf{B}^{\varepsilon}+Z^{\alpha}\Delta_h\mathbf{B}^{\varepsilon}+  Z^{\alpha}f^{\varepsilon}_3.
\end{align*}
Taking the $L^2$ norm on both sides of the above three identities, then we can obtain \eqref{G13}.
\end{proof}
\subsection{Synthesis}
We now chain all the estimates derived previously to conclude
the following lemma. To this end, for every  integer  $m\ge 2$, we define
\begin{equation}\label{G14}
L_{m}(t)\overset{def}{=}\left(\sum_{|\alpha|\leq m} \int_{0}^{t}  \left|\widetilde{A} _{\alpha}\right|  \mathrm{d} \tau+\sum_{|\alpha|\leq m-2}\int_{0}^{t}  \left|\widetilde{B} _{\alpha}\right|  \mathrm{d} \tau+ \sum_{|\alpha|\leq m-1} \int_{0}^{t}   \left(\left|\widetilde{C} _{\alpha}\right|
+\left|\widetilde{E} _{\alpha}\right|+\left|\widetilde{F} _{\alpha}\right|\right)   \mathrm{d} \tau\right)^{\frac{1}{2}}
\end{equation}
and
\begin{equation}\label{G15}
\begin{aligned}
N_{m}(t)\overset{\mathrm{def}}{=}{}
&\Bigg(\left\|\left(f^{\varepsilon}_2,f^{\varepsilon}_3\right)\right\|_{m-2}^{2}
+\int_{0}^{t}\Big(\left\|a^{\varepsilon}\operatorname{div}\mathbf{v}^{\varepsilon}\right\|_{m}^{2}
+\left\|\partial_{3}(a^{\varepsilon}\operatorname{div}\mathbf{v}^{\varepsilon})\right\|_{m-1}^{2}\Big)\mathrm{d}\tau\\
&+\int_{0}^{t}\Big(\left\|\partial_{\tau}(a^{\varepsilon}\operatorname{div}\mathbf{v}^{\varepsilon})\right\|_{m-2}^{2}
+\left\|f^{\varepsilon}_2\right\|_{m-1}^{2}\Big)\mathrm{d}\tau
+\sum_{|\alpha|\le m}\int_{0}^{t}\left|A_{\alpha}\right|\mathrm{d}\tau\\
&+\sum_{|\alpha|\le m-1}\int_{0}^{t}\big(\left|C_{\alpha}\right|+\left|D_{\alpha}\right|+\left|E_{\alpha}\right|\big)\mathrm{d}\tau
+\sum_{|\alpha|\le m-2}\int_{0}^{t}\left|B_{\alpha}\right|\mathrm{d}\tau\Bigg)^{\frac12}.
\end{aligned}
\end{equation}
Summing \eqref{G2}, \eqref{G5}--\eqref{G9}, \eqref{G11} and \eqref{G13} in order yields the following lemma.
\begin{lemm}\label{GL8}
Let $m\geq 4$ and assume $(a^{\varepsilon},\mathbf{v}^{\varepsilon},\mathbf{B}^{\varepsilon})$ is a solution to \eqref{eq6}. Then the following inequality holds for all $t \geq 0$.
\begin{equation}\label{G16}
\begin{aligned}
&E_{m}^{\varepsilon}(t)^{2}+D_{m}^{\varepsilon}(t)^{2}\\
 \lesssim{} &\left\|( a_{0} ,\mathbf{v}_{0},\mathbf{B}_{0})\right\|_{m}^2+
 \left\|\partial_3( a_{0} ,\mathbf{v}_{0},\mathbf{B}_{0})\right\|_{m-1}^2+\left\|\partial_3^2(\mathbf{v}_{0},\varepsilon\mathbf{B}_{0})\right\|_{m-2}^2\\
 &+L_{m}(t)^{2}+N_{m}(t)^{2}.
\end{aligned}
\end{equation}
\end{lemm}
\subsection{Estimates of $L_{m}(t)$ and $N_{m}(t)$}
To close the energy estimate, it suffices to establish the estimates for $L_{m}(t)$ and $N_{m}(t)$ which are stated in the following lemma.
\begin{lemm}\label{red}
Let $m\geq 4$ and assume that $E_m^{\varepsilon}(t) \leq 1$ for all $t \geq 0$. Then the following inequalities hold for all $t \geq 0$:
  \begin{equation}\label{G17}
    L_{m}(t)^2\lesssim \delta_4 D_{m}^{\varepsilon}(t)^2+D_{m-1}^{\varepsilon}(t)^2,
  \end{equation}
and
  \begin{equation}\label{G18}
    N_{m}(t)^2\lesssim \delta_4 \left(E_{m}^{\varepsilon}(t)^2+D_{m}^{\varepsilon}(t)^2\right)+\left(E_{m}^{\varepsilon}(t)^4+D_{m}^{\varepsilon}(t)^4\right)
  \end{equation}
for any \(\delta_4>0\).
\end{lemm}
\subsubsection{Commutator preliminaries}
To estimate the commutator terms in $L_m(t)$ and $N_m(t)$, we first state some properties of commutators which will be used later. For any integer $\alpha_3 \ge 1$,
\begin{align*}
	\left[Z_3^{\alpha_3},\partial_3\right]f = \sum_{k=0}^{\alpha_3-1} A_{k,\alpha_3}\left(\varphi'\right) \partial_3 Z_3^k f,
\end{align*}
where each $A_{k,\alpha_3}\left(\varphi'\right)$ is a harmless bounded function formed by finite linear combinations of $Z_3^i \left(\varphi'\right)^j$ with $i+j\le \alpha_3$ and $j\ge1$. Moreover, we have the weighted bound
\begin{align*}
	\left|Z_3^i \left(\varphi'\right)^j\right| \lesssim \frac{1}{x_3+1}, \quad \forall x_3\in\mathbb{R}_+, i\in\mathbb{N}, j\in\mathbb{N}^*,
\end{align*}
which yields the refined coefficient estimate
\begin{equation}\label{sawa}
	\left|A_{k,\alpha_3}\left(\varphi'\right)\right| \lesssim \frac{1}{x_3+1} \lesssim \frac{\varphi(x_3)}{x_3}, \quad \forall x_3\in\mathbb{R}_+, k\in\mathbb{N}.
\end{equation}
Similarly, we have
\begin{align*}
	\left[Z_3^{\alpha_3},\partial_3\right]f =&{} \sum_{k=0}^{\alpha_3-1} B_{k,\alpha_3}\left(\varphi'\right)Z_3^k \partial_3  f,\\
\left[Z_3^{\alpha_3},\partial_3^2\right]f =&{} \sum_{k=0}^{\alpha_3-1} C_{k,\alpha_3}\left(\varphi'\right)Z_3^k \partial_3^2  f+ \tilde{C}_{k,\alpha_3}\left(\varphi'\right)Z_3^k \partial_3  f,\\
=&{} \sum_{k=0}^{\alpha_3-1} D_{k,\alpha_3}\left(\varphi'\right)\partial_3^2Z_3^k   f+ \tilde{D}_{k,\alpha_3}\left(\varphi'\right)\partial_3Z_3^k   f,\\
=&{} \sum_{k=0}^{\alpha_3-1} E_{k,\alpha_3}\left(\varphi'\right)Z_3^k \partial_3^2  f+ \tilde{E}_{k,\alpha_3}\left(\varphi'\right)\partial_3Z_3^k   f,\\
=&{} \sum_{k=0}^{\alpha_3-1} F_{k,\alpha_3}\left(\varphi'\right)\partial_3^2 Z_3^k  f+ \tilde{F}_{k,\alpha_3}\left(\varphi'\right)Z_3^k \partial_3  f,
\end{align*}
where all coefficients $B_{k,\alpha_3}, C_{k,\alpha_3}, \tilde{C}_{k,\alpha_3}, D_{k,\alpha_3}, \tilde{D}_{k,\alpha_3}, E_{k,\alpha_3}, \tilde{E}_{k,\alpha_3}, F_{k,\alpha_3}$, and $\tilde{F}_{k,\alpha_3}$ enjoy similar property to that in \eqref{sawa}.
We omit the proof of the above properties, for which we refer the reader to \cite{9a}.
\subsubsection{Estimates of $L_{m}(t)$}
We first estimate the terms in $L_m(t)$. For
$\sum_{|\alpha|\leq m} \int_{0}^{t}  \left|\widetilde{A} _{\alpha}\right|  \mathrm{d} \tau$,
by the relevant properties of commutators and applying Young's inequality, we have
\begin{equation}\label{G19}
\begin{aligned}
&\sum_{|\alpha|\leq m} \int_{0}^{t}  \left|\widetilde{A} _{\alpha}\right|  \mathrm{d} \tau\\
\lesssim{}& \sum_{|\alpha|\leq m}\int_{0}^{t}\bigg(\left\|\left[Z^{\alpha}, \partial_3\right]\mathbf{v}^{\varepsilon}\right\|_0^2+\left\| [Z^{\alpha}, \partial_3]\mathbf{v}^{\varepsilon}\right\|_0\left\|  Z^{\alpha} a^{\varepsilon}\right\|_0
+\varepsilon\left\|\left[Z^{\alpha}, \partial_3^2\right]\mathbf{B}^{\varepsilon}\right\|_0\left\|  Z^{\alpha} \mathbf{B}^{\varepsilon} \right\|_0\\
&+\left(\left\|\left[Z^{\alpha}, \partial_3^2\right]\mathbf{v}^{\varepsilon}\right\|_0
+\left\|\partial_h[Z^{\alpha}, \partial_3]\mathbf{v}^{\varepsilon} \right\|_0
+\left\|[Z^{\alpha}, \partial_3]a^{\varepsilon}\right\|_0\right)\left\|   Z^{\alpha} \mathbf{v}^{\varepsilon}\right\|_0\bigg) \mathrm{d}\tau\\
\lesssim{}&\int_{0}^{t}\bigg(\delta_4\Big(\left\|\partial_3 a^{\varepsilon} \right\|_{m-1}^2+\left\|\partial_3\mathbf{v}^{\varepsilon}\right\|_{m}^2+\left\|\partial_3^2\mathbf{v}^{\varepsilon}\right\|_{m-1}^2+\varepsilon\left\| \partial_3\mathbf{B}^{\varepsilon}\right\|_{m}^2 +\varepsilon\left\| \partial_3^2 \mathbf{B}^{\varepsilon}\right\|_{m-1}^2\Big)\\
&+\left\|\partial_3\mathbf{v}^{\varepsilon}\right\|_{m-1}^2+\varepsilon \left\|\partial_3 \mathbf{B}^{\varepsilon}\right\|_{m-1}^2\bigg) \mathrm{d} \tau.
\end{aligned}
\end{equation}
We decompose $\widetilde{B}_{\alpha}$ into three parts
\begin{align*}
  \widetilde{B} _{\alpha}={}&\left\|\left[Z^{\alpha}, \partial_3\right]\partial_t\mathbf{v}^{\varepsilon}\right\|_0^2-\left( \left[Z^{\alpha}, \operatorname{div}\right]\partial_{t}\mathbf{v}^{\varepsilon},  Z^{\alpha}\partial_{t} a^{\varepsilon}\right)\\
  &+  \left(\varepsilon\left[Z^{\alpha}, \partial_3^2\right]\partial_{t}\mathbf{B}^{\varepsilon},  Z^{\alpha}\partial_{t} \mathbf{B}^{\varepsilon} \right)\\
  &+\left(\mu\left[Z^{\alpha}, \Delta\right]\partial_{t}\mathbf{v}^{\varepsilon} + (\mu + \lambda)\left[Z^{\alpha}, \nabla\operatorname{div}\right]\partial_{t}\mathbf{v}^{\varepsilon} - \left[Z^{\alpha}, \nabla\right]\partial_{t}a^{\varepsilon},  Z^{\alpha}\partial_{t} \mathbf{v}^{\varepsilon}\right)\\
\overset{def}{=}{}&\widetilde{B} _{\alpha,1}+\widetilde{B} _{\alpha,2}+\widetilde{B} _{\alpha,3}.
\end{align*}
For $\widetilde{B}_{\alpha,1}$, it holds
\begin{align*}
\left|\widetilde{B} _{\alpha,1}\right|\lesssim\left\|\left[Z^{\alpha}, \partial_3\right]\partial_t\mathbf{v}^{\varepsilon}\right\|_0^2+\left\| [Z^{\alpha}, \partial_3]\partial_{t}\mathbf{v}^{\varepsilon}\right\|_0\left\|  Z^{\alpha} \partial_{t}a^{\varepsilon}\right\|_0\lesssim\left\|\partial_t a^{\varepsilon}\right\|_{m-2}^2+\left\|\partial_3\partial_t\mathbf{v}^{\varepsilon}\right\|_{m-3}^2.
\end{align*}
Without loss of generality, we assume $\alpha_3>0$, and by the properties of commutators, integration by parts, and Young's inequality, we derive the following estimate for $\widetilde{B}_{\alpha,2}$.
\begin{align*}
\left|\widetilde{B} _{\alpha,2}\right|\lesssim{}& \varepsilon \sum_{k=0}^{\alpha_3-1} \left|\int_{\mathbb{R}_+^3} \partial_1^{\alpha_1}
	\partial_2^{\alpha_2} \left( \hat{C}_{k,\alpha_3}\left(\varphi'\right) Z^k_3 \partial_3 \partial_{t} \mathbf{B}^{\varepsilon}
	+ \tilde{C}_{k,\alpha_3}\left(\varphi'\right) \partial_3^2 Z^k_3 \partial_{t} \mathbf{B}^{\varepsilon} \right)\cdot Z^\alpha \partial_{t} \mathbf{B}^{\varepsilon} \mathrm{d}\mathbf{x}\right|\\
\lesssim{}& \varepsilon \sum_{k=0}^{\alpha_3-1} \left|\int_{\mathbb{R}_+^3} \hat{C}_{k,\alpha_3}\left(\varphi'\right) \partial_1^{\alpha_1}
	\partial_2^{\alpha_2}Z^k_3 \partial_3 \partial_{t} \mathbf{B}^{\varepsilon}\cdot Z^\alpha \partial_{t} \mathbf{B}^{\varepsilon} \mathrm{d}\mathbf{x}\right|\\
&+ \varepsilon \sum_{k=0}^{\alpha_3-1} \left|\int_{\mathbb{R}_+^3} \partial_3 \tilde{C}_{k,\alpha_3}\left(\varphi'\right) \partial_3 \partial_1^{\alpha_1} \partial_2^{\alpha_2} Z^k_3 \partial_t\mathbf{B}^{\varepsilon} \cdot Z^\alpha \partial_t\mathbf{B}^{\varepsilon} \mathrm{d}\mathbf{x}\right|\\
	&+ \varepsilon \sum_{k=0}^{\alpha_3-1}  \left|\int_{\mathbb{R}_+^3} \tilde{C}_{k,\alpha_3}\left(\varphi'\right) \partial_3 \partial_1^{\alpha_1} \partial_2^{\alpha_2} Z^k_3 \partial_t\mathbf{B}^{\varepsilon} \cdot \partial_3 Z^\alpha \partial_t\mathbf{B}^{\varepsilon} \mathrm{d}\mathbf{x}\right|\\
\lesssim{}&\delta_4\varepsilon\left\| \partial_3\partial_{t}\mathbf{B}^{\varepsilon}\right\|_{m-2}^2+\varepsilon\left\| \partial_3\partial_{t}\mathbf{B}^{\varepsilon}\right\|_{m-3}^2,
\end{align*}
where $\hat{C}_{k,\alpha_3}\left(\varphi'\right)$ and $\tilde{C}_{k,\alpha_3}\left(\varphi'\right)$ denote the harmless functions depending on derivatives of $\varphi$.
Similarly, it is direct to check that
\begin{align*}
\left|\widetilde{B} _{\alpha,3}\right|\lesssim{}&\sum_{k=0}^{\alpha_3-1} \left|\int_{\mathbb{R}_+^3} \partial_1^{\alpha_1}
	\partial_2^{\alpha_2} \left( \hat{C}_{k,\alpha_3}\left(\varphi'\right) Z^k_3 \partial_3 \partial_{t} \mathbf{v}^{\varepsilon}
	+ \tilde{C}_{k,\alpha_3}\left(\varphi'\right) \partial_3^2 Z^k_3 \partial_{t} \mathbf{v}^{\varepsilon} \right)\cdot Z^\alpha \partial_{t} \mathbf{v}^{\varepsilon} \mathrm{d}\mathbf{x}\right|\\
&+\sum_{k=0}^{\alpha_3-1} \left|\int_{\mathbb{R}_+^3}C_{k,\alpha_3}\left(\varphi'\right) \partial_{h}\partial_1^{\alpha_1}
	\partial_2^{\alpha_2} Z^k_3 \partial_3 \partial_{t} \mathbf{v}^{\varepsilon}\cdot Z^\alpha \partial_{t} \mathbf{v}^{\varepsilon} \mathrm{d}\mathbf{x}\right|\\
&+\sum_{k=0}^{\alpha_3-1} \left|\int_{\mathbb{R}_+^3}\bar{C}_{k,\alpha_3}\left(\varphi'\right) \partial_1^{\alpha_1}
	\partial_2^{\alpha_2} \partial_3Z^k_3 \partial_{t} a^{\varepsilon}\cdot Z^\alpha \partial_{t} \mathbf{v}^{\varepsilon} \mathrm{d}\mathbf{x}\right|\\
\lesssim{}&\sum_{k=0}^{\alpha_3-1} \left|\int_{\mathbb{R}_+^3} \hat{C}_{k,\alpha_3}\left(\varphi'\right)\partial_1^{\alpha_1}
	\partial_2^{\alpha_2} Z^k_3 \partial_3 \partial_{t} \mathbf{v}^{\varepsilon}\cdot Z^\alpha \partial_{t} \mathbf{v}^{\varepsilon} \mathrm{d}\mathbf{x}\right|\\
&+\sum_{k=0}^{\alpha_3-1} \left|\int_{\mathbb{R}_+^3} \partial_3 \tilde{C}_{k,\alpha_3}\left(\varphi'\right) \partial_3 \partial_1^{\alpha_1} \partial_2^{\alpha_2} Z^k_3 \partial_t\mathbf{v}^{\varepsilon} \cdot Z^\alpha \partial_t\mathbf{v}^{\varepsilon} \mathrm{d}\mathbf{x}\right|\\
	&+\sum_{k=0}^{\alpha_3-1}  \left|\int_{\mathbb{R}_+^3} \tilde{C}_{k,\alpha_3}\left(\varphi'\right) \partial_3 \partial_1^{\alpha_1} \partial_2^{\alpha_2} Z^k_3 \partial_t\mathbf{v}^{\varepsilon} \cdot \partial_3 Z^\alpha \partial_t\mathbf{v}^{\varepsilon} \mathrm{d}\mathbf{x}\right|\\
&+\sum_{k=0}^{\alpha_3-1} \left|\int_{\mathbb{R}_+^3}C_{k,\alpha_3}\left(\varphi'\right) \partial_{h}\partial_1^{\alpha_1}
	\partial_2^{\alpha_2} Z^k_3 \partial_3 \partial_{t} \mathbf{v}^{\varepsilon}\cdot Z^\alpha \partial_{t} \mathbf{v}^{\varepsilon} \mathrm{d}\mathbf{x}\right|\\
&+\sum_{k=0}^{\alpha_3-1} \left|\int_{\mathbb{R}_+^3} \partial_3 \bar{C}_{k,\alpha_3}\left(\varphi'\right) \partial_1^{\alpha_1} \partial_2^{\alpha_2} Z^k_3 \partial_t a^{\varepsilon} \cdot Z^\alpha \partial_t\mathbf{v}^{\varepsilon} \mathrm{d}\mathbf{x}\right|\\
	&+\sum_{k=0}^{\alpha_3-1}  \left|\int_{\mathbb{R}_+^3} \bar{C}_{k,\alpha_3}\left(\varphi'\right) \partial_1^{\alpha_1} \partial_2^{\alpha_2} Z^k_3 \partial_t a^{\varepsilon} \cdot \partial_3 Z^\alpha \partial_t\mathbf{v}^{\varepsilon} \mathrm{d}\mathbf{x}\right|\\
\lesssim{}&\delta_4\left\| \partial_3\partial_{t}\mathbf{v}^{\varepsilon}\right\|_{m-2}^2+\left\|\partial_{t}\left(a^{\varepsilon}, \mathbf{v}^{\varepsilon}\right)\right\|_{m-2}^2+\left\| \partial_3\partial_{t}\mathbf{v}^{\varepsilon}\right\|_{m-3}^2.
\end{align*}
Combining the above estimates gives
\begin{align*}
\left|\widetilde{B} _{\alpha}\right|\lesssim{}&\delta_4\left(\left\| \partial_3\partial_{t}\mathbf{v}^{\varepsilon}\right\|_{m-2}^2+\varepsilon\left\| \partial_3\partial_{t}\mathbf{B}^{\varepsilon}\right\|_{m-2}^2\right)+\left\|\partial_{t}\left(a^{\varepsilon}, \mathbf{v}^{\varepsilon},\mathbf{B}^{\varepsilon}\right)\right\|_{m-2}^2\\
&+\left\| \partial_3\partial_{t}\mathbf{v}^{\varepsilon}\right\|_{m-3}^2+\varepsilon\left\| \partial_3\partial_{t}\mathbf{B}^{\varepsilon}\right\|_{m-3}^2.
\end{align*}
Integrating it with respect to time variable in $[0, t]$ leads to
\begin{equation}\label{G25}
\begin{aligned}
&\sum_{|\alpha|\leq m-2}\int_{0}^{t}  \left|\widetilde{B} _{\alpha}\right|  \mathrm{d} \tau\\
\lesssim{}&\int_{0}^{t}\Big(\delta_4\left(\left\| \partial_3\partial_{\tau}\mathbf{v}^{\varepsilon}\right\|_{m-2}^2+\varepsilon\left\| \partial_3\partial_{\tau}\mathbf{B}^{\varepsilon}\right\|_{m-2}^2\right)+\left\|\partial_{\tau}\left(a^{\varepsilon}, \mathbf{v}^{\varepsilon},\mathbf{B}^{\varepsilon}\right)\right\|_{m-2}^2\\
&+\left\| \partial_3\partial_{\tau}\mathbf{v}^{\varepsilon}\right\|_{m-3}^2+\varepsilon\left\| \partial_3\partial_{\tau}\mathbf{B}^{\varepsilon}\right\|_{m-3}^2\Big) \mathrm{d}\tau.
\end{aligned}
\end{equation}
For
$\sum_{|\alpha|\leq m-1} \int_{0}^{t}  \left|\widetilde{C} _{\alpha}\right|  \mathrm{d} \tau$,
combining the relevant properties of commutators and Young's inequality yields
\begin{equation}\label{G26}
\begin{aligned}
&\sum_{|\alpha|\leq m-1} \int_{0}^{t}  \left|\widetilde{C} _{\alpha}\right|  \mathrm{d} \tau\\
\lesssim{}& \sum_{|\alpha|\leq m-1}\int_{0}^{t}\Big(\left\| [Z^{\alpha}, \partial_3]\mathbf{v}^{\varepsilon}\right\|_0\left\|  Z^{\alpha}\partial_{\tau} a^{\varepsilon}\right\|_0
+\varepsilon\left\|\left[Z^{\alpha}, \partial_3^2\right]\mathbf{B}^{\varepsilon}\right\|_0\left\|  Z^{\alpha} \partial_{\tau}\mathbf{B}^{\varepsilon} \right\|_0\\
&+\left(\left\|\left[Z^{\alpha}, \partial_3^2\right]\mathbf{v}^{\varepsilon}\right\|_0
+\left\|\partial_h[Z^{\alpha}, \partial_3]\mathbf{v}^{\varepsilon} \right\|_0
+\left\|[Z^{\alpha}, \partial_3]a^{\varepsilon}\right\|_0\right)\left\|   Z^{\alpha}\partial_{\tau} \mathbf{v}^{\varepsilon}\right\|_0\Big) \mathrm{d}\tau\\
\lesssim{}&\int_{0}^{t}\Big(\delta_4 \left\|\partial_{\tau}\left(a^{\varepsilon}, \mathbf{v}^{\varepsilon},\mathbf{B}^{\varepsilon}\right)\right\|_{m-1}^2+\left\|\partial_3 a^{\varepsilon} \right\|_{m-2}^2+\left\|\partial_3\mathbf{v}^{\varepsilon}\right\|_{m-1}^2+\left\|\partial_3^2\mathbf{v}^{\varepsilon}\right\|_{m-2}^2\\
&+\varepsilon \left\|\partial_3 \mathbf{B}^{\varepsilon}\right\|_{m-1}^2+\varepsilon\left\| \partial_3^2 \mathbf{B}^{\varepsilon}\right\|_{m-2}^2\Big) \mathrm{d} \tau.
\end{aligned}
\end{equation}
Similar to the previous argument, we obtain the following estimate for $\widetilde{E}_{\alpha}$.
\begin{align*}
\left|\widetilde{E} _{\alpha}\right|\lesssim{}& \varepsilon \sum_{k=0}^{\alpha_3-1} \left|\int_{\mathbb{R}_+^3} \partial_1^{\alpha_1}
	\partial_2^{\alpha_2} \left( \hat{C}_{k,\alpha_3}\left(\varphi'\right) Z^k_3\partial_{3}^2\mathbf{B}^{\varepsilon}
	+ \tilde{C}_{k,\alpha_3}\left(\varphi'\right) \partial_3^2 Z^k_3 \partial_{3} \mathbf{B}^{\varepsilon} \right)\cdot Z^\alpha \partial_{3} \mathbf{B}^{\varepsilon} \mathrm{d}\mathbf{x}\right|\\
\lesssim{}& \varepsilon \sum_{k=0}^{\alpha_3-1} \left|\int_{\mathbb{R}_+^3} \hat{C}_{k,\alpha_3}\left(\varphi'\right)\partial_1^{\alpha_1}
	\partial_2^{\alpha_2} Z^k_3\partial_{3}^2 \mathbf{B}^{\varepsilon}\cdot Z^\alpha \partial_{3} \mathbf{B}^{\varepsilon} \mathrm{d}\mathbf{x}\right|\\
&+ \varepsilon \sum_{k=0}^{\alpha_3-1} \left|\int_{\mathbb{R}_+^3} \partial_3 \tilde{C}_{k,\alpha_3}\left(\varphi'\right) \partial_3 \partial_1^{\alpha_1} \partial_2^{\alpha_2} Z^k_3 \partial_3\mathbf{B}^{\varepsilon} \cdot Z^\alpha \partial_3\mathbf{B}^{\varepsilon} \mathrm{d}\mathbf{x}\right|\\
	&+ \varepsilon \sum_{k=0}^{\alpha_3-1}  \left|\int_{\mathbb{R}_+^3} \tilde{C}_{k,\alpha_3}\left(\varphi'\right) \partial_3 \partial_1^{\alpha_1} \partial_2^{\alpha_2} Z^k_3 \partial_3\mathbf{B}^{\varepsilon} \cdot \partial_3 Z^\alpha \partial_3\mathbf{B}^{\varepsilon} \mathrm{d}\mathbf{x}\right|\\
\lesssim{}&\delta_4\varepsilon\left\|\partial_{3}^2\mathbf{B}^{\varepsilon}\right\|_{m-1}^2+\varepsilon\left\|\partial_{3}\mathbf{B}^{\varepsilon}\right\|_{m-1}^2+\varepsilon\left\| \partial_3^2\mathbf{B}^{\varepsilon}\right\|_{m-2}^2.
\end{align*}
Thus, its integration over $[0, t]$ satisfies
\begin{equation}\label{G28}
\begin{aligned}
\sum_{|\alpha|\leq m-1}\int_{0}^{t}  \left|\widetilde{E} _{\alpha}\right|  \mathrm{d} \tau\lesssim\varepsilon\int_{0}^{t}\Big(\delta_4\left\|\partial_{3}^2\mathbf{B}^{\varepsilon}\right\|_{m-1}^2+\left\|\partial_{3}\mathbf{B}^{\varepsilon}\right\|_{m-1}^2+\left\| \partial_3^2\mathbf{B}^{\varepsilon}\right\|_{m-2}^2\Big) \mathrm{d}\tau.
\end{aligned}
\end{equation}
Finally, let us deal with $\sum_{|\alpha|\leq m-1} \int_{0}^{t}  \left|\widetilde{F} _{\alpha}\right|  \mathrm{d} \tau$.
\begin{equation}\label{G29}
\begin{aligned}
&\sum_{|\alpha|\leq m-1} \int_{0}^{t}  \left|\widetilde{F} _{\alpha}\right|  \mathrm{d} \tau\\
\lesssim{}& \sum_{|\alpha|\leq m-1}\int_{0}^{t}\Big(\left\| \left[Z^{\alpha}, \partial_3^2\right]\mathbf{v}^{\varepsilon}\right\|_{0}^{2}+  \left\|[Z^{\alpha}, \partial_3]a^{\varepsilon}\right\|_{0}^{2}+\left\|\partial[Z^{\alpha}, \partial_3]\mathbf{v}^{\varepsilon}\right\|_{0}^{2}\Big) \mathrm{d}\tau\\
\lesssim{}&\int_{0}^{t}\Big(\left\|\partial_3 a^{\varepsilon} \right\|_{m-2}^2+\left\|\partial_3\mathbf{v}^{\varepsilon}\right\|_{m-1}^2+\left\|\partial_3^2\mathbf{v}^{\varepsilon}\right\|_{m-2}^2\Big) \mathrm{d} \tau.
\end{aligned}
\end{equation}
Then, substituting \eqref{G19}--\eqref{G29} into \eqref{G14}, we arrive at \eqref{G17}.
\subsubsection{Estimates of $N_{m}(t)$}
Next, we proceed to estimate the terms in $N_{m}(t)$. As for $\int_{0}^{t} \left\| a^{\varepsilon}  \operatorname{div}\mathbf{v}^{\varepsilon} \right\|_{m}^{2} \mathrm{d} \tau$, by Hölder's inequality, Lemmas \ref{Le1} and \ref{Le2}, we obtain
\begin{align}\label{G30}
	&\int_{0}^{t} \left\| a^{\varepsilon}  \operatorname{div}\mathbf{v}^{\varepsilon}\right\|_{m}^{2}\mathrm{d}\tau \nonumber\\
    \lesssim{}&\sum_{\substack{|\alpha|\leq m\\\beta+\gamma=\alpha}}\int_{0}^{t} \left\|Z^\beta a^{\varepsilon}  \cdot Z^\gamma \operatorname{div}\mathbf{v}^{\varepsilon}\right\|_0^2 \mathrm{d}\tau \nonumber\\
    \lesssim{}&\sum_{|\alpha|\leq m}\int_{0}^{t} \left(\left\| a^{\varepsilon}  \right\|_{L^\infty}^2\cdot\left\| Z^\alpha \operatorname{div}\mathbf{v}^{\varepsilon}\right\|_{0}^2+\left\|Z^\alpha a^{\varepsilon} \right\|_0^2 \cdot \left\|\operatorname{div}\mathbf{v}^{\varepsilon}\right\|_{L^\infty}^2\right) \mathrm{d}\tau \nonumber\\
    &+\sum_{\substack{|\alpha|\leq m\\\beta+\gamma=\alpha\\\beta\neq0,\gamma\neq0}}\int_{0}^{t} \left\|Z^\beta a^{\varepsilon} \right\|_{L^6}^2 \cdot \left\|Z^\gamma \operatorname{div}\mathbf{v}^{\varepsilon}\right\|_{L^3}^2 \mathrm{d}\tau \nonumber\\
	\lesssim{}&\sum_{|\alpha|\leq m}\int_{0}^{t}\left( \left\| a^{\varepsilon}  \right\|_2\left\| \partial_3 a^{\varepsilon}  \right\|_2\left\| Z^\alpha \operatorname{div}\mathbf{v}^{\varepsilon}\right\|_{0}^2
+ \left\|Z^\alpha a^{\varepsilon} \right\|_0^2\left\|\operatorname{div}\mathbf{v}^{\varepsilon}\right\|_{2}\left\|\partial_3\operatorname{div}\mathbf{v}^{\varepsilon}\right\|_{2}\right) \mathrm{d}\tau \nonumber\\
    &+\sum_{\substack{|\alpha|\leq m\\\beta+\gamma=\alpha\\\beta\neq0,\gamma\neq0}}\int_{0}^{t} \left\|\partial Z^\beta a^{\varepsilon} \right\|_{0}^2 \cdot \left\|Z^\gamma \operatorname{div}\mathbf{v}^{\varepsilon}\right\|_{H^1}^2 \mathrm{d}\tau \nonumber\\
    \lesssim{}&\sup_{0\leq \tau \leq t}\left(\left\| a^{\varepsilon}\right\|_{m}^{2} + \left\| \partial_3 a^{\varepsilon}\right\|_{m-1}^{2}\right)\int_{0}^{t}\left(\left\|\partial\mathbf{v}^{\varepsilon}\right\|_{m}^2+\left\|\partial_3\partial\mathbf{v}^{\varepsilon}\right\|_{m-1}^2\right) \mathrm{d}\tau.
\end{align}
We split the integral $\int_{0}^{t} \left\| \partial_3\left(a^{\varepsilon}  \operatorname{div}\mathbf{v}^{\varepsilon}\right) \right\|_{m-1}^{2} \mathrm{d} \tau$ into two terms to estimate.
\begin{align*}
	\int_{0}^{t} \left\| \partial_3\left(a^{\varepsilon}  \operatorname{div}\mathbf{v}^{\varepsilon}\right)\right\|_{m-1}^{2}\mathrm{d}\tau\lesssim\int_{0}^{t}\left( \left\| \partial_3 a^{\varepsilon}  \cdot\operatorname{div}\mathbf{v}^{\varepsilon}\right\|_{m-1}^{2}+ \left\| a^{\varepsilon}  \cdot\partial_3\operatorname{div}\mathbf{v}^{\varepsilon}\right\|_{m-1}^{2}\right)\mathrm{d}\tau.
\end{align*}
By Hölder's inequality and Lemma \ref{Le1}, the first term on the right-hand side can be estimated as follows.
\begin{align*}
	&\int_{0}^{t} \left\| \partial_3 a^{\varepsilon}  \cdot\operatorname{div}\mathbf{v}^{\varepsilon}\right\|_{m-1}^{2}\mathrm{d}\tau\\
    \lesssim{}&\sum_{\substack{|\alpha|\leq m-1\\\beta+\gamma=\alpha}}\int_{0}^{t} \left\|Z^\beta\partial_3 a^{\varepsilon}  \cdot Z^\gamma \operatorname{div}\mathbf{v}^{\varepsilon}\right\|_0^2 \mathrm{d}\tau\\
    \lesssim{}&\sum_{|\alpha|\leq m-1}\int_{0}^{t}\left( \left\|\partial_3  a^{\varepsilon}  \cdot Z^\alpha \operatorname{div}\mathbf{v}^{\varepsilon}\right\|_{0}^2+\left\|Z^\alpha \partial_3 a^{\varepsilon} \right\|_0^2 \cdot \left\|\operatorname{div}\mathbf{v}^{\varepsilon}\right\|_{L^\infty}^2\right) \mathrm{d}\tau\\
    &+\sum_{\substack{|\alpha|\leq m-1\\\beta+\gamma=\alpha\\\beta\neq0,\gamma\neq0}}\int_{0}^{t} \left\|Z^\beta\partial_3 a^{\varepsilon} \cdot Z^\gamma \operatorname{div}\mathbf{v}^{\varepsilon}\right\|_{0}^2 \mathrm{d}\tau\\
	\lesssim{}&\sum_{|\alpha|\leq m-1}\int_{0}^{t} \left\| \partial_{3} a^{\varepsilon}  \right\|_0^{\frac{1}{2}}\left\| \partial_{13} a^{\varepsilon}  \right\|_0^{\frac{1}{2}}\left\| \partial_{23} a^{\varepsilon}  \right\|_0^{\frac{1}{2}}\left\| \partial_{123} a^{\varepsilon}  \right\|_0^{\frac{1}{2}}\left\| Z^\alpha \operatorname{div}\mathbf{v}^{\varepsilon}\right\|_{0}\left\| \partial_{3}Z^\alpha \operatorname{div}\mathbf{v}^{\varepsilon}\right\|_{0} \mathrm{d}\tau\\
    &+\sum_{|\alpha|\leq m-1}\int_{0}^{t} \left\|Z^\alpha\partial_3 a^{\varepsilon} \right\|_0^2\left\|\operatorname{div}\mathbf{v}^{\varepsilon}\right\|_{2}\left\|\partial_3\operatorname{div}\mathbf{v}^{\varepsilon}\right\|_{2} \mathrm{d}\tau\\
    &+\sum_{\substack{|\alpha|\leq m-1\\\beta+\gamma=\alpha\\\beta\neq0,\gamma\neq0}}\int_{0}^{t} \left\|Z^\beta\partial_3 a^{\varepsilon} \right\|_0\left\|Z^\beta\partial_{13} a^{\varepsilon} \right\|_0
    \left\| Z^\gamma \operatorname{div}\mathbf{v}^{\varepsilon}\right\|_{0}^{\frac{1}{2}}\left\|\partial_{2} Z^\gamma \operatorname{div}\mathbf{v}^{\varepsilon}\right\|_{0}^{\frac{1}{2}}\\
    &\cdot\left\|\partial_{3} Z^\gamma \operatorname{div}\mathbf{v}^{\varepsilon}\right\|_{0}^{\frac{1}{2}}\left\|\partial_{23} Z^\gamma \operatorname{div}\mathbf{v}^{\varepsilon}\right\|_{0}^{\frac{1}{2}} \mathrm{d}\tau\\
    \lesssim{}&\sup_{0\leq \tau \leq t}\left\| \partial_3 a^{\varepsilon}\right\|_{m-1}^{2}\int_{0}^{t}\left(\left\|\partial\mathbf{v}^{\varepsilon}\right\|_{m}^2+\left\|\partial_3\partial\mathbf{v}^{\varepsilon}\right\|_{m-1}^2\right) \mathrm{d}\tau.
\end{align*}
Similarly,
\begin{align*}
	&\int_{0}^{t} \left\|  a^{\varepsilon}  \cdot\partial_3\operatorname{div}\mathbf{v}^{\varepsilon}\right\|_{m-1}^{2}\mathrm{d}\tau\\
    \lesssim{}&\sum_{\substack{|\alpha|\leq m-1\\\beta+\gamma=\alpha}}\int_{0}^{t} \left\|Z^\beta a^{\varepsilon}  \cdot Z^\gamma \partial_3\operatorname{div}\mathbf{v}^{\varepsilon}\right\|_0^2 \mathrm{d}\tau\\
    \lesssim{}&\sum_{|\alpha|\leq m-1}\int_{0}^{t}\left( \left\|a^{\varepsilon} \right\|_{L^{\infty}}^2  \left\|Z^\alpha \partial_3\operatorname{div}\mathbf{v}^{\varepsilon}\right\|_{0}^2+\left\|Z^\alpha a^{\varepsilon} \cdot\partial_3\operatorname{div}\mathbf{v}^{\varepsilon}\right\|_{0}^2 \right)\mathrm{d}\tau\\
    &+\sum_{\substack{|\alpha|\leq m-1\\\beta+\gamma=\alpha\\\beta\neq0,\gamma\neq0}}\int_{0}^{t} \left\|Z^\beta a^{\varepsilon} \cdot Z^\gamma\partial_3 \operatorname{div}\mathbf{v}^{\varepsilon}\right\|_{0}^2 \mathrm{d}\tau\\
	\lesssim{}&\sum_{|\alpha|\leq m-1}\int_{0}^{t} \left\| a^{\varepsilon}  \right\|_{2}\left\| \partial_{3} a^{\varepsilon}  \right\|_{2}\left\| Z^\alpha \partial_3\operatorname{div}\mathbf{v}^{\varepsilon}\right\|_{0}^2 \mathrm{d}\tau\\
    &+\sum_{|\alpha|\leq m-1}\int_{0}^{t} \left\|Z^\alpha a^{\varepsilon} \right\|_0\left\|\partial_3Z^\alpha a^{\varepsilon} \right\|_0\left\|\partial_{3}\operatorname{div}\mathbf{v}^{\varepsilon}\right\|_{0}^{\frac{1}{2}}\left\|\partial_{13}\operatorname{div}\mathbf{v}^{\varepsilon}\right\|_{0}^{\frac{1}{2}}\\
    &\cdot\left\|\partial_{23}\operatorname{div}\mathbf{v}^{\varepsilon}\right\|_{0}^{\frac{1}{2}}\left\|\partial_{123}\operatorname{div}\mathbf{v}^{\varepsilon}\right\|_{0}^{\frac{1}{2}} \mathrm{d}\tau\\
    &+\sum_{\substack{|\alpha|\leq m-1\\\beta+\gamma=\alpha\\\beta\neq0,\gamma\neq0}}\int_{0}^{t} \left\|Z^\beta a^{\varepsilon} \right\|_0^{\frac{1}{2}}\left\|\partial_{2}Z^\beta a^{\varepsilon} \right\|_0^{\frac{1}{2}}\left\|\partial_{3}Z^\beta a^{\varepsilon} \right\|_0^{\frac{1}{2}}\left\|\partial_{23}Z^\beta a^{\varepsilon} \right\|_0^{\frac{1}{2}}\\
    &\cdot\left\| Z^\gamma\partial_{3} \operatorname{div}\mathbf{v}^{\varepsilon}\right\|_{0}\left\| \partial_1 Z^\gamma \partial_{3}\operatorname{div}\mathbf{v}^{\varepsilon}\right\|_{0} \mathrm{d}\tau\\
    \lesssim{}&\sup_{0\leq \tau \leq t}\left(\left\| a^{\varepsilon}\right\|_{m}^{2} + \left\| \partial_3 a^{\varepsilon}\right\|_{m-1}^{2}\right)\int_{0}^{t}\left(\left\|\partial\mathbf{v}^{\varepsilon}\right\|_{m}^2+\left\|\partial_3\partial\mathbf{v}^{\varepsilon}\right\|_{m-1}^2\right) \mathrm{d}\tau.
\end{align*}
Combining the above two estimates implies
\begin{equation}\label{G34}
\begin{aligned}
	\int_{0}^{t} \left\| \partial_3\left(a^{\varepsilon}  \operatorname{div}\mathbf{v}^{\varepsilon}\right)\right\|_{m-1}^{2}\mathrm{d}\tau\lesssim\sup_{0\leq \tau \leq t}\left(\left\| a^{\varepsilon}\right\|_{m}^{2} + \left\| \partial_3 a^{\varepsilon}\right\|_{m-1}^{2}\right)\int_{0}^{t}\left(\left\|\partial\mathbf{v}^{\varepsilon}\right\|_{m}^2+\left\|\partial_3\partial\mathbf{v}^{\varepsilon}\right\|_{m-1}^2\right) \mathrm{d}\tau.
\end{aligned}
\end{equation}
We decompose the integral $\int_{0}^{t} \left\| \partial_{\tau}\left(a^{\varepsilon}  \operatorname{div}\mathbf{v}^{\varepsilon}\right)\right\|_{m-2}^{2}\mathrm{d}\tau$ into two terms.
\[
	\int_{0}^{t} \left\| \partial_{\tau}\left(a^{\varepsilon}  \operatorname{div}\mathbf{v}^{\varepsilon}\right)\right\|_{m-2}^{2}\mathrm{d}\tau\lesssim\int_{0}^{t}\left( \left\| \partial_{\tau} a^{\varepsilon}  \cdot\operatorname{div}\mathbf{v}^{\varepsilon}\right\|_{m-2}^{2}+ \left\| a^{\varepsilon}  \cdot\partial_{\tau}\operatorname{div}\mathbf{v}^{\varepsilon}\right\|_{m-2}^{2}\right)\mathrm{d}\tau.
\]
Using Lemma \ref{Le1}, the first term on the right-hand side satisfies
\begin{align*}
	&\int_{0}^{t} \left\| \partial_{\tau} a^{\varepsilon}  \cdot\operatorname{div}\mathbf{v}^{\varepsilon}\right\|_{m-2}^{2}\mathrm{d}\tau\\
    \lesssim{}&\sum_{\substack{|\alpha|\leq m-2\\\beta+\gamma=\alpha}}\int_{0}^{t} \left\|Z^\beta\partial_{\tau} a^{\varepsilon}  \cdot Z^\gamma \operatorname{div}\mathbf{v}^{\varepsilon}\right\|_0^2 \mathrm{d}\tau\\
	\lesssim{}&\sum_{\substack{|\alpha|\leq m-2\\\beta+\gamma=\alpha}}\int_{0}^{t} \left\|Z^\beta\partial_{\tau} a^{\varepsilon} \right\|_0\left\|Z^\beta\partial_{1}\partial_{\tau} a^{\varepsilon} \right\|_0\left\| Z^\gamma \operatorname{div}\mathbf{v}^{\varepsilon}\right\|_{0}^{\frac{1}{2}}\left\|\partial_{2} Z^\gamma \operatorname{div}\mathbf{v}^{\varepsilon}\right\|_{0}^{\frac{1}{2}}\\
	&\cdot\left\|\partial_{3} Z^\gamma \operatorname{div}\mathbf{v}^{\varepsilon}\right\|_{0}^{\frac{1}{2}}\left\|\partial_{23} Z^\gamma \operatorname{div}\mathbf{v}^{\varepsilon}\right\|_{0}^{\frac{1}{2}} \mathrm{d}\tau\\
    \lesssim{}&\sup_{0\leq \tau \leq t}\left(\left\| \mathbf{v}^{\varepsilon}\right\|_{m}^{2} + \left\| \partial_3 \mathbf{v}^{\varepsilon}\right\|_{m-1}^{2}+\left\| \partial_{\tau}a^{\varepsilon}\right\|_{m-2}^{2}\right)\\
    &\cdot\int_{0}^{t}\left( \left\|\partial\mathbf{v}^{\varepsilon}\right\|_{m}^2+\left\|\partial_3\partial\mathbf{v}^{\varepsilon}\right\|_{m-1}^2+\left\| \partial_{\tau}a^{\varepsilon}\right\|_{m-1}^{2}\right) \mathrm{d}\tau.
\end{align*}
In the same way, the second term is estimated as follows.
\begin{align*}
	&\int_{0}^{t} \left\|  a^{\varepsilon}  \cdot\partial_{\tau}\operatorname{div}\mathbf{v}^{\varepsilon}\right\|_{m-2}^{2}\mathrm{d}\tau\\
    \lesssim{}&\sum_{\substack{|\alpha|\leq m-2\\\beta+\gamma=\alpha}}\int_{0}^{t} \left\|Z^\beta a^{\varepsilon}  \cdot Z^\gamma \partial_{\tau}\operatorname{div}\mathbf{v}^{\varepsilon}\right\|_0^2 \mathrm{d}\tau\\
    \lesssim{}&\sum_{\substack{|\alpha|\leq m-2}}\int_{0}^{t} \left\| a^{\varepsilon} \right\|_{L^{\infty}}^2\left\|Z^\alpha\partial_{\tau}\operatorname{div}\mathbf{v}^{\varepsilon}\right\|_{0}^2 \mathrm{d}\tau\\
    &+\sum_{\substack{|\alpha|\leq m-2\\\beta+\gamma=\alpha, \beta\neq0}}\int_{0}^{t} \left\|Z^\beta a^{\varepsilon} \cdot Z^\gamma\partial_{\tau} \operatorname{div}\mathbf{v}^{\varepsilon}\right\|_{0}^2 \mathrm{d}\tau\\
	\lesssim{}&\sum_{\substack{|\alpha|\leq m-2}}\int_{0}^{t} \left\| a^{\varepsilon} \right\|_{2}\left\|\partial_3 a^{\varepsilon} \right\|_{2}\left\|Z^\alpha\partial_{\tau}\operatorname{div}\mathbf{v}^{\varepsilon}\right\|_{0}^2 \mathrm{d}\tau\\
	&+\sum_{\substack{|\alpha|\leq m-2\\\beta+\gamma=\alpha, \beta\neq0}}\int_{0}^{t} \left\|Z^\beta a^{\varepsilon} \right\|_0^{\frac{1}{2}}\left\|\partial_{2}Z^\beta a^{\varepsilon} \right\|_0^{\frac{1}{2}}
    \left\|\partial_{3}Z^\beta a^{\varepsilon} \right\|_0^{\frac{1}{2}}\left\|\partial_{23}Z^\beta a^{\varepsilon} \right\|_0^{\frac{1}{2}}\\
    &\cdot\left\| Z^\gamma\partial_{\tau} \operatorname{div}\mathbf{v}^{\varepsilon}\right\|_{0}\left\| \partial_1 Z^\gamma \partial_{\tau}\operatorname{div}\mathbf{v}^{\varepsilon}\right\|_{0} \mathrm{d}\tau\\
    \lesssim{}&\sup_{0\leq \tau \leq t}\left(\left\| a^{\varepsilon}\right\|_{m}^{2} +\left\| \partial_{3}a^{\varepsilon}\right\|_{m-1}^{2}\right)\int_{0}^{t}\left(\left\| \partial_{\tau}\mathbf{v}^{\varepsilon}\right\|_{m-1}^{2} +\left\| \partial_3\partial_{\tau}\mathbf{v}^{\varepsilon}\right\|_{m-2}^{2}\right) \mathrm{d}\tau.
\end{align*}
By combining the foregoing two estimates, we derive
\begin{equation}\label{G38}
\begin{aligned}
	&\int_{0}^{t} \left\| \partial_{\tau}\left(a^{\varepsilon}  \operatorname{div}\mathbf{v}^{\varepsilon}\right)\right\|_{m-2}^{2}\mathrm{d}\tau\\
\lesssim{}&\sup_{0\leq \tau \leq t}\left(\left\| \left(a^{\varepsilon},\mathbf{v}^{\varepsilon}\right)\right\|_{m}^{2} +\left\| \partial_{3}\left(a^{\varepsilon},\mathbf{v}^{\varepsilon}\right)\right\|_{m-1}^{2}+\left\| \partial_{\tau}a^{\varepsilon}\right\|_{m-2}^{2}\right)\\
&\cdot\int_{0}^{t}\left(\left\|\partial\mathbf{v}^{\varepsilon}\right\|_{m}^2+\left\|\partial_3\partial\mathbf{v}^{\varepsilon}\right\|_{m-1}^2+\left\| \partial_{\tau}\left(a^{\varepsilon},\mathbf{v}^{\varepsilon}\right)\right\|_{m-1}^{2}+\left\|\partial_3\partial_{\tau}\mathbf{v}^{\varepsilon}\right\|_{m-2}^2\right) \mathrm{d}\tau.
\end{aligned}
\end{equation}
We now move to estimate $\| f^{\varepsilon}_2\|_{m-2}^{2}$.
Recall that
\begin{align*}
f^{\varepsilon}_2 = &-\mathbf{v}^{\varepsilon} \cdot \nabla\mathbf{v}^{\varepsilon} + \mathbf{B}^{\varepsilon} \cdot \nabla\mathbf{B}^{\varepsilon} - \mathbf{B}^{\varepsilon}\nabla\mathbf{B}^{\varepsilon} - J\left(a^{\varepsilon}\right)\nabla a^{\varepsilon} \\
& - I\left(a^{\varepsilon}\right)\left(\mu\Delta\mathbf{v}^{\varepsilon} + (\mu + \lambda)\nabla\operatorname{div}\mathbf{v}^{\varepsilon} + \mathbf{B}^{\varepsilon} \cdot \nabla\mathbf{B}^{\varepsilon} - \mathbf{B}^{\varepsilon}\nabla\mathbf{B}^{\varepsilon}\right).
\end{align*}
Then, the following inequality holds
\begin{align*}
\left\|f^{\varepsilon}_2\right\|_{m-2}^{2} \lesssim{}&\left\|\mathbf{v}^{\varepsilon}\partial\mathbf{v}^{\varepsilon}\right\|_{m-2}^{2}+\left\|\mathbf{B}^{\varepsilon}\partial\mathbf{B}^{\varepsilon}\right\|_{m-2}^{2}+\left\| J\left(a^{\varepsilon}\right) \partial a^{\varepsilon}\right\|_{m-2}^{2}\\
&+\left\| I\left(a^{\varepsilon}\right) \partial^2\mathbf{v}^{\varepsilon}\right\|_{m-2}^{2} + \left\| I\left(a^{\varepsilon}\right)\mathbf{B}^{\varepsilon}\partial\mathbf{B}^{\varepsilon}\right\|_{m-2}^{2}.
\end{align*}
First, it follows from  Lemma \ref{Le1} that
\begin{align*}
	& \left\|\mathbf{v}^{\varepsilon} \partial \mathbf{v}^{\varepsilon}\right\|_{m-2}^{2}\\
    \lesssim{}&\sum_{\substack{|\alpha|\leq m-2\\\beta+\gamma=\alpha}} \left\|Z^\beta \mathbf{v}^{\varepsilon}  \cdot Z^\gamma \partial \mathbf{v}^{\varepsilon}\right\|_0^2\\
    \lesssim{}&\sum_{\substack{|\alpha|\leq m-2\\\beta+\gamma=\alpha}}\left\|Z^\beta \mathbf{v}^{\varepsilon} \right\|_{0}^{\frac{1}{2}}\left\|\partial_{2}Z^\beta \mathbf{v}^{\varepsilon} \right\|_{0}^{\frac{1}{2}}\left\|\partial_{3}Z^\beta \mathbf{v}^{\varepsilon} \right\|_{0}^{\frac{1}{2}}\left\|\partial_{23}Z^\beta \mathbf{v}^{\varepsilon} \right\|_{0}^{\frac{1}{2}}
    \left\|Z^\gamma \partial\mathbf{v}^{\varepsilon}\right\|_{0}\left\|\partial_{1}Z^\gamma \partial\mathbf{v}^{\varepsilon}\right\|_{0}\\
    \lesssim{}&\left\| \mathbf{v}^{\varepsilon}\right\|_{m}^{4} + \left\| \partial_3 \mathbf{v}^{\varepsilon}\right\|_{m-1}^{4}.
\end{align*}
In the same fashion, one can show that
\[
	\left\|\mathbf{B}^{\varepsilon}\partial\mathbf{B}^{\varepsilon}\right\|_{m-2}^{2}
    \lesssim\left\| \mathbf{B}^{\varepsilon}\right\|_{m}^{4} + \left\| \partial_3 \mathbf{B}^{\varepsilon}\right\|_{m-1}^{4}.
\]
Estimating the nonlinear terms involving composition functions requires a
more intricate approach.
Let $ F(\cdot) $ be a smooth function. Direct computation gives that for $ m\geq 4 $,
\[
\left\|F\left(a^{\varepsilon}\right)\right\|_m \lesssim \left\|a^{\varepsilon}\right\|_m + \left\|\partial_3 a^{\varepsilon}\right\|_{m-1},
\qquad
\left\|\partial F\left(a^{\varepsilon}\right)\right\|_{m-1} \lesssim \left\|\partial a^{\varepsilon}\right\|_{m-1},
\]
and if the condition $ F(0)=0 $ is further satisfied, we obtain
\[
\left\|F\left(a^{\varepsilon}\right)\right\|_{L^{\infty}} \lesssim \left\|a^{\varepsilon}\right\|_{L^{\infty}}.
\]
The proof is analogous to Lemma A.3 in \cite{k2} and is omitted here. Now it is ready to estimate the nonlinear terms involving composition functions.
According to  Lemma \ref{Le1}, we have
\begin{align*}
	&\left\| J\left(a^{\varepsilon}\right) \partial  a^{\varepsilon}\right\|_{m-2}^{2}\\
    \lesssim{}&\sum_{\substack{|\alpha|\leq m-2\\\beta+\gamma=\alpha}}\left\|Z^\beta J\left(a^{\varepsilon}\right)  \cdot Z^\gamma \partial a^{\varepsilon}\right\|_0^2\\
    \lesssim{}&\sum_{\substack{|\alpha|\leq m-2\\\beta+\gamma=\alpha}}\left\|Z^\beta  J\left(a^{\varepsilon}\right)  \right\|_{0}^{\frac{1}{2}}\left\|\partial_{2}Z^\beta  J\left(a^{\varepsilon}\right)  \right\|_{0}^{\frac{1}{2}}\left\|\partial_{3}Z^\beta  J\left(a^{\varepsilon}\right)  \right\|_{0}^{\frac{1}{2}}\left\|\partial_{23}Z^\beta  J\left(a^{\varepsilon}\right)  \right\|_{0}^{\frac{1}{2}}
    \left\|Z^\gamma \partial a^{\varepsilon}\right\|_{0}\left\|\partial_{1}Z^\gamma \partial a^{\varepsilon}\right\|_{0}\\
    \lesssim{}&\left\|a^{\varepsilon}\right\|_{m}^{4} + \left\| \partial_3 a^{\varepsilon}\right\|_{m-1}^{4}.
\end{align*}
With an aid of  Lemma \ref{Le1}, there holds
\begin{align*}
	&\left\| I\left(a^{\varepsilon}\right) \partial^2\mathbf{v}^{\varepsilon}\right\|_{m-2}^{2}\\
    \lesssim{}&\sum_{\substack{|\alpha|\leq m-2\\\beta+\gamma=\alpha}}\left\|Z^\beta I\left(a^{\varepsilon}\right)  \cdot Z^\gamma \partial^2\mathbf{v}^{\varepsilon}\right\|_0^2\\
    \lesssim{}& \sum_{\substack{|\alpha|\leq m-2}}\left\|  I\left(a^{\varepsilon}\right)   \right\|_{L^\infty}^2\cdot\left\| Z^\alpha \partial^2\mathbf{v}^{\varepsilon}\right\|_{0}^2+\sum_{\substack{|\alpha|\leq m-2\\\beta+\gamma=\alpha, \beta\neq0}}\left\|Z^\beta  I\left(a^{\varepsilon}\right) \cdot Z^\gamma \partial^2\mathbf{v}^{\varepsilon}\right\|_{0}^2\\
	\lesssim{}&\sum_{\substack{|\alpha|\leq m-2}}\left\| a^{\varepsilon}  \right\|_2\left\| \partial_3 a^{\varepsilon}  \right\|_2\left\| Z^\alpha \partial^2\mathbf{v}^{\varepsilon}\right\|_{0}^2\\
    &+\sum_{\substack{|\alpha|\leq m-2\\ \beta+\gamma=\alpha,\beta\neq0}}\left\|Z^\beta  I\left(a^{\varepsilon}\right)\right\|_{0}^{\frac{1}{2}}\left\|\partial_{1}Z^\beta  I\left(a^{\varepsilon}\right)\right\|_{0}^{\frac{1}{2}}\left\|\partial_{3}Z^\beta  I\left(a^{\varepsilon}\right)\right\|_{0}^{\frac{1}{2}}\left\|\partial_{13}Z^\beta  I\left(a^{\varepsilon}\right)\right\|_{0}^{\frac{1}{2}}\\
    &\cdot\left\| Z^\gamma \partial^2\mathbf{v}^{\varepsilon}\right\|_{0}\left\|\partial_{2} Z^\gamma \partial^2\mathbf{v}^{\varepsilon}\right\|_{0}\\
    \lesssim{}&\left\|\left(a^{\varepsilon}, \mathbf{v}^{\varepsilon}\right)\right\|_{m}^{4} + \left\| \partial_3 \left(a^{\varepsilon}, \mathbf{v}^{\varepsilon}\right)\right\|_{m-1}^{4}+\left\|\partial_3^2\mathbf{v}^{\varepsilon}\right\|_{m-2}^4.
\end{align*}
Finally, it follows from  Lemma \ref{Le1} that
\begin{align*}
	& \left\| I\left(a^{\varepsilon}\right)\mathbf{B}^{\varepsilon}\partial\mathbf{B}^{\varepsilon}\right\|_{m-2}^{2}\\
\lesssim{}&\sum_{\substack{|\alpha|\leq m-2\\\beta+\gamma+\delta=\alpha}} \left\|Z^\beta I\left(a^{\varepsilon}\right)\cdot Z^\gamma\mathbf{B}^{\varepsilon} \cdot Z^{\delta}\partial\mathbf{B}^{\varepsilon}\right\|_0^2\\
    \lesssim{}&\sum_{\substack{|\alpha|\leq m-2\\\gamma+\delta=\alpha}}\left\|I\left(a^{\varepsilon}\right)\right\|_{2}\left\|\partial_3 I\left(a^{\varepsilon}\right)\right\|_{2}\left\|Z^\gamma\mathbf{B}^{\varepsilon}\right\|_0^{\frac{1}{2}}\left\|\partial_{2}Z^\gamma\mathbf{B}^{\varepsilon}\right\|_0^{\frac{1}{2}}\\
    &\cdot\left\|\partial_{3}Z^\gamma\mathbf{B}^{\varepsilon}\right\|_0^{\frac{1}{2}}\left\|\partial_{23}Z^\gamma\mathbf{B}^{\varepsilon}\right\|_0^{\frac{1}{2}} \left\| Z^{\delta}\partial\mathbf{B}^{\varepsilon}\right\|_0\left\| \partial_1Z^{\delta}\partial\mathbf{B}^{\varepsilon}\right\|_0\\
    &+\sum_{\substack{|\alpha|\leq m-2\\\beta+\gamma+\delta=\alpha, \beta\neq0}}\left\|Z^\beta I\left(a^{\varepsilon}\right)\right\|_{0}^{\frac{1}{2}}\left\|\partial_2Z^\beta I\left(a^{\varepsilon}\right)\right\|_{0}^{\frac{1}{2}}\left\|\partial_3Z^\beta I\left(a^{\varepsilon}\right)\right\|_{0}^{\frac{1}{2}}\left\|\partial_{23}Z^\beta I\left(a^{\varepsilon}\right)\right\|_{0}^{\frac{1}{2}}\\
    &\cdot\left\|Z^\gamma\mathbf{B}^{\varepsilon}\right\|_2\left\|\partial_{3}Z^\gamma\mathbf{B}^{\varepsilon}\right\|_2\left\|Z^{\delta}\partial\mathbf{B}^{\varepsilon}\right\|_0\left\| \partial_{1}Z^{\delta}\partial\mathbf{B}^{\varepsilon}\right\|_0\\
    \lesssim{}&\left\|\left(a^{\varepsilon}, \mathbf{B}^{\varepsilon}\right)\right\|_{m}^{6} + \left\| \partial_3 \left(a^{\varepsilon}, \mathbf{B}^{\varepsilon}\right)\right\|_{m-1}^{6}.
\end{align*}
Collecting all the above estimates yields
\begin{equation}\label{G45}
\left\|f^{\varepsilon}_2\right\|_{m-2}^{2}
\lesssim \left\|\left(a^{\varepsilon},\mathbf{v}^{\varepsilon},\mathbf{B}^{\varepsilon}\right)\right\|_{m}^{4} + \left\|\partial_3\left(a^{\varepsilon},\mathbf{v}^{\varepsilon},\mathbf{B}^{\varepsilon}\right)\right\|_{m-1}^{4}
+ \left\|\partial_3^2\mathbf{v}^{\varepsilon}\right\|_{m-2}^{4}.
\end{equation}
By a similar argument, we can obtain
\begin{equation}\label{G45.5}
\left\|f^{\varepsilon}_3\right\|_{m-2}^{2}
\lesssim \left\|\left(\mathbf{v}^{\varepsilon},\mathbf{B}^{\varepsilon}\right)\right\|_{m}^{4} + \left\|\partial_3\left(\mathbf{v}^{\varepsilon},\mathbf{B}^{\varepsilon}\right)\right\|_{m-1}^{4}.
\end{equation}
Next, we deal with $\int_{0}^{t} \left\|f^{\varepsilon}_2\right\|_{m-1}^{2}\mathrm{d}\tau$.
Similarly,
\begin{align*}
\int_{0}^{t} \left\|f^{\varepsilon}_2\right\|_{m-1}^{2}\mathrm{d}\tau \lesssim{}&\int_{0}^{t}\Big(\left\|\mathbf{v}^{\varepsilon}\partial\mathbf{v}^{\varepsilon}\right\|_{m-1}^{2}+\left\|\mathbf{B}^{\varepsilon}\partial\mathbf{B}^{\varepsilon}\right\|_{m-1}^{2}+\left\| J\left(a^{\varepsilon}\right) \partial a^{\varepsilon}\right\|_{m-1}^{2}\\
&+\left\| I\left(a^{\varepsilon}\right) \partial^2\mathbf{v}^{\varepsilon}\right\|_{m-1}^{2} + \left\| I\left(a^{\varepsilon}\right)\mathbf{B}^{\varepsilon}\partial\mathbf{B}^{\varepsilon}\right\|_{m-1}^{2}\Big)\mathrm{d}\tau.
\end{align*}
First, it follows from  Lemma \ref{Le1} that
\begin{align*}
	&\int_{0}^{t} \left\|\mathbf{v}^{\varepsilon}\partial\mathbf{v}^{\varepsilon}\right\|_{m-1}^{2}\mathrm{d}\tau\\
    \lesssim{}&\sum_{\substack{|\alpha|\leq m-1\\\beta+\gamma=\alpha}}\int_{0}^{t} \left\|Z^\beta \mathbf{v}^{\varepsilon}  \cdot Z^\gamma \partial\mathbf{v}^{\varepsilon}\right\|_0^2 \mathrm{d}\tau\\
    \lesssim{}&\sum_{\substack{|\alpha|\leq m-1\\\beta+\gamma=\alpha}}\int_{0}^{t} \left\|Z^\beta \mathbf{v}^{\varepsilon} \right\|_{0}^{\frac{1}{2}}\left\|\partial_{1} Z^\beta \mathbf{v}^{\varepsilon} \right\|_{0}^{\frac{1}{2}}\left\|\partial_{3}Z^\beta \mathbf{v}^{\varepsilon} \right\|_{0}^{\frac{1}{2}}\left\|\partial_{13}Z^\beta \mathbf{v}^{\varepsilon} \right\|_{0}^{\frac{1}{2}}\\
    &\cdot\left\|Z^\gamma \partial\mathbf{v}^{\varepsilon}\right\|_{0}\left\|\partial_{2}Z^\gamma \partial\mathbf{v}^{\varepsilon}\right\|_{0} \mathrm{d}\tau\\
    \lesssim{}&\sup_{0\leq \tau \leq t}\left(\left\| \mathbf{v}^{\varepsilon}\right\|_{m}^{2} + \left\| \partial_3 \mathbf{v}^{\varepsilon}\right\|_{m-1}^{2}\right)\int_{0}^{t}\left\|\partial\mathbf{v}^{\varepsilon}\right\|_{m}^2 \mathrm{d}\tau.
\end{align*}
By the same argument, it follows that
\[
	\int_{0}^{t} \left\|\mathbf{B}^{\varepsilon}\partial\mathbf{B}^{\varepsilon}\right\|_{m-1}^{2}\mathrm{d}\tau
    \lesssim\sup_{0\leq \tau \leq t}\left(\left\| \mathbf{B}^{\varepsilon}\right\|_{m}^{2} + \left\| \partial_3 \mathbf{B}^{\varepsilon}\right\|_{m-1}^{2}\right)\int_{0}^{t}\left(\left\|\partial_{h}\mathbf{B}^{\varepsilon}\right\|_{m}^2+\left\|\partial_h\partial_3\mathbf{B}^{\varepsilon}\right\|_{m-1}^2\right) \mathrm{d}\tau.
\]
Then, we use  Lemma \ref{Le1} to get
\begin{align*}
	&\int_{0}^{t} \left\| J\left(a^{\varepsilon}\right) \partial a^{\varepsilon}\right\|_{m-1}^{2}\mathrm{d}\tau\\
    \lesssim{}&\sum_{\substack{|\alpha|\leq m-1\\\beta+\gamma=\alpha}}\int_{0}^{t} \left\|Z^\beta J\left(a^{\varepsilon}\right)  \cdot Z^\gamma \partial a^{\varepsilon}\right\|_0^2 \mathrm{d}\tau\\
    \lesssim{}&\sum_{\substack{|\alpha|\leq m-1}}\int_{0}^{t} \left\|  J\left(a^{\varepsilon}\right)   \right\|_{L^\infty}^2\cdot\left\| Z^\alpha \partial a^{\varepsilon}\right\|_{0}^2 \mathrm{d}\tau+\sum_{\substack{|\alpha|\leq m-1}}\int_{0}^{t} \left\|Z^\alpha  J\left(a^{\varepsilon}\right) \cdot \partial a^{\varepsilon}\right\|_{0}^2 \mathrm{d}\tau\\
    &+\sum_{\substack{|\alpha|\leq m-1\\\beta+\gamma=\alpha\\\beta\neq0,\gamma\neq0}}\int_{0}^{t} \left\|Z^\beta  J\left(a^{\varepsilon}\right) \cdot Z^\gamma \partial a^{\varepsilon}\right\|_{0}^2 \mathrm{d}\tau\\
	\lesssim{}&\sum_{\substack{|\alpha|\leq m-1}}\int_{0}^{t} \left\| a^{\varepsilon}  \right\|_2\left\| \partial_3 a^{\varepsilon}  \right\|_2\left\| Z^\alpha \partial a^{\varepsilon}\right\|_{0}^2 \mathrm{d}\tau\\
	&+\sum_{\substack{|\alpha|\leq m-1}}\int_{0}^{t} \left\|Z^\alpha J\left(a^{\varepsilon}\right) \right\|_0\left\|\partial_3 Z^\alpha J\left(a^{\varepsilon}\right) \right\|_0\left\|\partial a^{\varepsilon}\right\|_{0}^{\frac{1}{2}}\left\|\partial_{1}\partial a^{\varepsilon}\right\|_{0}^{\frac{1}{2}}\\
	&\cdot\left\|\partial_{2}\partial a^{\varepsilon}\right\|_{0}^{\frac{1}{2}}\left\|\partial_{12}\partial a^{\varepsilon}\right\|_{0}^{\frac{1}{2}} \mathrm{d}\tau\\
    &+\sum_{\substack{|\alpha|\leq m-1\\\beta+\gamma=\alpha\\\beta\neq0,\gamma\neq0}}\int_{0}^{t} \left\|Z^\beta  J\left(a^{\varepsilon}\right)\right\|_{0}^{\frac{1}{2}}\left\|\partial_{1}Z^\beta  J\left(a^{\varepsilon}\right)\right\|_{0}^{\frac{1}{2}}\left\|\partial_{3}Z^\beta  J\left(a^{\varepsilon}\right)\right\|_{0}^{\frac{1}{2}}\\
    &\cdot\left\|\partial_{13}Z^\beta  J\left(a^{\varepsilon}\right)\right\|_{0}^{\frac{1}{2}}\left\| Z^\gamma \partial a^{\varepsilon}\right\|_{0}\left\|\partial_{2} Z^\gamma \partial a^{\varepsilon}\right\|_{0} \mathrm{d}\tau\\
    \lesssim{}&\sup_{0\leq \tau \leq t}\left(\left\| a^{\varepsilon}\right\|_{m}^{2} + \left\| \partial_3 a^{\varepsilon}\right\|_{m-1}^{2}\right)\int_{0}^{t}\left\|\partial a^{\varepsilon}\right\|_{m-1}^2 \mathrm{d}\tau.
\end{align*}
Thanks to  Lemma \ref{Le1}, there holds
\begin{align*}
	&\int_{0}^{t} \left\| I\left(a^{\varepsilon}\right) \partial^2\mathbf{v}^{\varepsilon}\right\|_{m-1}^{2}\mathrm{d}\tau\\
    \lesssim{}&\sum_{\substack{|\alpha|\leq m-1\\\beta+\gamma=\alpha}}\int_{0}^{t} \left\|Z^\beta I\left(a^{\varepsilon}\right)  \cdot Z^\gamma \partial^2\mathbf{v}^{\varepsilon}\right\|_0^2 \mathrm{d}\tau\\
    \lesssim{}&\sum_{\substack{|\alpha|\leq m-1}}\int_{0}^{t} \left\|  I\left(a^{\varepsilon}\right)   \right\|_{L^\infty}^2\cdot\left\| Z^\alpha \partial^2\mathbf{v}^{\varepsilon}\right\|_{0}^2 \mathrm{d}\tau+\sum_{\substack{|\alpha|\leq m-1}}\int_{0}^{t} \left\|Z^\alpha  I\left(a^{\varepsilon}\right) \cdot \partial^2\mathbf{v}^{\varepsilon}\right\|_{0}^2 \mathrm{d}\tau\\
    &+\sum_{\substack{|\alpha|\leq m-1\\\beta+\gamma=\alpha\\\beta\neq0,\gamma\neq0}}\int_{0}^{t} \left\|Z^\beta  I\left(a^{\varepsilon}\right) \cdot Z^\gamma \partial^2\mathbf{v}^{\varepsilon}\right\|_{0}^2 \mathrm{d}\tau\\
	\lesssim{}&\sum_{\substack{|\alpha|\leq m-1}}\int_{0}^{t} \left\| a^{\varepsilon}  \right\|_2\left\| \partial_3 a^{\varepsilon}  \right\|_2\left\| Z^\alpha \partial^2\mathbf{v}^{\varepsilon}\right\|_{0}^2 \mathrm{d}\tau\\
	&+\sum_{\substack{|\alpha|\leq m-1}}\int_{0}^{t} \left\|Z^\alpha I\left(a^{\varepsilon}\right) \right\|_0\left\|\partial_3 Z^\alpha I\left(a^{\varepsilon}\right) \right\|_0\left\|\partial^2\mathbf{v}^{\varepsilon}\right\|_{0}^{\frac{1}{2}}\left\|\partial_{1}\partial^2\mathbf{v}^{\varepsilon}\right\|_{0}^{\frac{1}{2}}\\
	&\cdot\left\|\partial_{2}\partial^2\mathbf{v}^{\varepsilon}\right\|_{0}^{\frac{1}{2}}\left\|\partial_{12}\partial^2\mathbf{v}^{\varepsilon}\right\|_{0}^{\frac{1}{2}} \mathrm{d}\tau\\
    &+\sum_{\substack{|\alpha|\leq m-1\\\beta+\gamma=\alpha\\\beta\neq0,\gamma\neq0}}\int_{0}^{t} \left\|Z^\beta  I\left(a^{\varepsilon}\right)\right\|_{0}^{\frac{1}{2}}\left\|\partial_{1}Z^\beta  I\left(a^{\varepsilon}\right)\right\|_{0}^{\frac{1}{2}}\left\|\partial_{3}Z^\beta  I\left(a^{\varepsilon}\right)\right\|_{0}^{\frac{1}{2}}\\
    &\cdot\left\|\partial_{13}Z^\beta  I\left(a^{\varepsilon}\right)\right\|_{0}^{\frac{1}{2}}\left\| Z^\gamma \partial^2\mathbf{v}^{\varepsilon}\right\|_{0}\left\|\partial_{2} Z^\gamma \partial^2\mathbf{v}^{\varepsilon}\right\|_{0} \mathrm{d}\tau\\
    \lesssim{}&\sup_{0\leq \tau \leq t}\left(\left\| a^{\varepsilon}\right\|_{m}^{2} + \left\| \partial_3 a^{\varepsilon}\right\|_{m-1}^{2}\right)\int_{0}^{t}\left(\left\|\partial\mathbf{v}^{\varepsilon}\right\|_{m}^2+\left\|\partial\partial_3\mathbf{v}^{\varepsilon}\right\|_{m-1}^2\right) \mathrm{d}\tau.
\end{align*}
Finally, it follows from  Lemma \ref{Le1} that
\begin{align*}
	&\int_{0}^{t} \left\| I\left(a^{\varepsilon}\right)\mathbf{B}^{\varepsilon}\partial\mathbf{B}^{\varepsilon}\right\|_{m-1}^{2}\mathrm{d}\tau\\
\lesssim{}&\sum_{\substack{|\alpha|\leq m-1\\\beta+\gamma+\delta=\alpha}}\int_{0}^{t} \left\|Z^\beta I\left(a^{\varepsilon}\right)\cdot Z^\gamma\mathbf{B}^{\varepsilon} \cdot Z^{\delta}\partial\mathbf{B}^{\varepsilon}\right\|_0^2 \mathrm{d}\tau\\
    \lesssim{}&\sum_{\substack{|\alpha|\leq m-1\\|\beta|\leq m-3\\\beta+\gamma+\delta=\alpha}}\int_{0}^{t} \left\|Z^\beta I\left(a^{\varepsilon}\right)\right\|_{2}\left\|\partial_3Z^\beta I\left(a^{\varepsilon}\right)\right\|_{2}\left\|Z^\gamma\mathbf{B}^{\varepsilon}\right\|_0^{\frac{1}{2}}\left\|\partial_{2}Z^\gamma\mathbf{B}^{\varepsilon}\right\|_0^{\frac{1}{2}}\\
    &\cdot\left\|\partial_{3}Z^\gamma\mathbf{B}^{\varepsilon}\right\|_0^{\frac{1}{2}}\left\|\partial_{23}Z^\gamma\mathbf{B}^{\varepsilon}\right\|_0^{\frac{1}{2}} \left\| Z^{\delta}\partial\mathbf{B}^{\varepsilon}\right\|_0\left\| \partial_1Z^{\delta}\partial\mathbf{B}^{\varepsilon}\right\|_0 \mathrm{d}\tau\\
    &+\sum_{\substack{|\alpha|\leq m-1\\|\beta|\geq m-2\\\beta+\gamma+\delta=\alpha}}\int_{0}^{t} \left\|Z^\beta I\left(a^{\varepsilon}\right)\right\|_{0}\left\|\partial_3Z^\beta I\left(a^{\varepsilon}\right)\right\|_{0}\left\|Z^\gamma\mathbf{B}^{\varepsilon}\right\|_2\left\|\partial_{3}Z^\gamma\mathbf{B}^{\varepsilon}\right\|_2\\
    &\cdot\left\|Z^{\delta}\partial\mathbf{B}^{\varepsilon}\right\|_0^{\frac{1}{2}}\left\| \partial_{1}Z^{\delta}\partial\mathbf{B}^{\varepsilon}\right\|_0^{\frac{1}{2}}\left\| \partial_2 Z^{\delta}\partial\mathbf{B}^{\varepsilon}\right\|_0^{\frac{1}{2}}\left\| \partial_{12}Z^{\delta}\partial\mathbf{B}^{\varepsilon}\right\|_0^{\frac{1}{2}} \mathrm{d}\tau\\
    \lesssim{}&\sup_{0\leq \tau \leq t}\left(\left\| \left(a^{\varepsilon},\mathbf{B}^{\varepsilon}\right)\right\|_{m}^{4} + \left\| \partial_3 \left(a^{\varepsilon},\mathbf{B}^{\varepsilon}\right)\right\|_{m-1}^{4}\right)\\
    &\cdot\int_{0}^{t}\left(\left\|\partial a^{\varepsilon}\right\|_{m-1}^2+\left\|\partial_{h}\mathbf{B}^{\varepsilon}\right\|_{m}^2+\left\|\partial_h\partial_3\mathbf{B}^{\varepsilon}\right\|_{m-1}^2\right) \mathrm{d}\tau.
\end{align*}
Combining all the above estimates leads to
\begin{equation}\label{G51}
\begin{aligned}
\int_{0}^{t} \left\|f^{\varepsilon}_2\right\|_{m-1}^{2}\mathrm{d}\tau
\lesssim& \sup_{0\leq \tau \leq t}\Big(\left\|\left(a^{\varepsilon},\mathbf{v}^{\varepsilon},\mathbf{B}^{\varepsilon}\right)\right\|_{m}^{2} + \left\|\partial_3\left(a^{\varepsilon},\mathbf{v}^{\varepsilon},\mathbf{B}^{\varepsilon}\right)\right\|_{m-1}^{2}\Big)\\
&\cdot \int_{0}^{t}\Big(\left\|\partial a^{\varepsilon} \right\|_{m-1}^2+\left\|\left(\partial\mathbf{v}^{\varepsilon}, \partial_h\mathbf{B}^{\varepsilon}\right)\right\|_{m}^2
        +\left\|\partial_3\left(\partial\mathbf{v}^{\varepsilon}, \partial_h\mathbf{B}^{\varepsilon}\right)\right\|_{m-1}^2\Big)\mathrm{d}\tau.
\end{aligned}
\end{equation}
Now, we turn to derive the estimate for $\sum_{|\alpha|\leq m} \int_{0}^{t}  \left|{A} _{\alpha}\right|  \mathrm{d} \tau$.
The estimate can be divided into three parts.
\begin{align*}
  &\sum_{|\alpha|\leq m} \int_{0}^{t}  \left|{A} _{\alpha}\right|  \mathrm{d} \tau\\
\lesssim{}&\sum_{|\alpha|\leq m}\int_{0}^{t}\left( \left| \left( Z^{\alpha}f^{\varepsilon}_1,  Z^{\alpha} a^{\varepsilon}\right)\right|+\left| \left( Z^{\alpha}f^{\varepsilon}_2,  Z^{\alpha} \mathbf{v}^{\varepsilon}\right)\right|+\left| \left( Z^{\alpha}f^{\varepsilon}_3,  Z^{\alpha} \mathbf{B}^{\varepsilon}\right)\right|\right)\mathrm{d}\tau\\
\overset{def}{=}{}& H_{1}+H_{2}+H_{3}.
\end{align*}
As for $H_{1}$, we have the following splitting:
\begin{align*}
H_{1}
\lesssim{}&\sum_{|\alpha|\leq m}\int_{0}^{t} \left| \left( Z^{\alpha}\left(a^{\varepsilon}  \operatorname{div}\mathbf{v}^{\varepsilon}+\mathbf{v}^{\varepsilon}\cdot\nabla a^{\varepsilon}\right),  Z^{\alpha} a^{\varepsilon}\right)\right|\mathrm{d}\tau\\
\lesssim{}&\int_{0}^{t} \left| \left( a^{\varepsilon}  \operatorname{div}\mathbf{v}^{\varepsilon}, a^{\varepsilon}\right)\right|\mathrm{d}\tau+\sum_{1\leq|\alpha|\leq m}\int_{0}^{t} \left| \left( Z^{\alpha}\left(a^{\varepsilon}  \operatorname{div}\mathbf{v}^{\varepsilon}\right),  Z^{\alpha} a^{\varepsilon}\right)\right|\mathrm{d}\tau\\	
&+\sum_{|\alpha|\leq m}\int_{0}^{t}\int_{\mathbb{R}_+^3} \left|\operatorname{div}\mathbf{v}^{\varepsilon}\right| \cdot \left|Z^\alpha a^{\varepsilon}\right|^2 \mathrm{d}\mathbf{x}\mathrm{d}\tau+\sum_{|\alpha|\leq m}\int_{0}^{t}\left|\left( v_3^{\varepsilon} \left[Z^\alpha , \partial_3\right] a^{\varepsilon}, Z^\alpha a^{\varepsilon}\right) \right|\mathrm{d}\tau\\
&+\sum_{\substack{1\leq|\alpha|\leq m\\\beta+\gamma=\alpha,\beta \neq 0}}\int_{0}^{t} \left| \left( Z^\beta \mathbf{v}^{\varepsilon} \cdot Z^\gamma \nabla a^{\varepsilon}, Z^\alpha a^{\varepsilon} \right)\right|\mathrm{d}\tau\\
\overset{def}{=}{}& H_{1,1}+H_{1,2}+H_{1,3}+H_{1,4}+H_{1,5}.
\end{align*}
Applying Lemma \ref{Le1}, it holds
\begin{align*}
H_{1,1}
\lesssim{}&\int_{0}^{t} \left\| a^{\varepsilon}\right\|_0^{\frac{1}{2}}\left\| \partial_1a^{\varepsilon}\right\|_0^{\frac{1}{2}}\left\| a^{\varepsilon}\right\|_0^{\frac{1}{2}}\left\| \partial_3a^{\varepsilon}\right\|_0^{\frac{1}{2}} \left\| \operatorname{div}\mathbf{v}^{\varepsilon}\right\|_0^{\frac{1}{2}}\left\| \partial_2\operatorname{div}\mathbf{v}^{\varepsilon}\right\|_0^{\frac{1}{2}}\mathrm{d}\tau\\
\lesssim{}&\sup_{0\leq \tau \leq t}\left\| a^{\varepsilon}\right\|_{m}\int_{0}^{t}\left(\left\| \partial a^{\varepsilon}\right\|_{m-1}^{2}+\left\|\partial\mathbf{v}^{\varepsilon}\right\|_{m}^2\right) \mathrm{d}\tau.
\end{align*}
Using Young's inequality and substituting \eqref{G30}, we have
\begin{align*}
H_{1,2}
\lesssim{}&\delta_4\int_{0}^{t} \left\|\partial a^{\varepsilon}\right\|_{m-1}^2\mathrm{d}\tau+\int_{0}^{t} \left\|a^{\varepsilon}  \operatorname{div}\mathbf{v}^{\varepsilon}\right\|_m^2\mathrm{d}\tau\\
\lesssim{}&\delta_4\int_{0}^{t} \left\|\partial a^{\varepsilon}\right\|_{m-1}^2\mathrm{d}\tau+\sup_{0\leq \tau \leq t}\left(\left\| a^{\varepsilon}\right\|_{m}^2+ \left\| \partial_3 a^{\varepsilon}\right\|_{m-1}^2\right)\\
&\cdot\int_{0}^{t}\left(\left\|\partial\mathbf{v}^{\varepsilon}\right\|_{m}^2+\left\|\partial_3\partial\mathbf{v}^{\varepsilon}\right\|_{m-1}^2\right) \mathrm{d}\tau.
\end{align*}
The application of Lemma \ref{Le1} yields directly
\begin{align*}
H_{1,3}
\lesssim{}&\int_{0}^{t}\int_{\mathbb{R}_+^3} \left|\operatorname{div}\mathbf{v}^{\varepsilon}\right| \cdot \left| a^{\varepsilon}\right|^2 \mathrm{d}\mathbf{x} \mathrm{d}\tau+\sum_{1\leq|\alpha|\leq m}\int_{0}^{t}\int_{\mathbb{R}_+^3} \left|\operatorname{div}\mathbf{v}^{\varepsilon}\right| \cdot \left|Z^\alpha a^{\varepsilon}\right|^2 \mathrm{d}\mathbf{x} \mathrm{d}\tau\\
\lesssim{}&\int_{0}^{t} \left\| a^{\varepsilon}\right\|_0^{\frac{1}{2}}\left\| \partial_1a^{\varepsilon}\right\|_0^{\frac{1}{2}}\left\| a^{\varepsilon}\right\|_0^{\frac{1}{2}}\left\| \partial_3a^{\varepsilon}\right\|_0^{\frac{1}{2}} \left\| \operatorname{div}\mathbf{v}^{\varepsilon}\right\|_0^{\frac{1}{2}}\left\| \partial_2\operatorname{div}\mathbf{v}^{\varepsilon}\right\|_0^{\frac{1}{2}}\mathrm{d}\tau\\
&+\sum_{1\leq|\alpha|\leq m}\int_{0}^{t} \left\|Z^{\alpha} a^{\varepsilon}\right\|_0^2\left\| \operatorname{div}\mathbf{v}^{\varepsilon}\right\|_2^{\frac{1}{2}}\left\| \partial_3\operatorname{div}\mathbf{v}^{\varepsilon}\right\|_2^{\frac{1}{2}}\mathrm{d}\tau\\
\lesssim{}&\sup_{0\leq \tau \leq t}\left\| a^{\varepsilon}\right\|_{m}\int_{0}^{t}\left(\left\| \partial a^{\varepsilon}\right\|_{m-1}^{2}+\left\|\partial\mathbf{v}^{\varepsilon}\right\|_{m}^2+\left\|\partial\partial_3\mathbf{v}^{\varepsilon}\right\|_{m-1}^2\right) \mathrm{d}\tau.
\end{align*}
The term $H_{1,4}$ vanishes if $\alpha_3 =0$, so we only consider $\alpha_3 \ge 1$.
By the relevant properties of commutators and applying Lemma \ref{Le1}, we have
\begin{align*}
H_{1,4}
\lesssim{}&\sum_{\substack{|\alpha|\leq m\\\alpha_3 \geq 1}}  \sum_{k=0}^{\alpha_3-1} \int_{0}^{t}\int_{\mathbb{R}_+^3}\left|v_3^{\varepsilon}\right|
\left| Z^{\alpha_{h}}_{h} Z^k_3\partial_3 a^{\varepsilon} \right| \left| Z^\alpha a^{\varepsilon}\right| \mathrm{d}\mathbf{x}\mathrm{d}\tau\\
	\lesssim{}&\sum_{\substack{|\alpha|\leq m\\\alpha_3 \geq 1}}  \sum_{k=0}^{\alpha_3-1} \int_{0}^{t}
	\left\| v_3^{\varepsilon}\right\|_{2}^{\frac{1}{2}}\left\| \partial_3 v_3^{\varepsilon}\right\|_{2}^{\frac{1}{2}}\left\| Z^{\alpha_{h}}_{h}Z^{k}_3\partial_3 a^{\varepsilon} \right\|_0 \left\| Z^\alpha a^{\varepsilon}\right\|_0 \mathrm{d}\tau\\
\lesssim{}&\sup_{0\leq \tau \leq t}\left(\left\|\mathbf{v}^{\varepsilon}\right\|_{m}+\left\|\partial_3\mathbf{v}^{\varepsilon}\right\|_{m-1}\right)\int_{0}^{t}\left\| \partial a^{\varepsilon}\right\|_{m-1}^{2} \mathrm{d}\tau.
\end{align*}
Again, applying Lemma \ref{Le1}, it holds
\begin{align*}
H_{1,5}
\lesssim{}&\sum_{\substack{1\leq|\alpha|\leq m\\\beta+\gamma=\alpha\\|\beta|=1}}\int_{0}^{t} \left\|Z^\beta \mathbf{v}^{\varepsilon}\right\|_{2}^{\frac{1}{2}}\left\|\partial_{3}Z^\beta \mathbf{v}^{\varepsilon}\right\|_{2}^{\frac{1}{2}} \left\| Z^\gamma \partial a^{\varepsilon}\right\|_0\left\| Z^\alpha a^{\varepsilon} \right\|_0\mathrm{d}\tau\\
&+\sum_{\substack{1\leq|\alpha|\leq m\\\beta+\gamma=\alpha\\2\leq|\beta|\leq m-1}}\int_{0}^{t} \left\|Z^\beta \mathbf{v}^{\varepsilon}\right\|_{0}^{\frac{1}{4}}\left\|\partial_{2}Z^\beta \mathbf{v}^{\varepsilon}\right\|_{0}^{\frac{1}{4}}\left\|\partial_{3}Z^\beta \mathbf{v}^{\varepsilon}\right\|_{0}^{\frac{1}{4}}\left\|\partial_{23}Z^\beta \mathbf{v}^{\varepsilon}\right\|_{0}^{\frac{1}{4}}\\
&\cdot\left\| Z^\gamma \partial a^{\varepsilon}\right\|_0^{\frac{1}{2}}\left\| \partial_1Z^\gamma \partial a^{\varepsilon}\right\|_0^{\frac{1}{2}}\left\| Z^\alpha a^{\varepsilon} \right\|_0\mathrm{d}\tau\\
&+\sum_{1\leq|\alpha|\leq m}\int_{0}^{t} \left\|Z^{\alpha} \mathbf{v}^{\varepsilon}\right\|_{0}^{\frac{1}{2}}\left\|\partial_{3}Z^{\alpha} \mathbf{v}^{\varepsilon}\right\|_{0}^{\frac{1}{2}} \\
&\cdot\left\|\partial a^{\varepsilon}\right\|_0^{\frac{1}{4}}\left\| \partial_{1}\partial a^{\varepsilon}\right\|_0^{\frac{1}{4}}\left\| \partial_{2}\partial a^{\varepsilon}\right\|_0^{\frac{1}{4}}\left\| \partial_{12}\partial a^{\varepsilon}\right\|_0^{\frac{1}{4}}\left\| Z^\alpha a^{\varepsilon} \right\|_0\mathrm{d}\tau\\
\lesssim{}&\sup_{0\leq \tau \leq t}\left\| a^{\varepsilon}\right\|_{m}\int_{0}^{t}\left(\left\| \partial a^{\varepsilon}\right\|_{m-1}^{2}+\left\|\partial\mathbf{v}^{\varepsilon}\right\|_{m}^2+\left\|\partial\partial_3\mathbf{v}^{\varepsilon}\right\|_{m-1}^2\right) \mathrm{d}\tau.
\end{align*}
Combining the above estimates gives
\begin{align*}
H_1
\lesssim{}&\delta_4\int_{0}^{t} \left\|\partial a^{\varepsilon}\right\|_{m-1}^2\mathrm{d}\tau+\sup_{0\leq \tau \leq t}\left(\left\| \left(a^{\varepsilon},\mathbf{v}^{\varepsilon}\right)\right\|_{m}+ \left\| \partial_3 \left(a^{\varepsilon},\mathbf{v}^{\varepsilon}\right)\right\|_{m-1}\right)\\
&\cdot\int_{0}^{t}\left(\left\| \partial a^{\varepsilon}\right\|_{m-1}^{2}+\left\|\partial\mathbf{v}^{\varepsilon}\right\|_{m}^2+\left\|\partial_3\partial\mathbf{v}^{\varepsilon}\right\|_{m-1}^2\right) \mathrm{d}\tau.
\end{align*}
Next, we split $H_2$ into the following terms:
\begin{align*}
H_2
\lesssim{}&\sum_{|\alpha|\leq m}\int_{0}^{t} \Big(\left| \left( Z^{\alpha}\left(\mathbf{v}^{\varepsilon} \cdot \nabla\mathbf{v}^{\varepsilon}\right),  Z^{\alpha} \mathbf{v}^{\varepsilon}\right)\right|+\left| \left( Z^{\alpha}\left(\mathbf{B}^{\varepsilon} \cdot \nabla\mathbf{B}^{\varepsilon}\right),  Z^{\alpha} \mathbf{v}^{\varepsilon}\right)\right|
+\left| \left( Z^{\alpha}\left(\mathbf{B}^{\varepsilon}\nabla\mathbf{B}^{\varepsilon}\right),  Z^{\alpha} \mathbf{v}^{\varepsilon}\right)\right|\\
&+\left| \left( Z^{\alpha}\left(J\left(a^{\varepsilon}\right)\nabla a^{\varepsilon}\right),  Z^{\alpha} \mathbf{v}^{\varepsilon}\right)\right|+\left| \left( Z^{\alpha}\left(I\left(a^{\varepsilon}\right)\Delta\mathbf{v}^{\varepsilon}\right),  Z^{\alpha} \mathbf{v}^{\varepsilon}\right)\right|+\left| \left( Z^{\alpha}\left(I\left(a^{\varepsilon}\right)\nabla\operatorname{div}\mathbf{v}^{\varepsilon}\right),  Z^{\alpha} \mathbf{v}^{\varepsilon}\right)\right|\\
&+\left| \left( Z^{\alpha}\left(I\left(a^{\varepsilon}\right)\mathbf{B}^{\varepsilon} \cdot \nabla\mathbf{B}^{\varepsilon}\right),  Z^{\alpha} \mathbf{v}^{\varepsilon}\right)\right|+\left| \left( Z^{\alpha}\left(I\left(a^{\varepsilon}\right)\mathbf{B}^{\varepsilon}\nabla\mathbf{B}^{\varepsilon}\right),  Z^{\alpha} \mathbf{v}^{\varepsilon}\right)\right|\Big)\mathrm{d}\tau\\
\overset{def}{=}{}& H_{2,1}+H_{2,2}+H_{2,3}+H_{2,4}+H_{2,5}+H_{2,6}+H_{2,7}+H_{2,8}.
\end{align*}
By Hölder's inequality and Lemma \ref{Le2}, we obtain
\begin{align*}
H_{2,1}
\lesssim{}&\sum_{|\alpha|\leq m}\int_{0}^{t} \left\|\mathbf{v}^{\varepsilon}\right\|_{L^3}\left\| Z^\alpha \partial \mathbf{v}^{\varepsilon}\right\|_{0}\left\| Z^\alpha \mathbf{v}^{\varepsilon}\right\|_{L^6}\mathrm{d}\tau\\
&+\sum_{\substack{|\alpha|\leq m\\\beta+\gamma=\alpha,\beta \neq 0}}\int_{0}^{t} \left\|Z^{\beta}\mathbf{v}^{\varepsilon}\right\|_{L^3}\left\| Z^\gamma \partial \mathbf{v}^{\varepsilon}\right\|_{0}
\left\| Z^\alpha \mathbf{v}^{\varepsilon}\right\|_{L^6}\mathrm{d}\tau\\
\lesssim{}&\sum_{|\alpha|\leq m}\int_{0}^{t} \left\|\mathbf{v}^{\varepsilon}\right\|_{H^1}\left\| Z^\alpha \partial \mathbf{v}^{\varepsilon}\right\|_{0}\left\|\partial Z^\alpha \mathbf{v}^{\varepsilon}\right\|_{0}\mathrm{d}\tau\\
&+\sum_{\substack{|\alpha|\leq m\\\beta+\gamma=\alpha,\beta \neq 0}}\int_{0}^{t} \left\|Z^{\beta}\mathbf{v}^{\varepsilon}\right\|_{H^1}\left\| Z^\gamma \partial \mathbf{v}^{\varepsilon}\right\|_{0}\left\|\partial Z^\alpha \mathbf{v}^{\varepsilon}\right\|_{0}\mathrm{d}\tau\\
\lesssim{}&\sup_{0\leq \tau \leq t}\left(\left\| \mathbf{v}^{\varepsilon}\right\|_{m}+\left\|\partial_3\mathbf{v}^{\varepsilon}\right\|_{m-1}\right)\int_{0}^{t}\left\|\partial\mathbf{v}^{\varepsilon}\right\|_{m}^2\mathrm{d}\tau.
\end{align*}
Using integration by parts, we have
\begin{align*}
&H_{2,2}+H_{2,3}\\
={}&\sum_{|\alpha|\leq m}\int_{0}^{t} \Big(\left| \left( Z^{\alpha}\operatorname{div}\left(\mathbf{B}^{\varepsilon}\otimes\mathbf{B}^{\varepsilon}\right),  Z^{\alpha} \mathbf{v}^{\varepsilon}\right)\right|
+\frac{1}{2}\left| \left( Z^{\alpha}\nabla\left(\left|\mathbf{B}^{\varepsilon}\right|^2\right),  Z^{\alpha} \mathbf{v}^{\varepsilon}\right)\right|\Big)\mathrm{d}\tau\\
\lesssim{}&\sum_{|\alpha|\leq m}\int_{0}^{t} \Big(\left| \left( Z^{\alpha}\left(\mathbf{B}^{\varepsilon}\cdot\mathbf{B}^{\varepsilon}\right),  \partial Z^{\alpha} \mathbf{v}^{\varepsilon}\right)\right|+\left| \left( \left[Z^{\alpha},\partial_3\right]\left(\mathbf{B}^{\varepsilon}\cdot\mathbf{B}^{\varepsilon}\right), Z^{\alpha} \mathbf{v}^{\varepsilon}\right)\right|\Big)\mathrm{d}\tau\\
\overset{def}{=}{}&H_{2,2,1}+H_{2,3,1}.
\end{align*}
Applying Lemma \ref{Le1}, it holds
\begin{align*}
H_{2,2,1}
\lesssim{}&\sum_{\substack{|\alpha|\leq m\\\beta+\gamma=\alpha}}\int_{0}^{t}\left\| Z^{\beta}\mathbf{B}^{\varepsilon}\cdot Z^{\gamma}\mathbf{B}^{\varepsilon}\right\|_0  \left\|\partial Z^{\alpha} \mathbf{v}^{\varepsilon}\right\|_0\mathrm{d}\tau\\
\lesssim{}&\sum_{\substack{|\alpha|\leq m}}\int_{0}^{t}\left\|\mathbf{B}^{\varepsilon}\right\|_0^{\frac{1}{4}}\left\| \partial_{2}\mathbf{B}^{\varepsilon}\right\|_0^{\frac{1}{4}}\left\| \partial_{3}\mathbf{B}^{\varepsilon}\right\|_0^{\frac{1}{4}}\left\| \partial_{23}\mathbf{B}^{\varepsilon}\right\|_0^{\frac{1}{4}}  \\
&\cdot\left\| Z^{\alpha}\mathbf{B}^{\varepsilon}\right\|_0^{\frac{1}{2}}\left\| \partial_1 Z^{\alpha}\mathbf{B}^{\varepsilon}\right\|_0^{\frac{1}{2}}\left\|\partial Z^{\alpha}\mathbf{v}^{\varepsilon}\right\|_0\mathrm{d}\tau\\
&+\sum_{\substack{|\alpha|\leq m\\\beta+\gamma=\alpha,\beta\neq0}}\int_{0}^{t}\left\| Z^{\beta}\mathbf{B}^{\varepsilon}\right\|_0^{\frac{1}{2}}\left\| \partial_1 Z^{\beta}\mathbf{B}^{\varepsilon}\right\|_0^{\frac{1}{2}}\\
 &\cdot\left\|Z^{\gamma}\mathbf{B}^{\varepsilon}\right\|_0^{\frac{1}{4}}\left\| \partial_{2}Z^{\gamma}\mathbf{B}^{\varepsilon}\right\|_0^{\frac{1}{4}}\left\| \partial_{3}Z^{\gamma}\mathbf{B}^{\varepsilon}\right\|_0^{\frac{1}{4}}\left\| \partial_{23}Z^{\gamma}\mathbf{B}^{\varepsilon}\right\|_0^{\frac{1}{4}}  \left\|\partial Z^{\alpha} \mathbf{v}^{\varepsilon}\right\|_0\mathrm{d}\tau\\
\lesssim{}&\sup_{0\leq \tau \leq t}\left(\left\| \mathbf{B}^{\varepsilon}\right\|_{m}+\left\|\partial_3\mathbf{B}^{\varepsilon}\right\|_{m-1}\right)\int_{0}^{t}\left(\left\|\left(\partial\mathbf{v}^{\varepsilon},\partial_h\mathbf{B}^{\varepsilon}\right)\right\|_{m}^2+\left\|\partial_h\partial_3\mathbf{B}^{\varepsilon}\right\|_{m-1}^2\right) \mathrm{d}\tau.
\end{align*}
By virtue of the commutator property and Hardy's inequality, we derive
\begin{align*}
H_{2,3,1}
\lesssim{}&\sum_{\substack{|\alpha|\leq m\\\alpha_3 \geq 1}}  \sum_{k=0}^{\alpha_3-1} \int_{0}^{t}\int_{\mathbb{R}_+^3}
\left| Z^{\alpha_{h}}_{h} \varphi\partial_3Z^k_3 \left(\mathbf{B}^{\varepsilon}\cdot\mathbf{B}^{\varepsilon}\right) \right|\left| \frac{Z^{\alpha} \mathbf{v}^{\varepsilon}}{x_3} \right|\mathrm{d}\mathbf{x}\mathrm{d}\tau\\
\lesssim{}&\sum_{\substack{|\alpha|\leq m\\\alpha_3 \geq 1}}  \sum_{k=0}^{\alpha_3-1} \int_{0}^{t}
\left\| Z^{\alpha_{h}}_{h} Z^{k+1}_3 \left(\mathbf{B}^{\varepsilon}\cdot\mathbf{B}^{\varepsilon}\right) \right\|_0\left\| \partial_3Z^{\alpha} \mathbf{v}^{\varepsilon} \right\|_0\mathrm{d}\tau\\
\lesssim{}&\sum_{\substack{|\alpha|\leq m\\\beta+\gamma=\alpha}}\int_{0}^{t}\left\| Z^{\beta}\mathbf{B}^{\varepsilon}\cdot Z^{\gamma}\mathbf{B}^{\varepsilon}\right\|_0  \left\|\partial Z^{\alpha} \mathbf{v}^{\varepsilon}\right\|_0\mathrm{d}\tau\\
\lesssim{}&\sup_{0\leq \tau \leq t}\left(\left\| \mathbf{B}^{\varepsilon}\right\|_{m}+\left\|\partial_3\mathbf{B}^{\varepsilon}\right\|_{m-1}\right)\int_{0}^{t}\left(\left\|\left(\partial\mathbf{v}^{\varepsilon},\partial_h\mathbf{B}^{\varepsilon}\right)\right\|_{m}^2+\left\|\partial_h\partial_3\mathbf{B}^{\varepsilon}\right\|_{m-1}^2\right) \mathrm{d}\tau.
\end{align*}
Summing the above two estimates, we get
\begin{align*}
&H_{2,2}+H_{2,3}\\
\lesssim{}&\sup_{0\leq \tau \leq t}\left(\left\| \mathbf{B}^{\varepsilon}\right\|_{m}+\left\|\partial_3\mathbf{B}^{\varepsilon}\right\|_{m-1}\right)\int_{0}^{t}\left(\left\|\left(\partial\mathbf{v}^{\varepsilon},\partial_h\mathbf{B}^{\varepsilon}\right)\right\|_{m}^2+\left\|\partial_h\partial_3\mathbf{B}^{\varepsilon}\right\|_{m-1}^2\right) \mathrm{d}\tau.
\end{align*}
We only consider the case $\alpha_3 \ge 1$ for $H_{2,4}$. The other cases can be handled similarly and are omitted here.
Using integration by parts and Lemmas \ref{Le1} and \ref{Le2}, we derive
\begin{align*}
H_{2,4}={}&\sum_{\substack{|\alpha|\leq m\\\alpha_3 \geq 1}}\int_{0}^{t} \Big(\left| \left( Z^{\alpha}\left(J\left(a^{\varepsilon}\right)\nabla a^{\varepsilon}\right),  Z^{\alpha} \mathbf{v}^{\varepsilon}\right)\right|\Big)\mathrm{d}\tau\\
={}&\sum_{\substack{|\alpha|\leq m\\\alpha_3 \geq 1}}\int_{0}^{t} \Big(\left| \left( \varphi\partial_3Z^{\alpha-e_3}\left(J\left(a^{\varepsilon}\right)\nabla a^{\varepsilon}\right),  Z^{\alpha} \mathbf{v}^{\varepsilon}\right)\right|\Big)\mathrm{d}\tau\\
\lesssim{}&\sum_{\substack{|\alpha|\leq m, \alpha_3 \geq 1\\\beta+\gamma=\alpha-e_3}}\int_{0}^{t} \Big(\left| \left( \varphi'Z^{\alpha-e_3}\left(J\left(a^{\varepsilon}\right)\nabla a^{\varepsilon}\right),  Z^{\alpha} \mathbf{v}^{\varepsilon}\right)\right|\\
&+\left| \left( \varphi Z^{\alpha-e_3}\left(J\left(a^{\varepsilon}\right)\nabla a^{\varepsilon}\right),  \partial_3Z^{\alpha} \mathbf{v}^{\varepsilon}\right)\right|\Big)\mathrm{d}\tau\\
\lesssim{}&\sum_{\substack{|\alpha|\leq m, \alpha_3 \geq 1\\\beta+\gamma=\alpha-e_3}}\int_{0}^{t}\int_{\mathbb{R}_+^3}
\left|Z^{\beta}J\left(a^{\varepsilon}\right)  \right|\left|Z^{\gamma}\partial a^{\varepsilon} \right|\left(\left| Z^{\alpha} \mathbf{v}^{\varepsilon} \right|+\left| \partial_3Z^{\alpha} \mathbf{v}^{\varepsilon} \right|\right)\mathrm{d}\mathbf{x}\mathrm{d}\tau\\
\lesssim{}&\sum_{\substack{|\alpha|\leq m, \alpha_3 \geq 1\\\beta+\gamma=\alpha-e_3}}\int_{0}^{t}
\left\|Z^{\beta}J\left(a^{\varepsilon}\right)  \right\|_{L^3}\left\|Z^{\gamma}\partial a^{\varepsilon} \right\|_0\left\| Z^{\alpha} \mathbf{v}^{\varepsilon} \right\|_{L^6}\mathrm{d}\tau\\
&+\sum_{\substack{|\alpha|\leq m, \alpha_3 \geq 1\\\beta+\gamma=\alpha-e_3}}\int_{0}^{t}\int_{\mathbb{R}_+^3}
\left|Z^{\beta}J\left(a^{\varepsilon}\right)  \right|\left|Z^{\gamma}\partial a^{\varepsilon} \right|\left| \partial_3Z^{\alpha} \mathbf{v}^{\varepsilon} \right|\mathrm{d}\mathbf{x}\mathrm{d}\tau\\
\lesssim{}&\sum_{\substack{|\alpha|\leq m, \alpha_3 \geq 1\\\beta+\gamma=\alpha-e_3}}\int_{0}^{t}
\left\|Z^{\beta}J\left(a^{\varepsilon}\right)  \right\|_{H^1}\left\|Z^{\gamma}\partial a^{\varepsilon} \right\|_0\left\| \partial Z^{\alpha} \mathbf{v}^{\varepsilon} \right\|_{0}\mathrm{d}\tau\\
&+\sum_{\substack{|\alpha|\leq m, \alpha_3 \geq 1\\}}\int_{0}^{t}
\left\|a^{\varepsilon}  \right\|_2^{\frac{1}{2}}\left\|\partial_3 a^{\varepsilon}  \right\|_2^{\frac{1}{2}}\left\|Z^{\alpha-e_3}\partial a^{\varepsilon} \right\|_0\left\| \partial_3Z^{\alpha} \mathbf{v}^{\varepsilon} \right\|_0\mathrm{d}\tau\\
&+\sum_{\substack{|\alpha|\leq m, \alpha_3 \geq 1}}\int_{0}^{t}
\left\|Z^{\alpha-e_3}J\left(a^{\varepsilon}\right)  \right\|_0^{\frac{1}{2}}\left\|\partial_3Z^{\alpha-e_3}J\left(a^{\varepsilon}\right)  \right\|_0^{\frac{1}{2}}\\
&\cdot\left\|\partial a^{\varepsilon} \right\|_0^{\frac{1}{4}}\left\|\partial_{1}\partial a^{\varepsilon} \right\|_0^{\frac{1}{4}}\left\|\partial_{2}\partial a^{\varepsilon} \right\|_0^{\frac{1}{4}}\left\|\partial_{12}\partial a^{\varepsilon} \right\|_0^{\frac{1}{4}}\left\| \partial_3Z^{\alpha}\mathbf{v}^{\varepsilon} \right\|_0\mathrm{d}\tau\\
&+\sum_{\substack{|\alpha|\leq m, \alpha_3 \geq 1\\\beta+\gamma=\alpha-e_3\\\beta\neq0, \gamma\neq0}}\int_{0}^{t}
\left\|Z^{\beta}J\left(a^{\varepsilon}\right)  \right\|_0^{\frac{1}{4}}\left\|\partial_{2}Z^{\beta}J\left(a^{\varepsilon}\right)  \right\|_0^{\frac{1}{4}}\left\|\partial_{3}Z^{\beta}J\left(a^{\varepsilon}\right)  \right\|_0^{\frac{1}{4}}\\
&\cdot\left\|\partial_{23}Z^{\beta}J\left(a^{\varepsilon}\right)  \right\|_0^{\frac{1}{4}}\left\|Z^{\gamma}\partial a^{\varepsilon} \right\|_0^{\frac{1}{2}}\left\|\partial_1Z^{\gamma}\partial a^{\varepsilon} \right\|_0^{\frac{1}{2}}\left\| \partial_3Z^{\alpha} \mathbf{v}^{\varepsilon} \right\|_0\mathrm{d}\tau\\
\lesssim{}&\sup_{0\leq \tau \leq t}\left(\left\| a^{\varepsilon}\right\|_{m}+\left\|\partial_3 a^{\varepsilon}\right\|_{m-1}\right)\int_{0}^{t}\left(\left\|\partial a^{\varepsilon}\right\|_{m-1}^2+\left\|\partial\mathbf{v}^{\varepsilon}\right\|_{m}^2+\left\|\partial\partial_3\mathbf{v}^{\varepsilon}\right\|_{m-1}^2\right) \mathrm{d}\tau.
\end{align*}
Similarly,
\[
H_{2,5}+H_{2,6}
\lesssim\sup_{0\leq \tau \leq t}\left(\left\| a^{\varepsilon}\right\|_{m}+\left\|\partial_3 a^{\varepsilon}\right\|_{m-1}\right)\int_{0}^{t}\left(\left\|\partial\mathbf{v}^{\varepsilon}\right\|_{m}^2+\left\|\partial\partial_3\mathbf{v}^{\varepsilon}\right\|_{m-1}^2\right) \mathrm{d}\tau.
\]
By an argument similar to that for $H_{2,4}$, we obtain
\begin{align*}
&H_{2,7}+H_{2,8}\\
={}&\sum_{\substack{|\alpha|\leq m\\ \alpha_3 \geq 1}}\int_{0}^{t} \Big(\left| \left( Z^{\alpha}\left(I\left(a^{\varepsilon}\right)\mathbf{B}^{\varepsilon} \cdot \nabla\mathbf{B}^{\varepsilon}\right),  Z^{\alpha} \mathbf{v}^{\varepsilon}\right)\right|+\left| \left( Z^{\alpha}\left(I\left(a^{\varepsilon}\right)\mathbf{B}^{\varepsilon}\nabla\mathbf{B}^{\varepsilon}\right),  Z^{\alpha} \mathbf{v}^{\varepsilon}\right)\right|\Big)\mathrm{d}\tau\\
\lesssim{}&\sum_{\substack{|\alpha|\leq m, \alpha_3 \geq 1\\\beta+\gamma+\delta=\alpha-e_3}}\int_{0}^{t}\int_{\mathbb{R}_+^3}
\left|Z^{\beta}I\left(a^{\varepsilon}\right)  \right|\left|Z^{\gamma} \mathbf{B}^{\varepsilon} \right|\left|Z^{\delta}\partial \mathbf{B}^{\varepsilon} \right|\left(\left| Z^{\alpha} \mathbf{v}^{\varepsilon} \right|+\left| \partial_3Z^{\alpha} \mathbf{v}^{\varepsilon} \right|\right)\mathrm{d}\mathbf{x}\mathrm{d}\tau\\
\lesssim{}&\sum_{\substack{|\alpha|\leq m, \alpha_3 \geq 1\\\beta+\gamma+\delta=\alpha-e_3\\|\beta|\leq 1}}\int_{0}^{t}\left\|Z^{\beta}I\left(a^{\varepsilon}\right)  \right\|_{2}^{\frac{1}{2}}\left\|\partial_{3}Z^{\beta}I\left(a^{\varepsilon}\right)  \right\|_{2}^{\frac{1}{2}}
\left\|Z^{\gamma} \mathbf{B}^{\varepsilon}  \right\|_0^{\frac{1}{4}}\left\|\partial_{2}Z^{\gamma} \mathbf{B}^{\varepsilon}  \right\|_0^{\frac{1}{4}}\left\|\partial_{3}Z^{\gamma} \mathbf{B}^{\varepsilon}  \right\|_0^{\frac{1}{4}}\\
&\cdot\left\|\partial_{23}Z^{\gamma} \mathbf{B}^{\varepsilon}  \right\|_0^{\frac{1}{4}}
\left\|Z^{\delta}\partial \mathbf{B}^{\varepsilon} \right\|_0^{\frac{1}{2}}\left\|\partial_1Z^{\delta}\partial \mathbf{B}^{\varepsilon} \right\|_0^{\frac{1}{2}}\left(\left\|Z^{\alpha} \mathbf{v}^{\varepsilon} \right\|_0+\left\| \partial_3Z^{\alpha} \mathbf{v}^{\varepsilon} \right\|_0\right)\mathrm{d}\tau\\
&+\sum_{\substack{|\alpha|\leq m, \alpha_3 \geq 1\\\beta+\gamma+\delta=\alpha-e_3\\|\beta|\geq 2}}\int_{0}^{t}\left\|Z^{\beta}I\left(a^{\varepsilon}\right)  \right\|_{0}^{\frac{1}{2}}\left\|\partial_{3}Z^{\beta}I\left(a^{\varepsilon}\right)  \right\|_{0}^{\frac{1}{2}}
\left\|Z^{\gamma} \mathbf{B}^{\varepsilon}  \right\|_2^{\frac{1}{2}}\left\|\partial_{3}Z^{\gamma} \mathbf{B}^{\varepsilon}  \right\|_2^{\frac{1}{2}}\left\|Z^{\delta}\partial \mathbf{B}^{\varepsilon} \right\|_0^{\frac{1}{4}}\\
&\cdot\left\|\partial_{1}Z^{\delta}\partial \mathbf{B}^{\varepsilon} \right\|_0^{\frac{1}{4}}\left\|\partial_{2}Z^{\delta}\partial \mathbf{B}^{\varepsilon} \right\|_0^{\frac{1}{4}}\left\|\partial_{12}Z^{\delta}\partial \mathbf{B}^{\varepsilon} \right\|_0^{\frac{1}{4}}\left(\left\|Z^{\alpha} \mathbf{v}^{\varepsilon} \right\|_0+\left\| \partial_3Z^{\alpha} \mathbf{v}^{\varepsilon} \right\|_0\right)\mathrm{d}\tau\\
\lesssim{}&\sup_{0\leq \tau \leq t}\left(\left\| \left(a^{\varepsilon}, \mathbf{B}^{\varepsilon}\right)\right\|_{m}^2+\left\|\partial_3 \left(a^{\varepsilon}, \mathbf{B}^{\varepsilon}\right)\right\|_{m-1}^2\right)\\
&\cdot\int_{0}^{t}\left(\left\|\partial a^{\varepsilon}\right\|_{m-1}^2+\left\|\left(\partial\mathbf{v}^{\varepsilon}, \partial_h \mathbf{B}^{\varepsilon}\right)\right\|_{m}^2+\left\|\partial_3\left(\partial\mathbf{v}^{\varepsilon}, \partial_h \mathbf{B}^{\varepsilon}\right)\right\|_{m-1}^2\right) \mathrm{d}\tau.
\end{align*}
Summing up all the above estimates, we arrive at
\begin{align*}
H_2
\lesssim{}&\sup_{0\leq \tau \leq t}\left(\left\| \left(a^{\varepsilon}, \mathbf{v}^{\varepsilon}, \mathbf{B}^{\varepsilon}\right)\right\|_{m}+\left\|\partial_3 \left(a^{\varepsilon}, \mathbf{v}^{\varepsilon}, \mathbf{B}^{\varepsilon}\right)\right\|_{m-1}\right)\\
&\cdot\int_{0}^{t}\left(\left\|\partial a^{\varepsilon}\right\|_{m-1}^2+\left\|\left(\partial\mathbf{v}^{\varepsilon}, \partial_h \mathbf{B}^{\varepsilon}\right)\right\|_{m}^2+\left\|\partial_3\left(\partial\mathbf{v}^{\varepsilon}, \partial_h \mathbf{B}^{\varepsilon}\right)\right\|_{m-1}^2\right) \mathrm{d}\tau.
\end{align*}
Now, we split $H_3$ into the following terms:
\begin{align*}
H_3
\lesssim{}&\sum_{|\alpha|\leq m}\int_{0}^{t} \left| \left( Z^{\alpha}\left(-\mathbf{v}^{\varepsilon} \cdot \nabla\mathbf{B}^{\varepsilon} - \mathbf{B}^{\varepsilon}\operatorname{div}\mathbf{v}^{\varepsilon} + \mathbf{B}^{\varepsilon} \cdot \nabla\mathbf{v}^{\varepsilon}\right),  Z^{\alpha} \mathbf{B}^{\varepsilon}\right)\right|\mathrm{d}\tau\\
\lesssim{}&\sum_{|\alpha|\leq m}\int_{0}^{t}\int_{\mathbb{R}_+^3} \left|\operatorname{div}\mathbf{v}^{\varepsilon}\right| \cdot \left|Z^\alpha \mathbf{B}^{\varepsilon}\right|^2 \mathrm{d}\mathbf{x}\mathrm{d}\tau+\sum_{|\alpha|\leq m}\int_{0}^{t}\left|\left( v_3^{\varepsilon} \left[Z^\alpha , \partial_3\right] \mathbf{B}^{\varepsilon}, Z^\alpha \mathbf{B}^{\varepsilon}\right) \right|\mathrm{d}\tau\\
&+\sum_{\substack{|\alpha|\leq m\\\beta+\gamma=\alpha,\beta \neq 0}}\int_{0}^{t} \left| \left( Z^\beta \mathbf{v}^{\varepsilon} \cdot Z^\gamma \nabla \mathbf{B}^{\varepsilon}, Z^\alpha \mathbf{B}^{\varepsilon} \right)\right|\mathrm{d}\tau+\sum_{|\alpha|\leq m}\int_{0}^{t} \left| \left( Z^{\alpha}\left(\mathbf{B}^{\varepsilon} \partial\mathbf{v}^{\varepsilon}\right),  Z^{\alpha} \mathbf{B}^{\varepsilon}\right)\right|\mathrm{d}\tau\\
\overset{def}{=}{}& H_{3,1}+H_{3,2}+H_{3,3}+H_{3,4}.
\end{align*}
With the help of  Lemma \ref{Le1}, there holds
\begin{align*}
&H_{3,1}+H_{3,4}\\
\lesssim{}&\sum_{|\alpha|\leq m}\int_{0}^{t} \int_{\mathbb{R}_+^3}\left|\mathbf{B}^{\varepsilon}\right|\left| Z^{\alpha}\partial\mathbf{v}^{\varepsilon}\right|  \left|Z^{\alpha} \mathbf{B}^{\varepsilon}\right|\mathrm{d}\mathbf{x}\mathrm{d}\tau+\sum_{\substack{|\alpha|\leq m\\\beta+\gamma=\alpha,\beta \neq 0}}\int_{0}^{t} \int_{\mathbb{R}_+^3}\left| Z^{\beta}\mathbf{B}^{\varepsilon}\right|\left| Z^{\gamma}\partial\mathbf{v}^{\varepsilon}\right|  \left|Z^{\alpha} \mathbf{B}^{\varepsilon}\right|\mathrm{d}\mathbf{x}\mathrm{d}\tau\\
\lesssim{}&\sum_{|\alpha|\leq m}\int_{0}^{t} \left\|\mathbf{B}^{\varepsilon}\right\|_{0}^{\frac{1}{4}}\left\|\partial_{2}\mathbf{B}^{\varepsilon}\right\|_{0}^{\frac{1}{4}}\left\|\partial_{3}\mathbf{B}^{\varepsilon}\right\|_{0}^{\frac{1}{4}}\left\|\partial_{23}\mathbf{B}^{\varepsilon}\right\|_{0}^{\frac{1}{4}}\left\| Z^{\alpha}\partial\mathbf{v}^{\varepsilon}\right\|_0  \left\|Z^{\alpha} \mathbf{B}^{\varepsilon}\right\|_{0}^{\frac{1}{2}}\left\|\partial_1Z^{\alpha} \mathbf{B}^{\varepsilon}\right\|_{0}^{\frac{1}{2}}\mathrm{d}\tau\\
&+\sum_{\substack{|\alpha|\leq m\\\beta+\gamma=\alpha,\beta \neq 0}}\int_{0}^{t} \left\| Z^{\beta}\mathbf{B}^{\varepsilon}\right\|_{0}^{\frac{1}{2}}\left\| \partial_{1}Z^{\beta}\mathbf{B}^{\varepsilon}\right\|_{0}^{\frac{1}{2}}\left\| Z^{\gamma}\partial\mathbf{v}^{\varepsilon}\right\|_{0}^{\frac{1}{2}} \left\| \partial_{3}Z^{\gamma}\partial\mathbf{v}^{\varepsilon}\right\|_{0}^{\frac{1}{2}}
 \left\|Z^{\alpha} \mathbf{B}^{\varepsilon}\right\|_{0}^{\frac{1}{2}}\left\|\partial_{2}Z^{\alpha} \mathbf{B}^{\varepsilon}\right\|_{0}^{\frac{1}{2}}\mathrm{d}\tau\\
\lesssim{}&\sup_{0\leq \tau \leq t}\left(\left\|\mathbf{B}^{\varepsilon}\right\|_{m}+\left\|\partial_3\mathbf{B}^{\varepsilon}\right\|_{m-1}\right)\int_{0}^{t}\left(\left\|\left(\partial\mathbf{v}^{\varepsilon}, \partial_h \mathbf{B}^{\varepsilon}\right)\right\|_{m}^2+\left\|\partial_3\left(\partial\mathbf{v}^{\varepsilon}, \partial_h \mathbf{B}^{\varepsilon}\right)\right\|_{m-1}^2\right) \mathrm{d}\tau.
\end{align*}
By virtue of the commutator property,  Hardy's inequality and Lemma \ref{Le1}, we derive
\begin{align*}
H_{3,2}
\lesssim{}&\sum_{\substack{|\alpha|\leq m\\\alpha_3 \geq 1}}  \sum_{k=0}^{\alpha_3-1} \int_{0}^{t}\int_{\mathbb{R}_+^3}\left|\frac{v_3^{\varepsilon}}{x_3}\right|
\left| Z^{\alpha_{h}}_{h} \varphi\partial_3Z^k_3 \mathbf{B}^{\varepsilon} \right| \left| Z^\alpha \mathbf{B}^{\varepsilon}\right| \mathrm{d}\mathbf{x}\mathrm{d}\tau\\
	\lesssim{}&\sum_{\substack{|\alpha|\leq m\\\alpha_3 \geq 1}}  \sum_{k=0}^{\alpha_3-1} \int_{0}^{t}
	\left\|\partial_3 v_3^{\varepsilon}\right\|_{0}^{\frac{1}{2}}\left\| \partial_3^2 v_3^{\varepsilon}\right\|_{0}^{\frac{1}{2}}\left\| Z^{\alpha_{h}}_{h}Z^{k+1}_3 \mathbf{B}^{\varepsilon} \right\|_0^{\frac{1}{2}}\\
&\cdot\left\| \partial_1Z^{\alpha_{h}}_{h}Z^{k+1}_3 \mathbf{B}^{\varepsilon} \right\|_0^{\frac{1}{2}}
\left\| Z^\alpha \mathbf{B}^{\varepsilon}\right\|_0^{\frac{1}{2}}\left\|\partial_2 Z^\alpha \mathbf{B}^{\varepsilon}\right\|_0^{\frac{1}{2}} \mathrm{d}\tau\\
\lesssim{}&\sup_{0\leq \tau \leq t}\left\|\mathbf{B}^{\varepsilon}\right\|_{m}\int_{0}^{t}\left(\left\|\left(\partial\mathbf{v}^{\varepsilon}, \partial_h \mathbf{B}^{\varepsilon}\right)\right\|_{m}^2+\left\|\partial_3\partial\mathbf{v}^{\varepsilon}\right\|_{m-1}^2\right)\mathrm{d}\tau.
\end{align*}
Applying Lemma \ref{Le1}, we get
\begin{align*}
H_{3,3}
\lesssim{}&\sum_{\substack{|\alpha|\leq m\\\beta+\gamma=\alpha,\beta \neq 0}}\int_{0}^{t} \left\|Z^\beta \mathbf{v}^{\varepsilon}\right\|_{0}^{\frac{1}{2}}\left\|\partial_3Z^\beta \mathbf{v}^{\varepsilon}\right\|_{0}^{\frac{1}{2}} \left\|Z^\gamma \partial \mathbf{B}^{\varepsilon}\right\|_{0}^{\frac{1}{2}}\left\|\partial_1Z^\gamma \partial \mathbf{B}^{\varepsilon}\right\|_{0}^{\frac{1}{2}}
 \left\|Z^\alpha \mathbf{B}^{\varepsilon} \right\|_{0}^{\frac{1}{2}}\left\|\partial_2Z^\alpha \mathbf{B}^{\varepsilon} \right\|_{0}^{\frac{1}{2}}\mathrm{d}\tau\\
\lesssim{}&\sup_{0\leq \tau \leq t}\left(\left\|\mathbf{B}^{\varepsilon}\right\|_{m}+\left\|\partial_3\mathbf{B}^{\varepsilon}\right\|_{m-1}\right)\int_{0}^{t}\left(\left\|\left(\partial\mathbf{v}^{\varepsilon}, \partial_h \mathbf{B}^{\varepsilon}\right)\right\|_{m}^2+\left\|\partial_h \partial_3\mathbf{B}^{\varepsilon}\right\|_{m-1}^2\right)\mathrm{d}\tau.
\end{align*}
Combining all the above estimates, we obtain
\begin{equation*}
H_3
\lesssim\sup_{0\leq \tau \leq t}\left(\left\|\mathbf{B}^{\varepsilon}\right\|_{m}+\left\|\partial_3\mathbf{B}^{\varepsilon}\right\|_{m-1}\right)\int_{0}^{t}\left(\left\|\left(\partial\mathbf{v}^{\varepsilon}, \partial_h \mathbf{B}^{\varepsilon}\right)\right\|_{m}^2+\left\|\partial_3\left(\partial\mathbf{v}^{\varepsilon}, \partial_h \mathbf{B}^{\varepsilon}\right)\right\|_{m-1}^2\right) \mathrm{d}\tau.
\end{equation*}
Finally, combining the previous estimates, we derive
\begin{equation}\label{G75}
\begin{aligned}
&\sum_{|\alpha|\leq m} \int_{0}^{t}  \left|{A} _{\alpha}\right|  \mathrm{d} \tau\\
\lesssim{}&\delta_4\int_{0}^{t} \left\|\partial a^{\varepsilon}\right\|_{m-1}^2\mathrm{d}\tau+\sup_{0\leq \tau \leq t}\left(\left\| \left(a^{\varepsilon}, \mathbf{v}^{\varepsilon}, \mathbf{B}^{\varepsilon}\right)\right\|_{m}+\left\|\partial_3 \left(a^{\varepsilon}, \mathbf{v}^{\varepsilon}, \mathbf{B}^{\varepsilon}\right)\right\|_{m-1}\right)\\
&\cdot\int_{0}^{t}\left(\left\|\partial a^{\varepsilon}\right\|_{m-1}^2+\left\|\left(\partial\mathbf{v}^{\varepsilon}, \partial_h \mathbf{B}^{\varepsilon}\right)\right\|_{m}^2+\left\|\partial_3\left(\partial\mathbf{v}^{\varepsilon}, \partial_h \mathbf{B}^{\varepsilon}\right)\right\|_{m-1}^2\right) \mathrm{d}\tau.
\end{aligned}
\end{equation}
Applying Young's inequality, we arrive at
\begin{align*}
&\sum_{|\alpha|\leq m-1} \int_{0}^{t}  \left|{C} _{\alpha}\right|  \mathrm{d} \tau\\
\lesssim{}&\delta_4\int_{0}^{t} \left\| \partial_{\tau}\left(a^{\varepsilon}, \mathbf{v}^{\varepsilon}, \mathbf{B}^{\varepsilon}\right)\right\|_{m-1}^2\mathrm{d}\tau
+\int_{0}^{t} \Big(\left\| a^{\varepsilon}\operatorname{div}\mathbf{v}^{\varepsilon}\right\|_{m-1}^{2}+\left\|\mathbf{v}^{\varepsilon}\cdot \nabla a^{\varepsilon}\right\|_{m-1}^{2}
+\left\| \left(f^{\varepsilon}_2, f^{\varepsilon}_3\right)\right\|_{m-1}^{2}\Big)\mathrm{d} \tau.
\end{align*}
By Lemma \ref{Le1}, we have
\begin{align*}
	&\int_{0}^{t} \left\|\mathbf{v}^{\varepsilon}\cdot\nabla a^{\varepsilon}\right\|_{m-1}^{2}\mathrm{d}\tau\\
    \lesssim{}&\sum_{|\alpha|\leq m-1}\int_{0}^{t} \left\|\mathbf{v}^{\varepsilon}\right\|_{L^{\infty}}^2 \left\| Z^{\alpha} \partial a^{\varepsilon}\right\|_0^2 \mathrm{d}\tau
    +\sum_{\substack{|\alpha|\leq m-1\\\beta+\gamma=\alpha, \beta\neq0}}\int_{0}^{t} \left\|Z^\beta \mathbf{v}^{\varepsilon}  \cdot Z^\gamma \nabla a^{\varepsilon}\right\|_0^2 \mathrm{d}\tau\\
    \lesssim{}&\sum_{|\alpha|\leq m-1}\int_{0}^{t} \left\|\mathbf{v}^{\varepsilon}\right\|_{2}\left\|\partial_3\mathbf{v}^{\varepsilon}\right\|_{2} \left\| Z^{\alpha} \partial a^{\varepsilon}\right\|_0^2 \mathrm{d}\tau\\
    &+\sum_{\substack{|\alpha|\leq m-1\\\beta+\gamma=\alpha, \beta\neq0}}\int_{0}^{t} \left\|Z^\beta \mathbf{v}^{\varepsilon}\right\|_0^{\frac{1}{2}}
    \left\|\partial_{2}Z^\beta \mathbf{v}^{\varepsilon}\right\|_0^{\frac{1}{2}}\left\|\partial_{3}Z^\beta \mathbf{v}^{\varepsilon}\right\|_0^{\frac{1}{2}}
    \left\|\partial_{23}Z^\beta \mathbf{v}^{\varepsilon}\right\|_0^{\frac{1}{2}}\\
    &\cdot\left\|Z^\gamma \partial a^{\varepsilon}\right\|_0\left\| \partial_{1}Z^\gamma \partial a^{\varepsilon}\right\|_0 \mathrm{d}\tau\\
    \lesssim{}&\sup_{0\leq \tau \leq t}\Big(\left\| \left(a^{\varepsilon}, \mathbf{v}^{\varepsilon}\right)\right\|_{m}^{2} + \left\| \partial_3 \left(a^{\varepsilon}, \mathbf{v}^{\varepsilon}\right)\right\|_{m-1}^{2}\Big)\int_{0}^{t}\Big(\left\|\partial a^{\varepsilon}\right\|_{m-1}^2+\left\|\partial\mathbf{v}^{\varepsilon}\right\|_{m}^2\Big) \mathrm{d}\tau.
\end{align*}
By similar arguments, we derive
\begin{align*}
	&\int_{0}^{t} \left\|f^{\varepsilon}_3\right\|_{m-1}^{2}\mathrm{d}\tau\\
    \lesssim{}&\sum_{\substack{|\alpha|\leq m-1\\\beta+\gamma=\alpha}}\int_{0}^{t}\Big( \left\|Z^\beta \mathbf{v}^{\varepsilon}  \cdot Z^\gamma \partial\mathbf{B}^{\varepsilon}\right\|_0^2
    +\left\|Z^\beta \mathbf{B}^{\varepsilon}  \cdot Z^\gamma \partial\mathbf{v}^{\varepsilon}\right\|_0^2 \Big)\mathrm{d}\tau\\
    \lesssim{}&\sum_{\substack{|\alpha|\leq m-1\\\beta+\gamma=\alpha}}\int_{0}^{t}\Big(\left\|Z^\beta \mathbf{v}^{\varepsilon}\right\|_0^{\frac{1}{2}}
    \left\|\partial_{2}Z^\beta \mathbf{v}^{\varepsilon}\right\|_0^{\frac{1}{2}}\left\|\partial_{3}Z^\beta \mathbf{v}^{\varepsilon}\right\|_0^{\frac{1}{2}}
    \left\|\partial_{23}Z^\beta \mathbf{v}^{\varepsilon}\right\|_0^{\frac{1}{2}}
    \left\| Z^\gamma \partial\mathbf{B}^{\varepsilon}\right\|_0\left\| \partial_{1}Z^\gamma \partial\mathbf{B}^{\varepsilon}\right\|_0\\
    &+\left\|Z^\beta \mathbf{B}^{\varepsilon}\right\|_0^{\frac{1}{2}}
    \left\|\partial_{1}Z^\beta \mathbf{B}^{\varepsilon}\right\|_0^{\frac{1}{2}}\left\|\partial_{2}Z^\beta \mathbf{B}^{\varepsilon}\right\|_0^{\frac{1}{2}}
    \left\|\partial_{12}Z^\beta \mathbf{B}^{\varepsilon}\right\|_0^{\frac{1}{2}}
    \left\| Z^\gamma \partial\mathbf{v}^{\varepsilon}\right\|_0\left\| \partial_{3}Z^\gamma \partial\mathbf{v}^{\varepsilon}\right\|_0\Big)\mathrm{d}\tau\\
    \lesssim{}&\sup_{0\leq \tau \leq t}\Big(\left\| \left(\mathbf{v}^{\varepsilon}, \mathbf{B}^{\varepsilon}\right)\right\|_{m}^{2} + \left\| \partial_3 \left(\mathbf{v}^{\varepsilon}, \mathbf{B}^{\varepsilon}\right)\right\|_{m-1}^{2}\Big)\int_{0}^{t}\Big(\left\|\left(\partial\mathbf{v}^{\varepsilon}, \partial_h \mathbf{B}^{\varepsilon}\right)\right\|_{m}^2
    +\left\|\partial_3\left(\partial\mathbf{v}^{\varepsilon}, \partial_h \mathbf{B}^{\varepsilon}\right)\right\|_{m-1}^2\Big) \mathrm{d}\tau.
\end{align*}
Combining all the above estimates, \eqref{G30} and \eqref{G51}, we obtain
\begin{equation}\label{G80}
\begin{aligned}
&\sum_{|\alpha|\leq m-1} \int_{0}^{t}  \left|{C} _{\alpha}\right|  \mathrm{d} \tau\\
\lesssim{}&\delta_4\int_{0}^{t} \left\| \partial_{\tau}\left(a^{\varepsilon}, \mathbf{v}^{\varepsilon}, \mathbf{B}^{\varepsilon}\right)\right\|_{m-1}^2\mathrm{d}\tau
+\sup_{0\leq \tau \leq t}\Big(\left\| \left(a^{\varepsilon}, \mathbf{v}^{\varepsilon}, \mathbf{B}^{\varepsilon}\right)\right\|_{m}^2+\left\|\partial_3 \left(a^{\varepsilon}, \mathbf{v}^{\varepsilon}, \mathbf{B}^{\varepsilon}\right)\right\|_{m-1}^2\Big)\\
&\cdot\int_{0}^{t}\Big(\left\|\partial a^{\varepsilon}\right\|_{m-1}^2+\left\|\left(\partial\mathbf{v}^{\varepsilon}, \partial_h \mathbf{B}^{\varepsilon}\right)\right\|_{m}^2
+\left\|\partial_3\left(\partial\mathbf{v}^{\varepsilon}, \partial_h \mathbf{B}^{\varepsilon}\right)\right\|_{m-1}^2\Big) \mathrm{d}\tau.
\end{aligned}
\end{equation}
As for $\sum_{|\alpha|\leq m-1} \int_{0}^{t}  \left|{D} _{\alpha}\right|  \mathrm{d} \tau$, we have the following splitting:
\begin{align*}
\sum_{|\alpha|\leq m-1} \int_{0}^{t}  \left|{D} _{\alpha}\right|  \mathrm{d} \tau
\lesssim{}&\sum_{|\alpha|\leq m-1}\int_{0}^{t} \left| \left(  Z^{\alpha}\partial_3\left(\mathbf{v}^{\varepsilon}\cdot \nabla a^{\varepsilon} \right), Z^{\alpha}\partial_{3} a^{\varepsilon} \right)\right|\mathrm{d}\tau\\
\lesssim{}&\sum_{|\alpha|\leq m-1}\int_{0}^{t}\int_{\mathbb{R}_+^3} \left|\operatorname{div}\mathbf{v}^{\varepsilon}\right| \cdot \left|Z^\alpha \partial_{3}a^{\varepsilon}\right|^2 \mathrm{d}\mathbf{x}\mathrm{d}\tau\\
&+\sum_{|\alpha|\leq m-1}\int_{0}^{t}\left|\left( v_3^{\varepsilon} \left[Z^\alpha , \partial_3\right]\partial_{3} a^{\varepsilon}, Z^\alpha \partial_{3}a^{\varepsilon}\right) \right|\mathrm{d}\tau\\
&+\sum_{\substack{|\alpha|\leq m-1\\\beta+\gamma=\alpha,\beta \neq 0}}\int_{0}^{t} \left| \left( Z^\beta \mathbf{v}^{\varepsilon} \cdot Z^\gamma \nabla \partial_{3}a^{\varepsilon}, Z^\alpha \partial_{3}a^{\varepsilon} \right)\right|\mathrm{d}\tau\\
&+\sum_{|\alpha|\leq m-1}\int_{0}^{t} \left| \left(  Z^{\alpha}\left(\partial_3\mathbf{v}^{\varepsilon}\cdot \nabla a^{\varepsilon} \right), Z^{\alpha}\partial_{3} a^{\varepsilon} \right)\right|\mathrm{d}\tau\\
\overset{def}{=}{}& H_{4,1}+H_{4,2}+H_{4,3}+H_{4,4}.
\end{align*}
Lemma \ref{Le1} implies that
\begin{align*}
H_{4,1}
\lesssim{}&\sum_{|\alpha|\leq m-1}\int_{0}^{t} \left\| \operatorname{div}\mathbf{v}^{\varepsilon}\right\|_2^{\frac{1}{2}}\left\| \partial_3\operatorname{div}\mathbf{v}^{\varepsilon}\right\|_2^{\frac{1}{2}}\left\|Z^{\alpha} \partial_{3}a^{\varepsilon}\right\|_0^2\mathrm{d}\tau\\
\lesssim{}&\sup_{0\leq \tau \leq t}\left\| \partial_{3}a^{\varepsilon}\right\|_{m-1}\int_{0}^{t}\left(\left\| \partial a^{\varepsilon}\right\|_{m-1}^{2}+\left\|\partial\mathbf{v}^{\varepsilon}\right\|_{m}^2+\left\|\partial\partial_3\mathbf{v}^{\varepsilon}\right\|_{m-1}^2\right) \mathrm{d}\tau.
\end{align*}
By virtue of the commutator property, Hardy's inequality and Lemma \ref{Le1}, we derive
\begin{align*}
H_{4,2}
\lesssim{}&\sum_{\substack{|\alpha|\leq m-1\\\alpha_3 \geq 1}}  \sum_{k=0}^{\alpha_3-1} \int_{0}^{t}\int_{\mathbb{R}_+^3}\left|\frac{v_3^{\varepsilon}}{x_3}\right|
\left| Z^{\alpha_{h}}_{h} \varphi\partial_3Z^k_3\partial_3 a^{\varepsilon} \right| \left| Z^\alpha\partial_3 a^{\varepsilon}\right| \mathrm{d}\mathbf{x}\mathrm{d}\tau\\
\lesssim{}&\sum_{\substack{|\alpha|\leq m-1\\\alpha_3 \geq 1}}  \sum_{k=0}^{\alpha_3-1} \int_{0}^{t}\int_{\mathbb{R}_+^3}\left\|\partial_3v_3^{\varepsilon}\right\|_{L^{\infty}}
\left\| Z^{\alpha_{h}}_{h} Z^{k+1}_3\partial_3 a^{\varepsilon} \right\|_0 \left\| Z^\alpha\partial_3 a^{\varepsilon}\right\|_0 \mathrm{d}\mathbf{x}\mathrm{d}\tau\\
\lesssim{}&\sum_{\substack{|\alpha|\leq m-1\\\alpha_3 \geq 1}}  \sum_{k=0}^{\alpha_3-1} \int_{0}^{t}
\left\| \partial_3 v_3^{\varepsilon}\right\|_{2}^{\frac{1}{2}}\left\| \partial_3^2 v_3\right\|_{2}^{\frac{1}{2}}\left\| Z^{\alpha_{h}}_{h}Z^{k+1}_3\partial_3 a^{\varepsilon} \right\|_0 \left\| Z^\alpha \partial_3a^{\varepsilon}\right\|_0 \mathrm{d}\tau\\
\lesssim{}&\sup_{0\leq \tau \leq t}\left\| \partial_{3}a^{\varepsilon}\right\|_{m-1}\int_{0}^{t}\left(\left\| \partial a^{\varepsilon}\right\|_{m-1}^{2}+\left\|\partial\mathbf{v}^{\varepsilon}\right\|_{m}^2+\left\|\partial\partial_3\mathbf{v}^{\varepsilon}\right\|_{m-1}^2\right) \mathrm{d}\tau.
\end{align*}
It should be pointed out that the estimate of term $H_{4,3}$ is somewhat complicated. Thus, we decompose
this term into two components as follows:
\begin{align*}
H_{4,3}
\lesssim{}&\sum_{\substack{|\alpha|\leq m-1\\\beta+\gamma=\alpha,\beta \neq 0}}\int_{0}^{t} \Bigg(\left|\int_{\mathbb{R}_+^3} Z^\beta \mathbf{v}^{\varepsilon}_{h} \cdot Z^\gamma \partial_h \partial_3 a^{\varepsilon}  \cdot Z^\alpha \partial_3 a^{\varepsilon}  \mathrm{d}\mathbf{x}\right|\\
&+\left|\int_{\mathbb{R}_+^3} Z^\beta {v}_{3}^{\varepsilon}  Z^\gamma \partial_3^2 a^{\varepsilon}  \cdot Z^\alpha \partial_3 a^{\varepsilon}   \mathrm{d}\mathbf{x}\right|\Bigg)\mathrm{d}\tau\\
\overset{def}{=}{}&H_{4,3,1}+H_{4,3,2}.
\end{align*}
It follows from Lemma \ref{Le1} that
\begin{align*}
H_{4,3,1}
\lesssim{}&
\sum_{\substack{|\alpha|\leq m-1\\\beta+\gamma=\alpha,|\beta|=1}}\int_{0}^{t}
\left\| Z^\beta \mathbf{v}^{\varepsilon}_{h} \right\|_{2}^{\frac{1}{2}}
\left\| \partial_3 Z^\beta \mathbf{v}^{\varepsilon}_{h} \right\|_{2}^{\frac{1}{2}}
\left\| Z^\gamma \partial_h\partial_3 a^{\varepsilon}  \right\|_{0}
\left\|Z^\alpha \partial_3 a^{\varepsilon}  \right\|_{0}\mathrm{d}\tau\\
&+\sum_{\substack{|\alpha|\leq m-1\\\beta+\gamma=\alpha,|\beta|\geq2}}\int_{0}^{t}
\left\| Z^\beta \mathbf{v}^{\varepsilon}_{h} \right\|_{0}^{\frac{1}{4}}
\left\| \partial_1 Z^\beta \mathbf{v}^{\varepsilon}_{h} \right\|_{0}^{\frac{1}{4}}
\left\| \partial_3 Z^\beta \mathbf{v}^{\varepsilon}_{h} \right\|_{0}^{\frac{1}{4}}
\left\| \partial_{13}Z^\beta \mathbf{v}^{\varepsilon}_{h} \right\|_{0}^{\frac{1}{4}}\\
&\cdot\left\| Z^\gamma\partial_h  \partial_3 a^{\varepsilon}  \right\|_{0}^{\frac{1}{2}}
\left\| \partial_{2}Z^\gamma\partial_h  \partial_3 a^{\varepsilon}  \right\|_{0}^{\frac{1}{2}}
\left\|Z^\alpha \partial_3 a^{\varepsilon}  \right\|_{0}\mathrm{d}\tau \\
\lesssim{}&\sup_{0\leq \tau \leq t}\left\| \partial_{3}a^{\varepsilon}\right\|_{m-1}\int_{0}^{t}\left(\left\| \partial a^{\varepsilon}\right\|_{m-1}^{2}+\left\|\partial\mathbf{v}^{\varepsilon}\right\|_{m}^2\right) \mathrm{d}\tau.
\end{align*}
The term \(H_{4,3,2}\) requires further decomposition:
\begin{align*}
H_{4,3,2}
\lesssim{}&\sum_{\substack{|\alpha|\leq m-1\\\beta+\gamma=\alpha,\beta \neq 0}}\int_{0}^{t} \Bigg(\left|\int_{\mathbb{R}^2 \times [1,+\infty)} Z^\beta {v}_{3}^{\varepsilon}  Z^\gamma \partial_3^2 a^{\varepsilon}  \cdot Z^\alpha \partial_3 a^{\varepsilon}  \mathrm{d}\mathbf{x}\right|\\
&+\left|\int_{\mathbb{R}^2 \times [0,1)} Z^\beta {v}_{3}^{\varepsilon}  Z^\gamma \partial_3^2 a^{\varepsilon}  \cdot Z^\alpha \partial_3 a^{\varepsilon}  \mathrm{d}\mathbf{x}\right|\Bigg)\mathrm{d}\tau\\
\overset{def}{=}{}&H_{4,3,2,1}+H_{4,3,2,2}.
\end{align*}
We use \(\varphi \ge \frac{1}{2}\) on \([1,+\infty)\) to get
\begin{align*}
H_{4,3,2,1}\lesssim{}&\sum_{\substack{|\alpha|\leq m-1\\\beta+\gamma=\alpha,\beta \neq 0}}\int_{0}^{t}  \int_{\mathbb{R}^2 \times [1,+\infty)} \left|Z^\beta {v}_{3}^{\varepsilon}\right|  \left| \varphi Z^\gamma \partial_{3}^2 a^{\varepsilon} \right|  \left|Z^\alpha \partial_3 a^{\varepsilon} \right| \mathrm{d}\mathbf{x}\mathrm{d}\tau\\
\lesssim{}&\sum_{\substack{|\alpha|\leq m-1\\\beta+\gamma=\alpha,\beta \neq 0\\\tilde{\gamma}\leq\gamma+e_3}}\int_{0}^{t}  \int_{\mathbb{R}_{+}^3 } \left|Z^\beta {v}_{3}^{\varepsilon}\right|  \left| Z^{\tilde{\gamma}} \partial_3 a^{\varepsilon}  \right|  \left|Z^\alpha \partial_3 a^{\varepsilon} \right| \mathrm{d}\mathbf{x}\mathrm{d}\tau\\
\lesssim{}&\sum_{\substack{|\alpha|\leq m-1\\\beta+\gamma=\alpha, |\beta|=1\\\tilde{\gamma}\leq\gamma+e_3}}\int_{0}^{t}  \left\|Z^\beta {v}_{3}^{\varepsilon}\right\|_{2}^{\frac{1}{2}}\left\|\partial_3Z^\beta {v}_{3}^{\varepsilon}\right\|_{2}^{\frac{1}{2}}  \left\| Z^{\tilde{\gamma}} \partial_3 a^{\varepsilon}  \right\|_0  \left\|Z^\alpha \partial_3 a^{\varepsilon} \right\|_0 \mathrm{d}\tau\\	
&+\sum_{\substack{|\alpha|\leq m-1\\\beta+\gamma=\alpha, |\beta|\geq2\\\tilde{\gamma}\leq\gamma+e_3}}\int_{0}^{t}  \left\|Z^\beta {v}_{3}^{\varepsilon}\right\|_{0}^{\frac{1}{4}}\left\|\partial_{2}Z^\beta {v}_{3}^{\varepsilon}\right\|_{0}^{\frac{1}{4}}\left\|\partial_{3}Z^\beta {v}_{3}^{\varepsilon}\right\|_{0}^{\frac{1}{4}}\left\|\partial_{23}Z^\beta {v}_{3}^{\varepsilon}\right\|_{0}^{\frac{1}{4}}\\
&\cdot\left\| Z^{\tilde{\gamma}} \partial_3 a^{\varepsilon}  \right\|_0^{\frac{1}{2}}\left\| \partial_1Z^{\tilde{\gamma}} \partial_3 a^{\varepsilon}  \right\|_0^{\frac{1}{2}}  \left\|Z^\alpha \partial_3 a^{\varepsilon} \right\| \mathrm{d}\tau\\	
\lesssim{}&\sup_{0\leq \tau \leq t}\left\| \partial_{3}a^{\varepsilon}\right\|_{m-1}\int_{0}^{t}\left(\left\| \partial a^{\varepsilon}\right\|_{m-1}^{2}+\left\|\partial\mathbf{v}^{\varepsilon}\right\|_{m}^2\right) \mathrm{d}\tau.
\end{align*}
Using \(\frac{\varphi(x_3)}{x_3}\ge\frac{1}{2}\) on \([0,1)\), we obtain
\begin{align*}
H_{4,3,2,2}\lesssim{}&\sum_{\substack{|\alpha|\leq m-1\\\beta+\gamma=\alpha,\beta \neq 0}}\int_{0}^{t}  \int_{\mathbb{R}^2 \times [0, 1)} \left|\frac{Z^\beta {v}_{3}^{\varepsilon}}{x_3}\right|  \left| \varphi Z^\gamma \partial_{3}^2 a^{\varepsilon} \right|  \left|Z^\alpha \partial_3 a^{\varepsilon} \right| \mathrm{d}\mathbf{x}\mathrm{d}\tau\\
\lesssim{}&\sum_{\substack{|\alpha|\leq m-1\\\beta+\gamma=\alpha,\beta \neq 0\\\tilde{\gamma}\leq\gamma+e_3}}\int_{0}^{t}  \int_{\mathbb{R}^2 \times [0, 1)} \left|\frac{Z^\beta {v}_{3}^{\varepsilon}}{x_3}\right|  \left| Z^{\tilde{\gamma}} \partial_3 a^{\varepsilon}  \right|  \left|Z^\alpha \partial_3 a^{\varepsilon} \right| \mathrm{d}\mathbf{x}\mathrm{d}\tau\\
\lesssim{}&\sum_{\substack{|\alpha|\leq m-1\\\beta+\gamma=\alpha, |\beta|=1\\\tilde{\gamma}\leq\gamma+e_3}}\int_{0}^{t}  \left\|\partial_3Z^\beta {v}_{3}^{\varepsilon}\right\|_{2}^{\frac{1}{2}}\left\|\partial_3^2Z^\beta {v}_{3}^{\varepsilon}\right\|_{2}^{\frac{1}{2}}  \left\| Z^{\tilde{\gamma}} \partial_3 a^{\varepsilon}  \right\|_0  \left\|Z^\alpha \partial_3 a^{\varepsilon} \right\|_0 \mathrm{d}\tau\\	
&+\sum_{\substack{|\alpha|\leq m-1\\\beta+\gamma=\alpha, |\beta|\geq2\\\tilde{\gamma}\leq\gamma+e_3, \gamma\neq0}}\int_{0}^{t}  \left\|\partial_3Z^\beta {v}_{3}^{\varepsilon}\right\|_{0}^{\frac{1}{4}}\left\|\partial_{23}Z^\beta {v}_{3}^{\varepsilon}\right\|_{0}^{\frac{1}{4}}
\left\|\partial_{3}^2Z^\beta {v}_{3}^{\varepsilon}\right\|_{0}^{\frac{1}{4}}\left\|\partial_{233}Z^\beta {v}_{3}^{\varepsilon}\right\|_{0}^{\frac{1}{4}}\\
&\cdot\left\| Z^{\tilde{\gamma}} \partial_3 a^{\varepsilon}  \right\|_0^{\frac{1}{2}}\left\| \partial_1Z^{\tilde{\gamma}} \partial_3 a^{\varepsilon}  \right\|_0^{\frac{1}{2}}  \left\|Z^\alpha \partial_3 a^{\varepsilon} \right\|_0 \mathrm{d}\tau\\	
&+\sum_{\substack{2\leq|\alpha|\leq m-1\\\tilde{\gamma}\leq e_3}}\int_{0}^{t}  \left\|\partial_3Z^\alpha {v}_{3}^{\varepsilon}\right\|_{0}^{\frac{1}{2}}\left\|\partial_{3}^2Z^\alpha {v}_{3}^{\varepsilon}\right\|_{0}^{\frac{1}{2}}\left\|Z^\alpha \partial_3 a^{\varepsilon} \right\|_0\\
&\cdot\left\| Z^{\tilde{\gamma}} \partial_3 a^{\varepsilon}  \right\|_0^{\frac{1}{4}}\left\| \partial_1Z^{\tilde{\gamma}} \partial_3 a^{\varepsilon}  \right\|_0^{\frac{1}{4}}\left\| \partial_2Z^{\tilde{\gamma}} \partial_3 a^{\varepsilon}  \right\|_0^{\frac{1}{4}}\left\| \partial_{12}Z^{\tilde{\gamma}} \partial_3 a^{\varepsilon}  \right\|_0^{\frac{1}{4}}   \mathrm{d}\tau\\
\lesssim{}&\sup_{0\leq \tau \leq t}\left\| \partial_{3}a^{\varepsilon}\right\|_{m-1}\int_{0}^{t}\left(\left\| \partial a^{\varepsilon}\right\|_{m-1}^{2}+\left\|\partial\mathbf{v}^{\varepsilon}\right\|_{m}^2+\left\|\partial\partial_3\mathbf{v}^{\varepsilon}\right\|_{m-1}^2\right) \mathrm{d}\tau.
\end{align*}
Summing up the foregoing estimates, we arrive at
\[
H_{4,3,2}\lesssim\sup_{0\leq \tau \leq t}\left\| \partial_{3}a^{\varepsilon}\right\|_{m-1}\int_{0}^{t}\left(\left\| \partial a^{\varepsilon}\right\|_{m-1}^{2}+\left\|\partial\mathbf{v}^{\varepsilon}\right\|_{m}^2+\left\|\partial\partial_3\mathbf{v}^{\varepsilon}\right\|_{m-1}^2\right) \mathrm{d}\tau.
\]
Consequently, we have
\[
H_{4,3}\lesssim\sup_{0\leq \tau \leq t}\left\| \partial_{3}a^{\varepsilon}\right\|_{m-1}\int_{0}^{t}\left(\left\| \partial a^{\varepsilon}\right\|_{m-1}^{2}+\left\|\partial\mathbf{v}^{\varepsilon}\right\|_{m}^2+\left\|\partial\partial_3\mathbf{v}^{\varepsilon}\right\|_{m-1}^2\right) \mathrm{d}\tau.
\]
By Lemma \ref{Le1} again,  it holds that
\begin{align*}
H_{4,4}
\lesssim{}&\sum_{|\alpha|\leq m-1}\int_{0}^{t} \left(\left| \left( \partial_{3}\mathbf{v}^{\varepsilon} \cdot Z^\alpha \nabla a^{\varepsilon}, Z^\alpha \partial_{3}a^{\varepsilon} \right)\right|+\left| \left( Z^\alpha\partial_{3}\mathbf{v}^{\varepsilon} \cdot \nabla a^{\varepsilon}, Z^\alpha \partial_{3}a^{\varepsilon} \right)\right|\right)\mathrm{d}\tau\\
&+\sum_{\substack{|\alpha|\leq m-1\\\beta+\gamma=\alpha\\\beta \neq 0, \gamma \neq 0}}\int_{0}^{t} \left| \left( Z^\beta \partial_{3}\mathbf{v}^{\varepsilon} \cdot Z^\gamma \nabla a^{\varepsilon}, Z^\alpha \partial_{3}a^{\varepsilon} \right)\right|\mathrm{d}\tau\\
\lesssim{}&\sum_{|\alpha|\leq m-1}\int_{0}^{t}\Big( \left\|\partial_{3}\mathbf{v}^{\varepsilon}\right\|_{2}^{\frac{1}{2}}\left\|\partial_{3}^2\mathbf{v}^{\varepsilon}\right\|_{2}^{\frac{1}{2}}   \left\|Z^\alpha \partial a^{\varepsilon}\right\|_{0}  \left\|Z^\alpha \partial_{3}a^{\varepsilon}\right\|_{0}
+\left\|Z^\alpha\partial_{3}\mathbf{v}^{\varepsilon}\right\|_{0}^{\frac{1}{2}}\\
&\cdot\left\|\partial_{3}Z^\alpha\partial_{3}\mathbf{v}^{\varepsilon}\right\|_{0}^{\frac{1}{2}}\left\|\partial a^{\varepsilon}\right\|_{0}^{\frac{1}{4}}\left\|\partial_{1}\partial a^{\varepsilon}\right\|_{0}^{\frac{1}{4}}\left\|\partial_{2}\partial a^{\varepsilon}\right\|_{0}^{\frac{1}{4}}\left\|\partial_{12}\partial a^{\varepsilon}\right\|_{0}^{\frac{1}{4}}\left\|Z^\alpha \partial_{3}a^{\varepsilon}\right\|_{0}\Big)\mathrm{d}\tau\\
&+\sum_{\substack{|\alpha|\leq m-1\\\beta+\gamma=\alpha\\\beta \neq 0, \gamma \neq 0}}\int_{0}^{t}\left\| Z^\beta \partial_{3}\mathbf{v}^{\varepsilon}\right\|_{0}^{\frac{1}{4}}\left\|\partial_{2} Z^\beta \partial_{3}\mathbf{v}^{\varepsilon}\right\|_{0}^{\frac{1}{4}}\left\|\partial_{3} Z^\beta \partial_{3}\mathbf{v}^{\varepsilon}\right\|_{0}^{\frac{1}{4}}\left\|\partial_{23} Z^\beta \partial_{3}\mathbf{v}^{\varepsilon}\right\|_{0}^{\frac{1}{4}}\\
&\cdot  \left\|Z^\gamma \partial a^{\varepsilon}\right\|_{0}^{\frac{1}{2}}\left\|\partial_{1}Z^\gamma \partial a^{\varepsilon}\right\|_{0}^{\frac{1}{2}}\left\|Z^\alpha \partial_{3}a^{\varepsilon}\right\|_{0}\mathrm{d}\tau\\
\lesssim{}&\sup_{0\leq \tau \leq t}\left\| \partial_{3}a^{\varepsilon}\right\|_{m-1}\int_{0}^{t}\left(\left\| \partial a^{\varepsilon}\right\|_{m-1}^{2}+\left\|\partial\mathbf{v}^{\varepsilon}\right\|_{m}^2+\left\|\partial\partial_3\mathbf{v}^{\varepsilon}\right\|_{m-1}^2\right) \mathrm{d}\tau.
\end{align*}
Collecting all the above estimates, we conclude
\begin{equation}\label{G87}
\sum_{|\alpha|\leq m-1} \int_{0}^{t}  \left|{D} _{\alpha}\right|  \mathrm{d} \tau\lesssim\sup_{0\leq \tau \leq t}\left\| \partial_{3}a^{\varepsilon}\right\|_{m-1}\int_{0}^{t}\left(\left\| \partial a^{\varepsilon}\right\|_{m-1}^{2}+\left\|\partial\mathbf{v}^{\varepsilon}\right\|_{m}^2+\left\|\partial\partial_3\mathbf{v}^{\varepsilon}\right\|_{m-1}^2\right) \mathrm{d}\tau.
\end{equation}
Turning to estimate \(\sum_{|\alpha|\leq m-1} \int_{0}^{t}  \left|{E} _{\alpha}\right|  \mathrm{d} \tau\), we present the following splitting:
\begin{align*}
&\sum_{|\alpha|\leq m-1} \int_{0}^{t}  \left|{E} _{\alpha}\right|  \mathrm{d} \tau\\
\lesssim{}&\sum_{|\alpha|\leq m-1}\int_{0}^{t} \left| \left(  Z^{\alpha}\partial_3\left(-\mathbf{v}^{\varepsilon} \cdot \nabla\mathbf{B}^{\varepsilon} - \mathbf{B}^{\varepsilon}\operatorname{div}\mathbf{v}^{\varepsilon} + \mathbf{B}^{\varepsilon} \cdot \nabla\mathbf{v}^{\varepsilon}\right), Z^{\alpha}\partial_{3} \mathbf{B}^{\varepsilon} \right)\right|\mathrm{d}\tau\\
\lesssim{}&\sum_{|\alpha|\leq m-1}\int_{0}^{t}\int_{\mathbb{R}_+^3} \left|\operatorname{div}\mathbf{v}^{\varepsilon}\right| \cdot \left|Z^\alpha \partial_{3}\mathbf{B}^{\varepsilon}\right|^2 \mathrm{d}\mathbf{x}\mathrm{d}\tau+\sum_{|\alpha|\leq m-1}\int_{0}^{t}\left|\left( v_3^{\varepsilon} \left[Z^\alpha , \partial_3\right]\partial_{3} \mathbf{B}^{\varepsilon}, Z^\alpha \partial_{3}\mathbf{B}^{\varepsilon}\right) \right|\mathrm{d}\tau\\
&+\sum_{\substack{|\alpha|\leq m-1\\\beta+\gamma=\alpha,\beta \neq 0}}\int_{0}^{t} \left| \left( Z^\beta \mathbf{v}^{\varepsilon} \cdot Z^\gamma \nabla \partial_{3}\mathbf{B}^{\varepsilon}, Z^\alpha \partial_{3}\mathbf{B}^{\varepsilon} \right)\right|\mathrm{d}\tau
+\sum_{|\alpha|\leq m-1}\int_{0}^{t} \left| \left(  Z^{\alpha}\left(\partial\mathbf{v}^{\varepsilon}\cdot \partial\mathbf{B}^{\varepsilon} \right), Z^{\alpha}\partial_{3} \mathbf{B}^{\varepsilon} \right)\right|\mathrm{d}\tau\\
&+\sum_{|\alpha|\leq m-1}\int_{0}^{t} \left| \left(  Z^{\alpha}\left(\partial^2\mathbf{v}^{\varepsilon}\cdot \mathbf{B}^{\varepsilon}\right), Z^{\alpha}\partial_{3} \mathbf{B}^{\varepsilon} \right)\right|\mathrm{d}\tau\\
\overset{def}{=}{}& H_{5,1}+H_{5,2}+H_{5,3}+H_{5,4}+H_{5,5}.
\end{align*}
Lemma \ref{Le1} gives that
\begin{align*}
H_{5,1}
\lesssim{}&\sum_{|\alpha|\leq m-1}\int_{0}^{t}\left\| \operatorname{div}\mathbf{v}^{\varepsilon}\right\|_0^{\frac{1}{2}}\left\| \partial_3\operatorname{div}\mathbf{v}^{\varepsilon}\right\|_0^{\frac{1}{2}}\left\|Z^{\alpha} \partial_{3}\mathbf{B}^{\varepsilon}\right\|_0
\left\|\partial_{1}Z^{\alpha} \partial_{3}\mathbf{B}^{\varepsilon}\right\|_0^{\frac{1}{2}}\left\|\partial_{2}Z^{\alpha} \partial_{3}\mathbf{B}^{\varepsilon}\right\|_0^{\frac{1}{2}}\mathrm{d}\tau\\
\lesssim{}&\sup_{0\leq \tau \leq t}\left\| \partial_{3}\mathbf{B}^{\varepsilon}\right\|_{m-1}\int_{0}^{t}\left(\left\|\partial\mathbf{v}^{\varepsilon}\right\|_{m}^2+\left\|\partial_3\left(\partial\mathbf{v}^{\varepsilon}, \partial_h \mathbf{B}^{\varepsilon}\right)\right\|_{m-1}^2\right) \mathrm{d}\tau.
\end{align*}
In view of the commutator property, Hardy's inequality and Lemma \ref{Le1}, we deduce
\begin{align*}
H_{5,2}
\lesssim{}&\sum_{\substack{|\alpha|\leq m-1\\\alpha_3 \geq 1}}  \sum_{k=0}^{\alpha_3-1} \int_{0}^{t}\int_{\mathbb{R}_+^3}\left|\frac{{v}_{3}^{\varepsilon}}{x_3}\right|
\left| Z^{\alpha_{h}}_{h} \varphi\partial_3Z^k_3\partial_3 \mathbf{B}^{\varepsilon} \right| \left| Z^\alpha\partial_3 \mathbf{B}^{\varepsilon}\right| \mathrm{d}\mathbf{x}\mathrm{d}\tau\\
\lesssim{}&\sum_{\substack{|\alpha|\leq m-1\\\alpha_3 \geq 1}}  \sum_{k=0}^{\alpha_3-1} \int_{0}^{t}
\left\| \partial_3 {v}_{3}^{\varepsilon}\right\|_{0}^{\frac{1}{2}}\left\| \partial_3^2 {v}_{3}^{\varepsilon}\right\|_{0}^{\frac{1}{2}}\left\| Z^{\alpha_{h}}_{h}Z^{k+1}_3\partial_3 \mathbf{B}^{\varepsilon} \right\|_0^{\frac{1}{2}}\\
&\cdot\left\|\partial_1 Z^{\alpha_{h}}_{h}Z^{k+1}_3\partial_3 \mathbf{B}^{\varepsilon} \right\|_0^{\frac{1}{2}} \left\| Z^\alpha \partial_3\mathbf{B}^{\varepsilon}\right\|_0^{\frac{1}{2}}\left\|\partial_2 Z^\alpha \partial_3\mathbf{B}^{\varepsilon}\right\|_0^{\frac{1}{2}} \mathrm{d}\tau\\
\lesssim{}&\sup_{0\leq \tau \leq t}\left\| \partial_{3}\mathbf{B}^{\varepsilon}\right\|_{m-1}\int_{0}^{t}\left(\left\|\partial\mathbf{v}^{\varepsilon}\right\|_{m}^2+\left\|\partial_3\left(\partial\mathbf{v}^{\varepsilon}, \partial_h \mathbf{B}^{\varepsilon}\right)\right\|_{m-1}^2\right) \mathrm{d}\tau.
\end{align*}
We note that the estimate of term \(H_{5,3}\) is somewhat complex. Thus, we split this term into two components as follows:
\begin{align*}
H_{5,3}
\lesssim{}&\sum_{\substack{|\alpha|\leq m-1\\\beta+\gamma=\alpha,\beta \neq 0}}\int_{0}^{t} \Bigg(\left|\int_{\mathbb{R}_+^3} Z^\beta \mathbf{v}^{\varepsilon}_{h} \cdot Z^\gamma \partial_h \partial_3 \mathbf{B}^{\varepsilon}  \cdot Z^\alpha \partial_3 \mathbf{B}^{\varepsilon}  \mathrm{d}\mathbf{x}\right|\\
&+\left|\int_{\mathbb{R}_+^3} Z^\beta {v}_{3}^{\varepsilon}  Z^\gamma \partial_3^2 \mathbf{B}^{\varepsilon}  \cdot Z^\alpha \partial_3 \mathbf{B}^{\varepsilon}   \mathrm{d}\mathbf{x}\right|\Bigg)\mathrm{d}\tau\\
\overset{def}{=}{}&H_{5,3,1}+H_{5,3,2}.
\end{align*}
Lemma \ref{Le1} gives rise to
\begin{align*}
H_{5,3,1}
\lesssim{}&\sum_{\substack{|\alpha|\leq m-1\\\beta+\gamma=\alpha,|\beta|\neq0}}\int_{0}^{t}
\left\| Z^\beta \mathbf{v}^{\varepsilon}_{h} \right\|_{0}^{\frac{1}{2}}
\left\| \partial_3 Z^\beta \mathbf{v}^{\varepsilon}_{h} \right\|_{0}^{\frac{1}{2}}\\
&\cdot\left\| \partial_h Z^\gamma \partial_3 \mathbf{B}^{\varepsilon}  \right\|_{0}^{\frac{1}{2}}
\left\| \partial_{2}\partial_h Z^\gamma \partial_3 \mathbf{B}^{\varepsilon}  \right\|_{0}^{\frac{1}{2}}
\left\|Z^\alpha \partial_3 \mathbf{B}^{\varepsilon}  \right\|_{0}^{\frac{1}{2}}
\left\|\partial_1 Z^\alpha \partial_3 \mathbf{B}^{\varepsilon}  \right\|_{0}^{\frac{1}{2}}\mathrm{d}\tau \\
\lesssim{}&\sup_{0\leq \tau \leq t}\left\| \partial_{3}\mathbf{B}^{\varepsilon}\right\|_{m-1}
\int_{0}^{t}\left(\left\|\partial\mathbf{v}^{\varepsilon}\right\|_{m}^2+\left\|\partial_3\left(\partial\mathbf{v}^{\varepsilon}, \partial_h \mathbf{B}^{\varepsilon}\right)\right\|_{m-1}^2\right) \mathrm{d}\tau.
\end{align*}
Further decomposition is required for the term \(H_{5,3,2}\):
\begin{align*}
H_{5,3,2}
\lesssim{}&\sum_{\substack{|\alpha|\leq m-1\\\beta+\gamma=\alpha,\beta \neq 0}}\int_{0}^{t} \Bigg(\left|\int_{\mathbb{R}^2 \times [1,+\infty)} Z^\beta {v}_{3}^{\varepsilon}  Z^\gamma \partial_3^2 \mathbf{B}^{\varepsilon}  \cdot Z^\alpha \partial_3 \mathbf{B}^{\varepsilon}  \mathrm{d}\mathbf{x}\right|\\
&+\left|\int_{\mathbb{R}^2 \times [0,1)} Z^\beta {v}_{3}^{\varepsilon}  Z^\gamma \partial_3^2 \mathbf{B}^{\varepsilon}  \cdot Z^\alpha \partial_3 \mathbf{B}^{\varepsilon}   \mathrm{d}\mathbf{x}\right|\Bigg)\mathrm{d}\tau\\
\overset{def}{=}{}&H_{5,3,2,1}+H_{5,3,2,2}.
\end{align*}
Taking advantage of \(\varphi \ge \frac{1}{2}\) on \([1,+\infty)\), we arrive at
\begin{align*}
H_{5,3,2,1}\lesssim{}&\sum_{\substack{|\alpha|\leq m-1\\\beta+\gamma=\alpha,\beta \neq 0}}\int_{0}^{t}  \int_{\mathbb{R}^2 \times [1,+\infty)} \left|Z^\beta {v}_{3}^{\varepsilon}\right|  \left| \varphi Z^\gamma \partial_{3}^2 \mathbf{B}^{\varepsilon} \right|  \left|Z^\alpha \partial_3 \mathbf{B}^{\varepsilon} \right| \mathrm{d}\mathbf{x}\mathrm{d}\tau\\
\lesssim{}&\sum_{\substack{|\alpha|\leq m-1\\\beta+\gamma=\alpha,\beta \neq 0\\\tilde{\gamma}\leq\gamma+e_3}}\int_{0}^{t}  \int_{\mathbb{R}_{+}^3 } \left|Z^\beta {v}_{3}^{\varepsilon}\right|  \left| Z^{\tilde{\gamma}} \partial_3 \mathbf{B}^{\varepsilon}  \right|  \left|Z^\alpha \partial_3 \mathbf{B}^{\varepsilon} \right| \mathrm{d}\mathbf{x}\mathrm{d}\tau\\
\lesssim{}&\sum_{\substack{|\alpha|\leq m-1\\\beta+\gamma=\alpha,\beta \neq 0\\\tilde{\gamma}\leq\gamma+e_3}}\int_{0}^{t}  \left\|Z^\beta {v}_{3}^{\varepsilon}\right\|_{0}^{\frac{1}{2}}\left\|\partial_3Z^\beta {v}_{3}^{\varepsilon}\right\|_{0}^{\frac{1}{2}}  \\
&\cdot\left\| Z^{\tilde{\gamma}} \partial_3 \mathbf{B}^{\varepsilon}  \right\|_0^{\frac{1}{2}}\left\|\partial_2 Z^{\tilde{\gamma}} \partial_3 \mathbf{B}^{\varepsilon}  \right\|_0^{\frac{1}{2}}  \left\|Z^\alpha \partial_3 \mathbf{B}^{\varepsilon} \right\|_0^{\frac{1}{2}} \left\|\partial_1Z^\alpha \partial_3 \mathbf{B}^{\varepsilon} \right\|_0^{\frac{1}{2}}\mathrm{d}\tau\\	
\lesssim{}&\sup_{0\leq \tau \leq t}\left\| \partial_{3}\mathbf{B}^{\varepsilon}\right\|_{m-1}\int_{0}^{t}\left(\left\|\partial\mathbf{v}^{\varepsilon}\right\|_{m}^2+\left\|\partial_3\partial_h \mathbf{B}^{\varepsilon}\right\|_{m-1}^2\right) \mathrm{d}\tau.
\end{align*}
Making use of \(\dfrac{\varphi(x_3)}{x_3}\ge\dfrac{1}{2}\) on \([0,1)\), we derive
\begin{align*}
H_{5,3,2,2}\lesssim{}&\sum_{\substack{|\alpha|\leq m-1\\\beta+\gamma=\alpha,\beta \neq 0}}\int_{0}^{t}  \int_{\mathbb{R}^2 \times [0, 1)} \left|\frac{Z^\beta {v}_{3}^{\varepsilon}}{x_3}\right|  \left| \varphi Z^\gamma \partial_{3}^2 \mathbf{B}^{\varepsilon} \right|  \left|Z^\alpha \partial_3 \mathbf{B}^{\varepsilon} \right| \mathrm{d}\mathbf{x}\mathrm{d}\tau\\
\lesssim{}&\sum_{\substack{|\alpha|\leq m-1\\\beta+\gamma=\alpha,\beta \neq 0\\\tilde{\gamma}\leq\gamma+e_3}}\int_{0}^{t}  \int_{\mathbb{R}^2 \times [0, 1)} \left|\frac{Z^\beta {v}_{3}^{\varepsilon}}{x_3}\right|  \left| Z^{\tilde{\gamma}} \partial_3 \mathbf{B}^{\varepsilon}  \right|  \left|Z^\alpha \partial_3 \mathbf{B}^{\varepsilon} \right| \mathrm{d}\mathbf{x}\mathrm{d}\tau\\
\lesssim{}&\sum_{\substack{|\alpha|\leq m-1\\\beta+\gamma=\alpha,\beta \neq 0\\\tilde{\gamma}\leq\gamma+e_3}}\int_{0}^{t}  \left\|\partial_3Z^\beta {v}_{3}^{\varepsilon}\right\|_{0}^{\frac{1}{2}}\left\|\partial_3^2Z^\beta {v}_{3}^{\varepsilon}\right\|_{0}^{\frac{1}{2}}  \\
&\cdot\left\| Z^{\tilde{\gamma}} \partial_3 \mathbf{B}^{\varepsilon}  \right\|_0^{\frac{1}{2}}\left\|\partial_2 Z^{\tilde{\gamma}} \partial_3 \mathbf{B}^{\varepsilon}  \right\|_0^{\frac{1}{2}}  \left\|Z^\alpha \partial_3 \mathbf{B}^{\varepsilon} \right\|_0^{\frac{1}{2}} \left\|\partial_1Z^\alpha \partial_3 \mathbf{B}^{\varepsilon} \right\|_0^{\frac{1}{2}}\mathrm{d}\tau\\	
\lesssim{}&\sup_{0\leq \tau \leq t}\left\| \partial_{3}\mathbf{B}^{\varepsilon}\right\|_{m-1}\int_{0}^{t}\left(\left\|\partial\mathbf{v}^{\varepsilon}\right\|_{m}^2+\left\|\partial_3\left(\partial\mathbf{v}^{\varepsilon}, \partial_h \mathbf{B}^{\varepsilon}\right)\right\|_{m-1}^2\right) \mathrm{d}\tau.
\end{align*}
Summing the preceding estimates, we reach
\begin{align*}
H_{5,3,2}
\lesssim\sup_{0\leq \tau \leq t}\left\| \partial_{3}\mathbf{B}^{\varepsilon}\right\|_{m-1}\int_{0}^{t}\left(\left\|\partial\mathbf{v}^{\varepsilon}\right\|_{m}^2+\left\|\partial_3\left(\partial\mathbf{v}^{\varepsilon}, \partial_h \mathbf{B}^{\varepsilon}\right)\right\|_{m-1}^2\right) \mathrm{d}\tau.
\end{align*}
Gathering the foregoing estimates, we deduce
\begin{equation*}
H_{5,3}
\lesssim\sup_{0\leq \tau \leq t}\left\| \partial_{3}\mathbf{B}^{\varepsilon}\right\|_{m-1}
\int_{0}^{t}\left(\left\|\partial\mathbf{v}^{\varepsilon}\right\|_{m}^2+\left\|\partial_3\left(\partial\mathbf{v}^{\varepsilon}, \partial_h \mathbf{B}^{\varepsilon}\right)\right\|_{m-1}^2\right) \mathrm{d}\tau.
\end{equation*}
By Lemma \ref{Le1}, we have
\begin{align*}
H_{5,4}
\lesssim{}&\sum_{\substack{|\alpha|\leq m-1\\\beta+\gamma=\alpha}}\int_{0}^{t} \left| \left( Z^\beta \partial\mathbf{v}^{\varepsilon} \cdot Z^\gamma \partial \mathbf{B}^{\varepsilon}, Z^\alpha \partial_{3}\mathbf{B}^{\varepsilon} \right)\right|\mathrm{d}\tau\\
\lesssim{}&\sum_{\substack{|\alpha|\leq m-1\\\beta+\gamma=\alpha}}\int_{0}^{t}  \left\|Z^\beta \partial\mathbf{v}^{\varepsilon}\right\|_{0}^{\frac{1}{2}}\left\|\partial_3 Z^\beta \partial\mathbf{v}^{\varepsilon}\right\|_{0}^{\frac{1}{2}}  \\
&\cdot\left\| Z^\gamma \partial \mathbf{B}^{\varepsilon}\right\|_0^{\frac{1}{2}}\left\|\partial_2 Z^\gamma \partial \mathbf{B}^{\varepsilon}\right\|_0^{\frac{1}{2}}  \left\|Z^\alpha \partial_3 \mathbf{B}^{\varepsilon} \right\|_0^{\frac{1}{2}} \left\|\partial_1Z^\alpha \partial_3 \mathbf{B}^{\varepsilon} \right\|_0^{\frac{1}{2}}\mathrm{d}\tau\\	
\lesssim{}&\sup_{0\leq \tau \leq t}\left(\left\|\mathbf{B}^{\varepsilon}\right\|_{m}+\left\| \partial_{3}\mathbf{B}^{\varepsilon}\right\|_{m-1}\right)\int_{0}^{t}\left(\left\|\left(\partial\mathbf{v}^{\varepsilon}, \partial_h \mathbf{B}^{\varepsilon}\right)\right\|_{m}^2+\left\|\partial_3\left(\partial\mathbf{v}^{\varepsilon}, \partial_h \mathbf{B}^{\varepsilon}\right)\right\|_{m-1}^2\right) \mathrm{d}\tau.
\end{align*}
It follows from Lemma \ref{Le1} that
\begin{align*}
H_{5,5}
\lesssim{}&\sum_{\substack{|\alpha|\leq m-1\\\beta+\gamma=\alpha}}\int_{0}^{t} \left| \left( Z^\beta \partial^2\mathbf{v}^{\varepsilon} \cdot Z^\gamma \mathbf{B}^{\varepsilon}, Z^\alpha \partial_{3}\mathbf{B}^{\varepsilon} \right)\right|\mathrm{d}\tau\\
\lesssim{}&\sum_{\substack{|\alpha|\leq m-1\\\beta+\gamma=\alpha}}\int_{0}^{t}  \left\|Z^\beta \partial^2\mathbf{v}^{\varepsilon}\right\|_{0}\left\| Z^\gamma\mathbf{B}^{\varepsilon}\right\|_0^{\frac{1}{4}}\left\|\partial_2 Z^\gamma \mathbf{B}^{\varepsilon}\right\|_0^{\frac{1}{4}}\\
&\cdot\left\|\partial_3 Z^\gamma \mathbf{B}^{\varepsilon}\right\|_0^{\frac{1}{4}}\left\|\partial_{23} Z^\gamma \mathbf{B}^{\varepsilon}\right\|_0^{\frac{1}{4}}  \left\|Z^\alpha \partial_3 \mathbf{B}^{\varepsilon} \right\|_0^{\frac{1}{2}} \left\|\partial_1Z^\alpha \partial_3 \mathbf{B}^{\varepsilon} \right\|_0^{\frac{1}{2}}\mathrm{d}\tau\\	
\lesssim{}&\sup_{0\leq \tau \leq t}\left(\left\|\mathbf{B}^{\varepsilon}\right\|_{m}+\left\| \partial_{3}\mathbf{B}^{\varepsilon}\right\|_{m-1}\right)\int_{0}^{t}\left(\left\|\left(\partial\mathbf{v}^{\varepsilon}, \partial_h \mathbf{B}^{\varepsilon}\right)\right\|_{m}^2+\left\|\partial_3\left(\partial\mathbf{v}^{\varepsilon}, \partial_h \mathbf{B}^{\varepsilon}\right)\right\|_{m-1}^2\right) \mathrm{d}\tau.
\end{align*}
By assembling all the above estimates, we arrive at
\begin{equation}\label{G93}
\begin{aligned}
\sum_{|\alpha|\leq m-1} \int_{0}^{t}  \left|{E} _{\alpha}\right|  \mathrm{d} \tau
\lesssim{}&\sup_{0\leq \tau \leq t}\left(\left\|\mathbf{B}^{\varepsilon}\right\|_{m}+\left\| \partial_{3}\mathbf{B}^{\varepsilon}\right\|_{m-1}\right)\\
&\cdot\int_{0}^{t}\left(\left\|\left(\partial\mathbf{v}^{\varepsilon}, \partial_h \mathbf{B}^{\varepsilon}\right)\right\|_{m}^2+\left\|\partial_3\left(\partial\mathbf{v}^{\varepsilon}, \partial_h \mathbf{B}^{\varepsilon}\right)\right\|_{m-1}^2\right) \mathrm{d}\tau.
\end{aligned}
\end{equation}
By a similar argument as before, we derive the following estimate for $\sum_{|\alpha|\leq m-2} \int_{0}^{t}  \left|{B} _{\alpha}\right|  \mathrm{d} \tau$, and omit the details for brevity.
\begin{equation}\label{G76}
\begin{aligned}
&\sum_{|\alpha|\leq m-2} \int_{0}^{t}  \left|{B} _{\alpha}\right|  \mathrm{d} \tau\\
\lesssim{}&\sup_{0\leq \tau \leq t}\Big(\left\|\left(a^{\varepsilon}, \mathbf{v}^{\varepsilon}, \mathbf{B}^{\varepsilon}\right)\right\|_{m}+\left\|\partial_3 \left(a^{\varepsilon}, \mathbf{v}^{\varepsilon}, \mathbf{B}^{\varepsilon}\right)\right\|_{m-1}
+\left\| \partial_3^2\mathbf{v}^{\varepsilon}\right\|_{m-2}+\left\| \partial_{\tau}\left(a^{\varepsilon}, \mathbf{v}^{\varepsilon}, \mathbf{B}^{\varepsilon}\right)\right\|_{m-2}\Big)\\
&\cdot\int_{0}^{t}\Big(\left\|\partial a^{\varepsilon}\right\|_{m-1}^2+\left\|\left(\partial\mathbf{v}^{\varepsilon}, \partial_h \mathbf{B}^{\varepsilon}\right)\right\|_{m}^2
+\left\|\partial_3\left(\partial\mathbf{v}^{\varepsilon}, \partial_h \mathbf{B}^{\varepsilon}\right)\right\|_{m-1}^2+\left\| \partial_{\tau}\left(a^{\varepsilon}, \mathbf{v}^{\varepsilon}, \mathbf{B}^{\varepsilon}\right)\right\|_{m-1}^2 \\
&+\left\| \partial_3\partial_{\tau}\mathbf{v}^{\varepsilon}\right\|_{m-2}^2\Big) \mathrm{d}\tau.
\end{aligned}
\end{equation}
As a consequence, inserting \eqref{G30}--\eqref{G76} into \eqref{G15} leads to \eqref{G18}.
This completes the proof of Lemma \ref{red}.
\subsection{Proof of Theorem \ref{Th1}}
Inserting \eqref{G17} and \eqref{G18} into \eqref{G16}, we conclude that
\begin{lemm}\label{GL10}
  Let \(m \geq 4\) and assume that \(E_{m}^{\varepsilon}(t) \leq 1\) for all \(t \geq 0\).
Then the following inequality holds for all $t \geq 0$:
\begin{equation}\label{G94}
  \begin{aligned}
  E_{m}^{\varepsilon}(t)^{2}+D_{m}^{\varepsilon}(t)^{2}\lesssim{} &\left\|( a_{0},\mathbf{v}_{0},\mathbf{B}_{0})\right\|_{m}^2+\left\|\partial_3( a_{0},\mathbf{v}_{0},\mathbf{B}_{0})\right\|_{m-1}^2+\left\|\partial_3^2(\mathbf{v}_{0},\varepsilon\mathbf{B}_{0})\right\|_{m-2}^2\\
  &+\delta_4 \left(E_{m}^{\varepsilon}(t)^2+D_{m}^{\varepsilon}(t)^2\right)+E_{m}^{\varepsilon}(t)^{4}+D_{m}^{\varepsilon}(t)^{4}+D_{m-1}^{\varepsilon}(t)^2
  \end{aligned}
\end{equation}
for any \(\delta_4>0\).
\end{lemm}
By induction on $m$ via \eqref{G94}, we have for \(m \geq 4\)
\begin{equation}\label{G94.1}
  \begin{aligned}
  E_{m}^{\varepsilon}(t)^{2}+D_{m}^{\varepsilon}(t)^{2}\lesssim{} &\left\|( a_{0},\mathbf{v}_{0},\mathbf{B}_{0})\right\|_{m}^2+\left\|\partial_3( a_{0},\mathbf{v}_{0},\mathbf{B}_{0})\right\|_{m-1}^2+\left\|\partial_3^2(\mathbf{v}_{0},\varepsilon\mathbf{B}_{0})\right\|_{m-2}^2\\
  &+\delta_4 \left(E_{m}^{\varepsilon}(t)^2+D_{m}^{\varepsilon}(t)^2\right)+E_{m}^{\varepsilon}(t)^{4}+D_{m}^{\varepsilon}(t)^{4}+D_{3}^{\varepsilon}(t)^2.
  \end{aligned}
\end{equation}
Using the a priori assumption \eqref{G1} and proceeding analogously to the derivation of estimate \eqref{G94}, we obtain for all \(m \geq 4\)
\begin{equation}\label{G94.2}
  \begin{aligned}
  E_{3}^{\varepsilon}(t)^{2}+D_{3}^{\varepsilon}(t)^{2}\lesssim{} &\left\|( a_{0},\mathbf{v}_{0},\mathbf{B}_{0})\right\|_{m}^2+\left\|\partial_3( a_{0},\mathbf{v}_{0},\mathbf{B}_{0})\right\|_{m-1}^2+\left\|\partial_3^2(\mathbf{v}_{0},\varepsilon\mathbf{B}_{0})\right\|_{m-2}^2\\
  &+\delta_4 \left(E_{m}^{\varepsilon}(t)^2+D_{m}^{\varepsilon}(t)^2\right)+E_{m}^{\varepsilon}(t)^{4}+D_{m}^{\varepsilon}(t)^{4}+D_{2}^{\varepsilon}(t)^2,
\end{aligned}
\end{equation}
and
\begin{equation}\label{G94.3}
  \begin{aligned}
  E_{2}^{\varepsilon}(t)^{2}+D_{2}^{\varepsilon}(t)^{2}\lesssim{} &\left\|( a_{0},\mathbf{v}_{0},\mathbf{B}_{0})\right\|_{m}^2+\left\|\partial_3( a_{0},\mathbf{v}_{0},\mathbf{B}_{0})\right\|_{m-1}^2+\left\|\partial_3^2(\mathbf{v}_{0},\varepsilon\mathbf{B}_{0})\right\|_{m-2}^2\\
  &+\delta_4 \left(E_{m}^{\varepsilon}(t)^2+D_{m}^{\varepsilon}(t)^2\right)+E_{m}^{\varepsilon}(t)^{4}+D_{m}^{\varepsilon}(t)^{4}+D_{1}^{\varepsilon}(t)^2,
\end{aligned}
\end{equation}
where
\begin{align*}
  D_{1}^{\varepsilon}(t)\overset{def}{=} \Bigg( &\int_{0}^{t}\bigg(\left\|\partial a^{\varepsilon} \right\|_{0}^2+\left\|(\partial\mathbf{v}^{\varepsilon}, \partial_h\mathbf{B}^{\varepsilon})\right\|_{1}^2
        +\left\|\partial_3\left(\partial\mathbf{v}^{\varepsilon}, \partial_h\mathbf{B}^{\varepsilon}\right)\right\|_{0}^2+\left\| \partial_{\tau}\left(a^{\varepsilon}, \mathbf{v}^{\varepsilon}, \mathbf{B}^{\varepsilon}\right)\right\|_{0}^2 \\
&+\varepsilon \left( \left\| \partial_3 \mathbf{B}^{\varepsilon}\right\|_{1}^2+\left\| \partial_3^2 \mathbf{B}^{\varepsilon}\right\|_{0}^2 \right)\bigg) \mathrm{d}\tau\Bigg)^{\frac{1}{2}}.
\end{align*}
By the same argument used to establish \eqref{G16}--\eqref{G18}, for $m\geq 4$, we have
\begin{align*}
  E_{1}^{\varepsilon}(t)^{2}+D_{1}^{\varepsilon}(t)^{2} \lesssim{}&\left\|( a_{0} ,\mathbf{v}_{0},\mathbf{B}_{0})\right\|_{m}^2
  +\left\|\partial_3( a_{0},\mathbf{v}_{0},\mathbf{B}_{0})\right\|_{m-1}^2+L_{1}(t)^{2}+N_{1}(t)^{2},\\
  L_{1}(t)^2\lesssim{}& \delta_4 D_{m}^{\varepsilon}(t)^2+D_{0}^{\varepsilon}(t)^2,\\
  N_{1}(t)^2\lesssim{}& \delta_4 \left(E_{m}^{\varepsilon}(t)^2+D_{m}^{\varepsilon}(t)^2\right)+\left(E_{m}^{\varepsilon}(t)^4+D_{m}^{\varepsilon}(t)^4\right),
\end{align*}
where
\begin{align*}
E_{1}^{\varepsilon}(t)\overset{def}{=} {}& \sup_{0\leq \tau \leq t}\left\| \left(a^{\varepsilon}, \mathbf{v}^{\varepsilon}, \mathbf{B}^{\varepsilon}\right)(\tau)\right\|_{1} +\sup_{0\leq \tau \leq t} \left\| \partial_3\left(a^{\varepsilon}, \mathbf{v}^{\varepsilon}, \mathbf{B}^{\varepsilon}\right)(\tau)\right\|_{0},\\
L_{1}(t)\overset{def}{=}{}&\left(\sum_{|\alpha|\leq 1} \int_{0}^{t}  \left|\widetilde{A} _{\alpha}\right|  \mathrm{d} \tau\right)^{\frac{1}{2}},\\
N_{1}(t)\overset{\mathrm{def}}{=}{}
&\Bigg(\int_{0}^{t}\Big(\left\|a^{\varepsilon}\operatorname{div}\mathbf{v}^{\varepsilon}\right\|_{1}^{2}
+\left\|\partial_{3}(a^{\varepsilon}\operatorname{div}\mathbf{v}^{\varepsilon})\right\|_{0}^{2}+\left\|f^{\varepsilon}_2\right\|_{0}^{2}\Big)\mathrm{d}\tau\\
&+\sum_{|\alpha|\le1}\int_{0}^{t}\left|A_{\alpha}\right|\mathrm{d}\tau
+\sum_{|\alpha|=0}\int_{0}^{t}\big(\left|C_{\alpha}\right|+\left|D_{\alpha}\right|+\left|E_{\alpha}\right|\big)\mathrm{d}\tau\Bigg)^{\frac12},\\
D_{0}^{\varepsilon}(t)\overset{def}{=}{}& \Bigg( \int_{0}^{t}\bigg(\left\|(\partial\mathbf{v}^{\varepsilon}, \partial_h\mathbf{B}^{\varepsilon})\right\|_{0}^2
       +\varepsilon \left\| \partial_3 \mathbf{B}^{\varepsilon}\right\|_{0}^2\bigg) \mathrm{d}\tau\Bigg)^{\frac{1}{2}}.
\end{align*}
Collecting the above estimates yields for \(m \geq 4\)
\begin{equation}\label{G94.4}
  \begin{aligned}
    E_{1}^{\varepsilon}(t)^{2}+D_{1}^{\varepsilon}(t)^{2}\lesssim{}&\left\|( a_{0},\mathbf{v}_{0},\mathbf{B}_{0})\right\|_{m}^2
+\left\|\partial_3( a_{0},\mathbf{v}_{0},\mathbf{B}_{0})\right\|_{m-1}^2\\
&+\delta_4 \left(E_{m}^{\varepsilon}(t)^2+D_{m}^{\varepsilon}(t)^2\right)+E_{m}^{\varepsilon}(t)^{4}+D_{m}^{\varepsilon}(t)^{4}+D_{0}^{\varepsilon}(t)^2.
  \end{aligned}
\end{equation}
Analogously to \eqref{G94.4}, we obtain for \(m \geq 4\)
\begin{equation}\label{G94.5}
  E_{0}^{\varepsilon}(t)^{2}+D_{0}^{\varepsilon}(t)^{2}\lesssim\left\|( a_{0},\mathbf{v}_{0},\mathbf{B}_{0})\right\|_{m}^2+\delta_4 \left(E_{m}^{\varepsilon}(t)^2+D_{m}^{\varepsilon}(t)^2\right)+E_{m}^{\varepsilon}(t)^{4}+D_{m}^{\varepsilon}(t)^{4},
\end{equation}
where
\[
E_{0}^{\varepsilon}(t)\overset{def}{=}\sup_{0\leq \tau \leq t}\left\| \left(a^{\varepsilon}, \mathbf{v}^{\varepsilon}, \mathbf{B}^{\varepsilon}\right)(\tau)\right\|_{0}.
\]
Inserting estimates \eqref{G94.2}–\eqref{G94.5} into \eqref{G94.1} and taking \(\delta_4\) sufficiently small, we conclude that for all \(m \geq 4\)
\begin{equation}\label{G94.6}
  \begin{aligned}
  E_{m}^{\varepsilon}(t)^{2}+D_{m}^{\varepsilon}(t)^{2}\lesssim{} &\left\|( a_{0},\mathbf{v}_{0},\mathbf{B}_{0})\right\|_{m}^2+\left\|\partial_3( a_{0},\mathbf{v}_{0},\mathbf{B}_{0})\right\|_{m-1}^2+\left\|\partial_3^2(\mathbf{v}_{0},\varepsilon\mathbf{B}_{0})\right\|_{m-2}^2\\
  &+E_{m}^{\varepsilon}(t)^{4}+D_{m}^{\varepsilon}(t)^{4}.
  \end{aligned}
\end{equation}
Thus, the energy estimate \eqref{eq10} in Theorem \ref{Th1} immediately follows from the a priori assumption \eqref{G1} and estimate \eqref{G94.6}.
In a similar fashion as in \cite{MN1982} and \cite{Tani1977}, one can first establish the \(\varepsilon\)-uniform local-in-time well-posedness for equation \eqref{eq6}.
Then, by virtue of the uniform estimate \eqref{eq10}, the global well-posedness follows via the standard continuity argument.

Furthermore, the estimate \eqref{eq10} gives directly
\begin{equation}\label{G95}
\underset{0 \le \tau \le t}{\sup}
\left\|\left(a^{\varepsilon}, \mathbf{v}^{\varepsilon}, \mathbf{B}^{\varepsilon}\right)\left(\tau\right)\right\|_{m}^2
+\underset{0 \le \tau \le t}{\sup}
\left\|\partial_3\left(a^{\varepsilon}, \mathbf{v}^{\varepsilon}, \mathbf{B}^{\varepsilon}\right)\left(\tau\right)\right\|_{m-1}^2+\int_{0}^{t}\left\| \partial_{\tau}\left(a^{\varepsilon}, \mathbf{v}^{\varepsilon}, \mathbf{B}^{\varepsilon}\right)\right\|_{m-1}^2\mathrm{d}\tau \lesssim   \delta_0^2.
\end{equation}
Therefore, the estimate \eqref{G95}
and strong compactness argument (see \cite[Lemma 4]{MR1062395SIAM1990})
yield directly
\begin{equation*}
\left(a^{\varepsilon}, \mathbf{v}^{\varepsilon}, \mathbf{B}^{\varepsilon}\right) \rightarrow \left(a^{0}, \mathbf{v}^{0}, \mathbf{B}^{0}\right) \text{~strongly~in~}
L^\infty\left(0, t; H^{m-1}_{co,loc}\left(\mathbb{R}_+^3\right)\right)
\end{equation*}
for all $t>0$. Therefore, we complete the proof of Theorem \ref{Th1}.
\section{Asymptotic behavior of global solutions}\label{asymptotic-behavior}
\subsection{Decay rates of linear operators}
In this section we investigate large-time behavior of solutions to the nonlinear
problem \eqref{eq8}. For simplicity of notations, throughout this section,
we rewrite the initial boundary value problem \eqref{eq8} as
\begin{equation}\label{4-1}
		\left\{\begin{array}{*{5}{ll}}
			\partial_t \mathbf{u}^{\varepsilon} + A \mathbf{u}^{\varepsilon} = \operatorname{div} F^{\varepsilon} & {\rm in} ~~ \mathbb{R}_+^3,\\
			\partial_t\mathbf{B}^{\varepsilon} - \left(\Delta_h+\varepsilon\partial_3^2\right)\mathbf{B}^{\varepsilon} = F^{\varepsilon}_3 & {\rm in} ~~ \mathbb{R}_+^3,\\
			\operatorname{div}\mathbf{B}^{\varepsilon} = 0 & {\rm in} ~~ \mathbb{R}_+^3,\\
			\left(\mathbf{u}^{\varepsilon}, \mathbf{B}^{\varepsilon}\right)\big|_{t=0} = \left(\mathbf{u}_{0}, \mathbf{B}_{0}\right) & {\rm in} ~~ \mathbb{R}_+^3,\\
		\left(\mathbf{m}^{\varepsilon},\partial_3 \mathbf{B}^{\varepsilon}_h, B_3^\varepsilon\right) = \mathbf{0}\quad & {\rm on} ~~ \mathbb{R}^2 \times \left\{ x_3=0 \right\},
		\end{array}\right.
	\end{equation}
where \(\mathbf{u}^{\varepsilon}\overset{def}{=}\left(a^\varepsilon, \mathbf{m}^{\varepsilon}\right)\), \(\mathbf{u}_{0}\overset{def}{=}\left(a_0, \mathbf{m}_{0}\right)\), \(F^{\varepsilon}\overset{def}{=}\begin{pmatrix}0  \\ F^{\varepsilon}_1+F^{\varepsilon}_2\end{pmatrix}\) and \(A\overset{def}{=}\begin{pmatrix}0 &   \nabla^T \\   \nabla & -\mu \Delta-\left(\mu+\lambda\right) \nabla \operatorname{div}\end{pmatrix}\).

The solutions to system \eqref{4-1} can be represented in the integral form by Duhamel's principle. Based on the decay rate estimates of solutions to the corresponding linearized system, we can obtain the decay rates of solutions to the nonlinear problem by weighted energy estimates. To achieve the decay rates of solutions to linearized operator, we divide it into two cases: The first one specifies the decay rates of solutions of the linearized compressible Navier-Stokes equations on the half space with no-slip boundary condition.

Let \(\mathcal{U}(t) \mathbf{u}_{0}\) denote the solution of the following linearized problem,
\[
\left\{\begin{array}{*{5}{ll}}
\partial_t \mathbf{u}^{\varepsilon} + A \mathbf{u}^{\varepsilon} = \mathbf{0} & {\rm in} ~~ \mathbb{R}_+^3,\\
\mathbf{u}^{\varepsilon}\big|_{t=0} = \mathbf{u}_{0} & {\rm in} ~~ \mathbb{R}_+^3,\\
\mathbf{m}^{\varepsilon} = \mathbf{0} & {\rm on} ~~ \mathbb{R}^2 \times \left\{ x_3=0 \right\}.
\end{array}\right.
\]
where \(\mathbf{u}^{\varepsilon}=\left(a^\varepsilon, \mathbf{m}^{\varepsilon}\right)\), \(\mathbf{u}_{0}=\left(a_0, \mathbf{m}_{0}\right)\).
The decay rates of \(\mathcal{U}(t) \mathbf{u}_{0}\) have been established in \cite{k1, k2}; we only list the relevant results used in this paper in Lemmas \ref{L4-1}--\ref{aria4} without proofs.
For $\mathbf{u}^{\varepsilon}=\left(a^\varepsilon,\mathbf{m}^{\varepsilon}\right)$ (wherein $a^\varepsilon\in W^{\ell,p}$ and $\mathbf{m}^{\varepsilon}=\left(m_1^{\varepsilon},m_2^{\varepsilon},m_3^{\varepsilon}\right)\in W^{k,q}$), we define $\left\|\mathbf{u}^{\varepsilon}\right\|_{W^{\ell,p}\times W^{k,q}}$ as follows
$$\left\|\mathbf{u}^{\varepsilon}\right\|_{W^{\ell,p}\times W^{k,q}} \overset{def}{=} \left\|a^\varepsilon\right\|_{W^{\ell,p}} + \left\|\mathbf{m}^{\varepsilon}\right\|_{W^{k,q}}.$$
\begin{lemm}\label{L4-1}
Let \(\mathbf{u}_{0}=\left( a_{0}, \mathbf{m}_{0}\right)\) belong to required functional spaces. Then, the following inequalities hold uniformly in \(0< t \leq 1\).\\
(i)
\[\left\|\partial \mathcal{U}(t) \mathbf{u}_{0}\right\|_{L^2} \lesssim t^{-\frac{1}{2}}\left\| \mathbf{u}_{0}\right\|_{H^{1} \times L^2},\]
(ii) in particular, when \(a_0=0\), i.e., \(\mathbf{u}_{0}=\left(0,\mathbf{m}_{0}\right)\),
\[\left\| \mathcal{U}(t) \partial\mathbf{u}_{0}\right\|_{L^2} \lesssim t^{-\frac{1}{2}}\left\| \mathbf{m}_{0}\right\|_{L^2}.\]
\end{lemm}
\begin{lemm}\label{L4-2}
Let \(\mathbf{u}_{0}=\left( a_{0}, \mathbf{m}_{0}\right)\) belong to required functional spaces and \(0 \leq \ell \leq 2\). Then, there exists a constant $c_0>0$ such that the following estimates hold for all \(t \geq 1\).
\[
\left\| \partial^{\ell} \mathcal{U}(t) \mathbf{u}_{0}\right\|_{L^2} \lesssim  t^{-\frac{3}{4}-q(\ell)}\left\| \mathbf{u}_{0}\right\|_{L^1}+e^{-c_{0} t}\left(\left\|  a_{0}\right\|_{k_{2}(\ell)}+\left\| \mathbf{m}_{0}\right\|_{k_{2}(\ell-1)}\right),
\]
where
\(
q(s)=
\min\left\{\frac{s}{2}, \frac{s}{4}+\frac{3}{8}\right\}
\)
and
\(
k_{2}(s)=
\max\left\{0, s\right\}.
\)
\end{lemm}
\begin{lemm}\label{L4-3}
Let \(\mathbf{u}_{0}=\left(0, \mathbf{m}_{0}\right)\) belong to required functional spaces. Then, there exists a constant \(c_{0}>0\) such that
\[
\left\| \mathcal{U}(t)\left(\partial \mathbf{u}_{0}\right)\right\|_{L^2} \lesssim t^{-1}\left\| \mathbf{m}_{0}\right\|_{L^1}+e^{-c_{0} t}\left\| \mathbf{m}_{0}\right\|_{L^2}
\]
for all \(t \geq 1\).
\end{lemm}
\begin{lemm}\label{aria4}
Let \(\mathbf{u}_{0}=\left( a _{0}, \mathbf{m}_{0}\right) \in \left(H^{1} \times L^{2}\right) \cap \left(L^{1} \times L^{1}\right)\). If \(\int_{\mathbb{R}_{+}^{3}}  a _{0}(x)  \mathrm{d} x \neq 0\), then
\[
\left\| \mathcal{U}(t) \mathbf{u}_{0}\right\|_{L^2} \gtrsim t^{-\frac{3}{4}}
\]
as \(t \to \infty\).
\end{lemm}
The second one specifies the decay rates of solutions to the heat equation on the half space with homogeneous Neumann boundary condition or homogeneous Dirichlet boundary condition.
Let \(\mathcal{B}(t) \mathbf{B}_{0}\) denote the solution to the following linearized problem,
\[
\left\{\begin{array}{*{5}{ll}}
\partial_t \mathbf{B}^{\varepsilon} - \left(\Delta_h+\varepsilon\partial_3^2\right) \mathbf{B}^{\varepsilon} = \mathbf{0} & {\rm in} ~~ \mathbb{R}_+^3,\\
\mathbf{B}^{\varepsilon}\big|_{t=0} = \mathbf{B}_{0} & {\rm in} ~~ \mathbb{R}_+^3,\\
\left(\partial_3 \mathbf{B}^{\varepsilon}_h,  B_3^\varepsilon\right) = \mathbf{0}\quad & {\rm on} ~~ \mathbb{R}^2 \times \left\{ x_3=0 \right\}.
\end{array}\right.
\]
Then, we can obtain the following decay rate results of \(\mathcal{B}(t) \mathbf{B}_{0}\).
 \begin{lemm}\label{L4-5}
 For any \(1 \leq p_i \leq 2\) $(i=1,2)$, \(k \geq 0\) and $t>0$, there exists a positive constant \(C\) independent of \(t\) and $\varepsilon$, such that the solution \(\mathcal{B}(t) \mathbf{B}_{0}\) satisfies the following decay estimates.
\begin{equation}\label{4-2}
  \left\|\partial_{h}^{k}\mathcal{B}(t) \mathbf{B}_{0} \right\|_{L^{2}\left(\mathbb{R}^{3}_{+}\right)}\leq C t^{-\left(\frac{k}{2}+\frac{1}{2}\sum_{i=1}^{2}\left(\frac{1}{p_i}-\frac{1}{2}\right)\right)}
  \left\|\mathbf{B}_{0}\right\|_{L_{x_{1}}^{p_1}L_{x_{2}}^{p_2}L_{x_{3}}^{2}\left(\mathbb{R}^{3}_{+}\right)}
\end{equation}
and
\begin{equation}\label{4-3}
  \left\|\partial_{3}\mathcal{B}(t) \mathbf{B}_{0} \right\|_{L^{2}\left(\mathbb{R}^{3}_{+}\right)}\leq C t^{-\frac{1}{2}}\left\|\partial_3 \mathbf{B}_{0} \right\|_{L_{x_{1}x_{2}}^{1}L_{x_{3}}^{2}\left(\mathbb{R}^{3}_{+}\right)},
\end{equation}
where \(\mathbf{B}_0\) is subject to the boundary condition \((\mathbf{B}_0)_3\big|_{x_3=0} = 0\).
\end{lemm}
\begin{proof}
We treat the tangential and normal components of the magnetic field $\mathbf{B}^{\varepsilon}$ separately.
The horizontal component $\mathbf{B}^{\varepsilon}_h$ satisfies the heat equation with homogeneous Neumann boundary condition,
whereas the vertical component $B^{\varepsilon}_3$ satisfies the heat equation with homogeneous Dirichlet boundary condition.

We first extend the initial data $\mathbf{B}_0$ from the half-space $\mathbb{R}_+^3$ to the whole space $\mathbb{R}^3$ via the method of images.
For the horizontal component $(\mathbf{B}_0)_h$, we define the even reflection across $x_3=0$:
\[
\left(\widetilde{\mathbf{B}}_0(\mathbf{x})\right)_h \overset{def}{=}
\begin{cases}
(\mathbf{B}_0(\mathbf{x}))_h, & x_3 \ge 0, \\
\left(\mathbf{B}_0(\mathbf{x}_h, -x_3)\right)_h, & x_3 < 0.
\end{cases}
\]
For the vertical component $(\mathbf{B}_0)_3$, we define the odd reflection across $x_3=0$:
\[
\left(\widetilde{\mathbf{B}}_0(\mathbf{x})\right)_3 \overset{def}{=}
\begin{cases}
(\mathbf{B}_0(\mathbf{x}))_3, & x_3 \ge 0, \\
-\left(\mathbf{B}_0(\mathbf{x}_h, -x_3)\right)_3, & x_3 < 0.
\end{cases}
\]
Let $G(t, s)$ denote the one-dimensional heat kernel:
\[
G(t, s) \overset{def}{=} \frac{1}{\left(4\pi t\right)^{\frac{1}{2}}} \exp\left( -\frac{s^2}{4t} \right).
\]
It is straightforward to verify that for any integer $k \ge 0$ and $1 \le r \le \infty$, there exists a constant $C_{k,r} > 0$ depending only on $k$ and $r$ such that
\begin{equation}\label{win0}
  \left\| \partial_s^k G(t, \cdot) \right\|_{L^r(\mathbb{R})} \le C_{k,r} \, t^{ -\frac{1}{2}\left(k+1 - \frac{1}{r}\right) }.
\end{equation}
The anisotropic three-dimensional heat kernel $G_\varepsilon(t, \mathbf{x})$ is built by separating the horizontal and vertical variables:
\begin{align*}
G_\varepsilon(t, \mathbf{x})
&\overset{def}{=} G(t, x_1) \cdot G(t, x_2) \cdot G(\varepsilon t, x_3) \\
&= \frac{1}{\varepsilon^{\frac{1}{2}} \left(4\pi t\right)^{\frac{3}{2}}} \exp\left( -\frac{|\mathbf{x}_h|^2}{4t} - \frac{x_3^2}{4\varepsilon t} \right).
\end{align*}
By the reflection principle, the solution operator $\mathcal{B}(t)$ on $\mathbb{R}_+^3$ admits the following integral representation:
\begin{equation}\label{win1}
  \begin{aligned}
&(\mathcal{B}(t)\mathbf{B}_0)_h\\
={}& G_\varepsilon(t, \cdot) * \left(\widetilde{\mathbf{B}}_0\right)_h \\
={}& \int_{\mathbb{R}^3} G_\varepsilon(t, \mathbf{x} - \mathbf{y}) \left(\widetilde{\mathbf{B}}_0(\mathbf{y})\right)_h \mathrm{d}\mathbf{y} \\
={}& \int_{\mathbb{R}_+^3} G_\varepsilon(t, \mathbf{x} - \mathbf{y})  (\mathbf{B}_0(\mathbf{y}))_h \mathrm{d}\mathbf{y} + \int_{\mathbb{R}_-^3} G_\varepsilon(t, \mathbf{x} - \mathbf{y})  \left(\widetilde{\mathbf{B}}_0(\mathbf{y})\right)_h \mathrm{d}\mathbf{y} \\
={}& \int_{\mathbb{R}_+^3} G_\varepsilon(t, \mathbf{x}_h - \mathbf{y}_h, x_3 - y_3)  (\mathbf{B}_0(\mathbf{y}))_h \mathrm{d}\mathbf{y} + \int_{\mathbb{R}_+^3} G_\varepsilon(t, \mathbf{x}_h - \mathbf{y}_h, x_3 + y_3)  (\mathbf{B}_0(\mathbf{y}))_h \mathrm{d}\mathbf{y} \\
={}& \int_{\mathbb{R}_+^3} \left( G_\varepsilon(t, \mathbf{x}_h - \mathbf{y}_h, x_3 - y_3) + G_\varepsilon(t, \mathbf{x}_h - \mathbf{y}_h, x_3 + y_3) \right) (\mathbf{B}_0(\mathbf{y}))_h \mathrm{d}\mathbf{y},
  \end{aligned}
\end{equation}
\begin{equation}\label{win2}
\begin{aligned}
&(\mathcal{B}(t)\mathbf{B}_0)_3\\
={}& G_\varepsilon(t, \cdot) * \left(\widetilde{\mathbf{B}}_0\right)_3 \\
={}& \int_{\mathbb{R}^3} G_\varepsilon(t, \mathbf{x} - \mathbf{y})  \left(\widetilde{\mathbf{B}}_0(\mathbf{y})\right)_3 \mathrm{d}\mathbf{y} \\
={}& \int_{\mathbb{R}_+^3} G_\varepsilon(t, \mathbf{x} - \mathbf{y})  (\mathbf{B}_0(\mathbf{y}))_3 \mathrm{d}\mathbf{y} + \int_{\mathbb{R}_-^3} G_\varepsilon(t, \mathbf{x} - \mathbf{y})  \left(\widetilde{\mathbf{B}}_0(\mathbf{y})\right)_3 \mathrm{d}\mathbf{y} \\
={}& \int_{\mathbb{R}_+^3} G_\varepsilon(t, \mathbf{x}_h - \mathbf{y}_h, x_3 - y_3)  (\mathbf{B}_0(\mathbf{y}))_3 \mathrm{d}\mathbf{y} - \int_{\mathbb{R}_+^3} G_\varepsilon(t, \mathbf{x}_h - \mathbf{y}_h, x_3 + y_3)  (\mathbf{B}_0(\mathbf{y}))_3 \mathrm{d}\mathbf{y} \\
={}& \int_{\mathbb{R}_+^3} \left( G_\varepsilon(t, \mathbf{x}_h - \mathbf{y}_h, x_3 - y_3) - G_\varepsilon(t, \mathbf{x}_h - \mathbf{y}_h, x_3 + y_3) \right) (\mathbf{B}_0(\mathbf{y}))_3 \mathrm{d}\mathbf{y}.
\end{aligned}
\end{equation}

In what follows, we shall employ the following form of Minkowski's inequality: for an arbitrary measurable function $f$ defined on $\mathbb{R}^m \times \mathbb{R}^n$ and any exponents satisfying $1 \leq q\leq p \leq \infty$,
\begin{equation}\label{UBW2}
  \left\| \| f \|_{L_y^q(\mathbb{R}^n)} \right\|_{L_x^p(\mathbb{R}^m)} \leq \left\| \| f \|_{L_x^p(\mathbb{R}^m)} \right\|_{L_y^q(\mathbb{R}^n)},
\end{equation}
where $L_y^q(\mathbb{R}^n)$ stands for the $L^q$-norm with respect to the $y$-variable, and $L_x^p(\mathbb{R}^m)$ denotes the $L^p$-norm with respect to the $x$-variable.
A more general formulation of Minkowski's inequality, along with its rigorous proof, can be found in the reference \cite{30}.

Applying Young's convolution inequality, Minkowski's inequality \eqref{UBW2} and estimate \eqref{win0}, we deduce that
  \begin{align*}
  &\left\| \partial_1^{k_1} \partial_2^{k_2} \partial_3^{k_3} \left( \mathcal{B}(t) \mathbf{B}_0 \right) \right\|_{L^2\left(\mathbb{R}_+^3\right)}\\
  \leq{}& C \left\| \partial_1^{k_1} \partial_2^{k_2} \partial_3^{k_3} \left( G_\varepsilon(t, \cdot) * \widetilde{\mathbf{B}}_0 \right) \right\|_{L^2(\mathbb{R}^3)}\\
  ={}& C \left\| \left( \partial_{x_1}^{k_1} G(t, \cdot) \cdot \partial_{x_2}^{k_2} G(t, \cdot) \cdot \partial_{x_3}^{k_3} G(\varepsilon t, \cdot) \right) * \widetilde{\mathbf{B}}_0 \right\|_{L^2(\mathbb{R}^3)}\\
  ={}& C \left\| \partial_{x_1}^{k_1} G(t, \cdot) *_{x_1} \left( \partial_{x_2}^{k_2} G(t, \cdot) *_{x_2} \left( \partial_{x_3}^{k_3} G(\varepsilon t, \cdot) *_{x_3} \widetilde{\mathbf{B}}_0 \right) \right) \right\|_{L^2(\mathbb{R}^3)}\\
  \leq{}& C \left\| \left\| \partial_{x_1}^{k_1} G(t, \cdot) \right\|_{L_{x_1}^{r_1}} \cdot \left\| \partial_{x_2}^{k_2} G(t, \cdot) *_{x_2} \left( \partial_{x_3}^{k_3} G(\varepsilon t, \cdot) *_{x_3} \widetilde{\mathbf{B}}_0 \right) \right\|_{L_{x_1}^{p_1}} \right\|_{L_{x_3}^2 L_{x_2}^2}\\
  \leq{}& C \left\| \partial_{x_1}^{k_1} G(t, \cdot) \right\|_{L_{x_1}^{r_1}} \cdot \left\| \left\| \partial_{x_2}^{k_2} G(t, \cdot) *_{x_2} \left( \partial_{x_3}^{k_3} G(\varepsilon t, \cdot) *_{x_3} \widetilde{\mathbf{B}}_0 \right) \right\|_{L_{x_2}^2} \right\|_{L_{x_1}^{p_1} L_{x_3}^2}\\
  \leq{}& C \left\| \partial_{x_1}^{k_1} G(t, \cdot) \right\|_{L_{x_1}^{r_1}} \cdot \left\| \partial_{x_2}^{k_2} G(t, \cdot) \right\|_{L_{x_2}^{r_2}} \cdot \left\| \left\| \partial_{x_3}^{k_3} G(\varepsilon t, \cdot) *_{x_3} \widetilde{\mathbf{B}}_0 \right\|_{L_{x_3}^2} \right\|_{L_{x_1}^{p_1} L_{x_2}^{p_2}}\\
  \leq{}& C \left\| \partial_{x_1}^{k_1} G(t, \cdot) \right\|_{L_{x_1}^{r_1}} \cdot \left\| \partial_{x_2}^{k_2} G(t, \cdot) \right\|_{L_{x_2}^{r_2}} \cdot \left\| \partial_{x_3}^{k_3} G(\varepsilon t, \cdot) \right\|_{L_{x_3}^{r_3}} \cdot \left\| \widetilde{\mathbf{B}}_0 \right\|_{L_{x_1}^{p_1} L_{x_2}^{p_2} L_{x_3}^{p_3}(\mathbb{R}^3)}\\
  \leq{}& C t^{ -\frac{1}{2}\sum_{i=1}^2 \left( k_i + \frac{1}{p_i} - \frac{1}{2} \right) } \cdot \left(\varepsilon t\right)^{ -\frac{1}{2}\left( k_3 + \frac{1}{p_3} - \frac{1}{2} \right) } \cdot \left\| \mathbf{B}_0 \right\|_{L_{x_1}^{p_1} L_{x_2}^{p_2} L_{x_3}^{p_3}\left(\mathbb{R}_+^3\right)},
  \end{align*}
where \(1 + \frac{1}{2} = \frac{1}{p_i} + \frac{1}{r_i}\) with \(1 \le p_i \le 2\) for \(i = 1, 2, 3\), and \(C > 0\) is a generic positive constant independent of $t$ and \(\varepsilon\), which may vary from line to line.

Consequently, we arrive at the following time decay estimate for the derivatives of the solution:
\begin{equation}\label{win3}
\boxed{
  \left\| \partial_1^{k_1} \partial_2^{k_2} \partial_3^{k_3} \left( \mathcal{B}(t) \mathbf{B}_0 \right) \right\|_{L^2\left(\mathbb{R}_+^3\right)}
  \leq C t^{ -\frac{1}{2}\sum_{i=1}^2 \left( k_i + \frac{1}{p_i} - \frac{1}{2} \right) } \cdot \left(\varepsilon t\right)^{ -\frac{1}{2}\left( k_3 + \frac{1}{p_3} - \frac{1}{2} \right) } \cdot \left\| \mathbf{B}_0 \right\|_{L_{x_1}^{p_1} L_{x_2}^{p_2} L_{x_3}^{p_3}\left(\mathbb{R}_+^3\right)}.
}
\end{equation}
For estimate \eqref{win3}, only by setting $p_3 = 2$ and $k_3 = 0$ can we eliminate all $\varepsilon$-dependence from the resulting decay estimate, and we then obtain estimate \eqref{4-2}.
For the case $k_3 = 1$, we apply integration by parts to transfer the normal derivative onto the initial data $\mathbf{B}_0$, which ensures the decay rate is free of $\varepsilon$.

We now turn to the proof of estimate \eqref{4-3}. For identity \eqref{win1}, applying integration by parts, we obtain
\begin{equation}\label{4-4}
  \begin{aligned}
  \partial_3\left(\mathcal{B}(t) \mathbf{B}_0\right)_h ={}& \int_{\mathbb{R}^3_+} \partial_{x_3} \left(G_\varepsilon\left(t, \mathbf{x}_h-\mathbf{y}_h, x_3-y_3\right) + G_\varepsilon\left(t, \mathbf{x}_h-\mathbf{y}_h, x_3+y_3\right)\right) \left(\mathbf{B}_0(\mathbf{y})\right)_h  \mathrm{d}\mathbf{y} \\
  ={}& \int_{\mathbb{R}^3_+} \partial_{y_3} \left(-G_\varepsilon\left(t, \mathbf{x}_h-\mathbf{y}_h, x_3-y_3\right) + G_\varepsilon\left(t, \mathbf{x}_h-\mathbf{y}_h, x_3+y_3\right)\right) \left(\mathbf{B}_0(\mathbf{y})\right)_h  \mathrm{d}\mathbf{y} \\
  ={}& \int_{\mathbb{R}^3_+} \left(G_\varepsilon\left(t, \mathbf{x}_h-\mathbf{y}_h, x_3-y_3\right) - G_\varepsilon\left(t, \mathbf{x}_h-\mathbf{y}_h, x_3+y_3\right)\right) \partial_{y_3} \left(\mathbf{B}_0(\mathbf{y})\right)_h  \mathrm{d}\mathbf{y}.
  \end{aligned}
\end{equation}
Similarly, for the identity \eqref{win2}, integrating by parts and using the boundary condition $\left(\mathbf{B}_0\right)_3\big|_{x_3=0} = 0$, we obtain
\begin{equation}\label{4-5}
  \begin{aligned}
  \partial_3\left(\mathcal{B}(t) \mathbf{B}_0\right)_3 ={}& \int_{\mathbb{R}^3_+} \partial_{x_3} \left(G_\varepsilon\left(t, \mathbf{x}_h-\mathbf{y}_h, x_3-y_3\right) - G_\varepsilon\left(t, \mathbf{x}_h-\mathbf{y}_h, x_3+y_3\right)\right) \left(\mathbf{B}_0(\mathbf{y})\right)_3  \mathrm{d}\mathbf{y} \\
  ={}& -\int_{\mathbb{R}^3_+} \partial_{y_3} \left(G_\varepsilon\left(t, \mathbf{x}_h-\mathbf{y}_h, x_3-y_3\right) + G_\varepsilon\left(t, \mathbf{x}_h-\mathbf{y}_h, x_3+y_3\right)\right) \left(\mathbf{B}_0(\mathbf{y})\right)_3  \mathrm{d}\mathbf{y} \\
  ={}& \int_{\mathbb{R}^3_+} \left(G_\varepsilon\left(t, \mathbf{x}_h-\mathbf{y}_h, x_3-y_3\right) + G_\varepsilon\left(t, \mathbf{x}_h-\mathbf{y}_h, x_3+y_3\right)\right) \partial_{y_3} \left(\mathbf{B}_0(\mathbf{y})\right)_3  \mathrm{d}\mathbf{y}.
  \end{aligned}
\end{equation}
Applying the same argument used for \eqref{win3} to \eqref{4-4} and \eqref{4-5}, we arrive at estimate \eqref{4-3}.
\end{proof}
\subsection{Decay rates of global solutions}
Once the above preparations are in hand, we represent the solution $\left(\mathbf{u}^{\varepsilon}, \mathbf{B}^{\varepsilon}\right)$ to \eqref{4-1} as follows:
\begin{equation}\label{4-6}
  \begin{aligned}
  \mathbf{u}^{\varepsilon}\left(t\right)&=\mathcal{U}(t) \mathbf{u}_{0}+\int_{0}^{t} \mathcal{U}(t-\tau) \operatorname{div} F^{\varepsilon}\left(\tau\right) \mathrm{d} \tau,\\
  \mathbf{B}^{\varepsilon}\left(t\right)&=\mathcal{B}(t) \mathbf{B}_{0}+\int_{0}^{t} \mathcal{B}(t-\tau) F^{\varepsilon}_3\left(\tau\right) \mathrm{d} \tau.\\
  \end{aligned}
\end{equation}
To state the results of decay rates of solutions, we introduce the following notation. Set
\begin{equation*}
\begin{split}
M(t)\overset{def}{=}{}&\sup_{0\leq \tau \leq t}(1+\tau)^{\frac{3}{4}}\left\|\mathbf{u}^{\varepsilon}\left(\tau\right)\right\|_{H^1}+\sup_{0\leq \tau \leq t}(1+\tau)^{\frac{1}{2}}\left\|\mathbf{B}^{\varepsilon}\left(\tau\right)\right\|_{L^2} \\
&+\sup_{0\leq \tau \leq t}(1+\tau)\left\| \partial_h\mathbf{B}^{\varepsilon}\left(\tau\right)\right\|_{L^2}+\sup_{0\leq \tau \leq t}(1+\tau)^{\frac{1}{2}-\delta_2}\left\|\partial_3\mathbf{B}^{\varepsilon}\left(\tau\right)\right\|_{L^2},
\end{split}
\end{equation*}
where $\delta_2 \in \left(0, \frac{1}{32}\right)$ is an arbitrary constant.

Below, the main task is to prove that
\begin{equation}\label{L4-6}
  M(t) \lesssim\left\| \mathbf{u}_{0} \right\|_{L^1}+\left\| \mathbf{u}_{0}\right\|_{H^{1} \times L^2}+\left\|\left(\mathbf{B}_{0},\partial_3 \mathbf{B}_{0}\right)\right\|_{L_{x_1 x_2}^1L_{x_3}^2}+ E_{m}^{\varepsilon}(t)^{2}+ D_{m}^{\varepsilon}(t)^{2}+M(t)^{2}
\end{equation}
for all $t \geq 0$, provided that $E_{m}^{\varepsilon}(\infty) \leq 1$. The desired $H^{1}$ estimates in Theorem \ref{Th3} $(i)$
then follow from Theorem \ref{Th1} and \eqref{L4-6},
provided that
\begin{align*}
      &\left\| ( a_0, \mathbf{v}_{0},\mathbf{B}_{0}) \right\|_{m}
			+\left\| \partial_3 (a_0, \mathbf{v}_{0}, \mathbf{B}_{0}) \right\|_{m-1}
+\left\| \partial_3^2\left(\mathbf{v}_{0}, \varepsilon\mathbf{B}_{0}\right) \right\|_{m-2}\\
&+ \left\| (a_0, \mathbf{m}_{0}) \right\|_{L^1}+\left\|\left(\mathbf{B}_{0},\partial_3 \mathbf{B}_{0}\right)\right\|_{L_{x_1 x_2}^1L_{x_3}^2} \le \delta_1
    \end{align*}
for some small constant $\delta_{1}>0$.

By Theorem \ref{Th1}, it is enough to show \eqref{L4-6} for $t \geq 1$, and hence hereafter we will assume that $t \geq 1$.
We decompose \eqref{4-6} as
\begin{equation}\label{ubw8}
  \begin{aligned}
  \mathbf{u}^{\varepsilon}\left(t\right)&=\mathcal{U}(t) \mathbf{u}_{0}+\int_{t-1}^{t} \mathcal{U}(t-\tau) \operatorname{div} F^{\varepsilon}\left(\tau\right) \mathrm{d} \tau+\int_{0}^{t-1} \mathcal{U}(t-\tau) \operatorname{div} F^{\varepsilon}\left(\tau\right) \mathrm{d} \tau\\
  &\overset{def}{=} I_0(t)+I_1(t)+I_2(t),\\
  \mathbf{B}^{\varepsilon}\left(t\right)&=\mathcal{B}(t) \mathbf{B}_{0}+\int_{t-1}^{t} \mathcal{B}(t-\tau) F^{\varepsilon}_3\left(\tau\right) \mathrm{d} \tau+\int_{0}^{t-1} \mathcal{B}(t-\tau) F^{\varepsilon}_3\left(\tau\right) \mathrm{d} \tau\\
  &\overset{def}{=} II_0(t)+II_1(t)+II_2(t).\\
  \end{aligned}
\end{equation}
By Lemma \ref{L4-2}, we have
\begin{equation}\label{ubw9}
\begin{aligned}
\left\|I_{0}(t)\right\|_{L^2} \lesssim  {} &  t^{-\frac{3}{4}}\left\| \mathbf{u}_{0}\right\|_{L^1}+e^{-c_{0} t}\left\| \mathbf{u}_{0}\right\|_{L^2}\\
\lesssim  {} &  (1+t)^{-\frac{3}{4}}\left(\left\| \mathbf{u}_{0}\right\|_{L^1}+\left\| \mathbf{u}_{0}\right\|_{L^2}\right),\\
\end{aligned}
\end{equation}
and
\begin{equation}\label{ubw10}
\begin{aligned}
\left\| \partial I_{0}(t)\right\|_{L^2} \lesssim  {} &  t^{-\frac{5}{4}}\left\| \mathbf{u}_{0}\right\|_{L^1}+e^{-c_{0} t}\left\| \mathbf{u}_{0}\right\|_{H^{1} \times L^2}\\
\lesssim  {} &  (1+t)^{-\frac{5}{4}}\left(\left\| \mathbf{u}_{0}\right\|_{L^1}+\left\| \mathbf{u}_{0}\right\|_{H^{1} \times L^2}\right).\\
\end{aligned}
\end{equation}
Similarly, by Lemma \ref{L4-5}, we get
\begin{equation}\label{ubw11}
\begin{aligned}
\left\| II_{0}(t)\right\|_{L^{2}\left(\mathbb{R}_+^3\right)}
 \lesssim {} &  t^{-\frac{1}{2}} \left\| \mathbf{B}_{0}\right\|_{L_{x_1 x_2}^1L_{x_3}^2}\\
 \lesssim {} &  (1+t)^{-\frac{1}{2}} \left\| \mathbf{B}_{0}\right\|_{L_{x_1 x_2}^1L_{x_3}^2},\\
\end{aligned}
\end{equation}
\begin{equation}\label{ubw12}
\begin{aligned}
\left\| \partial_h II_{0}(t)\right\|_{L^{2}\left(\mathbb{R}_+^3\right)}
\lesssim {} &  t^{-1} \left\| \mathbf{B}_{0}\right\|_{L_{x_1 x_2}^1L_{x_3}^2}\\
\lesssim {} &  (1+t)^{-1} \left\| \mathbf{B}_{0}\right\|_{L_{x_1 x_2}^1L_{x_3}^2},\\
\end{aligned}
\end{equation}
and
\begin{equation}\label{ubw13}
\begin{aligned}
\left\| \partial_3II_{0}(t)\right\|_{L^{2}\left(\mathbb{R}_+^3\right)}
 \lesssim {} &  t^{-\frac{1}{2}} \left\| \partial_3\mathbf{B}_{0}\right\|_{L_{x_1 x_2}^1L_{x_3}^2}\\
 \lesssim {} &  (1+t)^{-\frac{1}{2}} \left\| \partial_3\mathbf{B}_{0}\right\|_{L_{x_1 x_2}^1L_{x_3}^2}.\\
\end{aligned}
\end{equation}
To estimate $I_1(t)$ and $I_2(t)$, we will use the following estimates for the nonlinearities $F^{\varepsilon}_1$ and $F^{\varepsilon}_2$.

\begin{lemm}\label{L4-7}
Let $m\geq6$. Assume that $ E_{m}^{\varepsilon}(t)\leq 1$ and $a^\varepsilon\left(t\right)\geq-\frac{1}{2}$ for all $t \geq 0$. Then the
following inequalities hold for all $t \geq 0$:

(i)
\[
\left\| F^{\varepsilon}_1\left(t\right)\right\|_{L^1} \lesssim(1+t)^{-\frac{3}{2}}M(t)^2,
\]

(ii)
\[
\left\| F^{\varepsilon}_1\left(t\right)\right\|_{L^2} \lesssim(1+t)^{-\frac{9}{8}}\left( E_m^{\varepsilon}(t)^2+M(t)^2\right),
\]

(iii)
\[
\left\| \partial F^{\varepsilon}_1\left(t\right)\right\|_{L^1} \lesssim (1+t)^{-\frac{3}{4}}\left( E_m^{\varepsilon}(t)^2+M(t)^2\right),
\]

(iv)
\[
\left\| \partial F^{\varepsilon}_1\left(t\right)\right\|_{L^2} \lesssim (1+t)^{-\frac{3}{4}}\left( E_m^{\varepsilon}(t)^2+M(t)^2\right),
\]

(v)
\[
\left\| F^{\varepsilon}_2\left(t\right)\right\|_{L^1} \lesssim(1+t)^{-1}M(t)^2,
\]

(vi)
\[
\left\| F^{\varepsilon}_2\left(t\right)\right\|_{L^2} \lesssim(1+t)^{-\frac{5}{4}+\frac{1}{4}\delta_2}\left( E_m^{\varepsilon}(t)^2+M(t)^2\right),
\]

(vii)
\[
\left\| \partial F^{\varepsilon}_2\left(t\right)\right\|_{L^1} \lesssim (1+t)^{-1+\delta_2}M(t)^2,
\]

(viii)
\[
\left\| \partial F^{\varepsilon}_2\left(t\right)\right\|_{L^2} \lesssim (1+t)^{-\frac{33}{32}+\frac{21}{16}\delta_2}\left( E_m^{\varepsilon}(t)^2+M(t)^2\right).
\]

\end{lemm}

\begin{proof}
Since $ E_{m}^{\varepsilon}(t)\leq 1$ and $a^\varepsilon\left(t\right)\geq-\frac{1}{2}$ by uniform estimate, we first derive the boundedness of the $L^\infty$-norm as follows:
\[
\left\|\left(a^\varepsilon,\mathbf{m}^{\varepsilon},\mathbf{v}^{\varepsilon},\mathbf{B}^{\varepsilon}\right)\right\|_{L^{\infty}}\lesssim 1.
\]
To facilitate the transformation between $\mathbf{v}^{\varepsilon}$ and $\mathbf{m}^{\varepsilon}$, we have
\begin{equation}\label{ubw14}
  \left|\mathbf{v}^{\varepsilon}\right|\lesssim \left|\frac{\mathbf{m}^{\varepsilon}}{{a}^{\varepsilon}+1}\right|\lesssim \left|\mathbf{m}^{\varepsilon}\right|\lesssim \left|\mathbf{u}^{\varepsilon}\right|
\end{equation}
and
\begin{equation}\label{ubw15}
  \left|\partial\mathbf{v}^{\varepsilon}\right|\lesssim \left|\partial\left(\frac{\mathbf{m}^{\varepsilon}}{{a}^{\varepsilon}+1}\right)\right|\lesssim \left|\partial{a}^{\varepsilon}\right|+\left|\partial\mathbf{m}^{\varepsilon}\right|\lesssim \left|\partial\mathbf{u}^{\varepsilon}\right|.
\end{equation}
Based on this boundedness, we can obtain the following pointwise estimates of $F^{\varepsilon}_1$, $\partial F^{\varepsilon}_1$, $F^{\varepsilon}_2$ and $\partial F^{\varepsilon}_2$ respectively.
\begin{equation}\label{ubw16}
  \left|F^{\varepsilon}_1\right|\lesssim\left|\mathbf{u}^{\varepsilon}\right|^{2}+\left| a^\varepsilon \right|\left|\partial \mathbf{v}^{\varepsilon}\right|+\left|\partial  a^\varepsilon \right|\left|\mathbf{v}^{\varepsilon}\right|,
\end{equation}
\begin{equation}\label{ubw17}
 \left|\partial F^{\varepsilon}_1\right| \lesssim \left| \mathbf{u}^{\varepsilon} \right|\left|\partial \mathbf{u}^{\varepsilon}\right|+\left|\partial {a}^{\varepsilon}\right|\left|\partial \mathbf{v}^{\varepsilon}\right|+\left| {a}^{\varepsilon} \right|\left|\partial^{2}  \mathbf{v}^{\varepsilon}\right|+\left|\partial^{2}a^\varepsilon \right|\left|\mathbf{v}^{\varepsilon}\right|,
\end{equation}
\begin{equation}\label{ubw18}
  \left|F^{\varepsilon}_2\right|\lesssim\left|\mathbf{B}^{\varepsilon}\right|^2,
\end{equation}
\begin{equation}\label{ubw19}
  \left|\partial F^{\varepsilon}_2\right|\lesssim\left|\mathbf{B}^{\varepsilon}\right|\left|\partial\mathbf{B}^{\varepsilon}\right|.
\end{equation}
We first establish the estimate in $(i)$. Applying  Hölder's inequality, \eqref{ubw14} and \eqref{ubw15} to \eqref{ubw16}, we have
\[
\left\| F^{\varepsilon}_1\right\|_{L^1} \lesssim\left\| \mathbf{u}^{\varepsilon}\right\|_{0}^{2}+\left\| \mathbf{u}^{\varepsilon}\right\|_{0}\left\| \partial \mathbf{u}^{\varepsilon}\right\|_{0}\lesssim(1+t)^{-\frac{3}{2}}M(t)^2.
\]
By virtue of Lemmas \ref{Le1},\ref{Le2}, \eqref{ubw14}--\eqref{ubw16}, we derive that
\begin{align*}
\left\| F^{\varepsilon}_1\right\|_{L^2} \lesssim{}&\left\| \mathbf{u}^{\varepsilon}\right\|^2_{L^4}+\left\|a^\varepsilon\right\|_{0}^{\frac{1}{8}}\left\|\partial_{1}a^\varepsilon\right\|_{0}^{\frac{1}{8}}\left\|\partial_{2}a^\varepsilon\right\|_{0}^{\frac{1}{8}}\left\|\partial_{12}a^\varepsilon\right\|_{0}^{\frac{1}{8}}\\
&\cdot\left\|\partial_{3}a^\varepsilon\right\|_{0}^{\frac{1}{8}}\left\|\partial_{13}a^\varepsilon\right\|_{0}^{\frac{1}{8}}\left\|\partial_{23}a^\varepsilon\right\|_{0}^{\frac{1}{8}}\left\|\partial_{123}a^\varepsilon\right\|_{0}^{\frac{1}{8}}\left\|\partial\mathbf{v}^{\varepsilon}\right\|_0\\
&+\left\|\mathbf{v}^{\varepsilon}\right\|_{0}^{\frac{1}{8}}\left\|\partial_{1}\mathbf{v}^{\varepsilon}\right\|_{0}^{\frac{1}{8}}\left\|\partial_{2}\mathbf{v}^{\varepsilon}\right\|_{0}^{\frac{1}{8}}\left\|\partial_{12}\mathbf{v}^{\varepsilon}\right\|_{0}^{\frac{1}{8}}\\
&\cdot\left\|\partial_{3}\mathbf{v}^{\varepsilon}\right\|_{0}^{\frac{1}{8}}\left\|\partial_{13}\mathbf{v}^{\varepsilon}\right\|_{0}^{\frac{1}{8}}\left\|\partial_{23}\mathbf{v}^{\varepsilon}\right\|_{0}^{\frac{1}{8}}\left\|\partial_{123}\mathbf{v}^{\varepsilon}\right\|_{0}^{\frac{1}{8}}\left\|\partial{a}^{\varepsilon}\right\|_0\\
\lesssim{}&\left\| \mathbf{u}^{\varepsilon}\right\|^{\frac{1}{2}}_{0}\left\| \partial\mathbf{u}^{\varepsilon}\right\|^{\frac{3}{2}}_{0}+\left(\left\| \left(a^\varepsilon,\mathbf{v}^{\varepsilon}\right)\right\|^{\frac{1}{2}}_{m}+\left\| \partial_3\left(a^\varepsilon,\mathbf{v}^{\varepsilon}\right)\right\|^{\frac{1}{2}}_{m-1}\right)\left\| \mathbf{u}^{\varepsilon}\right\|^{\frac{1}{8}}_{0}\left\| \partial\mathbf{u}^{\varepsilon}\right\|^{\frac{11}{8}}_{0}\\
\lesssim{}&(1+t)^{-\frac{9}{8}}\left( E_m^{\varepsilon}(t)^2+M(t)^2\right),
\end{align*}
which immediately gives the desired inequality in $(ii)$.

Then, we turn to the estimate in $(iii)$.
Using Hölder's inequality and \eqref{ubw15} to \eqref{ubw17}, we derive that
\[
\left\| \partial F^{\varepsilon}_1\right\|_{L^1} \lesssim \left\| \mathbf{u}^{\varepsilon}\right\|_{0}\left\| \partial \mathbf{u}^{\varepsilon}\right\|_{0}+\left\| \partial \mathbf{u}^{\varepsilon}\right\|_{0}^{2}+\left\| \mathbf{u}^{\varepsilon}\right\|_{0}\left\|\partial^{2}  \mathbf{v}^{\varepsilon}\right\|_{0}+\left\|\partial^{2}  a^\varepsilon \cdot\mathbf{v}^{\varepsilon}\right\|_{L^1}.
\]
By Hardy's inequality, we derive
\begin{align*}
\left\|\partial^2 a^\varepsilon\cdot  \mathbf{v}^{\varepsilon}\right\|_{L^1}
\lesssim{}&\left\|\partial_3\partial a^\varepsilon\cdot  \mathbf{v}^{\varepsilon}\right\|_{L^1}+\left\|\partial_h^2 a^\varepsilon\cdot  \mathbf{v}^{\varepsilon}\right\|_{L^1}\\
\lesssim{}&\left\|\chi(x_3) \varphi\partial_3 \partial a^\varepsilon\cdot  \left(\frac{\mathbf{v}^{\varepsilon}}{x_3}\right)\right\|_{L^1}+\left\|\left(1-\chi(x_3)\right) \varphi\partial_3 \partial a^\varepsilon\cdot  \mathbf{v}^{\varepsilon}\right\|_{L^1}+\left\|\partial_h^2 a^\varepsilon\cdot  \mathbf{v}^{\varepsilon}\right\|_{L^1}\\
\lesssim{}&\left\| Z_3\partial a^\varepsilon \right\|_{0}\left\| \partial_3 \mathbf{v}^{\varepsilon}\right\|_{0}+\left\| Z_3\partial a^\varepsilon \right\|_{0}\left\|\mathbf{v}^{\varepsilon}\right\|_{0}+\left\|\partial_{h}^2 a^\varepsilon \right\|_{0}\left\|\mathbf{v}^{\varepsilon}\right\|_{0},
\end{align*}
where $\chi(x_3)$ is a smooth cutoff function supported near the boundary $x_3=0$. Combining all the above estimates, we obtain
\begin{align*}
\left\| \partial F^{\varepsilon}_1\right\|_{L^1}
\lesssim{}& \left\| \mathbf{u}^{\varepsilon}\right\|_{0}\left\| \partial \mathbf{u}^{\varepsilon}\right\|_{0}
+\left\| \partial \mathbf{u}^{\varepsilon}\right\|_{0}^{2}+\Big(\left\| \left(a^\varepsilon,\mathbf{v}^{\varepsilon}\right)\right\|_{m}+\left\| \partial_3\left(a^\varepsilon,\mathbf{v}^{\varepsilon}\right)\right\|_{m-1}\\
&+\left\|\partial_3^{2}  \mathbf{v}^{\varepsilon}\right\|_{m-2}\Big)
\left(\left\| \mathbf{u}^{\varepsilon}\right\|_{0}
+\left\|\partial\mathbf{u}^{\varepsilon}\right\|_{0}\right)\\
\lesssim{}& (1+t)^{-\frac{3}{4}}\left( E_m^{\varepsilon}(t)^2+M(t)^2\right).
\end{align*}
Next, we prove the inequality in $(iv)$.
We establish the $L^2$-norm estimate for $\partial F^{\varepsilon}_1$ as follows
\begin{align*}
\left\| \partial F^{\varepsilon}_1\right\|_{L^2} \lesssim{}&\left\| \mathbf{u}^{\varepsilon}\cdot \partial \mathbf{u}^{\varepsilon}\right\|_{0}+\left\|\partial {a}^{\varepsilon}\cdot\partial\mathbf{v}^{\varepsilon}\right\|_{0}+\left\| {a}^{\varepsilon} \cdot\partial^{2}  \mathbf{v}^{\varepsilon}\right\|_{0}
+\left\|\partial^2 a^\varepsilon\cdot  \mathbf{v}^{\varepsilon}\right\|_{0}.
\end{align*}
Applying Hölder's inequality, Lemma \ref{Le1} and \eqref{ubw14} yields
\begin{align*}
  &\left\| \mathbf{u}^{\varepsilon}\cdot \partial \mathbf{u}^{\varepsilon}\right\|_{0}+\left\|\partial {a}^{\varepsilon}\cdot\partial\mathbf{v}^{\varepsilon}\right\|_{0}\\
  \lesssim{}&\left\| \mathbf{u}^{\varepsilon}\right\|_{L^\infty} \left\|\partial \mathbf{u}^{\varepsilon}\right\|_{0}
+\left\|\partial {a}^{\varepsilon} \right\|_{0}\left\|\partial\mathbf{v}^{\varepsilon}\right\|_{L^{\infty}}\\
\lesssim{}&\Big(\left\|\left({a}^{\varepsilon},\mathbf{v}^{\varepsilon}\right)\right\|_{m}+\left\| \partial_3\left({a}^{\varepsilon},\mathbf{v}^{\varepsilon}\right)\right\|_{m-1}
+\left\|\partial_3^{2}  \mathbf{v}^{\varepsilon}\right\|_{m-2}\Big) \left\|\partial \mathbf{u}^{\varepsilon}\right\|_{0}.
\end{align*}
Using Lemma \ref{Le1}, we get
\begin{align*}
&\left\| {a}^{\varepsilon} \cdot\partial^{2}  \mathbf{v}^{\varepsilon}\right\|_{0}\\
\lesssim{}&\left\|{a}^{\varepsilon}\right\|_{0}^{\frac{1}{2}}\left\| \partial_3 {a}^{\varepsilon}\right\|_{0}^{\frac{1}{2}}\left\| \partial^2 \mathbf{v}^{\varepsilon} \right\|_{0}^{\frac{1}{4}}
			\left\| \partial_1 \partial^2 \mathbf{v}^{\varepsilon} \right\|_{0}^{\frac{1}{4}}
			\left\| \partial_2 \partial^2 \mathbf{v}^{\varepsilon} \right\|_{0}^{\frac{1}{4}}
			\left\| \partial_{12} \partial^2 \mathbf{v}^{\varepsilon} \right\|_{0}^{\frac{1}{4}}\\
\lesssim{}&\Big(\left\|\mathbf{v}^{\varepsilon}\right\|_{m}+\left\| \partial_3\mathbf{v}^{\varepsilon}\right\|_{m-1}
+\left\|\partial_3^{2}  \mathbf{v}^{\varepsilon}\right\|_{m-2}\Big) \left(\left\| \mathbf{u}^{\varepsilon}\right\|_{0}+\left\| \partial\mathbf{u}^{\varepsilon}\right\|_{0}\right).
\end{align*}
In a similar manner as before, we can deduce that
\begin{align*}
\left\|\partial^2 a^\varepsilon\cdot  \mathbf{v}^{\varepsilon}\right\|_{0}
\lesssim{}&\left\|\partial_3\partial a^\varepsilon\cdot  \mathbf{v}^{\varepsilon}\right\|_{0}+\left\|\partial_h^2 a^\varepsilon\cdot  \mathbf{v}^{\varepsilon}\right\|_{0}\\
\lesssim{}&\left\|\chi(x_3) \varphi\partial_3 \partial a^\varepsilon\cdot  \left(\frac{\mathbf{v}^{\varepsilon}}{x_3}\right)\right\|_{0}+\left\|\left(1-\chi(x_3)\right) \varphi\partial_3 \partial a^\varepsilon\cdot  \mathbf{v}^{\varepsilon}\right\|_{0}+\left\|\partial_h^2 a^\varepsilon\cdot  \mathbf{v}^{\varepsilon}\right\|_{0}\\
\lesssim{}&
\left\| \partial_3 \mathbf{v}^{\varepsilon} \right\|_{0}^{\frac{1}{4}}
			\left\| \partial_{13} \mathbf{v}^{\varepsilon} \right\|_{0}^{\frac{1}{4}}
			\left\| \partial_3^2 \mathbf{v}^{\varepsilon} \right\|_{0}^{\frac{1}{4}}
			\left\| \partial_{133} \mathbf{v}^{\varepsilon} \right\|_{0}^{\frac{1}{4}}
			\left\|Z_3 \partial a^\varepsilon\right\|_{0}^{\frac{1}{2}}
			\left\| \partial_2 Z_3 \partial a^\varepsilon \right\|_{0}^{\frac{1}{2}}\\
&+\left\|\mathbf{v}^{\varepsilon} \right\|_{0}^{\frac{1}{4}}
			\left\| \partial_{1} \mathbf{v}^{\varepsilon} \right\|_{0}^{\frac{1}{4}}
			\left\| \partial_3 \mathbf{v}^{\varepsilon} \right\|_{0}^{\frac{1}{4}}
			\left\| \partial_{13} \mathbf{v}^{\varepsilon} \right\|_{0}^{\frac{1}{4}}
			\left\|Z_3 \partial a^\varepsilon\right\|_{0}^{\frac{1}{2}}
			\left\| \partial_2 Z_3 \partial a^\varepsilon \right\|_{0}^{\frac{1}{2}}\\
&+\left\|\partial_h^2 a^\varepsilon \right\|_{0}^{\frac{1}{4}}
			\left\| \partial_{1} \partial_h^2 a^\varepsilon \right\|_{0}^{\frac{1}{4}}
			\left\| \partial_2 \partial_h^2 a^\varepsilon \right\|_{0}^{\frac{1}{4}}
			\left\| \partial_{12} \partial_h^2 a^\varepsilon \right\|_{0}^{\frac{1}{4}}
			\left\|\mathbf{v}^{\varepsilon}\right\|_{0}^{\frac{1}{2}}
			\left\| \partial_3 \mathbf{v}^{\varepsilon} \right\|_{0}^{\frac{1}{2}}\\
\lesssim{}&
\left\|\partial \mathbf{v}^{\varepsilon}\right\|_{0}^{\frac{1}{4}}\left(
    \left\|\partial_3 \mathbf{v}^{\varepsilon}\right\|_{0}^{\frac{3}{4}}
    \left\|\partial_1^4 \partial_3 \mathbf{v}^{\varepsilon}\right\|_{0}^{\frac{1}{4}}
\right)^{\frac{1}{4}}
\left\|\partial_3^2 \mathbf{v}^{\varepsilon}\right\|_{1}^{\frac{1}{2}} \\
&\cdot
\left(
    \left\|\partial a^\varepsilon\right\|_{0}
    +
    \left\|\partial a^\varepsilon\right\|_{0}^{\frac{3}{4}}
    \left\|Z_3^4 \partial a^\varepsilon\right\|_{0}^{\frac{1}{4}}
\right)^{\frac{5}{6}}
\cdot
\left\|\partial_2^3 Z_3 \partial a^\varepsilon\right\|_{0}^{\frac{1}{6}}\\
&+
\left\| \mathbf{v}^{\varepsilon}\right\|_{0}^{\frac{1}{4}}\left\|\partial \mathbf{v}^{\varepsilon}\right\|_{0}^{\frac{1}{2}}
\left(\left\|\partial_3 \mathbf{v}^{\varepsilon}\right\|_{0}^{\frac{3}{4}}
    \left\|\partial_1^4 \partial_3 \mathbf{v}^{\varepsilon}\right\|_{0}^{\frac{1}{4}}
\right)^{\frac{1}{4}} \\
&\cdot
\left(
    \left\|\partial a^\varepsilon\right\|_{0}
    +
    \left\|\partial a^\varepsilon\right\|_{0}^{\frac{3}{4}}
    \left\|Z_3^4 \partial a^\varepsilon\right\|_{0}^{\frac{1}{4}}
\right)^{\frac{5}{6}}
\cdot
\left\|\partial_2^3 Z_3 \partial a^\varepsilon\right\|_{0}^{\frac{1}{6}}\\
&+\left\| a^\varepsilon \right\|_{m}\left\|\mathbf{v}^{\varepsilon}\right\|_{0}^{\frac{1}{2}}
			\left\| \partial_3 \mathbf{v}^{\varepsilon} \right\|_{0}^{\frac{1}{2}}\\
\lesssim{}&\left(\left\| \left(a^\varepsilon,\mathbf{v}^{\varepsilon}\right)\right\|_{m}+\left\| \partial_3\left(a^\varepsilon,\mathbf{v}^{\varepsilon}\right)\right\|_{m-1}+\left\| \partial_3^2\mathbf{v}^{\varepsilon}\right\|_{m-2}\right)
\left(\left\| \mathbf{u}^{\varepsilon}\right\|_{0}+\left\| \partial\mathbf{u}^{\varepsilon}\right\|_{0}\right),
\end{align*}
in the fourth inequality, we apply the interpolation inequality for conormal derivatives, namely \eqref{a8} in Lemma \ref{Le1}.
Combining all the above estimates, we obtain
\begin{align*}
\left\| \partial F^{\varepsilon}_1\right\|_{L^2} \lesssim{}&\left(\left\| \left(a^\varepsilon,\mathbf{v}^{\varepsilon}\right)\right\|_{m}+\left\| \partial_3\left(a^\varepsilon,\mathbf{v}^{\varepsilon}\right)\right\|_{m-1}+\left\| \partial_3^2\mathbf{v}^{\varepsilon}\right\|_{m-2}\right)
\left(\left\| \mathbf{u}^{\varepsilon}\right\|_{0}+\left\| \partial\mathbf{u}^{\varepsilon}\right\|_{0}\right)\\
\lesssim{}& (1+t)^{-\frac{3}{4}}\left( E_m^{\varepsilon}(t)^2+M(t)^2\right).
\end{align*}
Now, we consider the estimates in $(v)$ and $(vii)$. By  Hölder's inequality applied to \eqref{ubw18} and \eqref{ubw19}, we have
\[
\left\| F^{\varepsilon}_2\right\|_{L^1} \lesssim\left\| \mathbf{B}^{\varepsilon}\right\|_{0}^{2}\lesssim(1+t)^{-1}M(t)^2
\]
and
\[
\left\| \partial F^{\varepsilon}_2\right\|_{L^1} \lesssim \left\| \mathbf{B}^{\varepsilon}\right\|_{0}\left\| \partial \mathbf{B}^{\varepsilon}\right\|_{0}\lesssim (1+t)^{-1+\delta_2}M(t)^2.
\]
As for the estimate in $(vi)$, applying Lemma \ref{Le1} to \eqref{ubw18}, we obtain
\[
\left\| F^{\varepsilon}_2\right\|_{L^2} \lesssim\left\| \mathbf{B}^{\varepsilon}\right\|_{0}^{\frac{1}{4}}\left\| \partial_1 \mathbf{B}^{\varepsilon}\right\|_{0}^{\frac{1}{4}}\left\| \partial_3 \mathbf{B}^{\varepsilon}\right\|_{0}^{\frac{1}{4}}\left\| \partial_{13} \mathbf{B}^{\varepsilon}\right\|_{0}^{\frac{1}{4}}\left\| \mathbf{B}^{\varepsilon}\right\|_{0}^{\frac{1}{2}}\left\| \partial_2 \mathbf{B}^{\varepsilon}\right\|_{0}^{\frac{1}{2}}\lesssim(1+t)^{-\frac{5}{4}+\frac{1}{4}\delta_2}\left( E_m^{\varepsilon}(t)^2+M(t)^2\right).
\]
Lastly, we prove the inequality in $(viii)$. Using Lemma \ref{Le1} together with the interpolation inequality to \eqref{ubw19}, we can derive that
\begin{align*}
\left\| \partial F^{\varepsilon}_2\right\|_{L^2} \lesssim {} & \left\| \mathbf{B}^{\varepsilon}\right\|_{0}^{\frac{1}{4}}\left\| \partial_1 \mathbf{B}^{\varepsilon}\right\|_{0}^{\frac{1}{4}}\left\| \partial_3 \mathbf{B}^{\varepsilon}\right\|_{0}^{\frac{1}{4}}\left\| \partial_{13} \mathbf{B}^{\varepsilon}\right\|_{0}^{\frac{1}{4}}
\left\| \partial\mathbf{B}^{\varepsilon}\right\|_{0}^{\frac{1}{2}}\left\| \partial_2 \partial \mathbf{B}^{\varepsilon}\right\|_{0}^{\frac{1}{2}}\\
\lesssim {} & \left\| \mathbf{B}^{\varepsilon}\right\|_{0}^{\frac{1}{4}}\left\| \partial_h \mathbf{B}^{\varepsilon}\right\|_{0}^{\frac{1}{4}}\left\| \partial \mathbf{B}^{\varepsilon}\right\|_{0}^{\frac{3}{4}}\left\| \partial_{h}\partial \mathbf{B}^{\varepsilon}\right\|_{0}^{\frac{3}{4}}\\
\lesssim {} & \left\| \mathbf{B}^{\varepsilon}\right\|_{0}^{\frac{1}{4}}\left\| \partial_h \mathbf{B}^{\varepsilon}\right\|_{0}^{\frac{1}{4}}\left\| \partial \mathbf{B}^{\varepsilon}\right\|_{0}^{\frac{3}{4}}\left(\left\|\partial \mathbf{B}^{\varepsilon}\right\|^{\frac{3}{4}}_{0}\left\| \partial_{h}^4\partial \mathbf{B}^{\varepsilon}\right\|^{\frac{1}{4}}_{0}\right)^{\frac{3}{4}}\\
\lesssim {} & (1+t)^{-\frac{33}{32}+\frac{21}{16}\delta_2}\left( E_m^{\varepsilon}(t)^2+M(t)^2\right).
\end{align*}
\end{proof}
We now deal with $I_1(t)$. By Lemmas \ref{L4-1} $(ii)$ and \ref{L4-7} $(ii)$, $(vi)$, we obtain
\begin{align}\label{ubw20}
\left\| I_1(t)\right\|_{{L^2}}  \lesssim{}& \int_{t-1}^{t}(t-\tau)^{-\frac{1}{2}}\left(\left\| F^{\varepsilon}_1\left(\tau\right)\right\|_{L^2}+\left\| F^{\varepsilon}_2\left(\tau\right)\right\|_{L^2}\right) \mathrm{d} \tau \nonumber\\
 \lesssim {} &  \left( E_m^{\varepsilon}(t)^2+M(t)^2\right)\int_{t-1}^{t}(t-\tau)^{-\frac{1}{2}}(1+\tau)^{-\frac{9}{8}} \mathrm{d} \tau \nonumber\\
  &+  \left( E_m^{\varepsilon}(t)^2+M(t)^2\right)\int_{t-1}^{t}(t-\tau)^{-\frac{1}{2}}(1+\tau)^{-\frac{5}{4}+\frac{1}{4}\delta_2} \mathrm{d} \tau \nonumber\\
\lesssim {} & (1+t)^{-\frac{9}{8}} \left( E_m^{\varepsilon}(t)^2+M(t)^2\right).
\end{align}
By Lemmas \ref{L4-1} $(i)$ and \ref{L4-7} $(iv)$, $(viii)$, we see that
\begin{equation}\label{ubw21}
\begin{aligned}
\left\|\partial I_1(t)\right\|_{{L^2}}  \lesssim{}& \int_{t-1}^{t}(t-\tau)^{-\frac{1}{2}}\left(\left\| \operatorname{div}F^{\varepsilon}_1\left(\tau\right)\right\|_{L^2}+\left\|\operatorname{div} F^{\varepsilon}_2\left(\tau\right)\right\|_{L^2}\right) \mathrm{d} \tau \\
 \lesssim {} &  \left( E_m^{\varepsilon}(t)^2+M(t)^2\right)\int_{t-1}^{t}(t-\tau)^{-\frac{1}{2}}(1+\tau)^{-\frac{3}{4}} \mathrm{d} \tau \\
&+\left( E_m^{\varepsilon}(t)^2+M(t)^2\right)\int_{t-1}^{t}(t-\tau)^{-\frac{1}{2}}(1+\tau)^{-\frac{33}{32}+\frac{21}{16}\delta_2} \mathrm{d} \tau\\
\lesssim {} & (1+t)^{-\frac{3}{4}} \left( E_m^{\varepsilon}(t)^2+M(t)^2\right).
\end{aligned}
\end{equation}
As for $I_2(t)$, we apply Lemmas \ref{L4-3}, \ref{L4-7} $(i)$, $(ii)$, $(v)$, $(vi)$ and \ref{Le3} to obtain,
\begin{align}\label{ubw22}
\left\| I_2(t)\right\|_{{L^2}} \lesssim {} &  \int_{0}^{t-1}(t-\tau)^{-1}\left(\left\| F^{\varepsilon}_1\left(\tau\right)\right\|_{L^1}+\left\| F^{\varepsilon}_2\left(\tau\right)\right\|_{L^1}\right) \mathrm{d} \tau \nonumber\\
&+\int_{0}^{t-1}e^{-c_0(t-\tau)}\left(\left\| F^{\varepsilon}_1\left(\tau\right)\right\|_{L^2}+\left\| F^{\varepsilon}_2\left(\tau\right)\right\|_{L^2}\right) \mathrm{d} \tau \nonumber\\
 \lesssim {} &  M(t)^2\int_{0}^{t-1}(t-\tau)^{-1}(1+\tau)^{-\frac{3}{2}} \mathrm{d} \tau \nonumber\\
&+M(t)^2\int_{0}^{t-1}(t-\tau)^{-1}(1+\tau)^{-1} \mathrm{d} \tau\nonumber\\
&+ \left( E_m^{\varepsilon}(t)^2+M(t)^2\right)\int_{0}^{t-1}e^{-c_0(t-\tau)}(1+\tau)^{-\frac{9}{8}} \mathrm{d} \tau \nonumber\\
&+ \left( E_m^{\varepsilon}(t)^2+M(t)^2\right)\int_{0}^{t-1}e^{-c_0(t-\tau)}(1+\tau)^{-\frac{5}{4}+\frac{1}{4}\delta_2} \mathrm{d} \tau \nonumber\\
\lesssim {} & (1+t)^{-1}\ln (1+t) \left( E_m^{\varepsilon}(t)^2+M(t)^2\right).
\end{align}
By Lemmas \ref{L4-2}, \ref{L4-7} $(iii)$, $(iv)$, $(vii)$, $(viii)$ and \ref{Le3}, we can arrive at
\begin{align}\label{ubw23}
\left\|\partial I_2(t)\right\|_{{L^2}} \lesssim {} &  \int_{0}^{t-1}(t-\tau)^{-\frac{5}{4}}\left(\left\|\operatorname{div} F^{\varepsilon}_1\left(\tau\right)\right\|_{L^1}+\left\|\operatorname{div} F^{\varepsilon}_2\left(\tau\right)\right\|_{L^1}\right) \mathrm{d} \tau \nonumber\\
&+\int_{0}^{t-1}e^{-c_0(t-\tau)}\left(\left\|\operatorname{div} F^{\varepsilon}_1\left(\tau\right)\right\|_{L^2}
+\left\|\operatorname{div} F^{\varepsilon}_2\left(\tau\right)\right\|_{L^2}\right) \mathrm{d} \tau \nonumber\\
 \lesssim {} &  \left( E_m^{\varepsilon}(t)^2+M(t)^2\right)\int_{0}^{t-1}(t-\tau)^{-\frac{5}{4}}(1+\tau)^{-\frac{3}{4}} \mathrm{d} \tau \nonumber\\
&+M(t)^2\int_{0}^{t-1}(t-\tau)^{-\frac{5}{4}}(1+\tau)^{-1+\delta_2} \mathrm{d} \tau\nonumber\\
&+\left( E_m^{\varepsilon}(t)^2+M(t)^2\right)\int_{0}^{t-1}e^{-c_0(t-\tau)}(1+\tau)^{-\frac{3}{4}} \mathrm{d} \tau\nonumber\\
&+\left( E_m^{\varepsilon}(t)^2+M(t)^2\right)\int_{0}^{t-1}e^{-c_0(t-\tau)}(1+\tau)^{-\frac{33}{32}+\frac{21}{16}\delta_2} \mathrm{d} \tau\nonumber\\
\lesssim {} & (1+t)^{-\frac{3}{4}} \left( E_m^{\varepsilon}(t)^2+M(t)^2\right).
\end{align}
To deal with $II_1(t)$ and $II_2(t)$, we use the following estimates for the nonlinearity $F^{\varepsilon}_3$.
\begin{lemm}\label{L4-10}
Let $m\geq6$. Assume that $ E_{m}^{\varepsilon}(t)\leq 1$ and $a^\varepsilon\left(t\right)\geq-\frac{1}{2}$ for all $t \geq 0$. Then the
following inequalities hold for all $t \geq 0$:

(i)
\[
\left\| F^{\varepsilon}_3\left(t\right)\right\|_{L_{x_1x_2}^1L_{x_3}^2} \lesssim(1+t)^{-\frac{5}{4}+\delta_2}M(t)^2,
\]

(ii)
\[
\left\| F^{\varepsilon}_3\left(t\right)\right\|_{L_{x_1}^1L_{x_2x_3}^2} \lesssim(1+t)^{-\frac{19}{16}+\frac{7}{8}\delta_2}\left( E_m^{\varepsilon}(t)^2+M(t)^2\right),
\]

(iii)
\begin{align*}
\left\| \partial_3 F^{\varepsilon}_3\left(t\right)\right\|_{L_{x_1x_2}^1L_{x_3}^2} \lesssim {} &  (1+t)^{-\frac{31}{40}+\frac{4}{5}\delta_2}\left( E_m^{\varepsilon}(t)^{\frac{3}{2}}+M(t)^{\frac{3}{2}}\right)\left(\left\|\partial\mathbf{v}^{\varepsilon}\right\|_0+\left\|\partial^2\mathbf{v}^{\varepsilon}\right\|_0\right)^{\frac{1}{2}}\\
&+(1+t)^{-\frac{1}{2}+\frac{1}{2}\delta_2} M(t)\left\|\partial^2\mathbf{v}^{\varepsilon}\right\|_0.
\end{align*}
\end{lemm}
\begin{proof}
Under the assumptions $E_{m}^{\varepsilon}(t)\leq 1$ and $a^\varepsilon\left(t\right)\geq-\frac{1}{2}$, we first derive the pointwise bounds for $F^{\varepsilon}_3$ and $\partial_3 F^{\varepsilon}_3$ as follows:
\begin{equation}\label{ubw24}
  \left|F^{\varepsilon}_3\right| \lesssim\left|\mathbf{u}^{\varepsilon}\right|\left|\partial\mathbf{B}^{\varepsilon}\right|+\left|\partial \mathbf{u}^{\varepsilon}\right|\left|\mathbf{B}^{\varepsilon}\right|
\end{equation}
and
\begin{equation}\label{ubw25}
  \left|\partial_3 F^{\varepsilon}_3\right| \lesssim \left| \partial\mathbf{v}^{\varepsilon} \right|\left|\partial\mathbf{B}^{\varepsilon} \right|+\left|\mathbf{v}^{\varepsilon} \right|\left|\partial_3\partial  \mathbf{B}^{\varepsilon} \right|
+\left|\mathbf{B}^{\varepsilon} \right|\left|\partial^2\mathbf{v}^{\varepsilon} \right|.
\end{equation}
We start with the proof of statement $(i)$.
Applying Lemma \ref{Le1} to \eqref{ubw24}, we obtain
\begin{align*}
\left\| F^{\varepsilon}_3\right\|_{L_{x_1x_2}^1L_{x_3}^2}
\lesssim{} & \left\| \mathbf{u}^{\varepsilon}\right\|_{0}^{\frac{1}{2}}\left\| \partial_3 \mathbf{u}^{\varepsilon}\right\|_{0}^{\frac{1}{2}}\left\| \partial \mathbf{B}^{\varepsilon}\right\|_{0}+\left\| \mathbf{B}^{\varepsilon}\right\|_{0}^{\frac{1}{2}}\left\| \partial_3 \mathbf{B}^{\varepsilon}\right\|_{0}^{\frac{1}{2}}\left\| \partial \mathbf{u}^{\varepsilon}\right\|_{0} \\
\lesssim{}& (1+t)^{-\frac{5}{4}+\delta_2}M(t)^2.
\end{align*}
We next turn to the estimate in $(ii)$.
Using Lemma \ref{Le1} and the interpolation inequality to \eqref{ubw24}, we have
\begin{align*}
\left\| F^{\varepsilon}_3\right\|_{L_{x_1}^1L_{x_2x_3}^2}
\lesssim{} &\left\| \mathbf{u}^{\varepsilon}\right\|_{0}^{\frac{1}{2}}\left\| \partial_3 \mathbf{u}^{\varepsilon}\right\|_{0}^{\frac{1}{2}}\left\| \partial \mathbf{B}^{\varepsilon}\right\|_{0}^{\frac{1}{2}}
\left(\left\|\partial \mathbf{B}^{\varepsilon}\right\|_{0}^{\frac{3}{4}}\left\| \partial_2^4\partial \mathbf{B}^{\varepsilon}\right\|_{0}^{\frac{1}{4}}\right)^{\frac{1}{2}} \\
&+ \left\| \mathbf{B}^{\varepsilon}\right\|_{0}^{\frac{1}{4}}\left\| \partial_2 \mathbf{B}^{\varepsilon}\right\|_{0}^{\frac{1}{4}}\left\| \partial_3 \mathbf{B}^{\varepsilon}\right\|_{0}^{\frac{1}{4}}\left\| \partial_{23} \mathbf{B}^{\varepsilon}\right\|_{0}^{\frac{1}{4}}\left\| \partial \mathbf{u}^{\varepsilon}\right\|_{0} \\
\lesssim{}& (1+t)^{-\frac{19}{16}+\frac{7}{8}\delta_2}\left( E_m^{\varepsilon}(t)^2+M(t)^2\right).
\end{align*}
Finally, we establish the inequality in $(iii)$.
Using Lemma \ref{Le1}, \eqref{ubw14} and \eqref{ubw15} to \eqref{ubw25}, there holds
\begin{align*}
\left\| \partial_3 F^{\varepsilon}_3\right\|_{L_{x_1x_2}^1L_{x_3}^2}
\lesssim{}&\left\|\partial\mathbf{v}^{\varepsilon} \cdot\partial\mathbf{B}^{\varepsilon}\right\|_{L_{x_1x_2}^1L_{x_3}^2}+\left\|\mathbf{v}^{\varepsilon} \cdot\partial_3\partial  \mathbf{B}^{\varepsilon}\right\|_{L_{x_1x_2}^1L_{x_3}^2}+\left\| \mathbf{B}^{\varepsilon} \cdot\partial^2\mathbf{v}^{\varepsilon}\right\|_{L_{x_1x_2}^1L_{x_3}^2}\\
\lesssim{}&\left\|\partial\mathbf{v}^{\varepsilon} \cdot\partial\mathbf{B}^{\varepsilon}\right\|_{L_{x_1x_2}^1L_{x_3}^2}+\left\|\chi(x_3)\left(\frac{\mathbf{v}^{\varepsilon}}{x_3}\right) \cdot\varphi\partial_3\partial  \mathbf{B}^{\varepsilon}\right\|_{L_{x_1x_2}^1L_{x_3}^2}\\
&+\left\|\left(1-\chi(x_3)\right)\mathbf{v}^{\varepsilon}\cdot\varphi\partial_3\partial  \mathbf{B}^{\varepsilon}\right\|_{L_{x_1x_2}^1L_{x_3}^2}+\left\| \mathbf{B}^{\varepsilon} \cdot\partial^2\mathbf{v}^{\varepsilon}\right\|_{L_{x_1x_2}^1L_{x_3}^2}\\
\lesssim{}&\left\| \partial\mathbf{v}^{\varepsilon}\right\|_{0}^{\frac{1}{2}}\left\| \partial_3 \partial\mathbf{v}^{\varepsilon}\right\|_{0}^{\frac{1}{2}}\left\| \partial \mathbf{B}^{\varepsilon}\right\|_{0}+\left\| \partial_3\mathbf{v}^{\varepsilon}\right\|_{0}^{\frac{1}{2}}\left\| \partial_3^2 \mathbf{v}^{\varepsilon}\right\|_{0}^{\frac{1}{2}}\left\| Z_3 \partial \mathbf{B}^{\varepsilon}\right\|_{0}\\
&+\left\|\mathbf{v}^{\varepsilon}\right\|_{0}^{\frac{1}{2}}\left\| \partial_3 \mathbf{v}^{\varepsilon}\right\|_{0}^{\frac{1}{2}}\left\| Z_3 \partial \mathbf{B}^{\varepsilon}\right\|_{0}+\left\|\mathbf{B}^{\varepsilon}\right\|_{0}^{\frac{1}{2}}\left\|\partial_3\mathbf{B}^{\varepsilon}\right\|_{0}^{\frac{1}{2}}\left\| \partial^2 \mathbf{v}^{\varepsilon}\right\|_{0}\\
\lesssim{}&\left\| \partial\mathbf{v}^{\varepsilon}\right\|_{0}^{\frac{1}{2}}\left\| \partial_3 \partial\mathbf{v}^{\varepsilon}\right\|_{0}^{\frac{1}{2}}\left\| \partial \mathbf{B}^{\varepsilon}\right\|_{0}+\left\|\mathbf{B}^{\varepsilon}\right\|_{0}^{\frac{1}{2}}\left\|\partial_3\mathbf{B}^{\varepsilon}\right\|_{0}^{\frac{1}{2}}\left\| \partial^2 \mathbf{v}^{\varepsilon}\right\|_{0}\\
&+\left\|\left(\mathbf{v}^{\varepsilon},\partial_3 \mathbf{v}^{\varepsilon}\right)\right\|_{0}^{\frac{1}{2}}\left\| \partial_3 \left(\mathbf{v}^{\varepsilon},\partial_3 \mathbf{v}^{\varepsilon}\right)\right\|_{0}^{\frac{1}{2}}\left(\left\| \partial \mathbf{B}^{\varepsilon}\right\|_{0}+\left\| \partial \mathbf{B}^{\varepsilon}\right\|_{0}^{\frac{4}{5}}\left\| Z_3^5 \partial \mathbf{B}^{\varepsilon}\right\|_{0}^{\frac{1}{5}}\right)\\
\lesssim {} &  (1+t)^{-\frac{31}{40}+\frac{4}{5}\delta_2}\left( E_m^{\varepsilon}(t)^{\frac{3}{2}}+M(t)^{\frac{3}{2}}\right)\left(\left\|\partial\mathbf{v}^{\varepsilon}\right\|_0+\left\|\partial^2\mathbf{v}^{\varepsilon}\right\|_0\right)^{\frac{1}{2}}\\
&+(1+t)^{-\frac{1}{2}+\frac{1}{2}\delta_2} M(t)\left\|\partial^2\mathbf{v}^{\varepsilon}\right\|_0.
\end{align*}
\end{proof}
We now deal with $II_1(t)$ and $II_2(t)$. By Lemmas \ref{L4-5},  \ref{L4-10} $(i)$ and \ref{Le4}, we obtain
\begin{equation}\label{ubw26}
\begin{aligned}
\left\| \left(II_{1}+II_{2}\right)\left(t\right)\right\|_{L^{2}\left(\mathbb{R}_+^3\right)}
\lesssim{}& \int_{0}^{t}(t-\tau)^{-\frac{1}{2}} \left\| F^{\varepsilon}_3\left(\tau\right)\right\|_{L_{x_1 x_2}^1L_{x_3}^2} \mathrm{d} \tau\\
\lesssim{}& M(t)^2\int_{0}^{t}(t-\tau)^{-\frac{1}{2}} (1+\tau)^{-\frac{5}{4}+\delta_2} \mathrm{d} \tau\\
\lesssim{}&(1+t)^{-\frac{1}{2}} M(t)^2.\\
\end{aligned}
\end{equation}
Then, we turn to $\partial_h II_1(t)$ and $\partial_h II_2(t)$. By Lemmas \ref{L4-5} and \ref{L4-10} $(ii)$, we have
\begin{equation}\label{ubw27}
\begin{aligned}
\left\|\partial_h II_{1}(t)\right\|_{L^{2}\left(\mathbb{R}_+^3\right)}
\lesssim{}& \int_{t-1}^{t}(t-\tau)^{-\frac{3}{4}} \left\|F^{\varepsilon}_{3}\left(\tau\right)\right\|_{L_{x_1}^1L_{x_2x_3}^2}  \mathrm{d} \tau\\
\lesssim{}& \left( E_m^{\varepsilon}(t)^2+M(t)^2\right)\int_{t-1}^{t}(t-\tau)^{-\frac{3}{4}} (1+\tau)^{-\frac{19}{16}+\frac{7}{8}\delta_2} \mathrm{d} \tau\\
\lesssim{}&(1+t)^{-\frac{19}{16}+\frac{7}{8}\delta_2} \left( E_m^{\varepsilon}(t)^2+M(t)^2\right).\\
\end{aligned}
\end{equation}
Similarly, it follows from Lemmas \ref{L4-5}, \ref{L4-10} $(i)$ and \ref{Le3} that
\begin{equation}\label{ubw28}
\begin{aligned}
\left\| \partial_{h}II_{2}(t)\right\|_{L^{2}\left(\mathbb{R}_+^3\right)}
\lesssim{}& \int_{0}^{t-1}(t-\tau)^{-1} \left\|F^{\varepsilon}_{3}(\tau)\right\|_{L_{x_1x_2}^1L_{x_3}^2}  \mathrm{d} \tau\\
\lesssim{}& M(t)^2\int_{0}^{t-1}(t-\tau)^{-1} (1+\tau)^{-\frac{5}{4}+\delta_2} \mathrm{d} \tau\\
\lesssim{}&(1+t)^{-1} M(t)^2.
\end{aligned}
\end{equation}
Finally, by Lemmas \ref{L4-5}, \ref{L4-10} $(iii)$, \ref{Le3} and the boundary condition $\left(F^{\varepsilon}_3\right)_3\big|_{x_3=0} = 0$, we deduce
\begin{equation}\label{ubw28.5}
\begin{aligned}
&\left\|\partial_3 II_{1}(t)\right\|_{L^{2}\left(\mathbb{R}_+^3\right)}\\
\lesssim{}& \int_{t-1}^{t}(t-\tau)^{-\frac{1}{2}} \left\| \partial_3F^{\varepsilon}_3(\tau)\right\|_{L_{x_1 x_2}^1L_{x_3}^2} \mathrm{d} \tau\\
\lesssim{}& \left( E_m^{\varepsilon}(t)^{\frac{3}{2}}+M(t)^{\frac{3}{2}}\right)\int_{t-1}^{t}(t-\tau)^{-\frac{1}{2}} (1+\tau)^{-\frac{31}{40}+\frac{4}{5}\delta_2}\left(\left\|\partial\mathbf{v}^{\varepsilon}\right\|_0+\left\|\partial^2\mathbf{v}^{\varepsilon}\right\|_0\right)^{\frac{1}{2}} \mathrm{d} \tau\\
&+M(t)\int_{t-1}^{t}(t-\tau)^{-\frac{1}{2}}(1+\tau)^{-\frac{1}{2}+\frac{1}{2}\delta_2} \left\|\partial^2\mathbf{v}^{\varepsilon}(\tau)\right\|_0 \mathrm{d} \tau\\
\lesssim{}& \left( E_m^{\varepsilon}(t)^{2}+M(t)^{2}\right)\int_{t-1}^{t}(t-\tau)^{-\frac{1}{2}} (1+\tau)^{-\frac{31}{40}+\frac{4}{5}\delta_2} \mathrm{d} \tau\\
&+\left( E_m^{\varepsilon}(t)^{2}+M(t)^{2}\right)\int_{t-1}^{t}(t-\tau)^{-\frac{1}{2}}(1+\tau)^{-\frac{1}{2}+\frac{1}{2}\delta_2}\mathrm{d} \tau\\
\lesssim{}&(1+t)^{-\frac{1}{2}+\frac{1}{2}\delta_2}\left( E_m^{\varepsilon}(t)^2+M(t)^2\right),
\end{aligned}
\end{equation}
and
\begin{equation}\label{ubw29}
  \begin{aligned}
&\left\|\partial_3 II_{2}(t)\right\|_{L^{2}\left(\mathbb{R}_+^3\right)}\\
\lesssim{}& \int_{0}^{t-1}(t-\tau)^{-\frac{1}{2}} \left\| \partial_3F^{\varepsilon}_3(\tau)\right\|_{L_{x_1 x_2}^1L_{x_3}^2} \mathrm{d} \tau\\
\lesssim{}& \left( E_m^{\varepsilon}(t)^{\frac{3}{2}}+M(t)^{\frac{3}{2}}\right)\int_{0}^{t-1}(t-\tau)^{-\frac{1}{2}} (1+\tau)^{-\frac{31}{40}+\frac{4}{5}\delta_2}\left(\left\|\partial\mathbf{v}^{\varepsilon}\right\|_0+\left\|\partial^2\mathbf{v}^{\varepsilon}\right\|_0\right)^{\frac{1}{2}} \mathrm{d} \tau\\
&+M(t)\int_{0}^{t-1}(t-\tau)^{-\frac{1}{2}}(1+\tau)^{-\frac{1}{2}+\frac{1}{2}\delta_2} \left\|\partial^2\mathbf{v}^{\varepsilon}(\tau)\right\|_0 \mathrm{d} \tau\\
\lesssim{}& \left( E_m^{\varepsilon}(t)^{\frac{3}{2}}+M(t)^{\frac{3}{2}}\right)\left(\int_{0}^{t-1}(t-\tau)^{-\frac{2}{3}}(1+\tau)^{-\frac{31}{30}+\frac{16}{15}\delta_2} \mathrm{d} \tau\right)^{\frac{3}{4}}\\
&\cdot\left(\int_{0}^{t-1}\left(\left\|\partial\mathbf{v}^{\varepsilon}(\tau)\right\|_0^{2}+\left\|\partial^2\mathbf{v}^{\varepsilon}(\tau)\right\|_0^{2}\right) \mathrm{d} \tau\right)^{\frac{1}{4}}\\
&+M(t)\left(\int_{0}^{t-1}(t-\tau)^{-1}(1+\tau)^{-1+\delta_2} \mathrm{d} \tau\right)^{\frac{1}{2}}\left(\int_{0}^{t-1}\left\|\partial^2\mathbf{v}^{\varepsilon}(\tau)\right\|_0^{2} \mathrm{d} \tau\right)^{\frac{1}{2}}\\
\lesssim{}&(1+t)^{-\frac{1}{2}+\frac{1}{2}\delta_2}\left(\ln(1+t)\right)^{\frac{1}{2}}\left( E_m^{\varepsilon}(t)^2+D_m^{\varepsilon}(t)^2+M(t)^2\right),
\end{aligned}
\end{equation}
where the last inequality invokes the condition \(\delta_2 \in \left(0, \frac{1}{32}\right)\).

Consequently, combining the estimates \eqref{ubw9}--\eqref{ubw13}, \eqref{ubw20}--\eqref{ubw23} and \eqref{ubw26}--\eqref{ubw29}, we arrive at
\begin{align*}
M(t)
\lesssim{}&\sup_{0\leq \tau \leq t}(1+\tau)^{\frac{3}{4}}\left\|(I_0+I_1+I_2)(\tau)\right\|_{H^1}+\sup_{0\leq \tau \leq t}(1+\tau)^{\frac{1}{2}}\left\|(II_0+II_1+II_2)(\tau)\right\|_{L^2}\\
&+\sup_{0\leq \tau \leq t}(1+\tau)\left\|\partial_h\left(II_0+II_1+II_2\right)(\tau)\right\|_{L^2}+\sup_{0\leq \tau \leq t}(1+\tau)^{\frac{1}{2}-\delta_2}\left\|\partial_3(II_0+II_1+II_2)(\tau)\right\|_{L^2}\\
\lesssim{}&\left\| \mathbf{u}_{0} \right\|_{L^1}+\left\| \mathbf{u}_{0}\right\|_{H^{1} \times L^2}+\left\|(\mathbf{B}_{0},\partial_3 \mathbf{B}_{0})\right\|_{L_{x_1 x_2}^1L_{x_3}^2}+ E_{m}^{\varepsilon}(t)^{2}+ D_{m}^{\varepsilon}(t)^{2}+M(t)^{2},
\end{align*}
which verifies inequality \eqref{L4-6}. As stated before, under the small initial data condition
\begin{align*}
&\| ( a_0 ,\mathbf{v}_{0},\mathbf{B}_{0}) \|_{m}
+\left\| \partial_3 (a_0, \mathbf{v}_{0}, \mathbf{B}_{0}) \right\|_{m-1}
+\left\| \partial_3^2(\mathbf{v}_{0}, \varepsilon\mathbf{B}_{0}) \right\|_{m-2}\\
&+ \left\| (a_0, \mathbf{m}_{0}) \right\|_{L^1}+\left\|(\mathbf{B}_{0},\partial_3 \mathbf{B}_{0})\right\|_{L_{x_1 x_2}^1L_{x_3}^2} \le \delta_1
\end{align*}
for sufficiently small constant $\delta_{1}>0$, the desired $H^1$ decay results in Theorem \ref{Th3} $(i)$ and the uniform bound $M(t) \lesssim \delta_1$ for all $t\geq0$ follow directly from Theorem \ref{Th1} and \eqref{L4-6}.

Furthermore, utilizing \eqref{ubw20} and \eqref{ubw22}, we obtain
\begin{equation}\label{blade1}
\begin{aligned}
\left\|\mathbf{u}^\varepsilon(t)-\mathcal{U}(t) \mathbf{u}_{0}\right\|_{L^2}\lesssim{}&\left\| I_{1}(t)\right\|_{L^2}+\left\| I_{2}(t)\right\|_{L^2}\\
\lesssim{}& (1+t)^{-1}\ln (1+t) \left(E_m^{\varepsilon}(t)^2+M(t)^2\right).
\end{aligned}
\end{equation}
Then Theorem \ref{Th3} $(ii)$ follows immediately from \eqref{blade1} and Lemma \ref{aria4}. The proof of Theorem \ref{Th3} is thus completed.

\section{Convergence rate of global solutions}\label{Convergence rate of global solution}
In this section, we establish the convergence rate of the solutions between systems \eqref{eq1} and \eqref{eq4}.
For notational convenience, we introduce the following notations
\[
\bar{\rho} \overset{def}{=} \rho^\varepsilon - \rho^0, \quad \bar{a} \overset{def}{=} a^\varepsilon - a^0, \quad \bar{\rho} = \bar{a}, \quad \bar{\mathbf{v}} \overset{def}{=} \mathbf{v}^\varepsilon - \mathbf{v}^0, \quad \bar{\mathbf{B}} \overset{def}{=} \mathbf{B}^\varepsilon - \mathbf{B}^0, \quad \bar{P} \overset{def}{=} P\left(\rho^\varepsilon\right) - P\left(\rho^0\right).
\]
From the equations \eqref{eq1} and \eqref{eq4}, we can obtain
\begin{equation}\label{Con1}
	\left\{
	\begin{array}{*{5}{ll}}
		\partial_t \bar{\rho} + \mathbf{v}^{0}\cdot\nabla \bar{\rho} + \bar{\rho}\operatorname{div}\mathbf{v}^{0} = \bar{f}_1 &{\rm in} ~~ \mathbb{R}_+^3,\\[4pt]
		\rho^{0}\partial_t \bar{\mathbf{v}} + \rho^{0} \mathbf{v}^{0}\cdot\nabla \bar{\mathbf{v}} - \mu\Delta \bar{\mathbf{v}} - (\mu+\lambda)\nabla\operatorname{div}\bar{\mathbf{v}} + \nabla\bar{P} = \bar{f}_2 &{\rm in} ~~ \mathbb{R}_+^3,\\[4pt]
		\partial_t \bar{\mathbf{B}} - \Delta_h \bar{\mathbf{B}} = \bar{f}_3+\varepsilon\partial_3^2\mathbf{B}^{\varepsilon} &{\rm in} ~~ \mathbb{R}_+^3,\\[4pt]
		\operatorname{div}\bar{\mathbf{B}} = 0 &{\rm in} ~~ \mathbb{R}_+^3,\\[4pt]
		\left.\left(\bar{\rho},\bar{\mathbf{v}},\bar{\mathbf{B}}\right)\right|_{t=0} = \boldsymbol{0} &{\rm in} ~~ \mathbb{R}_+^3,\\[4pt]
		\left(\bar{\mathbf{v}},\bar{\mathbf{B}}_3\right) = \boldsymbol{0} &{\rm on} ~~ \mathbb{R}^2 \times \left\{ x_3=0 \right\},
	\end{array}
	\right.
\end{equation}
where the source terms $\bar{f}_1,\bar{f}_2,\bar{f}_3$ are given explicitly as follows:
\[
\left\{
\begin{aligned}
&\bar{f}_1 \overset{def}{=} -\bar{\mathbf{v}}\cdot\nabla\rho^\varepsilon - \rho^\varepsilon\operatorname{div}\bar{\mathbf{v}},\\[4pt]
&\bar{f}_2 \overset{def}{=} \operatorname{div}\left(\mathbf{B}^{0}\otimes \bar{\mathbf{B}} + \bar{\mathbf{B}}\otimes \mathbf{B}^\varepsilon\right) - \frac{1}{2}\nabla\left(\mathbf{B}^{0}\cdot\bar{\mathbf{B}} + \bar{\mathbf{B}}\cdot \mathbf{B}^\varepsilon\right)\\[3pt]
&\quad\quad -\bar{\rho}\partial_t \mathbf{v}^\varepsilon - \left(\rho^{0}\bar{\mathbf{v}} + \bar{\rho} \mathbf{v}^\varepsilon\right)\cdot\nabla \mathbf{v}^\varepsilon,\\[4pt]
&\bar{f}_3 \overset{def}{=} -\mathbf{v}^\varepsilon\cdot\nabla\bar{\mathbf{B}} - \mathbf{B}^{0}\operatorname{div}\bar{\mathbf{v}} + \mathbf{B}^{0}\cdot\nabla\bar{\mathbf{v}} - \bar{\mathbf{v}}\cdot\nabla \mathbf{B}^{0} - \bar{\mathbf{B}}\operatorname{div}\mathbf{v}^\varepsilon + \bar{\mathbf{B}}\cdot\nabla \mathbf{v}^\varepsilon.
\end{aligned}
\right.
\]
Taking the $L^2$ inner product of the momentum and magnetic field equations with $\left(\bar{\mathbf{v}},\bar{\mathbf{B}}\right)$ respectively, we can derive the basic energy identity
\begin{equation}\label{Con2}
  \begin{aligned}
&\frac{1}{2}\frac{\mathrm{d}}{\mathrm{d}t}\left(\int_{\mathbb{R}_+^3} \rho^{0}\left|\bar{\mathbf{v}}\right|^2  \mathrm{d}\mathbf{x}+ \left\|\bar{\mathbf{B}}\right\|_{L^2}^2\right)
+\mu\left\|\nabla\bar{\mathbf{v}}\right\|_{L^2}^2 + (\mu+\lambda)\left\|\operatorname{div}\bar{\mathbf{v}}\right\|_{L^2}^2\\
&+\left\|\nabla_h\bar{\mathbf{B}}\right\|_{L^2}^2 + \int_{\mathbb{R}_+^3}\bar{\mathbf{v}} \cdot \nabla \bar{P} \mathrm{d}\mathbf{x} =\left(\bar{f}_2,\bar{\mathbf{v}}\right) + \left(\bar{f}_3,\bar{\mathbf{B}}\right)-\left(\varepsilon\partial_3\mathbf{B}^{\varepsilon},\partial_3\bar{\mathbf{B}}\right).
\end{aligned}
\end{equation}
First, by integration by parts, Hölder's inequality and Young's inequality, it holds
\begin{align*}
&\int_{\mathbb{R}_+^3}\left( \operatorname{div}\left(\mathbf{B}^{0}\otimes\bar{\mathbf{B}}+\bar{\mathbf{B}}\otimes \mathbf{B}^\varepsilon\right)-\frac{1}{2}\nabla\left(\mathbf{B}^{0}\cdot\bar{\mathbf{B}}+\bar{\mathbf{B}}\cdot \mathbf{B}^\varepsilon\right) \right)\cdot\bar{\mathbf{v}} \mathrm{d}\mathbf{x} \\
\lesssim{}& \left\|\left(\mathbf{B}^{0},\mathbf{B}^\varepsilon\right)\right\|_{L^\infty}\left\|\bar{\mathbf{B}}\right\|_{L^2}\left\|\partial\bar{\mathbf{v}}\right\|_{L^2} \\
\lesssim{}& \delta_5\left\|\partial\bar{\mathbf{v}}\right\|_{L^2}^2 + \left\|\left(\mathbf{B}^{0},\mathbf{B}^\varepsilon\right)\right\|_{L^\infty}^2\left\|\bar{\mathbf{B}}\right\|_{L^2}^2,
\end{align*}
where $\delta_5$ is a sufficiently small positive constant.

Second, using Lemma \ref{Le2}, we have
\begin{align*}
&-\int_{\mathbb{R}_+^3}\bar{\rho}\cdot\partial_t \mathbf{v}^\varepsilon\cdot\bar{\mathbf{v}} \mathrm{d}\mathbf{x}
\lesssim \left\|\bar{\rho}\right\|_{L^2}\left\|\partial_t \mathbf{v}^\varepsilon\right\|_{L^3}\left\|\bar{\mathbf{v}}\right\|_{L^6}
\lesssim \delta_5\left\|\partial\bar{\mathbf{v}}\right\|_{L^2}^2 + \left\|\partial_t \mathbf{v}^\varepsilon\right\|_{H^1}^2\left\|\bar{a}\right\|_{L^2}^2.
\end{align*}
Similarly, the same arguments yield
\begin{align*}
&-\int_{\mathbb{R}_+^3}\left(\rho^{0}\bar{\mathbf{v}}\right)\cdot\nabla \mathbf{v}^\varepsilon\cdot\bar{\mathbf{v}} \mathrm{d}\mathbf{x}
\lesssim \left\|\bar{\mathbf{v}}\right\|_{L^2}\left\|\nabla \mathbf{v}^\varepsilon\right\|_{L^3}\left\|\bar{\mathbf{v}}\right\|_{L^6}
\lesssim \delta_5\left\|\partial\bar{\mathbf{v}}\right\|_{L^2}^2 + \left\|\partial \mathbf{v}^\varepsilon\right\|_{H^1}^2\left\|\bar{\mathbf{v}}\right\|_{L^2}^2,
\end{align*}
and
\begin{align*}
&-\int_{\mathbb{R}_+^3}\left(\bar{\rho} \mathbf{v}^\varepsilon\right)\cdot\nabla \mathbf{v}^\varepsilon\cdot\bar{\mathbf{v}} \mathrm{d}\mathbf{x}
\lesssim \left\|\bar{\rho}\right\|_{L^2}\left\|\nabla\mathbf{v}^\varepsilon\right\|_{L^3}\left\|\bar{\mathbf{v}}\right\|_{L^6}
\lesssim \delta_5\left\|\partial\bar{\mathbf{v}}\right\|_{L^2}^2 + \left\|\partial \mathbf{v}^\varepsilon\right\|_{H^1}^2\left\|\bar{a}\right\|_{L^2}^2.
\end{align*}
Next, by integration by parts, Lemma \ref{Le1} and Young's inequality, we have
\begin{align*}
&-\int_{\mathbb{R}_+^3}\mathbf{v}^{\varepsilon}\cdot\nabla\bar{\mathbf{B}}\cdot\bar{\mathbf{B}} \mathrm{d}\mathbf{x}\\
={}& -\frac{1}{2}\int_{\mathbb{R}_+^3}\mathbf{v}^{\varepsilon}\cdot\nabla\left|\bar{\mathbf{B}}\right|^2 \mathrm{d}\mathbf{x}\\
={}& \frac{1}{2}\int_{\mathbb{R}_+^3}\operatorname{div}\mathbf{v}^{\varepsilon}\left|\bar{\mathbf{B}}\right|^2 \mathrm{d}\mathbf{x} \\
\lesssim{}& \left\|\operatorname{div}\mathbf{v}^{\varepsilon}\right\|_{0}^{\frac{1}{2}}\left\|\partial_3\operatorname{div}\mathbf{v}^{\varepsilon}\right\|_{0}^{\frac{1}{2}}\left\|\bar{\mathbf{B}}\right\|_{0}^{\frac{1}{2}}\left\|\partial_1\bar{\mathbf{B}}\right\|_{0}^{\frac{1}{2}}\left\|\bar{\mathbf{B}}\right\|_{0}^{\frac{1}{2}}\left\|\partial_2\bar{\mathbf{B}}\right\|_{0}^{\frac{1}{2}}\\
\lesssim{}& \delta_5\left\|\partial_h\bar{\mathbf{B}}\right\|_{L^2}^2 + \left\|\partial \mathbf{v}^{\varepsilon}\right\|_{H^1}^2\left\|\bar{\mathbf{B}}\right\|_{L^2}^2.
\end{align*}
And it follows from Hölder's inequality  and Young's inequality that
\begin{align*}
\int_{\mathbb{R}_+^3}\left(-\mathbf{B}^{0}\cdot\operatorname{div}\bar{\mathbf{v}}+\mathbf{B}^0\cdot\nabla\bar{\mathbf{v}}\right)\cdot\bar{\mathbf{B}} \mathrm{d}\mathbf{x}
\lesssim{}& \left\|\mathbf{B}^{0}\right\|_{L^\infty}\left\|\partial\bar{\mathbf{v}}\right\|_{L^2}\left\|\bar{\mathbf{B}}\right\|_{L^2}\\
\lesssim{}&\delta_5\left\|\partial\bar{\mathbf{v}}\right\|_{L^2}^2 + \left\|\mathbf{B}^{0}\right\|_{L^\infty}^2\left\|\bar{\mathbf{B}}\right\|_{L^2}^2.
\end{align*}
Then, with an aid of Lemma \ref{Le1} and Young's inequality, we obtain
\begin{align*}
-\int_{\mathbb{R}_+^3}\bar{\mathbf{v}}\cdot\nabla \mathbf{B}^{0}\cdot\bar{\mathbf{B}} \mathrm{d}\mathbf{x}
\lesssim{}& \left\|\bar{\mathbf{v}}\right\|_{0}^{\frac{1}{2}}\left\|\partial_3\bar{\mathbf{v}}\right\|_{0}^{\frac{1}{2}}\left\|\bar{\mathbf{B}}\right\|_{0}^{\frac{1}{2}}\left\|\partial_2\bar{\mathbf{B}}\right\|_{0}^{\frac{1}{2}}\left\|\partial \mathbf{B}^{0}\right\|_{0}^{\frac{1}{2}}\left\|\partial_1\partial \mathbf{B}^{0}\right\|_{0}^{\frac{1}{2}} \\
\lesssim{}& \delta_5\left\|\left(\partial\bar{\mathbf{v}}, \partial_h\bar{\mathbf{B}}\right)\right\|_{L^2}^2
+ \left\|\partial \mathbf{B}^{0}\right\|_{L^2}\left\|\partial_h\partial \mathbf{B}^{0}\right\|_{L^2}\left\|\left(\bar{\mathbf{v}},\bar{\mathbf{B}}\right)\right\|_{L^2}^2.
\end{align*}
Finally, using the anisotropic type inequality in Lemma \ref{Le1} and Young's inequality lead to
\begin{align*}
\int_{\mathbb{R}_+^3}\left(-\bar{\mathbf{B}}\operatorname{div}\mathbf{v}^\varepsilon+\bar{\mathbf{B}}\cdot\nabla \mathbf{v}^\varepsilon\right)\bar{\mathbf{B}} \mathrm{d}\mathbf{x}
\lesssim{}& \left\|\bar{\mathbf{B}}\right\|_{0}^{\frac{1}{2}}\left\|\partial_1\bar{\mathbf{B}}\right\|_{0}^{\frac{1}{2}}\left\|\bar{\mathbf{B}}\right\|_{0}^{\frac{1}{2}}\left\|\partial_2\bar{\mathbf{B}}\right\|_{0}^{\frac{1}{2}}\left\|\partial\mathbf{v}^\varepsilon\right\|_{0}^{\frac{1}{2}}
\left\|\partial_3\partial \mathbf{v}^\varepsilon\right\|_{0}^{\frac{1}{2}} \\
\lesssim{}& \delta_5\left\|\partial_h\bar{\mathbf{B}}\right\|_{L^2}^2 + \left\|\partial \mathbf{v}^\varepsilon\right\|_{H^1}^2\left\|\bar{\mathbf{B}}\right\|_{L^2}^2.
\end{align*}
Substituting the above estimates into the basic energy identity \eqref{Con2}, we arrive at the following inequality
\begin{equation}\label{Con3}
  \begin{aligned}
&\frac{\mathrm{d}}{\mathrm{d}t}\left(\int_{\mathbb{R}_+^3} \rho^{0}\left|\bar{\mathbf{v}}\right|^2  \mathrm{d}\mathbf{x} + \left\|\bar{\mathbf{B}}\right\|_{L^2}^2\right)
+\left\|\left(\partial\bar{\mathbf{v}},\partial_h\bar{\mathbf{B}}\right)\right\|_{L^2}^2 + \int_{\mathbb{R}_+^3}\bar{\mathbf{v}} \cdot \nabla \bar{P} \mathrm{d}\mathbf{x} \\
\lesssim{}& \delta_5\left\|\left(\partial\bar{\mathbf{v}}, \partial_h\bar{\mathbf{B}}\right)\right\|_{L^2}^2+ \left( \left\|\left(\partial_t \mathbf{v}^\varepsilon,\partial \mathbf{v}^\varepsilon\right)\right\|_{H^1}^2+\left\|\left(\mathbf{B}^{0},\mathbf{B}^\varepsilon\right)\right\|_{L^\infty}^2
+ \left\|\partial \mathbf{B}^{0}\right\|_{L^2}\left\|\partial_h\partial \mathbf{B}^{0}\right\|_{L^2}\right)\\
&\cdot\left\|\left(\bar{a},\bar{\mathbf{v}},\bar{\mathbf{B}}\right)\right\|_{L^2}^2-\left(\varepsilon\partial_3\mathbf{B}^{\varepsilon},\partial_3\bar{\mathbf{B}}\right).
\end{aligned}
\end{equation}
We now turn to the pressure gradient term $\int_{\mathbb{R}_+^3}\bar{\mathbf{v}} \cdot \nabla \bar{P} \mathrm{d}\mathbf{x}$. For convenience, we define the auxiliary function
\[
g(\rho) \overset{def}{=}\rho \int_{1}^{\rho} \frac{P(s) - P(1)}{s^2} \mathrm{d}s,
\]
which satisfies the following relations
\[
\rho \nabla g'(\rho) = \nabla P(\rho), \quad g''(\rho) = \frac{P'(\rho)}{\rho},
\]
and for any fixed positive constant $c$, if $c \leq \rho \leq c^{-1}$, then $g(\rho) \sim (\rho-1)^2$.

By the continuity equation, we have the energy identity
\begin{equation}\label{Con3.1}
  \int_{\mathbb{R}_+^3} \mathbf{v} \cdot \nabla P(\rho) \mathrm{d}\mathbf{x} = \frac{\mathrm{d}}{\mathrm{d}t} \int_{\mathbb{R}_+^3} g(\rho) \mathrm{d}\mathbf{x}.
\end{equation}
And the pressure gradient term can be rewritten as the following formula.
\[
\int_{\mathbb{R}_+^3}\bar{\mathbf{v}} \cdot \nabla \bar{P} \mathrm{d}\mathbf{x}= \int_{\mathbb{R}_+^3}\Big(\mathbf{v}^\varepsilon \cdot \nabla P\left(\rho^\varepsilon\right) - \mathbf{v}^0 \cdot \nabla P\left(\rho^0\right) - \mathbf{v}^0 \cdot \nabla \bar{P} - \bar{\mathbf{v}} \cdot \nabla P\left(\rho^0\right)\Big)\mathrm{d}\mathbf{x}.
\]
By \eqref{Con3.1} and Taylor expansion, we obtain
\begin{equation}\label{Con3.2}
  \begin{aligned}
&\int_{\mathbb{R}_+^3} \left( \mathbf{v}^\varepsilon \cdot \nabla P\left(\rho^\varepsilon\right) - \mathbf{v}^0 \cdot \nabla P\left(\rho^0\right) \right) \mathrm{d}\mathbf{x} \\
={}& \frac{\mathrm{d}}{\mathrm{d}t} \int_{\mathbb{R}_+^3} \left( g\left(\rho^\varepsilon\right) - g\left(\rho^0\right) \right) \mathrm{d}\mathbf{x} \\
={}& \frac{\mathrm{d}}{\mathrm{d}t} \int_{\mathbb{R}_+^3} \mathcal{G}\left(\rho^\varepsilon, \rho^0\right) \mathrm{d}\mathbf{x} + \int_{\mathbb{R}_+^3} g''\left(\rho^0\right) \partial_t \rho^0 \bar{\rho} \mathrm{d}\mathbf{x} + \int_{\mathbb{R}_+^3} g'\left(\rho^0\right) \partial_t \bar{\rho} \mathrm{d}\mathbf{x}\\
={}& \frac{\mathrm{d}}{\mathrm{d}t} \int_{\mathbb{R}_+^3} \mathcal{G}\left(\rho^\varepsilon, \rho^0\right) \mathrm{d}\mathbf{x} + \int_{\mathbb{R}_+^3} \frac{P'\left(\rho^0\right)}{\rho^0} \partial_t \rho^0 \bar{\rho} \mathrm{d}\mathbf{x} + \int_{\mathbb{R}_+^3} g'\left(\rho^0\right) \partial_t \bar{\rho} \mathrm{d}\mathbf{x},
\end{aligned}
\end{equation}
where
\[
\mathcal{G}\left(\rho^\varepsilon, \rho^0\right) \overset{def}{=} g\left(\rho^\varepsilon\right) - g\left(\rho^0\right) - g'\left(\rho^0\right) \bar{\rho}.
\]
It is straightforward to check that for any positive constant $c$, if \(c \leq \rho^0, \rho^\varepsilon \leq c^{-1}\), then \(\mathcal{G}\left(\rho^\varepsilon, \rho^0\right) \sim \bar{a}^2\).

Using the continuity equation of difference function $\partial_t \bar{\rho} + \operatorname{div}\left( \rho^0 \bar{\mathbf{v}} + \bar{\rho} \mathbf{v}^\varepsilon \right) = 0$, we rewrite the term involving $\partial_t \bar{\rho}$ as
\begin{align*}
\int_{\mathbb{R}_+^3} g'\left(\rho^0\right) \partial_t \bar{\rho} \mathrm{d}\mathbf{x}
={}& -\int_{\mathbb{R}_+^3} g'\left(\rho^0\right) \operatorname{div}\left( \rho^0 \bar{\mathbf{v}} + \bar{\rho} \mathbf{v}^\varepsilon \right) \mathrm{d}\mathbf{x} \\
= {}&\int_{\mathbb{R}_+^3} \nabla g'\left(\rho^0\right) \cdot \left( \rho^0 \bar{\mathbf{v}} + \bar{\rho} \mathbf{v}^\varepsilon \right) \mathrm{d}\mathbf{x} \\
= {}&\int_{\mathbb{R}_+^3} \rho^0 \nabla g'\left(\rho^0\right) \cdot \bar{\mathbf{v}} \mathrm{d}\mathbf{x} + \int_{\mathbb{R}_+^3} g''\left(\rho^0\right) \left( \nabla \rho^0 \cdot \mathbf{v}^\varepsilon \right) \bar{\rho} \mathrm{d}\mathbf{x}\\
={}&\int_{\mathbb{R}_+^3} \bar{\mathbf{v}} \cdot \nabla P\left(\rho^0\right) \mathrm{d}\mathbf{x} + \int_{\mathbb{R}_+^3} \frac{P'\left(\rho^0\right)}{\rho^0} \left( \mathbf{v}^\varepsilon \cdot \nabla \rho^0 \right) \bar{\rho} \mathrm{d}\mathbf{x},
\end{align*}
where we have used the facts $\rho^0 \nabla g'\left(\rho^0\right) = \nabla P\left(\rho^0\right)$ and $g''\left(\rho^0\right) = \frac{P'\left(\rho^0\right)}{\rho^0}$.

Substituting this back into \eqref{Con3.2} gives
\begin{align*}
&\int_{\mathbb{R}_+^3} \left( \mathbf{v}^\varepsilon \cdot \nabla P\left(\rho^\varepsilon\right) - \mathbf{v}^0 \cdot \nabla P\left(\rho^0\right) \right) \mathrm{d}\mathbf{x} \\
={}& \frac{\mathrm{d}}{\mathrm{d}t} \int_{\mathbb{R}_+^3} \mathcal{G}\left(\rho^\varepsilon, \rho^0\right) \mathrm{d}\mathbf{x} + \int_{\mathbb{R}_+^3} \frac{P'\left(\rho^0\right)}{\rho^0} \left( \partial_t \rho^0 + \mathbf{v}^\varepsilon \cdot \nabla \rho^0 \right) \bar{\rho} \mathrm{d}\mathbf{x} + \int_{\mathbb{R}_+^3} \bar{\mathbf{v}} \cdot \nabla P\left(\rho^0\right) \mathrm{d}\mathbf{x}.
\end{align*}
Integrating by parts and using the Taylor expansion up to the third-order, we obtain
\begin{align*}
  -\int_{\mathbb{R}_+^3} \mathbf{v}^0 \cdot \nabla \bar{P}  \mathrm{d}\mathbf{x} ={}& \int_{\mathbb{R}_+^3} \bar{P}  \operatorname{div} \mathbf{v}^0  \mathrm{d}\mathbf{x}\\
   ={}& \int_{\mathbb{R}_+^3} \left( P'\left(\rho^0\right) \bar{\rho} + \frac{P''\left(\rho^0\right)}{2} \bar{\rho}^2 +\frac{P'''\left(\rho^0\right)}{6} \bar{\rho}^3 + F\left(\rho^0, \bar{\rho}\right) \bar{\rho}^4 \right) \operatorname{div} \mathbf{v}^0  \mathrm{d}\mathbf{x},
\end{align*}
where
\[
F\left(\rho^0, \bar{\rho}\right) \overset{def}{=} \frac{1}{6} \int_{0}^{1} (1-\theta)^3 P''''\left(\rho^0 + \theta \bar{\rho}\right)  \mathrm{d}\theta.
\]

Combining all terms yields the exact identity
\begin{equation}\label{con1}
\begin{aligned}
\int_{\mathbb{R}_+^3} \bar{\mathbf{v}} \cdot \nabla \bar{P} \mathrm{d}\mathbf{x}
={}& \frac{\mathrm{d}}{\mathrm{d}t} \int_{\mathbb{R}_+^3} \mathcal{G}\left(\rho^\varepsilon, \rho^0\right) \mathrm{d}\mathbf{x}  + \int_{\mathbb{R}_+^3} \frac{P'\left(\rho^0\right)}{\rho^0} \left(  \bar{\mathbf{v}}\cdot \nabla \rho^0 \right) \bar{\rho} \mathrm{d}\mathbf{x}  \\
&+ \int_{\mathbb{R}_+^3} \left( \frac{P''\left(\rho^0\right)}{2} \bar{\rho}^2+\frac{P'''\left(\rho^0\right)}{6} \bar{\rho}^3 + F\left(\rho^0, \bar{\rho}\right) \bar{\rho}^4 \right) \operatorname{div} \mathbf{v}^0 \mathrm{d}\mathbf{x},
\end{aligned}
\end{equation}
where we use the continuity equation $\partial_t \rho^0 + \mathbf{v}^0 \cdot \nabla \rho^0 + \rho^0 \operatorname{div} \mathbf{v}^0=0$.

First, by Hölder's inequality, Lemma \ref{Le2} and Young's inequality, we have
\begin{equation}\label{con2}
\begin{aligned}
\left|\int_{\mathbb{R}_+^3} \frac{P'\left(\rho^0\right)}{\rho^0} \left(  \bar{\mathbf{v}}\cdot \nabla \rho^0 \right) \bar{\rho} \mathrm{d}\mathbf{x}\right|
\lesssim \left\|\bar{\mathbf{v}}\right\|_{L^6}\left\|\partial {\rho}^0\right\|_{L^3}\left\|\bar{\rho}\right\|_{L^2}
\lesssim \delta_5\left\|\partial\bar{\mathbf{v}}\right\|_{L^2}^2+\left\|\partial {a}^0\right\|_{H^1}^2\left\|\bar{a}\right\|_{L^2}^2.
\end{aligned}
\end{equation}
For the fourth-order remainder term in \eqref{con1}, by Hölder's inequality, Sobolev's inequality $\|f\|_{L^\infty}\lesssim\|f\|_{H^2}$ and Lemma \ref{Le1}, we obtain
\begin{equation}\label{con2.5}
\begin{aligned}
\left|\int_{\mathbb{R}_+^3} F\left(\rho^0, \bar{\rho}\right) \bar{\rho}^4  \operatorname{div} \mathbf{v}^0 \mathrm{d}\mathbf{x}\right|
\lesssim \left\|\partial\mathbf{v}^0\right\|_{L^\infty}\left\|\bar{\rho}\right\|_{L^\infty}^2\left\|\bar{\rho}\right\|_{L^2}^2
\lesssim \left(\left\|\partial \bar{a}\right\|_{2}^2+\left\|\partial\mathbf{v}^0\right\|_{H^2}^2\right)\left\|\bar{a}\right\|_{L^2}^2.
\end{aligned}
\end{equation}
For the second and third order Taylor terms in \eqref{con1}, since we do not have the decay rate of \(\|\operatorname{div} \mathbf{v}^0\|_{L^\infty}\), we use the continuity equation to rewrite them in a tractable form:
\begin{align*}
&\int_{\mathbb{R}_+^3} \left( \frac{P''\left(\rho^0\right)}{2} \bar{\rho}^2+\frac{P'''\left(\rho^0\right)}{6} \bar{\rho}^3 \right) \operatorname{div} \mathbf{v}^0 \mathrm{d}\mathbf{x}\\
={}&-\int_{\mathbb{R}_+^3} \left( \frac{P''\left(\rho^0\right)}{2\rho^0} \bar{\rho}^2+\frac{P'''\left(\rho^0\right)}{6\rho^0} \bar{\rho}^3 \right) \left(\partial_t\rho^0+\mathbf{v}^0\cdot\nabla \rho^0\right) \mathrm{d}\mathbf{x}.
\end{align*}
To handle the term involving the time derivative of density, we introduce the auxiliary function
\begin{align*}
J_1\left(\rho^0\right) \overset{def}{=} \int_{1}^{\rho^0} \frac{P''\left(s\right)}{2s} \mathrm{d}s.
\end{align*}
Then, we have
\begin{align*}
-\int_{\mathbb{R}_+^3} \frac{P''\left(\rho^0\right)}{2\rho^0} \bar{\rho}^2\partial_t\rho^0 \mathrm{d}\mathbf{x}
=-\frac{\mathrm{d}}{\mathrm{d}t}\int_{\mathbb{R}_+^3} J_1\left(\rho^0\right) \bar{a}^2\mathrm{d}\mathbf{x}+2\int_{\mathbb{R}_+^3} J_1\left(\rho^0\right) \bar{\rho}\partial_t\bar{\rho}\mathrm{d}\mathbf{x}.
\end{align*}
According to the continuity equation and Leibniz's formula, we have
\begin{align*}
\int_{\mathbb{R}_+^3} J_1\left(\rho^0\right) \bar{\rho}\partial_t\bar{\rho}\mathrm{d}\mathbf{x}
={}&-\int_{\mathbb{R}_+^3} J_1\left(\rho^0\right) \bar{\rho}\operatorname{div}\left(\rho^0\bar{\mathbf{v}}+\bar{\rho}\mathbf{v}^{\varepsilon}\right)\mathrm{d}\mathbf{x}\\
={}&-\int_{\mathbb{R}_+^3} J_1\left(\rho^0\right) \bar{\rho}\left(\rho^0\operatorname{div}\bar{\mathbf{v}}+\bar{\mathbf{v}}\cdot\nabla \rho^0+\bar{\rho}\operatorname{div}\mathbf{v}^{\varepsilon}+\mathbf{v}^{\varepsilon}\cdot\nabla \bar{\rho}\right)\mathrm{d}\mathbf{x}.
\end{align*}
Below, we estimate these four terms separately. By Hölder's inequality, Sobolev's inequality $\|f\|_{L^\infty}\lesssim\|\partial f\|_{L^2}^{\frac{1}{2}}\|\partial^2 f\|_{L^2}^{\frac{1}{2}}$, Lemma \ref{Le1} and Young's inequality, we obtain
\begin{align*}
-\int_{\mathbb{R}_+^3} J_1\left(\rho^0\right) \bar{\rho}\cdot\rho^0\operatorname{div}\bar{\mathbf{v}}\mathrm{d}\mathbf{x}
\lesssim{}&\left\|J_1\left(\rho^0\right)\right\|_{L^\infty}\left\|\bar{\rho}\right\|_{L^2}\left\|\partial\bar{\mathbf{v}}\right\|_{L^2}\\
\lesssim{}& \delta_5\left\|\partial\bar{\mathbf{v}}\right\|_{L^2}^2+\left\|\partial {a}^0\right\|_{H^1}^2\left\|\bar{a}\right\|_{L^2}^2,
\end{align*}
\begin{align*}
-\int_{\mathbb{R}_+^3} J_1\left(\rho^0\right) \bar{\rho}\cdot\bar{\mathbf{v}}\cdot\nabla \rho^0\mathrm{d}\mathbf{x}
\lesssim{}& \left\|J_1\left(\rho^0\right)\right\|_{L^\infty}\left\|\bar{\rho}\right\|_{L^2}\left\|\bar{\mathbf{v}}\right\|_{L^2}\left\|\partial \rho^0\right\|_{L^\infty}\\
\lesssim{}& \left\|\partial {a}^0\right\|_{H^2}^2\left\|\left(\bar{a},\bar{\mathbf{v}}\right)\right\|_{L^2}^2,
\end{align*}
\begin{align*}
-\int_{\mathbb{R}_+^3} J_1\left(\rho^0\right) \bar{\rho}\cdot\bar{\rho}\operatorname{div}\mathbf{v}^{\varepsilon}\mathrm{d}\mathbf{x}
\lesssim{}& \left\|J_1\left(\rho^0\right)\right\|_{L^\infty}\left\|\bar{\rho}\right\|_{L^2}^2\left\|\operatorname{div}\mathbf{v}^{\varepsilon}\right\|_{L^\infty}\\
\lesssim{}& \left(\left\|\partial {a}^0\right\|_{H^1}^2+\left\|\partial \mathbf{v}^{\varepsilon}\right\|_{3}^2+\left\|\partial^2 \mathbf{v}^{\varepsilon}\right\|_{2}^2\right)\left\|\bar{a}\right\|_{L^2}^2,
\end{align*}
and
\begin{align*}
-\int_{\mathbb{R}_+^3} J_1\left(\rho^0\right) \bar{\rho}\cdot\mathbf{v}^{\varepsilon}\cdot\nabla \bar{\rho}\mathrm{d}\mathbf{x}
={}&-\frac{1}{2}\int_{\mathbb{R}_+^3} J_1\left(\rho^0\right) \mathbf{v}^{\varepsilon}\cdot\nabla \bar{\rho}^2\mathrm{d}\mathbf{x}\\
={}&\frac{1}{2}\int_{\mathbb{R}_+^3} \operatorname{div}\left(J_1\left(\rho^0\right) \mathbf{v}^{\varepsilon}\right) \bar{\rho}^2\mathrm{d}\mathbf{x}\\
={}&\frac{1}{2}\int_{\mathbb{R}_+^3} \left(J_1\left(\rho^0\right) \operatorname{div}\mathbf{v}^{\varepsilon}+J'_1\left(\rho^0\right) \mathbf{v}^{\varepsilon}\cdot\nabla\rho^0\right) \bar{\rho}^2\mathrm{d}\mathbf{x}\\
\lesssim{}& \left\|J_1\left(\rho^0\right)\right\|_{L^\infty}\left\|\operatorname{div}\mathbf{v}^{\varepsilon}\right\|_{L^\infty}\left\|\bar{\rho}\right\|_{L^2}^2+\left\|\partial\rho^0\right\|_{L^\infty}\left\|\mathbf{v}^{\varepsilon}\right\|_{L^\infty}\left\|\bar{\rho}\right\|_{L^2}^2\\
\lesssim{}&\left(\left\|\partial {a}^0\right\|_{H^2}^2+\left\|\partial \mathbf{v}^{\varepsilon}\right\|_{3}^2+\left\|\partial^2 \mathbf{v}^{\varepsilon}\right\|_{2}^2\right)\left\|\bar{a}\right\|_{L^2}^2.
\end{align*}
Combining these estimates, we get
\begin{align*}
\left|\int_{\mathbb{R}_+^3} J_1\left(\rho^0\right) \bar{\rho}\partial_t\bar{\rho}\mathrm{d}\mathbf{x}\right|
\lesssim \delta_5\left\|\partial\bar{\mathbf{v}}\right\|_{L^2}^2+\left(\left\|\partial {a}^0\right\|_{H^2}^2+\left\|\partial \mathbf{v}^{\varepsilon}\right\|_{3}^2+\left\|\partial^2 \mathbf{v}^{\varepsilon}\right\|_{2}^2\right)\left\|\left(\bar{a},\bar{\mathbf{v}}\right)\right\|_{L^2}^2.
\end{align*}
Thus, we conclude that the second-order Taylor term satisfies
\begin{align*}
&-\int_{\mathbb{R}_+^3} \frac{P''\left(\rho^0\right)}{2\rho^0} \bar{\rho}^2\partial_t\rho^0 \mathrm{d}\mathbf{x}\\
\geq{}&-\frac{\mathrm{d}}{\mathrm{d}t}\int_{\mathbb{R}_+^3} J_1\left(\rho^0\right) \bar{a}^2\mathrm{d}\mathbf{x}-C\bigg(\delta_5\left\|\partial\bar{\mathbf{v}}\right\|_{L^2}^2+\left(\left\|\partial {a}^0\right\|_{H^2}^2+\left\|\partial \mathbf{v}^{\varepsilon}\right\|_{3}^2+\left\|\partial^2 \mathbf{v}^{\varepsilon}\right\|_{2}^2\right)\left\|\left(\bar{a},\bar{\mathbf{v}}\right)\right\|_{L^2}^2\bigg).
\end{align*}
By the same argument, for the third-order Taylor term, we define
\begin{align*}
J_2\left(\rho^0\right) \overset{def}{=} \int_{1}^{\rho^0} \frac{P'''\left(s\right)}{6s} \mathrm{d}s,
\end{align*}
and obtain
\begin{align*}
&-\int_{\mathbb{R}_+^3} \frac{P'''\left(\rho^0\right)}{6\rho^0} \bar{\rho}^3\partial_t\rho^0 \mathrm{d}\mathbf{x}\\
\geq{}&-\frac{\mathrm{d}}{\mathrm{d}t}\int_{\mathbb{R}_+^3} J_2\left(\rho^0\right) \bar{a}^3\mathrm{d}\mathbf{x}-C\bigg(\delta_5\left\|\partial\bar{\mathbf{v}}\right\|_{L^2}^2+\left(\left\|\partial {a}^0\right\|_{H^2}^2+\left\|\partial \mathbf{v}^{\varepsilon}\right\|_{3}^2+\left\|\partial^2 \mathbf{v}^{\varepsilon}\right\|_{2}^2\right)\left\|\left(\bar{a},\bar{\mathbf{v}}\right)\right\|_{L^2}^2\bigg).
\end{align*}
Similarly, we can obtain
\begin{align*}
\left|\int_{\mathbb{R}_+^3} \left( \frac{P''\left(\rho^0\right)}{2\rho^0} \bar{\rho}^2+\frac{P'''\left(\rho^0\right)}{6\rho^0} \bar{\rho}^3 \right)\mathbf{v}^0\cdot\nabla \rho^0 \mathrm{d}\mathbf{x}\right|
\lesssim \left\|\partial \left({a}^0,\mathbf{v}^0\right)\right\|_{H^2}^2\left\|\bar{a}\right\|_{L^2}^2.
\end{align*}
Combining the the above estimates, we get
\begin{equation}\label{con3}
\begin{aligned}
&\int_{\mathbb{R}_+^3} \left( \frac{P''\left(\rho^0\right)}{2} \bar{\rho}^2+\frac{P'''\left(\rho^0\right)}{6} \bar{\rho}^3 \right) \operatorname{div} \mathbf{v}^0 \mathrm{d}\mathbf{x} \\
\geq{}&-\frac{\mathrm{d}}{\mathrm{d}t}\int_{\mathbb{R}_+^3} \left(J_1\left(\rho^0\right) \bar{a}^2+J_2\left(\rho^0\right) \bar{a}^3\right)\mathrm{d}\mathbf{x}\\
&-C\bigg(\delta_5\left\|\partial\bar{\mathbf{v}}\right\|_{L^2}^2+\left(\left\|\partial \left({a}^0,\mathbf{v}^0\right)\right\|_{H^2}^2+\left\|\partial \mathbf{v}^{\varepsilon}\right\|_{3}^2+\left\|\partial^2 \mathbf{v}^{\varepsilon}\right\|_{2}^2\right)\left\|\left(\bar{a},\bar{\mathbf{v}}\right)\right\|_{L^2}^2\bigg).
\end{aligned}
\end{equation}
Substituting \eqref{con2}--\eqref{con3} into \eqref{con1} yields the final energy estimate for the pressure gradient term
\begin{equation}\label{con4}
\begin{aligned}
&\int_{\mathbb{R}_+^3} \bar{\mathbf{v}} \cdot \nabla \bar{P} \mathrm{d}\mathbf{x} \\
\geq{}&\frac{\mathrm{d}}{\mathrm{d}t}\int_{\mathbb{R}_+^3} \left(\mathcal{G}\left(\rho^\varepsilon, \rho^0\right)- J_1\left(\rho^0\right) \bar{a}^2-J_2\left(\rho^0\right) \bar{a}^3\right)\mathrm{d}\mathbf{x} \\
&-C\bigg(\delta_5\left\|\partial\bar{\mathbf{v}}\right\|_{L^2}^2+\left(\left\|\partial \left({a}^0,\mathbf{v}^0\right)\right\|_{H^2}^2+\left\|\partial {a}^{\varepsilon}\right\|_{2}^2+\left\|\partial \mathbf{v}^{\varepsilon}\right\|_{3}^2+\left\|\partial^2 \mathbf{v}^{\varepsilon}\right\|_{2}^2\right)\left\|\left(\bar{a},\bar{\mathbf{v}}\right)\right\|_{L^2}^2\bigg).
\end{aligned}
\end{equation}
Substituting \eqref{con4} into \eqref{Con3} and taking $\delta_5$ sufficiently small gives
\begin{equation}\label{Con5}
  \begin{aligned}
&\frac{\mathrm{d}}{\mathrm{d}t}\Bigg(
\int_{\mathbb{R}_+^3} \rho^{0}\left|\bar{\mathbf{v}}\right|^2  \mathrm{d}\mathbf{x}
+ \left\|\bar{\mathbf{B}}\right\|_{L^2}^2
+ \int_{\mathbb{R}_+^3} \left(\mathcal{G}\left(\rho^\varepsilon, \rho^0\right)
- J_1\left(\rho^0\right) \bar{a}^2
- J_2\left(\rho^0\right) \bar{a}^3\right)\mathrm{d}\mathbf{x}
\Bigg) \\
&+ \left\|\left(\partial\bar{\mathbf{v}},\partial_h\bar{\mathbf{B}}\right)\right\|_{L^2}^2 \\
\lesssim{}& \Bigg(\left\|\partial a^\varepsilon\right\|_{2}^2+\left\|\partial \left(a^0,\mathbf{v}^0\right)\right\|_{H^2}^2
+ \left\|\partial \mathbf{v}^\varepsilon\right\|_{3}^2+ \left\|\partial^2 \mathbf{v}^\varepsilon\right\|_{2}^2+
\left\|\partial_t \mathbf{v}^\varepsilon\right\|_{H^1}^2 + \left\|\left(\mathbf{B}^{0},\mathbf{B}^\varepsilon\right)\right\|_{L^\infty}^2 \\
&
+ \left\|\partial \mathbf{B}^{0}\right\|_{L^2}\left\|\partial_h\partial \mathbf{B}^{0}\right\|_{L^2}
\Bigg)
\left\|\left(\bar{a},\bar{\mathbf{v}},\bar{\mathbf{B}}\right)\right\|_{L^2}^2+\left|\left(\varepsilon\partial_3\mathbf{B}^{\varepsilon},\partial_3\bar{\mathbf{B}}\right)\right|.
  \end{aligned}
\end{equation}
To simplify notation, we define the modified energy functional
\[
\bar{E}(t) \overset{def}{=}
\left(\int_{\mathbb{R}_+^3} \rho^{0}\left|\bar{\mathbf{v}}\right|^2  \mathrm{d}\mathbf{x}
+\left\|\bar{\mathbf{B}}\right\|_{L^2}^2
+ \int_{\mathbb{R}_+^3} \left(\mathcal{G}\left(\rho^\varepsilon, \rho^0\right)
- J_1\left(\rho^0\right) \bar{a}^2
- J_2\left(\rho^0\right) \bar{a}^3\right)\mathrm{d}\mathbf{x}\right)^{\frac{1}{2}},
\]
\begin{align*}
\bar{C}(t) \overset{def}{=}{}& \left\|\partial a^\varepsilon\right\|_{2}^2+\left\|\partial \left(a^0,\mathbf{v}^0\right)\right\|_{H^2}^2
+ \left\|\partial \mathbf{v}^\varepsilon\right\|_{3}^2+ \left\|\partial^2 \mathbf{v}^\varepsilon\right\|_{2}^2+
\left\|\partial_t \mathbf{v}^\varepsilon\right\|_{H^1}^2\\
&+\left\|\left(\mathbf{B}^{0},\mathbf{B}^\varepsilon\right)\right\|_{L^\infty}^2
+\left\|\partial \mathbf{B}^{0}\right\|_{L^2}\left\|\partial_h\partial \mathbf{B}^{0}\right\|_{L^2},
\end{align*}
and
\[
\bar{K}(t)\overset{def}{=}\left|\left(\varepsilon^{\frac{1}{2}}\partial_3\mathbf{B}^{\varepsilon},\partial_3\bar{\mathbf{B}}\right)\right|.
\]
It is direct to verify that
\begin{equation}\label{Con5.5}
  \bar{E}(t) \sim \left\|\left(\bar{a},\bar{\mathbf{v}},\bar{\mathbf{B}}\right)\right\|_{L^2}
\end{equation}
and
\begin{equation}\label{Con6}
  \int_0^t \left(\bar{C}(\tau)+\bar{K}(\tau)\right) \mathrm{d}\tau \leq C
\end{equation}
for some positive constant $C$ independent of $t$.
In what follows, we will establish the estimates for several representative terms in $\int_0^t \bar{C}(\tau)\mathrm{d}\tau$.

First, for the term $\int_{0}^{t}\left\|\mathbf{B}^\varepsilon\right\|_{L^\infty}^2\mathrm{d}\tau$, by Lemma \ref{Le1} and Hölder's inequality, we have
\begin{align*}
  \int_{0}^{t}\left\|\mathbf{B}^\varepsilon\right\|_{L^\infty}^2\mathrm{d}\tau\lesssim{}&\int_{0}^{t}\left\|\mathbf{B}^\varepsilon\right\|_{0}^{\frac{1}{4}}\left\|\partial_h\mathbf{B}^\varepsilon\right\|_{0}^{\frac{1}{2}}
\left\|\partial_3\mathbf{B}^\varepsilon\right\|_{0}^{\frac{1}{4}}\left\|\partial_h^2\mathbf{B}^\varepsilon\right\|_{0}^{\frac{1}{4}}\left\|\partial_h \partial_3\mathbf{B}^\varepsilon\right\|_{0}^{\frac{1}{2}}\left\|\partial_h^2\partial_3\mathbf{B}^\varepsilon\right\|_{0}^{\frac{1}{4}}\mathrm{d}\tau\\
\lesssim{}&\left(\int_{0}^{t}\left\|\mathbf{B}^\varepsilon\right\|_{0}^{\frac{1}{2}}\left\|\partial_h\mathbf{B}^\varepsilon\right\|_{0}
\left\|\partial_3\mathbf{B}^\varepsilon\right\|_{0}^{\frac{1}{2}}\mathrm{d}\tau\right)^{\frac{1}{2}}\cdot\left(\int_{0}^{t}\left(\left\|\partial_h\mathbf{B}^\varepsilon\right\|_{1}^{2}+\left\|\partial_h\partial_3\mathbf{B}^\varepsilon\right\|_{1}^{2}\right)\mathrm{d}\tau\right)^{\frac{1}{2}}\\
\lesssim{}&\left(\int_{0}^{t}\left\|\mathbf{B}^\varepsilon\right\|_{0}^{\frac{1}{2}}\left\|\partial_h\mathbf{B}^\varepsilon\right\|_{0}
\left\|\partial_3\mathbf{B}^\varepsilon\right\|_{0}^{\frac{1}{2}}\mathrm{d}\tau\right)^{\frac{1}{2}}\cdot D_{2}^{\varepsilon}(t)\\
\lesssim{}&\left(\int_{0}^{t}(1+\tau)^{-\frac{3}{2}+\frac{1}{2}\delta_2}\mathrm{d}\tau\right)^{\frac{1}{2}}\cdot D_{2}^{\varepsilon}(t)\\
\lesssim{}& D_{2}^{\varepsilon}(t).
\end{align*}
Next, we estimate the term $\int_{0}^{t}\left\|\partial \mathbf{B}^{0}\right\|_{L^2}\left\|\partial_h\partial \mathbf{B}^{0}\right\|_{L^2}\mathrm{d}\tau$. By the interpolation inequality and Hölder's inequality, it implies
\begin{align*}
  \int_{0}^{t}\left\|\partial \mathbf{B}^{0}\right\|_{L^2}\left\|\partial_h\partial \mathbf{B}^{0}\right\|_{L^2}\mathrm{d}\tau\lesssim{}&\int_{0}^{t}\left\|\partial \mathbf{B}^{0}\right\|_{L^2}\left\|\partial_h\mathbf{B}^{0}\right\|_{L^2}^{\frac{2}{3}}\left\|\partial_h\partial^3 \mathbf{B}^{0}\right\|_{L^2}^{\frac{1}{3}}\mathrm{d}\tau\\
\lesssim{}&\left(\int_{0}^{t}\left\|\partial \mathbf{B}^{0}\right\|_{L^2}^{\frac{6}{5}}\left\|\partial_h\mathbf{B}^{0}\right\|_{L^2}^{\frac{4}{5}}\mathrm{d}\tau\right)^{\frac{5}{6}}\cdot\left(\int_{0}^{t}\left\|\partial_h \mathbf{B}^{0}\right\|_{H^3}^2\mathrm{d}\tau\right)^{\frac{1}{6}}\\
\lesssim{}&\left(\int_{0}^{t}(1+\tau)^{-\frac{7}{5}}\mathrm{d}\tau\right)^{\frac{5}{6}}\cdot D_{3}^{0}(t)^{\frac{1}{3}}\\
\lesssim{}& D_{3}^{0}(t)^{\frac{1}{3}}.
\end{align*}
The remaining terms can be handled similarly and are omitted for brevity.

To prove that $\int_0^t \bar{K}(\tau) \mathrm{d}\tau \leq C$, we first establish some preliminary estimates.
Applying the standard basic energy estimate to system \eqref{eq1}, we obtain
\begin{equation}\label{Con3.5}
\begin{aligned}
  &\frac{1}{2}\frac{\mathrm{d}}{\mathrm{d}t}\int_{\mathbb{R}_+^3} \left(2g\left(\rho^{\varepsilon}\right) + \rho^{\varepsilon}\left|\mathbf{v}^{\varepsilon}\right|^2 + \left|\mathbf{B}^{\varepsilon}\right|^2\right) \mathrm{d}\mathbf{x} \\
  &+ \mu\left\|\nabla \mathbf{v}^{\varepsilon}\right\|_{L^2}^2 + (\mu+\lambda)\left\|\operatorname{div}\mathbf{v}^{\varepsilon}\right\|_{L^2}^2 + \left\|\nabla_h \mathbf{B}^{\varepsilon}\right\|_{L^2}^2 + \varepsilon\left\|\partial_3 \mathbf{B}^{\varepsilon}\right\|_{L^2}^2 = 0.
\end{aligned}
\end{equation}
Straightforward calculation gives
\[
\frac{1}{2}\int_{\mathbb{R}_+^3} \left(2g\left(\rho^{\varepsilon}\right) + \rho^{\varepsilon}\left|\mathbf{v}^{\varepsilon}\right|^2 + \left|\mathbf{B}^{\varepsilon}\right|^2\right) \mathrm{d}\mathbf{x} \sim \left\| \left(a^{\varepsilon}, \mathbf{v}^{\varepsilon}, \mathbf{B}^{\varepsilon}\right)(t)\right\|_{L^2}^{2}.
\]
Applying the weighted energy method, we obtain
\begin{equation}\label{Con3.6}
  \begin{aligned}
&\left((1+t)^{1-\delta_2}\left\| \left(a^{\varepsilon}, \mathbf{v}^{\varepsilon}, \mathbf{B}^{\varepsilon}\right)(t)\right\|_{L^2}^{2}\right) + \int_{0}^{t}(1+\tau)^{1-\delta_2}\left(\left\|\left(\partial \mathbf{v}^{\varepsilon},\partial_h \mathbf{B}^{\varepsilon}\right)\right\|_{L^2}^2 + \varepsilon\left\|\partial_3 \mathbf{B}^{\varepsilon}\right\|_{L^2}^2\right)\mathrm{d}\tau \\
\lesssim{}& \left\| \left(a^{\varepsilon}, \mathbf{v}^{\varepsilon}, \mathbf{B}^{\varepsilon}\right)(0)\right\|_{L^2}^{2} + \int_{0}^{t}(1+\tau)^{-\delta_2}\left\| \left(a^{\varepsilon}, \mathbf{v}^{\varepsilon}, \mathbf{B}^{\varepsilon}\right)(\tau)\right\|_{L^2}^{2}\mathrm{d}\tau \\
\lesssim{}& \left\| \left(a_0, \mathbf{v}_0, \mathbf{B}_0\right)\right\|_{L^2}^{2} + \int_{0}^{t}(1+\tau)^{-\delta_2}\left\|\left(a^{\varepsilon},\mathbf{v}^{\varepsilon},\mathbf{B}^{\varepsilon}\right)(\tau)\right\|_{L^2}^2\mathrm{d}\tau \\
\lesssim{}& C,
\end{aligned}
\end{equation}
where $C$ is a positive constant independent of $\varepsilon$ and $t$. The last inequality follows from the decay rate result of $\left\|\left(a^{\varepsilon},\mathbf{v}^{\varepsilon},\mathbf{B}^{\varepsilon}\right)\right\|_{L^2}$ established in Theorem \ref{Th3} $(i)$.

Using Hölder's inequality and \eqref{Con3.6}, we get
\begin{align*}
\int_{0}^{t}\bar{K}(\tau)\mathrm{d}\tau
\lesssim{}&\left(\varepsilon\int_{0}^{t}(1+\tau)^{1-\delta_2}\left\|\partial_3\mathbf{B}^{\varepsilon}\right\|_{L^2}^{2}\mathrm{d}\tau\right)^{\frac{1}{2}}\cdot\left(\int_{0}^{t}(1+\tau)^{-1+\delta_2}\left\|\partial_3\bar{\mathbf{B}}\right\|_{L^2}^2\mathrm{d}\tau\right)^{\frac{1}{2}}\\
\lesssim{}&\left(\int_{0}^{t}(1+\tau)^{-2+3\delta_2}\mathrm{d}\tau\right)^{\frac{1}{2}}\\
\lesssim{}& C.
\end{align*}
This completes the proof of \eqref{Con6}.

Then, \eqref{Con5} implies
\begin{equation}\label{Con8}
\frac{\mathrm{d}}{\mathrm{d}t}\bar{E}(t)^2
\lesssim \bar{C}(t) \bar{E}(t)^2+\varepsilon^{\frac{1}{2}}\bar{K}(t).
\end{equation}
Applying Gronwall's inequality to \eqref{Con8} with $\bar{E}(0) = 0$ and \eqref{Con6} gives
\[
\bar{E}(t)^2 \lesssim  \varepsilon^{\frac{1}{2}}\int_0^t  \bar{K}(s) e^{\int_s^t \bar{C}(r)\mathrm{d}r} \mathrm{d}s\lesssim
 \varepsilon^{\frac{1}{2}}.
\]
Combining this with the equivalence \eqref{Con5.5} yields the global convergence rate
\[
\sup_{0\leq \tau \leq t}\left\|\left(\bar{a},\bar{\mathbf{v}},\bar{\mathbf{B}}\right)(\tau)\right\|_{ L^2\left(\mathbb{R}_+^3\right)} \lesssim \varepsilon^{\frac{1}{4}}.
\]
This completes the proof of Theorem \ref{Th9}.

\begin{sloppypar}
\noindent\textbf{Acknowledgements~}
Jiahong Wu was partially supported by the National Science Foundation of the United States
(DMS 2104682, DMS 2309748).
Feng Xie was supported by National Natural Science Foundation of China No.12271359.
\end{sloppypar}
	
\begin{appendices}
\section{Some useful inequalities}	\label{ApA}

First, we introduce the anisotropic Sobolev inequalities used frequently in Sections \ref{global-estimate}, \ref{asymptotic-behavior} and \ref{Convergence rate of global solution}.

\begin{lemm}\label{Le1}
For the index numbers $i,j,k \in \{ 1,2,3\}$, the following inequalities hold when the right-hand side are all bounded:
\begin{equation}\label{a1}
\int_{\mathbb{R}_+^3}\left|fgh\right| \mathrm{d}\mathbf{x}
\lesssim
\left\| f \right\|_{0}^{\frac{1}{2}}
\left\| \partial_1 f \right\|_{0}^{\frac{1}{2}}
\left\| g \right\|_{0}^{\frac{1}{2}}
\left\| \partial_2 g \right\|_{0}^{\frac{1}{2}}
\left\| h \right\|_{0}^{\frac{1}{2}}
\left\| \partial_3 h \right\|_{0}^{\frac{1}{2}},
\end{equation}
\begin{equation}\label{a2}
			\int_{\mathbb{R}_+^3}
			\left|fgh\right| \mathrm{d}\mathbf{x}
			\lesssim
			\left\| f \right\|_{0}^{\frac{1}{4}}
			\left\| \partial_i f \right\|_{0}^{\frac{1}{4}}
			\left\| \partial_j f \right\|_{0}^{\frac{1}{4}}
			\left\| \partial_{ij} f \right\|_{0}^{\frac{1}{4}}
			\left\| g \right\|_{0}^{\frac{1}{2}}
			\left\| \partial_k g \right\|_{0}^{\frac{1}{2}}
			\left\| h \right\|_{0},
\end{equation}
\begin{equation}\label{a3}
\begin{aligned}
  \int_{\mathbb{R}_+^3}\left|fghv\right|\mathrm{d}\mathbf{x}\lesssim{}&\left\|f\right\|_{0}^{\frac{1}{4}}\left\|\partial_{i}f\right\|_{0}^{\frac{1}{4}}\left\|\partial_{j}f\right\|_{0}^{\frac{1}{4}}\left\|\partial_{ij}f\right\|_{0}^{\frac{1}{4}}\left\|g\right\|_{0}^{\frac{1}{2}}\left\|\partial_{k}g\right\|_{0}^{\frac{1}{2}}\\
  {}&\left\|h\right\|_{0}^{\frac{1}{4}}\left\|\partial_{i}h\right\|_{0}^{\frac{1}{4}}\left\|\partial_{j}h\right\|_{0}^{\frac{1}{4}}\left\|\partial_{ij}h\right\|_{0}^{\frac{1}{4}}\left\|v\right\|_{0}^{\frac{1}{2}}\left\|\partial_{k}v\right\|_{0}^{\frac{1}{2}},
\end{aligned}
\end{equation}
\begin{equation}\label{a4}
  \left\|fg\right\|_{0}\lesssim\left\|f\right\|_{0}^{\frac{1}{4}}\left\|\partial_{i}f\right\|_{0}^{\frac{1}{4}}\left\|\partial_{j}f\right\|_{0}^{\frac{1}{4}}\left\|\partial_{ij}f\right\|_{0}^{\frac{1}{4}}\left\|g\right\|_{0}^{\frac{1}{2}}\left\|\partial_{k}g\right\|_{0}^{\frac{1}{2}},
\end{equation}
\begin{equation}\label{a5}
  \left\|f\right\|_{L^\infty\left(\mathbb{R}_+^3\right)}\lesssim\left\|f\right\|_{0}^{\frac{1}{8}}\left\|\partial_{1}f\right\|_{0}^{\frac{1}{8}}\left\|\partial_{2}f\right\|_{0}^{\frac{1}{8}}\left\|\partial_{12}f\right\|_{0}^{\frac{1}{8}}
  \left\|\partial_{3}f\right\|_{0}^{\frac{1}{8}}\left\|\partial_{13}f\right\|_{0}^{\frac{1}{8}}\left\|\partial_{23}f\right\|_{0}^{\frac{1}{8}}\left\|\partial_{123}f\right\|_{0}^{\frac{1}{8}}
  \lesssim\left\| f \right\|_{2}^{\frac{1}{2}}\left\| \partial_3 f \right\|_{2}^{\frac{1}{2}},
\end{equation}
\begin{equation}\label{a6}
  \left\| f g\right\|_{L_{x_{1} x_{2}}^{1}L_{x_{3}}^{2}} \lesssim \left\| f\right\|_{0}^{\frac{1}{2}}\left\| \partial_{3} f\right\|_{0}^{\frac{1}{2}}\left\| g\right\|_{0},
\end{equation}
\begin{equation}\label{a7}
  \left\| f g\right\|_{L_{x_{1}}^{1} L_{x_{2} x_{3}}^{2}} \lesssim \left\| f\right\|_{L^{2}}^{\frac{1}{4}}\left\| \partial_{2} f\right\|_{0}^{\frac{1}{4}}\left\| \partial_{3} f\right\|_{0}^{\frac{1}{4}}\left\| \partial_{23} f\right\|_{0}^{\frac{1}{4}}\left\| g\right\|_{0},
\end{equation}
\begin{equation}\label{a7.5}
  \left\| f g\right\|_{L_{x_{1}}^{1} L_{x_{2} x_{3}}^{2}} \lesssim \left\| f\right\|_{L^{2}}^{\frac{1}{2}}\left\| \partial_{2} f\right\|_{0}^{\frac{1}{2}}\left\| g\right\|_{0}^{\frac{1}{2}}\left\| \partial_{3} g\right\|_{0}^{\frac{1}{2}},
\end{equation}
\begin{equation}\label{a8}
\left\|Z_3^k f\right\|_{0}
\lesssim \left\|f\right\|_{0}+\left\|f\right\|_{0}^{1-\frac{k}{n}}\left\|Z_3^n f\right\|_{0}^{\frac{k}{n}}.
\end{equation}
\end{lemm}
\begin{proof}
See \cite{9a}, \cite{59}  and \cite{31c} for the proofs.
\end{proof}
We next state the following Sobolev embedding inequality.
\begin{lemm}\label{Le2}
Let \(2\leq p\leq 6\) and assume $f\in H^1\left(\mathbb{R}_+^3\right)$, it holds that
\begin{equation*}
\left\| f\right\|_{L^p\left(\mathbb{R}_+^3\right)} \lesssim\left\| f\right\|_{L^2\left(\mathbb{R}_+^3\right)}^{1-d}\left\| \partial f\right\|_{L^2\left(\mathbb{R}_+^3\right)}^{d},
\end{equation*}
where \(d = \frac{3}{2} - \frac{3}{p}\).
\end{lemm}
\begin{proof}
The proof can be found in \cite{AL}.
\end{proof}
Finally, we need the following elementary inequalities providing upper bounds for a convolution type integral whose proofs are straightforward.
\begin{lemm}\label{Le3}
Assume \(0 < s_1 \leq s_2\), it holds that
\begin{align*}
\int_{0}^{t} (1 + t - \tau)^{-s_1} (1 + \tau)^{-s_2} \,\mathrm{d}\tau
&= \int_{0}^{t} (1 + t - \tau)^{-s_2} (1 + \tau)^{-s_1} \,\mathrm{d}\tau \\
&\lesssim
\begin{cases}
(1 + t)^{-s_1}, & \text{if } s_2 > 1, \\
(1 + t)^{-s_1} \ln(1 + t), & \text{if } s_2 = 1, \\
(1 + t)^{1 - s_1 - s_2}, & \text{if } s_2 < 1.
\end{cases}
\end{align*}
\end{lemm}
\begin{lemm}\label{Le4}
Assume \(0<s_{1}<1\) and \(s_{2}>0\), it holds that
\[
\int_{0}^{t}\left(t-\tau\right)^{-s_{1}}\left(1+\tau\right)^{-s_{2}} \mathrm{d} \tau \lesssim
\begin{cases}
\left(1+t\right)^{-s_{1}}, & \text{if }  s_{2}>1, \\
\left(1+t\right)^{-s_{1}} \ln \left(1+t\right), & \text{if }  s_{2}=1, \\
\left(1+t\right)^{1-s_{1}-s_{2}}, & \text{if }  s_{2}<1 .
\end{cases}
\]
\end{lemm}
\end{appendices}
		\phantomsection
		\addcontentsline{toc}{section}{\refname}
		\footnotesize
		\addtolength{\itemsep}{-1.0em}
		\setlength{\itemsep}{-4pt}
		\bibliographystyle{abbrv}

\end{document}